\documentclass[11pt, leqno]{article}

\usepackage{verbatim}
\usepackage{hyperref}
\usepackage{amsmath}
\usepackage{enumerate, amsthm}
\usepackage{mathrsfs}
\usepackage{rotating}
\usepackage{xcolor}
\usepackage[margin=1.1 in]{geometry}
\usepackage[margin={1.5cm, 1.5cm},font=small, labelsep=endash]{caption}
\usepackage{mathtools}
\usepackage{tikz}
\usetikzlibrary{patterns}
\usetikzlibrary{arrows}
\usepackage{tikz-dimline}
\usepackage{ifthen}

\usepackage{bm}

      \usepackage{amssymb}
      \usepackage{amsmath}
      \usepackage{centernot}

      \numberwithin{equation}{section}
      \theoremstyle{plain}
      \newtheorem{theorem}{Theorem}[section]

            \newtheorem{thm}[theorem]{Theorem}
      \newtheorem{lemma}[theorem]{Lemma}
      \newtheorem{lem}[theorem]{Lemma}
      \newtheorem{corollary}[theorem]{Corollary}
      
      \newtheorem{prop}[theorem]{Proposition}

      \theoremstyle{definition}
      \newtheorem{defn}[theorem]{Definition}

      \theoremstyle{remark}
      \newtheorem{remark}[theorem]{Remark}

\renewcommand{\P}{\mathbb P}
\newcommand{\R}{\mathbb R}
\newcommand{\E}{\mathbb E}
\newcommand{\e}{\mathrm{e}}

\newcommand{\Z}{\mathbb Z}
\newcommand{\piv}{\mathrm{Piv}}

\newcommand{\N}{\mathbb N}

\newcommand{\lr}[4]{#3\xleftrightarrow[#1]{#2} #4}
     \newcommand{\nlr}[4]{#3\mathrel{\mathop{\centernot\longleftrightarrow}_{#1}^{#2}} #4}

\usepackage{constants}
\usepackage{titletoc}
\usepackage{stackrel}
\usepackage{etex}
\usepackage{enumitem}
\usepackage{enumerate}
\usepackage{vwcol}
\usepackage{stackengine}

\usepackage{hyperref}
\hypersetup{linktocpage,
  colorlinks   = true,
  urlcolor     = blue,
  linkcolor    = blue,
  citecolor   = black
}

\newcommand\couprad{10} 
\newcommand\rangeofdep{40}
\newcommand\pivrange{\the\numexpr\couprad + \rangeofdep\relax}

\newconstantfamily{c}{symbol=c}
\makeatother

\usepackage{titlesec}
\titleformat{\subsection}[runin]{\normalfont\bfseries}{\thesubsection.}{.5em}{}[.]\titlespacing{\subsection}{0pt}{2ex plus .1ex minus .2ex}{.8em}
\titleformat{\subsubsection}[runin]{\normalfont\bfseries}{\thesubsubsection.}{.5em}{}[.]
\titlespacing{\subsubsection}{0pt}{2ex plus .1ex minus .2ex}{.8em}

\usepackage{multicol}
\usepackage{parcolumns}

\usepackage{float}

\title{{\textbf{\normalsize{PHASE TRANSITION FOR THE VACANT SET OF RANDOM WALK \\ AND RANDOM INTERLACEMENTS}}}}
\date{}
\begin{document}
\thispagestyle{empty}
\maketitle
\vspace{0.1cm}
\begin{center}
\vspace{-1.7cm}
Hugo Duminil-Copin$^{1,2}$, Subhajit Goswami$^3$, Pierre-Fran\c{c}ois Rodriguez$^4$,\\[1em] Franco Severo$^{5}$ and Augusto Teixeira$^6$

\end{center}
\vspace{0.1cm}
\begin{abstract}
We consider the set of points visited by the random walk on the discrete torus $(\mathbb{Z}/N\mathbb{Z})^d$, for $d \geq 3$, at times of order $uN^d$, for a parameter $u>0$  in the large-$N$ limit. We prove that the vacant set left by the walk undergoes a phase transition across a non-degenerate critical value $u_* = u_*(d)$, as follows. For all $u< u_*$, the vacant set contains a giant connected component with high probability, which has a non-vanishing asymptotic density and satisfies a certain local uniqueness property. In stark contrast, for all $u> u_*$ the vacant set scatters into tiny connected components. Our results further imply that the threshold $u_*$ precisely equals the critical value, introduced by Sznitman in \textit{Ann.~Math., 171 (2010), 2039--2087}, which characterizes the percolation transition of the corresponding local limit, the vacant set of random interlacements on $\Z^d$. Our findings also yield the analogous infinite-volume result, i.e.~the long purported equality of three critical parameters $\bar u$, $u_*$ and $u_{**}$ naturally associated to the vacant set of random interlacements.
\end{abstract}

\vspace{2.3cm}

\begin{flushleft}
\thispagestyle{empty}
\vspace{0.1cm}
{\footnotesize
\noindent\rule{6cm}{0.35pt} \hfill {\normalsize August 2023} \\[2em]

\begin{multicols}{2}
\small
$^1$Institut des Hautes \'Etudes Scientifiques
 \\   35, route de Chartres \\ 91440 -- Bures-sur-Yvette, France.\\ \url{duminil@ihes.fr}\\[2em]
 
$^2$Universit\'e de Gen\`eve\\
 Section de Math\'ematiques\\
 2-4 rue du Li\`evre \\
1211 Gen\`eve 4, Switzerland.\\
\url{hugo.duminil@unige.ch} \\[2em]

 $^3$School of Mathematics\\
 Tata Institute of Fundamental Research\\
 1, Homi Bhabha Road\\
 Colaba, Mumbai 400005, India. \\ \url{goswami@math.tifr.res.in} \columnbreak
 
\hfill$^4$Imperial College London\\
\hfill Department of Mathematics\\
\hfill London SW7 2AZ \\
\hfill United Kingdom.\\
\hfill \url{p.rodriguez@imperial.ac.uk}  \\[2em]

\hfill$^5$ETH Zurich\\
\hfill Department of Mathematics\\
\hfill R\"amistrasse 101\\
\hfill 8092 Zurich, Switzerland.\\
\hfill \url{franco.severo@math.ethz.ch}\\ [2em]

\hfill $^6$Instituto de Matem\'atica Pura e Aplicada\\
\hfill  Estrada dona Castorina, 110\\
\hfill  22460-320, Rio de Janeiro - RJ, Brazil.\\
\hfill  \url{augusto@impa.br}

\end{multicols}
}
\end{flushleft}

\newpage

{\setcounter{tocdepth}{2} \tableofcontents}
\thispagestyle{empty}
\newpage
\setcounter{page}{1}

\section{Introduction}
\label{Sec:intro}

Geometric properties of random walks display a rich phenomenology. To mention but a few examples, in planar setups one knows for instance that the outer boundary of a Brownian motion has Hausdorff dimension $\frac43$, see \cite{zbMATH01730749,zbMATH01930784}, and that several natural `observables' (its occupation measure, thick points, uncovered set,...) exhibit a (multi-)fractal structure \cite{zbMATH01697476, zbMATH02157793}. In higher dimensions various covering and fragmentation problems relate to an intriguing percolation phase transition \cite{zbMATH05054008,benjamini2008giant,MR2680403}, which is the subject of the present work.

Let $ Z= (Z_n)_{n \geq 0}$ denote the symmetric random walk on the $d$-dimensional discrete torus $\mathbb{T}= (\Z / N\Z)^d$ of side length $N \geq1$, for $d \geq3$, and $P$ be its law when started from the uniform distribution on $\mathbb{T}$, which is stationary for $Z$. If $u$ is an arbitrary positive number, one knows that the walk $Z$ has a probability to hit a given point of $\mathbb{T}$ up to time $uN^d$ which is bounded away from $0$ and $1$ uniformly in $N$. It is then natural to investigate connectivity properties of the vacant set of the walk at these time scales, i.e.~to study
\begin{equation}
\label{e:V-RW}
\mathcal{V}_N^u \stackrel{\text{def.}}{=} \mathbb{T} \setminus Z_{[0,uN^d]}, \ u >0,
\end{equation}
where $Z_{[0, t]}= \{ x\in \mathbb{T} : \text{for some } 0\leq n \leq t, \,  Z_n=x\}$, $ t\geq  0$. Originating in work of Benjamini and Sznitman~\cite{benjamini2008giant}, who studied the set $\mathcal{V}_N^u$ for small $u$ and exhibited a giant connected component (i.e.,~having a positive asymptotic density as $N \to \infty$), it has long been conjectured that $\mathcal{V}_N^u$ undergoes an abrupt phase transition across a non-trivial value $u_c=u_c(d) \in (0,\infty)$, independent of $N$, above which $\mathcal{V}_N^u$ scatters into tiny pieces with high probability as $N \to \infty$. The same fate is expected for all but the largest cluster in the regime $u<u_c$.

In a landmark paper \cite{MR2680403}, which subsequently spurred a lot of activity, Sznitman introduced an infinite-volume version $\mathcal{V}^u$ of the set $\mathcal{V}_N^u$ in \eqref{e:V-RW}, the so-called vacant set of random interlacements at level $u$, and along with it a compelling candidate for $u_c$, characterized entirely in terms of the infinite model. The set $\mathcal{V}^u$ is a random subset of $\Z^d$, decreasing in $u$ (as is $\mathcal{V}_N^u$ in \eqref{e:V-RW}), its law is invariant under lattice symmetries and characterised by the property that
\begin{equation}
\label{eq:I_u_intro}
\P[\mathcal{V}^u \supset K] = \exp\{ - u \text{cap}(K) \},
\end{equation}
for all finite $K \subset \Z^d $, where $\text{cap}(K)$ refers to the capacity of $K$; see \eqref{eq:cap_K}. Informally, the set $\mathcal{V}^u$ is constructed as follows. One introduces under a measure $\P$ a Poisson point process on $ W^* \times \R_+$, the space of labeled bi-infinite transient $\Z^d$-valued trajectories modulo time-shift. We refer to Section~\ref{subsec:RI} for its precise definition; see in particular \eqref{e:RI-intensity} regarding its intensity measure. 
The interlacement set $ \mathcal{I}^u$ at level $u$ is obtained as the trace of all trajectories in this Poisson cloud with label at most $u$ and $ \mathcal{V}^u = \Z^d\setminus \mathcal{I}^u$ is defined as its complement. Thus $u$ acts as an intensity parameter: the larger $u$ is, the more trajectories enter the picture. Loosely speaking, when viewed from a point $x$, the different interlacement trajectories present near $x$ correspond to excursions of $Z$ in the neighborhood of its projection on the torus.

As a matter of fact, one knows that $\mathcal{V}^u$ is the local limit in law of $\mathcal{V}_N^u$ as $N \to \infty$. That is, for finite $K\subset\Z^d$, with $\pi: \Z^d \to \mathbb{T} $ denoting the canonical projection, one has that 
\begin{equation}
\label{eq:loc-limit}\lim_N P[\mathcal{V}_N^u \supset \pi(K)]=\P[\mathcal{V}^u \supset K],
\end{equation}
see \cite[Chap.~3]{RIbook2014} or \cite{MR2386070}; in fact, rather more is true, cf.~\eqref{eq:coup-RW-RI} and refs.~below. The set $\mathcal{V}^u$ undergoes a non-trivial percolation phase transition: defining for $u,R \geq 0$ the function
\begin{equation}\label{eq:order}
\theta_R(u)= \P[\lr{}{ \mathcal{V}^u}{0}{\partial B_R}], 
\end{equation}
where $B_R=([-R, R]\cap \Z)^d$, $\partial B_R$ is its inner (vertex) boundary, see Section~\ref{s:not} for notation, and the event in question refers to a nearest-neighbor path in $\mathcal{V}^u$ connecting $0$ and $\partial B_R$, one introduces the critical parameter for percolation of $\mathcal{V}^u$ as
\begin{equation}
\label{eq:u_*}
 u_*= u_*(d) \stackrel{\text{def.}}{=} \sup \big\{u >0 :  \theta_{\infty}(u) > 0 \big\}\,,
\end{equation}
where $ \theta_{\infty} = \inf_{R} \theta_R$. As shown in successive works \cite{MR2680403,MR2512613}, see also \cite{10.1214/ECP.v20-3734}, this phase transition is non-trivial, i.e.
\begin{equation}
\label{e:u_*}
u_*(d) \in (0,\infty), \text{ for all $d \geq 3$.}
\end{equation}

\subsection{Main results}\label{subsec_thms}
It has long been believed that the conjectured phase transition for the vacant set $\mathcal{V}_N^u$ of the walk, with the presumed features outlined below~\eqref{e:V-RW}, occurs across the value $u_c(d)=u_*(d)$ given by \eqref{eq:u_*}. Our first main result confirms these predictions.

For $x \in \mathbb{T}$, let $\mathcal{C}^u(x) \subset \mathbb{T}$ denote the connected component (cluster) of $x$ in $\mathcal{V}_N^u$. The maximal cluster size is defined as $|\mathcal{C}^u_{\textnormal{max}}| =\sup_x |\mathcal{C}^u(x)|$, where $x$ ranges over $\mathbb{T}$ and $|K|$ denotes the cardinality of the set $K$. Let $\mathcal{C}^u_{\geqslant t}$ denote the collection of clusters in $\mathcal{V}_N^u$ of diameter (with respect to the graph distance on $\mathbb{T}$) larger than or equal to $t$.

\begin{thm}[$d \geq 3$]
	\label{t:main1} With $u_*$ as in \eqref{eq:u_*}, the following holds.
	\begin{enumerate}
		\item[i)] For all $u > u_*$, there exist $c= c(d)$ and $C, \lambda \in (0,\infty)$ depending only on $d$ and $u$ such that, with $t_N= \log^{\lambda} N$, one has
		\begin{equation}
			\label{t:main1-sub}
		P \big[ | \mathcal{C}^u_{\textnormal{max}} | \geq t \big]   \leq  Ce^{-({t}/{t_N})^c},  \text{ for all }  t , N \geq 1.
		\end{equation}
		\item[ii)] For all $u < u_*$ and $\varepsilon>0$ there exist $c, C, \lambda$ as above but possibly depending on $\varepsilon$ such that
		\begin{equation}
			\label{t:main1-super}
			\begin{split} 
				& \lim_N P\left[   \frac{|\mathcal{C}^u_{\textnormal{max}}| }{N^d} \geq  \theta_{\infty}(u) - \varepsilon \right]=1, \text{ and }
				\\
				& P\bigg[ \nlr{}{ \mathcal{V}_N^{u(1-\varepsilon)}}{\mathcal C}{\mathcal C'} \text{ for some } \mathcal C, \mathcal C' \in \mathcal{C}^u_{\geqslant t}\bigg]  \leq  Ce^{-({t}/{t_N})^c},  \text{ for all }  t,  N \geq 1.
			\end{split}
		\end{equation}
	\end{enumerate}
\end{thm}
In words, Theorem~\ref{t:main1} asserts that the vacant set $\mathcal{V}_N^u$ of the walk undergoes a percolation phase transition and that $u_*$ defined by \eqref{eq:u_*} describes the corresponding critical point. For $u< u_*$, the vacant set $\mathcal{V}_N^u$ contains a macroscopic (`giant') component with asymptotic density bounded from below by $\theta_{\infty}(u)$ (which is $>0$ in view of \eqref{eq:u_*}) in probability. On the contrary, throughout the subcritical regime $u> u_*$, this giant component disappears completely and all clusters of $\mathcal{V}_N^u$ are poly-logarithmically small in $N$ with probability tending (rapidly) to $1$ as $N \to \infty$. Moreover, as $u$ drops below the threshold $u_*$, not only does a giant component emerge, but upon exceeding diameter $ t \geq t_N$, any two clusters in $\mathcal{V}_N^u$ are actually part of the same cluster of $\mathcal{V}_N^{u(1-\varepsilon)}$ with high probability. 

The picture emanating from Theorem~\ref{t:main1} is reminiscent of various classical results, the best-known of which is perhaps the famed Erd\"os-R\'enyi random graph \cite{erdos59a}, which in a loose sense corresponds (locally) to letting $d \to \infty$ in the above setup. `Low-dimensional' critical phenomena however, such as the one studied here, bear very significant differences. We defer a thorough discussion of these matters to~\S\ref{subsec:applications}. One overarching aspect of the problem is the long-range dependence induced by the random walk $Z$. Combining \eqref{eq:I_u_intro}, \eqref{eq:loc-limit} and \cite[(1.68)]{MR2680403}, one knows for instance that for all $x,y \in \mathbb{Z}^d$, as $N \to \infty$,
\begin{equation}\label{eq:LR}
\text{Cov}_P \big(1_{\{\pi( x) \in \mathcal{V}_N^u\}}, 1_{\{ \pi(y) \in \mathcal{V}_N^u\}} \big) \sim c (d,u) |x-y|^{2-d},
\end{equation}
with $|\cdot|$ denoting the standard Euclidean distance and where $\sim$ means that the ratio of both sides tends to $1$ in the limit. The (strong) polynomial correlations implied by \eqref{eq:LR}, along with the inherent (local) transience of the problem, forcing $\mathbb{T}$ to be at least three-dimensional, pose a very significant challenge, cf.~\S\ref{subsec-pf-outline}. The transition established in Theorem~\ref{t:main1} is a benchmark example with these features, in what is arguably the simplest possible framework.

\bigskip

Our second main result concerns the vacant set $\mathcal{V}^u$ of random interlacements~defined by~\eqref{eq:I_u_intro}, which offers an infinite volume version of the problem (cf.~\eqref{eq:loc-limit}) and inevitably inherits the polynomial correlations of \eqref{eq:LR}. It expresses a sharpness result for the sets $\mathcal{V}= (\mathcal{V}^u)_{u> 0}$, thus addressing an important open problem, see, e.g.,~\cite[Remark 4.4,3)]{MR2680403}, \cite[Remark 4.2]{MR2891880} or \cite[Remark 3.1]{MR3602841}. In order to state a meaningful theorem, we introduce two events,
\begin{align}
\text{Exist}(R,u)&=\left\{\begin{array}{c}  \text{there exists a cluster in }\\ \text{$\mathcal V^u \cap B_R$ with diameter at least  $\frac R5$}\end{array}\right\}, \label{eq:EXIST} \\[0.3cm]
\text{Unique}(R,u,v)&= \left\{\begin{array}{c}\text{any two clusters in $ \mathcal V^u \cap B_{R}$ having diameter at}\\ \text{least $\frac R{10}$ are connected to each other in $ \mathcal V^v \cap B_{2R}$ } \end{array}\right\}.\label{eq:UNIQUE1}
\end{align}
The events defined by \eqref{eq:EXIST}-\eqref{eq:UNIQUE1}, which are similar in spirit to those employed in \cite{AntPis96} in a different context, pin down a subset of the percolative phase that is very robust, in the sense that one has strong quantitative control on the existence and uniqueness of large local clusters, as discussed further below in \S\ref{subsec:crit-par}. Recall $u_*=u_*(d)$ from \eqref{eq:u_*} and \eqref{e:u_*}.

\begin{thm}[$d \geq 3$]  \label{thm:hh1} $\quad$
 \begin{itemize}
 \item[i)] For all $u>u_*$, there exist $c=c(d)$ and $C= C(u,d)$ in $(0,\infty)$ such that for all $R\ge1$,
 \begin{equation}\label{eq:subcrit}
\P[\lr{}{\mathcal{V}^u}{0}{\partial B_{R}} ] \le Ce^{-R^c} .
 \end{equation}
 \item[ii)] For all $0<v<u<u_*$, $\mathcal{V}$ {\em strongly percolates} at levels $u,v$, in the sense that there exist constants $c=c(d)$ and $C= C(u,v,d)$ in  $(0,\infty)$  such that for every $R\ge1$,
\begin{align}
\label{eq:barhEXIST}
\P[\mathrm{Exist}(R,u)]&\ge 1-C{ e}^{-R^c},\\
\P\left[\mathrm{Unique}(R,u,v)\right]&\geq 1- C{ e}^{-R^c}.
\label{eq:barhUNIQUE}
\end{align}
 \end{itemize}
\end{thm}

\subsection{Applications} \label{subsec:applications} We first discuss various consequences of Theorems~\ref{t:main1} and~\ref{thm:hh1} and their links to existing literature, and mention a few open questions. 
As detailed at the beginning of Section~\ref{Sec:intro}, Theorem~\ref{t:main1} gives strong answers to the conjectured phase transition for the vacant set $\mathcal{V}_N^u$ of the random walk.
Prior results concerning the different phases of $\mathcal{V}_N^u$ are of one of two types: either i) valid in (a-priori) perturbative regimes, see e.g.~\cite{benjamini2008giant} 
and \cite[Theorems 1.3-1.4]{MR2838338} for small $u$, and \cite[Theorem 1.2]{MR2838338} for large $u$; or ii) `$0-1$-law' type results, valid without such an assumption but non-quantitative, see e.g.~\cite[(1.8)]{PopTeix} and \cite[Theorem 1.1]{MR3563197}. In fact results of type ii) are immediate consequences of (`hard') coupling results such as \eqref{eq:coup-RW-RI} below (and its predecessors) but use otherwise rather `soft' properties of the interlacement; namely, whether $\theta_R$ in \eqref{eq:order} vanishes in the limit $R\to \infty$ or not. All these results are subsumed by Theorem~\ref{t:main1} with the exception of \eqref{t:main1-super}, which can be strengthened when $u \ll 1$, essentially by removing the sprinkling $\varepsilon$; see \cite{MR2838338, Tei11}. This is related to the possible strengthening of the notion of ``strongly percolating'' from \eqref{eq:barhEXIST}-\eqref{eq:barhUNIQUE}, see \eqref{eq:baru3} and the subsequent discussion in \S\ref{subsec:crit-par}. 

We refer to \cite{zbMATH06148893,zbMATH06220715} for results akin to Theorem~\ref{t:main1} on random regular graphs and expanders of logarithmic girth; see also~\cite{cerny2023critical,conchonkerjan2023speed,zbMATH07225519, zbMATH07343338, 10.1214/23-EJP920} for recent related results concerning excursion sets of the Gaussian free field. These `mean-field' type results include quantitative information in $u$ and $n$ (the number of vertices) on the size of the critical window, which has width $ n^{-1/3}$. Any result quantifying the zero-one law describing the transition on $\mathbb T$, even for $d \gg 3$, would be novel and interesting.

Theorem~\ref{thm:hh1} readily implies that the two-point function of $\mathcal{V}^u$ satisfies
\begin{equation}
\label{eq:2pt}
\P[\lr{}{\mathcal{V}^u}{0}{x} ] \leq C e^{-|x|^c}, \, x \in \Z^d, \, u> u_*,
\end{equation}
for some constants $C,c$ depending on $d$ and $u$ only, and the bound \eqref{eq:2pt} remains valid for $u<u_*$ if one includes the truncation $\displaystyle \{\nlr{}{}{0}{\infty} \text{ in } \mathcal{V}^{u'} \}$ for any $u' < u$, on the left-hand side.
By means of a coupling such as \eqref{eq:coup-RW-RI} these connectivity estimates immediately transfer to the vacant set $\mathcal{V}_N^u$ of the walk on $\mathbb{T}$. In fact, one obtains using the results of \cite{PopTeix} that the decay in \eqref{eq:2pt} is exponential in $|x|$ for $ d\geq 4$, and sub-exponential when $d=3$. 

Large-deviation questions in the supercritical regime $u< u_*$ have also attracted considerable attraction in recent years. These include disconnection questions of macroscopic `regular' bodies in the supercritical regime of $\mathcal{V}^u$ \cite{zbMATH06257634, MR3602841, zbMATH07227743, zbMATH07226362}, and the related (but much harder) `droplet' problem, which relates to the possible emergence of a macroscopic shape when an excessive fraction of sites inside a large box gets disconnected by $\mathcal{I}^u$ for $u< u_*$; see \cite{zbMATH07114721, https://doi.org/10.48550/arxiv.1906.05809,zbMATH07483480,https://doi.org/10.48550/arxiv.2105.12110}. 
The resulting upper and lower bounds on these deviant events can now be propitiously combined with the knowledge of Theorem~\ref{thm:hh1},\textit{i)} and \textit{ii)} (which is tantamount to the equality $\bar{u}= u_{**}$ between critical parameters; see \eqref{eq:crit-equal} and the discussion in \S\ref{subsec:crit-par} below) to produce precise matching asymptotics. This is compelling notably because it gives credit to certain `scenarios' used to derive these bounds as identifying the correct phenomenology lurking behind these large-deviation constraints. 

To give but two examples of this, let $\mathcal{V}$ denote the complement in $\Z^d$ of the range of a simple random walk started at the origin under $P_0$. Owing to \cite[Corollary 7.4]{MR3602841} on the one hand and \cite[Theorem 0.1]{zbMATH06797082} (see also (0.5) therein, as well as \cite{zbMATH06257634} for a corresponding result for interlacements) on the other, one knows that for all $d \geq 3$,
\begin{equation}\label{eq:ldp_rw}
\begin{split}
& \liminf_N \frac1{N^{d-2}} \log P_0 [\nlr{}{ \mathcal{V}}{\partial B_N}{\partial B_{2N}}] \geq - \frac{u_{**}}d \text{cap}_{\R^d}\big([-1,1]^d \big), \\
& \limsup_N \frac1{N^{d-2}} \log P_0 [\nlr{}{ \mathcal{V}}{\partial B_N}{\partial B_{2N}}] \leq - \frac{\overline{u}}d  \text{cap}_{\R^d}\big([-1,1]^d \big);
\end{split}
\end{equation}
see also \cite{zbMATH06797082} and \cite[Corollary 4.4]{zbMATH07227743} when the disconnected set is not a box, and possibly non-convex. Here $u_{**}$ and $\overline{u}$ (see~\eqref{eq:u_**} and \eqref{eq:baru3} below for their precise definition) refer to the aforementioned auxiliary critical parameters, which satisfy $\bar{u} \leq u_* \leq u_{**}$, and above (resp.~below) which \eqref{eq:subcrit} (resp.~\eqref{eq:barhEXIST}-\eqref{eq:barhUNIQUE}) hold by definition. As a consequence of our main result, these thresholds each coincide with $u_*$ (cf.~\eqref{eq:crit-equal} below), thus yielding, together with \eqref{eq:ldp_rw}, that
\begin{equation}\label{eq:ldp_rw-full}
 \lim_N \frac1{N^{d-2}} \log P_0 [\nlr{}{ \mathcal{V}}{\partial B_N}{\partial B_{2N}}] = - \frac{u_*}d \text{cap}_{\R^d}\big([-1,1]^d \big).
\end{equation}
Importantly, and much in spirit as in the statement of Theorem~\ref{t:main1}, \eqref{eq:ldp_rw-full} exhibits the threshold $u_{*}$ one gains access upon introducing interlacements (see \eqref{eq:u_*}) as an intrinsic quantity associated to the simple random walk on $\Z^d$.

In a similar vein, consider now $\mathcal{C}^u_N$, defined as the union of $\partial B_N$ and the connected components of $\mathcal{V}^u$ intersecting it. By combining Theorem~\ref{thm:hh1} with \cite[Theorem 5.1]{https://doi.org/10.48550/arxiv.2105.12110}, \cite[Theorem 6.1]{https://doi.org/10.48550/arxiv.1906.05809} (see also Prop.~6.5 and (6.32) therein) and \cite[Theorem 0.2]{zbMATH07395560}, one obtains that for all $u< u_*$ and $\nu \in [ \bar \theta_{\infty}(u), 1)$, with $\bar \theta_{\infty} = 1-\theta_{\infty}$, cf.~\eqref{eq:u_*}),
\begin{equation}
\label{eq:LB-ex}
 \lim_N \frac1{N^{d-2}} \log \P\big[ |B_N \setminus \mathcal{C}^u_N| \geq \nu |B_N|\big] = - \overline{J}_{u,\nu},
\end{equation}
where $ \overline{J}_{u,\nu}$ is a rate function encompassing a certain constrained variational problem for the Dirichlet energy over a well-chosen class of (non-negative) test functions $\varphi$, see \cite[(6.32)]{https://doi.org/10.48550/arxiv.1906.05809}, for which $2^{-d}\int_{[-1,1]^d} \bar \theta_{\infty} \big((\sqrt{u} + \varphi)^2\big) dz > \nu$, thus reflecting at the continuous level the density constraint appearing in \eqref{eq:LB-ex}. An intriguing question concerns the nature of the subset (of $\R^d$) where minimizers, which are known to exist \cite{zbMATH07395560}, 
attain their maximal value $\sqrt{u_*}- \sqrt{u}$, cf.~\cite{https://doi.org/10.48550/arxiv.2105.12110}. We refer to \cite{zbMATH01415882,10.1214/aop/1019160324,zbMATH01495033} for works on corresponding questions in the context of the Ising model and Bernoulli percolation for $d \geq 3$, which due to their short-range nature, lead to surface order rather than capacitary problems in analogues of \eqref{eq:LB-ex}.

We conclude with a few remarks on the critical regime defined by the transition of Theorems~\ref{t:main1} and~\ref{thm:hh1}. Very little is known rigorously about $\mathcal{V}^{u_*}$ (or $\mathcal{V}_N^{u_*}$). Recent simulations \cite{Ch22} on $\mathbb{T}$ indicate that the transition is indeed continuous, and (within error bars) that the critical exponents describing the critical and near-critical behavior of $\mathcal{V}^{u_*}$ in dimension~$3$ correspond to those derived rigorously in 
\cite{DPR22,zbMATH07529630} for a related model bond percolation model involving the GFF, which exhibits the same type of long-range decay as \eqref{eq:LR}. These exponents exhibit both scaling and hyperscaling, but do not coincide (numerically) with the expected exponents for short-range percolation models, and thus appear to constitute a different, long-range `universality class'. Any progress on questions aimed at rigorously describing the (scaling) behavior of $\mathcal{V}^u$ for $u$ near $u_*$ would of course be a significant advance.

\subsection{Critical parameters for random interlacements} \label{subsec:crit-par} We now discuss how our results relate to various critical parameters previously introduced in the literature. We refer to \cite[Section 9.3]{RIbook2014} for pertinent (if slightly outdated) historical background; see also further refs.~below. Two important such parameters, alluded to in \S\ref{subsec:applications}, are
\begin{align}
  \label{eq:u_**}
 u_{**}(d) &\stackrel{\text{def.}}{=} \inf \big\{ u > 0 : \exists c=c(u,d)> 0\text{ such that }\lim_{R}  e^{R^c} \P[\lr{}{\mathcal{V}^u}{0}{\partial B_{R}} ] = 0 \big\}\,\\
\label{eq:baru3}
\bar{u}(d) &\stackrel{\text{def.}}{=} \sup\big\{s >0:~ \mathcal V \textrm{ strongly percolates at level } u,v ~\textrm{for all } 0< v< u < s \big\},
\end{align}
where strong percolation refers to the occurrence of the events \eqref{eq:EXIST}-\eqref{eq:UNIQUE1}, cf.~also Theorem~\ref{thm:hh1},\textit{ii)}. With the help of these, Theorem~\ref{thm:hh1} can be rephrased as follows:
\begin{equation}\label{eq:crit-equal}\bar{u}(d) =  {u}_*(d) = u_{**}(d) , \quad\text{for all $d \geq { 3}$.}
\end{equation}
The definitions of $u_{**}$ and $\bar{u}$ given in \eqref{eq:u_**} and \eqref{eq:baru3} are most useful in applications as they give strong quantitative information on the different phases of the model. As explained in \S\ref{subsec:applications}, there is by now a series of works which are `conditional' on Theorem~\ref{thm:hh1}, in the sense that their results become effective with the equality \eqref{eq:crit-equal}: a prototypical example is a pair of upper and lower bounds on a quantity of interest, each involving one of $\bar{u}$ or $u_{**}$, which match as a consequence of Theorem~\ref{thm:hh1}; see the above discussion around \eqref{eq:LB-ex} for more on these matters.

In order to establish their equality, it helps to work with the weakest possible definitions of $u_{**}$ and $\bar{u}$, thus making them intuitively `closer' to each other. We now introduce these weaker versions of the critical values $u_{**}$ and $\bar{u}$, which will be instrumental in our proof. Considerable effort has been previously devoted to weakening the defining condition in $u_{**}$. For our purposes, it will be sufficient to know that (cf.~\eqref{eq:u_**})
\begin{equation}
  \label{eq:equivalent_**}
  u_{**}(d) = \inf \big\{u >0 : \inf_{R}\P[\lr{}{\mathcal{V}^u}{B_R}{\partial B_{2R}} ] = 0 \big\}\,.
\end{equation}
In particular, the condition appearing in \eqref{eq:equivalent_**} yields useful connectivity estimates in the regime $u<u_{**}$, see e.g.~\eqref{eq:twopointsbound} below or 
 \cite[Lemma 2.2]{RI-II}. In fact $u_{**}$ was originally introduced in \cite{10.1214/09-AOP450} with a polynomial decay condition for the probability appearing in \eqref{eq:equivalent_**}, which was shown in \cite{MR2744881} to imply the stretched exponential decay asserted in \eqref{eq:u_**}. The polynomial speed condition was then removed (among others) in \cite{MR2891880}, yielding \eqref{eq:equivalent_**} in its present form, and as shown in \cite{PopTeix}, it is enough for the infimum to fall below an explicit constant $c=c(d)>0$.

The supercritical phase has comparatively seen less progress. By renormalization arguments, one knows for instance that the rapid decay exhibited in \eqref{eq:barhEXIST}-\eqref{eq:barhUNIQUE} follows as soon as the probabilities in question decay to $0$ as $R \to \infty$, but little is known otherwise. Note that, by requiring $\text{Exist}(R,u)$ and $\text{Unique}(R,u,v)$ to occur simultaneously for all scales $R=R_0 2^k$ for $R_0 \geq 1$ and $k \geq0$, and at levels $v<u< \bar{u}$, one readily infers using \eqref{eq:barhEXIST}-\eqref{eq:barhUNIQUE}, a union bound and a straightforward gluing argument involving \eqref{eq:EXIST} and \eqref{eq:UNIQUE1}, that $\P[ \lr{}{ \mathcal{V}^{v}}{B_{R_0}}{\infty}] \to 1$ as $R_0 \to \infty$, whence $\bar u \leq u_*$ in view of \eqref{eq:baru3}. Moreover, by \cite[Theorem 1.1]{MR3269990}, see also \cite{Tei11} for $d \ge 5$, one knows that $\bar{u} $ is non-trivial, i.e.~$\bar{u} >0$ for all $d \ge 3$. For completeness, we mention that even stronger notions than \eqref{eq:baru3} have appeared in the literature, involving any of: i) no sprinkling, i.e.~picking $u=v$ in \eqref{eq:UNIQUE1} and \eqref{eq:baru3}, see e.g.~\cite[(1.3)]{MR3269990}; or even ii) requiring strong percolation in \eqref{eq:baru3} to hold for all $u< s$ and some $v \in (u,s)$, see e.g.~\cite[(2.16)]{MR2838338}, see also \cite[Remark 8.9,3)]{drewitz2018geometry}, where this stronger uniqueness property is established in a non-perturbative regime; or iii) requiring uniformity of the constants $c,C$ in \eqref{eq:barhEXIST}--\eqref{eq:barhUNIQUE} over compact intervals of $u,v$, see \cite[(2)-(3)]{zbMATH07395560}, where it is used in combination with i). We will not deal with these stronger notions in the present work; see~\cite{GRS23+} for more on this.

As much as the defining features of $\bar{u}$ yield deep insights in the phase $u< \bar{u}$, that, together with Theorem~\ref{thm:hh1}, turn out to apply to the entire supercritical regime $u< u_*$, the implications of the condition $u > \bar{u}$ are unwieldy.
Closer in spirit to \eqref{eq:equivalent_**}, we define a parameter, first introduced in \cite{RI-II}, given by
\begin{equation}
\label{eq:tildeu}
\tilde{u}= \tilde{u}(d) = \sup \big\{ u > 0 : \liminf_{R} \, (M/R)^{d} \,\P[\nlr{}{ \mathcal{V}^u}{B_{R}}{\partial B_{M}}] \leq \alpha \big\},
\end{equation}
with $\alpha=\alpha(d)>0$ as supplied by \cite[Theorem 1.1]{RI-II},
\begin{equation}
\label{eq:def_M}
M=M(R)=\exp\big\{(\log R)^{\gamma_M}\big\}
\end{equation}
and $\gamma_M$ large enough, as for \cite[Corollary~1.2]{RI-II} to hold. The regime $u > \tilde{u}$ characterizes a region of parameters in which, borrowing a term from \cite{zbMATH07227743}, \textit{de-solidification} effects occur. Indeed, note that for every $u>\tilde u$, the negation of  \eqref{eq:tildeu} implies that {\em disconnection} events are not too unlikely, in a manner which is quantitative in the scale $M(R)$. This force, together with its counterpart corresponding to the condition $u<u_{**}$ (see \eqref{eq:equivalent_**}), will play a fundamental role in the sequel.

\subsection{Discussion of the proof of Theorem~\ref{thm:hh1}} \label{subsec-pf-outline} Recent sharp threshold results for a wide range of dependent percolation models, see e.g.~\cite{DumRaoTas17a,DumRaoTas17c, https://doi.org/10.48550/arxiv.2102.12123, AHL_2022__5__987_0,dembin2022sharp} and references~below, see also \cite{Men86,AizBar87,DumTas15} in the case of~Bernoulli percolation and the Ising model, have shown that the probabilities of connection events quickly decay when incrementing the parameters starting from a value for which disconnection events are not too unlikely (as is the case above $\tilde u$). Among such models, so-called $k$-dependent models that satisfy a uniform finite-energy property are of particular interest. Here $k \geq 0$ parametrizes the (finite) range of spatial dependence.

\bigskip

\noindent \textbf{Key features.} The vacant set of random interlacements however is anything but a $k$-dependent percolation model, and very far from satisfying the finite-energy property, let alone a uniform one, as we will shortly elaborate. The high-level strategy of the proof will be to interpolate, in a sense to be made precise, between our model and $k$-dependent percolation models for varying choices of $k$. An important stepping stone towards this interpolation is a propitious approximation of the vacant set $\mathcal{V}^u= \Z^d \setminus \mathcal{I}^u$ by a truncated version $\mathcal{V}^{u,L}= \Z^d \setminus \mathcal{I}^{u,L}$ `localized' (temporally) at scale $L$, i.e.~comprising trajectories of (time-)length $L$, which roughly corresponds to choosing $k\approx \sqrt{L}$. We will soon describe this interpolation in more detail. It is delicate. To wit, see~for instance \eqref{eq:V-L-informal}, \eqref{eq:V-L-k-informal} and \eqref{eq:V-T-informal} below, see also Figure~\ref{F:V_k}.

There are \textit{several} serious obstructions to implementing anything close to the strategy outlined above. We now highlight some of these, which gives insights into some of the central issues we have to face up to. Our previous work~\cite{DCGRS20} successfully managed to leverage a certain finite-range approximation of the Gaussian free field (GFF), which bears a long-range dependence akin to \eqref{eq:LR}, in order to derive an analogue of the equality $\bar{u}=u_*=u_{**}$ for excursion sets of the GFF; see also \cite{https://doi.org/10.48550/arxiv.2206.10724} for a different argument yielding subcritical sharpness, i.e.~the analogue of the equality $u_*=u_{**}$, including generalizations to a class of Gaussian percolation models, and also \cite{pcnontriv18,AizGri91} for inspirational interpolation techniques, albeit in a different context. These works all crucially exploit a very specific (multi-scale white-noise) {decomposition} of the underlying Gaussian field over scales, which harnesses the Gaussian nature of the problem; cf.~also \cite{https://doi.org/10.48550/arxiv.2212.05756, Bau13}. 

In the present context, a first and immediate obstruction is to \textit{give meaning} to a multi-scale approximation of $\mathcal{V}^u$. One can no longer exploit the structural properties of the Gaussian setup. In fact matters are rather worse owing to {\em degeneracies} in the law of $\mathcal V^u$, which arise in multiple ways. For example, they
preclude the `ellipticity' of the conditional law of $\mathcal{V}^u$ that any analogue of a finite-range decomposition would necessarily imply: indeed, unlike in the setup of \cite{DCGRS20} for instance, a point is forced to lie in $\mathcal{V}^u$ whenever its neighbors do (a manifestation of the lack of finite energy mentioned above). In particular, this means that there is \textit{no} analogue of a finite-range decomposition in the present context. 
This absence is also linked to the fact that  for \textit{every} $u>0$, an infinite component is present in $\mathcal{I}^u= \Z^d \setminus \mathcal{V}^u$ (in fact, $\mathcal{I}^u$ is connected for every $u>0$, see \cite[Cor.~2.3]{MR2680403}, so $\mathcal{I}^u$ consists of a single infinite component), and that $\mathcal{I}^u$ corresponds to a (degenerate) `hard threshold' limit $\alpha \downarrow 0$ for the excursion sets $\{\ell^u_{\cdot} > \alpha\}$ of the occupation time field $(\ell^u_x)_{x\in \Z^d}$ of interlacements, see \cite{MR3163210}. Incidentally, let us also mention \cite{AHL_2022__5__987_0,severo2022uniqueness}, where a (soft) shift argument is employed to deal with certain degeneracy issues stemming from analytic rigidity effects in the context of (smooth) Gaussian fields with short-range correlations.

To tackle the issue of decomposing the problem over scales, we initiated in our companion article \cite{RI-III} a different pathway using coupling, which will play a central role in this work. Although appealing, an approach involving `massive' interlacements, i.e.~including a uniform killing measure, does not distinguish sharply between scales, see Remark~\ref{eq:mass} for more on this. By pushing existing coupling techniques, see e.g.~\cite{MR2891880,MR2838338,PopTeix,MR3563197, PRS23, zbMATH06247265, CaioSerguei2018, 10.1214/23-EJP950}, one can compare $\mathcal V^u$ with $\mathcal V^{u, L}$ favorably in the sub-diffusive range, i.e.~inside regions with diameter $\ll \sqrt{L}$. However, in 
order to navigate the dependence inherent to the model $\mathcal V^{u, L}$, which has an `effective' range $\approx \sqrt{L}$, we need to be able to compare the two models in regions well above  the diffusive scale $\sqrt{L}$. Indeed, leveraging the independence properties of $\mathcal V^{u, L}$ typically warrants `losing information' at scales $\gtrsim \sqrt{L}$, which in turn inevitably leads to `reconstruction' problems around these scales.  Notice that across range $\sqrt{L}$, the length-$L$ trajectories essentially become 
stripped of their long-range structure and behave increasingly like `dust particles'. One of the most important features of 
our coupling results in \cite{RI-III} is to manage to cross over this barrier between super- and sub-diffusive scales. This feature also permeates the present paper, as will become apparent in the discussion below: sub- and super-diffusive scales are treated in distinctive manners, and the most uncompromising difficulties arise at near-diffusive scales, at which the cross-over for length-$L$ trajectories occurs. 

Finally, let us point out that the severe degeneracies in the conditional law of $\mathcal{V}^u$ alluded to above have {very} serious ramifications for performing surgery arguments involving clusters, for which some form of finite energy is often a key. This is felt all the more so in 
situations where we want to preserve a non-local condition like {\em pivotality}, see, e.g. \eqref{def:coarse_piv} below. To get a sense of what this entails in practice, we refer the reader to the `path reconstruction' arguments described at the start of Section~\ref{sec:penelope}.

\bigskip

\noindent\textbf{Overview of the proof.}  We now return to our interpolation scheme and discuss informally the truncated models $\mathcal{V}^{u,L}$ localized (temporally) at scale $L$ that will be used in our approximation; their formal definition is postponed to Section~\ref{sec:toolbox}. They correspond to a special (spatially homogenous, cf.~\S\ref{subsec:fr-models}) example drawn from a more general class of models $\mathcal{I}^{\rho}$ introduced in Section~\ref{sec:truncation} (see \eqref{e:this-is-mu}-\eqref{eq:prelim2}), which will account for  all our needs. Let $P_x$ denote the canonical law of the discrete-time (lazy) random walk on $\Z^d$ started at $x$ and $X = (X_n)_{n \ge 0}$ the corresponding process; see \S\ref{sec:RW} for precise definitions. We consider the product measure $\nu$ on $\R_+ \times W_+$, where $W_+$ is the space of forward $\Z^d$-valued trajectories (supporting $P_x$), see below \eqref{eq:Wdef}, characterized by
\begin{equation}
  \label{eq:mu_intensity}
  \nu ([0,u] \times B)  =u \sum_{x \in \Z^d} P_x[X \in  B],
\end{equation}
for any event $B$ measurable for $X$. We then introduce the Poisson point process $ \omega$ on $\R_+ \times W_+$, defined on its canonical space $(\Omega_+, \mathcal{A}_+)$, having intensity $\nu$. Its construction is standard as $\nu$ is $\sigma$-finite. For an arbitrary (density) 
function $f:\Z^d \to \R_+$, we then define, if $\omega = \sum_{i } \delta_{(u_i, w_i)}$,
\begin{equation}
\label{eq:J}
\mathcal{J}^{f ,L}= \mathcal{J}^{f ,L} (\omega) =\bigcup_{i \, : u_i \leq \frac{4d}{L} f(w_i(0))} w_i[0,L-1],
\end{equation}
where $w_i[0,L] = \{ x\in \Z^d: x= w_i(t)  \text{ for some $0 \leq t \leq L$}\}$.
The factor $4d$ appearing in \eqref{eq:J} is a matter of convenience. In words, $\mathcal{J}^{f ,L}$ comprises the first $L$ steps (including the initial position) of the traces of a Poissonian number of random walk trajectories, started with density proportional to $\frac{1}{L}f(\cdot)$. For $u  \geq 0$, we write $\mathcal{J}^{u ,L} $  whenever 
$f(x)=u$ for all $x \in \Z^d$. The random set $\mathcal{J}^{u ,L} $ is translation invariant, and converges in law (in the sense of finite-dimensional marginals) as $L \to \infty$ to $\mathcal{I}^u$ defined by \eqref{eq:I_u_intro}, see \eqref{e:loc-limit-J-u-L1}. Models of this and similar kind have appeared  in the literature, see e.g.~\cite{MR3962876, zbMATH07577023, 10.1214/23-EJP950}. 

The approximation $ \mathcal I^{u,L}$ of $ \mathcal I^{u}$ mentioned above shares this property, see \eqref{e:loc-limit-I-u-L}, but corresponds to a carefully chosen `noisy' version of $\mathcal{J}^{u,L}$, see \eqref{eq:I^Lu} in \S\ref{subsec:fr-models}. The more involved definition of $\mathcal{I}^{u,L}$ over $\mathcal{J}^{u,L}$ has technical reasons. Informally,
the noised version $\mathcal{I}^{u,L}$ of the process $\mathcal J^{u,L}$, whose vacant set will define $\mathcal V^{u,L}$, is obtained in two steps:
\begin{equation}\label{eq:V-L-informal}
\begin{cases}
\text{\parbox{13cm}{\begin{itemize}
\item[-] first, one defines another random interlacement set $\mathcal  J^{\varepsilon_L\sigma,L}$, where $\varepsilon_L$ is of order inverse poly-log (see \eqref{eq:varepsilon_L} for the precise value) and $\sigma$ is  {\em random} and chosen as follows: paving $\mathbb Z^d$ into boxes $B$ of radius $L$, one roughly sets
$
\sigma =\sum_{B}\sigma_B,
$
where the $\sigma_B$'s are i.i.d.~mean one Poisson variables;
\item[-] second, one resamples the state of every vertex in the set $\mathcal J^{u,L}\cup\mathcal J^{\varepsilon_L\sigma,L}$ in an i.i.d.~fashion with a very small probability $e^{-L}$.
\end{itemize}}}
\end{cases}
\end{equation}

We are now ready to state the first main step of the proof, which in itself already highlights a number of key issues. Recalling $M(R)$ from \eqref{eq:def_M}, set
\begin{equation}
\label{eq:def_M_0}
M_0(L)= 10^3M(10^3L).
\end{equation}
We focus on the comparison of `connection events' $\{B_r\longleftrightarrow \partial B_R\}$ between the full and truncated vacant sets $(\mathcal{V}^u)_{u>0}$ and $(\mathcal{V}^{u,L})_{u>0}$. Following our policy regarding constants stated at the end of this introduction, $c,C \in (0,\infty)$ denote generic constants depending only on the dimension $d$. The following result will be obtained as part of Corollary~\ref{prop:comparison} below, see also Remark~\ref{R:Prop-intro}.
\begin{prop}
  \label{prop:L to infinity}
  	For all $\delta\in(0,\frac12)$ and
	$\gamma \geq C$, there exists $L_0(\delta, \gamma) > 1$ such that for all $\tilde{u} \delta^{-1}>u>\tilde u(1+\delta)$, all $L \geq L_0$ integer power of 2 
	and $r,R \geq 1$ satisfying $2r\le R\le 2M_0(L)$, 
	\begin{align}
	 &\P[\lr{}{ \mathcal V^{u,L}}{B_{r}}{\partial B_{R}}]\ge\P[\lr{}{ {\mathcal V}^{u(1+(\log L)^{-3})}}{B_{r}}{\partial B_{R}}]-\exp\{-(\log R)^{c\gamma}\}, \label{eq:intro-compa1}\\
	 &\P[\lr{}{ \mathcal V^{u,L}}{B_{r}}{\partial B_{R}}]\leq \P[\lr{}{ {\mathcal V}^{u(1-(\log L)^{-3})}}{B_{r}}{\partial B_{R}}]+\exp\{-(\log R)^{c\gamma}\}. \label{eq:intro-compa2}
	\end{align}
\end{prop}

This result may be surprising at first sight. For, when looking at a box of size $R$, it is fairly believable (but not that easy) to compare $\mathcal V^{u}$ and $\mathcal V^{u,L}$ when $L\gg R^2$. Indeed, the latter is mostly composed of walks of length $L$ that rarely start inside the box and therefore naturally lend themselves to a comparison with the walks `arriving from infinity' comprising $\mathcal{V}^u$. It is much more surprising that one may achieve the comparison of Proposition~\ref{prop:L to infinity} up to walks of length $L(R)\approx \exp\{ (\log R)^{1/\gamma_M}\}$ (which is sub-polynomial in $R$), corresponding to the smallest scale $L$ which is a power of 2 and for which $2M_0(L)\ge R$, cf.~\eqref{eq:def_M_0} and \eqref{eq:def_M}. The proof of this proposition will already occupy large parts of Sections~\ref{sec:toolbox} and \ref{subsec:first_reduct} (until the end of \S\ref{subsec:compa}) and will involve a series of couplings; see in particular Theorems~\ref{thm:short_long} and~\ref{thm:long_short}, as well as Proposition~\ref{prop:couple_global} (derived from them), which plays a central role in the argument. It is important to realize that  the proof of Proposition~\ref{prop:L to infinity} (which is but a first step) and above all the underlying couplings that allow to compare $(\mathcal{V}^u)_{u>0}$ and $(\mathcal{V}^{u,L})_{u>0}$ rely on novel techniques, some of which are delegated to another article \cite{RI-III} not to make the present one too long. 

These couplings are in fact used in several places and of independent interest.
At their heart lies the fact that we can afford to compare the two random interlacements of interest on $B_R\setminus \mathcal O$, where $\mathcal O$ is called \textit{obstacle set}. This possibility is offered by the fact that we work above $\tilde u$ (cf.~the statement of Proposition~\ref{prop:L to infinity}) and that we can use the disconnection events to reconstruct the geometry of certain connected components in the vacant sets in an efficient way. A good mental picture is that the obstacle set $\mathcal O$ is the union of many small boxes (the obstacles) inside $B_R$, in which incoming pieces of random walk trajectories can be glued together to form longer ones. In practice, we typically only build a small fraction of trajectories at a time. The remaining bulk contribution is left untouched and used to generate $\mathcal O$, which is random. In a loose sense, the set $\mathcal O$ exploits a certain \textit{exchangeability} present in the models at mesoscopic scales.

The obstacle set $\mathcal{O}$ will only feature indirectly in the present article: it is a crucial ingredient for the proof of the coupling exhibited in Theorem~\ref{thm:long_short}, which is obtained as a direct consequence of the results of \cite{RI-III}. We wish to emphasize that the definition of $\mathcal{O}$, which lurks in the background of Theorem~\ref{thm:long_short}, is a delicate matter. In particular, the parameters associated to the obstacle set (obstacle size vs.~separation) need to balance opposite forces: indeed one intuitively wants `as much exchangeability' as possible, which manifests itself as requiring a high `surface density' of incoming trajectories on each obstacle comprising $\mathcal{O}$. This feature tends to improve the smaller the obstacles get. On the other hand, they need to remain sufficiently visible for the walks.
We defer a thorough discussion of these matters to \cite{RI-III}; see, in particular, (1.9) and (1.10) therein, along with the discussion in \cite[\S1.2]{RI-III}.

\medskip

Suppose now that Proposition~\ref{prop:L to infinity} is proved. At this stage, notice that walks involved in the definition of $\mathcal V^{u,L(R)}$ will be of size much smaller than $R$. Still, we need to pursue our comparison to reach the set $\mathcal V^{u,L_0}$, with $L_0$ independent of $R$, which is a $2L_0$-dependent percolation models with a finite-energy property, for which we can use  available sharp threshold results. To go down from scale $L(R)$ to $L_0$, we will compare $\mathcal V^{u,2L}$ and $\mathcal V^{u',L}$, where $u'$ is close to $u$. This will be done by incrementing between $\mathcal V^{u,2L}$ and $\mathcal V^{u',L}$ using intermediate models $\widetilde{\mathcal V}^{u,L}_k$ and $\overline{\mathcal V}^{u,L}_k$, each corresponding to one of two possible directions (cf.~\eqref{eq:intro-compa1} and \eqref{eq:intro-compa2}). We simply write ${\mathcal V}^{u,L}_k$ when referring to either choice.

We now provide an idea of what these two processes look like by defining a baby version of $ {\mathcal V}_k={\mathcal V}^{u,L}_k$, as follows (we return to the legitimate question as to why what follows is not the full story at the end of this proof outline). With a slight abuse of notation, we still refer to these simplified processes as ${\mathcal V}_k$, but stress that their informal character (see e.g.~\eqref{eq:V-L-k-informal} below) only serves the expository purposes of this introduction. The reader is referred to Section~\ref{sec:toolbox} for precise definitions. Consider now the partition $\mathcal B_L$ of $\Z^d$ provided by the boxes used to define $\mathcal V^{u,L}$ and set $A_k$ for every $k\ge0$ to be the union of the $(k+1)$ first boxes in this collection. Then, (the baby version of) ${\mathcal V}_k$ is a noised version of the process $$\mathcal J^{u'1_{A_k},L}\cup \mathcal J^{u(1-1_{A_k}),2L},$$ with $u'$ close to $u$, obtained in two steps:
\begin{equation}\label{eq:V-L-k-informal}
\begin{cases}
\text{\parbox{13cm}{\begin{itemize}
\item[-] first, one introduces to the picture another random interlacement set $\mathcal  J^{\varepsilon_L\sigma 1_{A_k},L}\cup\mathcal J^{\varepsilon_{2L}\sigma (1-1_{A_k}),2L}$, where $\varepsilon_L,\varepsilon_{2L}$ and $\sigma$ are defined as in \eqref{eq:V-L-informal};
\item[-] second, one resamples independently the state of each vertex in $\mathcal J^{u'1_{A_k},L}\cup \mathcal J^{u(1-1_{A_k}),2L}\cup\mathcal  J^{\varepsilon_L\sigma 1_{A_k},L}\cup\mathcal J^{\varepsilon_{2L}\sigma (1-1_{A_k}),2L}$ with a probability $e^{-L}$ for sites in $A_k$, and $e^{-2L}$ for the remaining sites.
\end{itemize}}}
\end{cases}
\end{equation}
For simplicity, we ignore in the following discussion the second step in \eqref{eq:V-L-k-informal}, which is anyways simple to handle.
The construction of ${\mathcal V}_k$ roughly resembles the one of $\mathcal V^{u',L}$ in $A_k$ and $\mathcal V^{u,2L}$ outside $A_k$; see Figure~\ref{F:V_k}. Notice that this process is $2L$-dependent and spatially inhomogenous. We also define ${\mathcal V}_{k+1/2}$ exactly as ${\mathcal V}_{k+1}$ except that we do not include the union with $\mathcal J^{\varepsilon_{2L}\sigma (1-1_{A_k}),2L}$ in the first step of \eqref{eq:V-L-k-informal}. Hence, we immediately find that ${\mathcal V}_{k+1/2}\supset{\mathcal V}_{k+1}$.

\begin{figure}[h!]
  \centering 
  \includegraphics[scale=0.7]{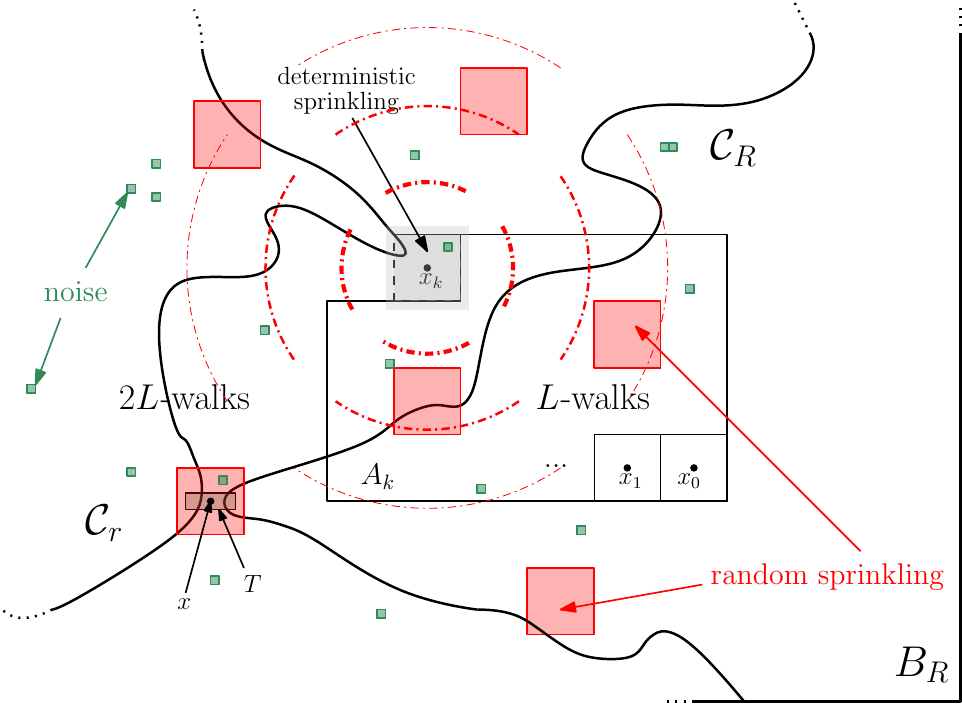}
  \caption{\textbf{From ${\mathcal{V}}_k$ to ${\mathcal{V}}_{k+1}$ in a pivotal configuration around $x$.} In going from $k$ to $k+1$, length-$L$ walks of intensity $u$ in the box around $x_k$ are replaced by length-$2L$ walks, with a sprinkled intensity $u'$ (grey area). The sets $\mathcal{C}_r$ and $\mathcal{C}_R$, representing the clusters of $\partial B_r$ and $\partial B_R$ in $\mathcal{V}_{k+1} \cap B_R$ are disconnected and in a pivotal configuration near $x$, as typically encountered in the process of bounding $f(x)$ in \eqref{eq:intro-f-ineq}. The circular red arcs around $x_k$ hint at the fact that, contrary to the `baby' version of the process $\mathcal{V}_k$ from \eqref{eq:V-L-k-informal}, the `true' process $\mathcal{V}_k$ introduced in Section~\ref{sec:toolbox} includes random sprinkling not only localized around $x_k$, but spreading everywhere in space (red), with intensity decreasing radially away from $x_k$. The red boxes correspond to regions in which the random sprinkling is (abnormally) large, i.e.~of order $\varepsilon_L$, as needed to perform surgery in their vicinity. The necessity for this is related to the iterative scheme in which the difference estimate \eqref{eq:intro-f-ineq} is applied, which requires understanding $f(x)$, see \eqref{eq:f-intro}, for a generic argument $x$, and not only for $x=x_k$.}
  \label{F:V_k}
\end{figure}

In an ideal world, one would manage to compare ${\mathcal V}_{k}$ and ${\mathcal V}_{k+1}$ (for instance looking at the probability of the event $\{B_r\longleftrightarrow \partial B_R\}$) using the fact that the two processes only differ because of walks (and noise) sampled from the $(k+1)$-st box $B$. In order to prove such a fact, we will rely on a coupling similar to the one used to prove Proposition~\ref{prop:L to infinity} above, which allows to compare interlacements comprising walks of length $L$ and length $2L$. This will essentially yield that with good probability, ${\mathcal V}_{k}\supset {\mathcal V}_{k+1/2}(\supset {\mathcal V}_{k+1})$. Unfortunately, it will happen that the coupling between these processes fails. 
This is not an artefact of our method: indeed coupling with perfect inclusion across all scales would imply similar large-scale behavior of observables, which are however known to change as $L \to \infty$, e.g.~from surface to capacity order; see \eqref{eq:I_u_intro} or \eqref{eq:ldp_rw-full} for instance.

In case of coupling failure, we aim to leverage the fact that the probability of such a failure is the product of two contributions: first the coupling needs to fail locally, which will have tiny probability of order $\exp\{-(\log L)^\gamma\}$ and depends only on what happens in the box $\widetilde{B} \supset B$ of size $CL$ concentric with $B$, but in addition this failure should impact the occurrence of $\{B_r\longleftrightarrow \partial B_R\}$. This is only possible if the pivotal event $\piv_{\widetilde{B}}({\mathcal V}_{k+1/2})$ occurs, where 
\begin{equation}\label{def:coarse_piv}
\piv_{K}({\mathcal V}) \stackrel{{\rm def}.}{=} \big\{\lr{}{ {{\mathcal V} \cup K}}{B_r}{\partial B_R}, \nlr{}{ {{\mathcal V} \setminus K}}{B_r}{\partial B_R}\big\},
\end{equation}
for $ \mathcal{V}, K \subset \Z^d$ and $R \geq r \geq 0.$
Overall, we will thus (roughly, cf.~\eqref{eq:compa1}) get that
\begin{equation}\label{eq:step1-compa}
\mathbb P[\lr{}{ {\mathcal V}_{k+1/2}}{B_{r}}{\partial B_{R}}]\le \mathbb P[\lr{}{ {\mathcal V}_{k}}{B_{r}}{\partial B_{R}}]+\exp\{-(\log L)^\gamma\} f(x_k),
\end{equation}
where $x_k$ refers to the center of the box $\widetilde{B}$ and, writing $\mathbb{L}$ for the lattice consisting of all centers of boxes in the collection $\mathcal{B}_L$ that partition $\Z^d$ (to which $x_k$ belongs), we introduce
\begin{equation}
\label{eq:f-intro}
f(x)= \mathbb P[\piv_{B(x,CL)}({\mathcal V}_{k+1/2})], \text{ for } x\in \mathbb{L}.
\end{equation}
The core of the argument will be to prove that the small additive error term arising in \eqref{eq:step1-compa} can be compensated in a second step by passing from ${\mathcal V}_{k+1/2}$ to ${\mathcal V}_{k+1}$, i.e.~that it is bounded for values of $u$ around $u_*$ and at scales $R(\geq 2M_0(L))$ in the regime complementary to that of Proposition~\ref{prop:L to infinity} by a `discrete gradient' of the form
$$
b= \mathbb P[\lr{}{ {\mathcal V}_{k+1/2}}{B_{r}}{\partial B_{R}}] - \mathbb P[\lr{}{ {\mathcal V}_{k+1}}{B_{r}}{\partial B_{R}}] = P[\lr{}{ {\mathcal V}_{k+1/2}}{B_{r}}{\partial B_{R}}, \nlr{}{ {\mathcal V}_{k+1}}{B_{r}}{\partial B_{R}}]   \ (>0).
$$
We refer to Proposition~\ref{prop:comparisonLk} for the exact result. In summary, owing to \eqref{eq:step1-compa}, the game is over once we have that $e^{-(\log L)^{\gamma}} f(x_k) \leq b$. As we now explain, the prof of this inequality is a very difficult game to win, and it occupies most of this article. We will obtain the desired estimate by iterating a functional inequality for the function $f=f(\cdot)$ in \eqref{eq:f-intro} of the form
\begin{equation}\label{eq:intro-f-ineq}
f \leq b\cdot g + e^{-(\log L)^{c\gamma}} \bar{A}f
\end{equation}
(see Proposition~\ref{P:compa_a} for the precise statement). Observe that \eqref{eq:intro-f-ineq} is an inequality between functions defined on $\mathbb{L}$.
Here $b$ is our discrete gradient around $x_k$ (a scalar), $g=g(\cdot)$ is a certain cost function satisfying $\log ( g) = o((\log L)^{\gamma})$ as $L \to \infty$ that captures the cumulated price of reconstructing $b$ out of the box pivotality, and $(\bar{A}f)(x)$ refers to a local average of $f$ in a neighborhood of $x \in \mathbb{L}$ roughly of size $M_0$. Iterating \eqref{eq:intro-f-ineq} and evaluating at $x=x_k$, the desired bound relating $f(x_k)$ to $b$ quickly follows. Proving \eqref{eq:intro-f-ineq}, which is key to the argument, will take us on a long expedition starting from \S\ref{subsec:2lemmas} onwards.

Let us now outline how \eqref{eq:intro-f-ineq} comes about.
First, we will use disconnection estimates to prove that the probability of pivotality of the box $\widetilde{B}$ inherent to $f(x_k)$ can be expressed in terms of the probability of closed pivotality of a much bigger box of size roughly $M_0(L)$, where closed pivotality of a set $K$ refers to the event 
\begin{equation}\label{def:closed_piv}
\overline \piv_{K}({\mathcal V}) \stackrel{{\rm def}.}{=} \piv_{K}({\mathcal V})\cap\{\nlr{}{ {\mathcal V}}{B_r}{\partial B_R}\}.
\end{equation}
This readily takes care of the disconnection condition that forms part of $b$, at the cost of forfeiting information inside the large box of size approximately $ M_0(L)$. However, producing a configuration in $b$ further requires building a full connection in $\widehat{\mathcal{V}}_{k+1/2}$, all while preserving this disconnection. We think of this in terms of progressively reducing the separation between the clusters of $\partial B_r$ and $\partial B_R$, or, equivalently, narrowing down the region of (closed) pivotality, which to begin with is a box of size about $M_0(L)$, until we eventually reconstruct a configuration in $b$, up to a not too large multiplicative cost factor given by $g$. We will describe the process of shrinking the cluster separation (or pivotal region) momentarily. Notice though that all efforts to do this may fail at various stages of the argument, but if they do then typically on some bad event of small probability $e^{-(\log L)^{c\gamma}} $. Roughly speaking, the second term on the right-hand side of \eqref{eq:intro-f-ineq} accounts for this possibility.

We now enter the heart of the argument: reducing the size of the pivotal region. This is performed in two steps, which deal with complementary sets of scales, and are fundamentally different. First, in Lemma~\ref{L:closed_piv2piv}, we will essentially find a closed pivotal box $T$ of intermediate (but super-diffusive) size  $R_T =\sqrt L \cdot (\log L)^C$. Second, Lemma~\ref{lem:reduce_distance} will establish that one can in fact go from the closed pivotality of $T$ to a configuration in $b$ (which roughly corresponds to a closed pivotality at scale $1$). 

The scale $\sqrt L$ does not come out of thin air. At scales $ \gg \sqrt{L}$, the walks of length $ L$ and $2L$ constituting ${\mathcal V}_{k+1}$ are essentially dust-like. At scales significantly below $ \sqrt{L}$, they start to look back like random interlacements and become infinitely long for practical purposes. Managing the cross-over at the diffusive barrier to `transform the dust into random walks' is the most challenging part of the argument. For this purpose, it is in fact crucial that $T$ not be a box, but rather a tube of thin cross-section $r_T= \sqrt{L}\cdot (\log L)^{-C'}$, whose long direction $R_T$ minimizes the distance between the two clusters at the outcome of the first step. Inside $T$, we can no longer work directly with ${\mathcal V}_{k+1}$. Its finite-range property is essentially useless, at the same time the truncation to time-length $\approx L$ is severely felt from a random walk perspective. To deal with this, we introduce a local modification, $\mathcal{V}_T$ roughly obtained as follows: starting from $\mathcal{V}_{k+1}$,
\begin{equation}\label{eq:V-T-informal}
\begin{cases}
\text{\parbox{13cm}{\begin{itemize}
\item[-] first, one removes all trajectories that start inside a slightly bigger tube $T^{\circ}$ (with rounded cross-section);
\item[-] second, one lets the walk `run freely' inside an intermediate tube $T'$ with $T\subset T' \subset T^{\circ}$, so that time only accumulates towards the fixed time horizon ($L$ or $2L$) when the random walk travels outside of $T'$.
\end{itemize}}}
\end{cases}
\end{equation}
These properties are designed to exhibit a picture resembling that of a normal interlacement inside $T$, which in turn reinstates various desirable tools, notably a conditional decoupling which is heavily relied on in our construction of the path. Crucially, the geometric features of $T$ ensure that \eqref{eq:V-T-informal} is not tampering too much with the configuration (for instance walks tend to quickly exit through the short sides), so that with high probability, one has that 
\begin{equation}\label{eq:coupling-V-T}
{\mathcal V}_{k+1} \subset \mathcal{V}_T \subset {\mathcal V}_{k+1/2}
\end{equation}
(cf.~Lemma~\ref{L:superhardcoupling}). Yet again, the general coupling results developed separately in our companion article \cite{RI-III} crucially enter in establishing \eqref{eq:coupling-V-T}, which is far from innocent to show.

\begin{figure}[h!]
  \centering 
  \includegraphics[scale=0.70]{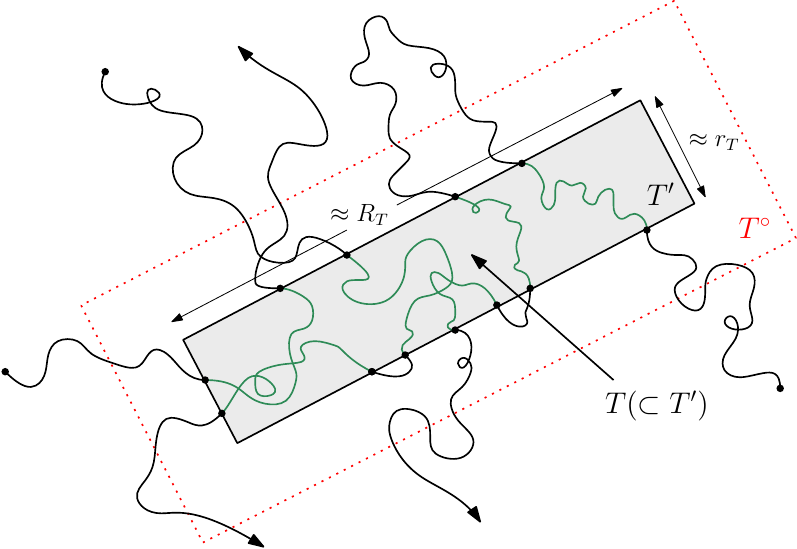}
  \caption{\textbf{The model ${\mathcal{V}}_T$.} Green trajectories correspond to unconstrained random walk bridges. Time only accumulates towards the fixed horizon in the black pieces. These modifications are instrumental in order to `cross' the diffusive barrier $R \approx \sqrt{L}$.}
  \label{F:V_T}
\end{figure}

The actual path is then built in the configuration $\mathcal{V}_T$, which is typically in ${\mathcal V}_{k+1/2}$ on account of \eqref{eq:coupling-V-T}. We will not describe the specifics of the construction here (for further details the reader is referred to the beginning of Section~\ref{sec:penelope}). 
We should mention that this part of the proof will rely on a delicate hierarchical bridge construction, which is in fact also crucially at play in the proof of the equality $\bar u=\tilde u$ in our work \cite{RI-II}. This bridge, constructed in \cite{RI-II}, is very different from the `croissants' used in the context of the GFF in our earlier work \cite{DCGRS20}. The present bridge is much less `rigid,' it leaves gaps at all scales, which are in order owing to the afore mentioned absence of finite energy and degeneracies in the occupation laws. Importantly, the configuration $\mathcal{V}_T$, which looks like interlacements at near-critical intensity $u \approx u_*$ inside $T$, inherits (almost) polynomial lower bounds from $\mathcal{V}^u$, which are used to build the path efficiently `on' the bridge. The resulting picture of the reconstructed path is that of a critical object: for instance, the bridge region occupied by the path has Hausdorff dimension $>1$.

To finish, we wish to highlight one last thing. Among other things, this also explains why ${\mathcal{V}}_k$ as described informally in \eqref{eq:V-L-k-informal} is but a baby version of the actual processes $\widetilde{\mathcal V}^{u,L}_k$ and $\overline{\mathcal V}^{u,L}_k$, which are more carefully designed. Due to the functional nature of the key inequality \eqref{eq:intro-f-ineq}, the argument $x$ in $f(x)$ bears no connection to the
centre $x_k$ of the box where trajectories are currently being swapped. This is not a side-note and owed to the non-local nature of our argument: although eventually we aim to bound $f(x_k)$, iterating \eqref{eq:intro-f-ineq} really requires controlling $f$ anywhere in space.
This is why $\overline {\mathcal V}^{u,L}_k$ and $\widetilde {\mathcal V}^{u,L}_k$ differ from \eqref{eq:V-L-k-informal} 
in that the random intensity profile used to go from step $k+\frac12$ to $k+1$ will in fact be a polynomially decaying infinite-range profile instead of simply concentrated around the $(k+1)$-st box. Moreover, the reader may legitimately wonder why this sprinkling is random. This is because an inclusion such as \eqref{eq:coupling-V-T} actually requires a sprinkling of order $\varepsilon_L= (\log L)^{-C}$, no matter the location of the tube $T$ where the surgery is currently being performed. When iterating over $k$, sprinkling would thus add up in units of $\varepsilon_L$ everywhere in space, which would be catastrophic. Instead the \textit{average} sprinkling follows the profile of $\sigma$, which decays away from $x_k$ in step $k\to (k+1)$, but the randomness leaves the possibility to ask for an atypical sprinkling when need be, which, as will turn out, comes at an affordable cost that can be absorbed in the pre-factor $g$.

\subsection{Organization of the article} 
Section~\ref{s:not} introduces some notation and recalls various facts about the simple random walk, potential theory, and random interlacements. It concludes with the proof of how Theorem~\ref{t:main1} is classically deduced from Theorem~\ref{thm:hh1}. 

Section~\ref{sec:truncation} brings in the models $\mathcal I^\rho$, parametrized by an intensity measure $\rho$ allowing for trajectories with both variable time-length and spatial intensity (\S\ref{subsec:setup}). The setup is general enough to fit all our needs, which comprise the models described informally in \eqref{eq:V-L-informal}, \eqref{eq:V-L-k-informal} and \eqref{eq:V-T-informal}. In \S\ref{subsec-coup-gen} we then present the important couplings between models of type $\mathcal I^\rho$ as $\rho$ varies. These will play a central role throughout. The proof of these couplings, which are of independent interest, is included in \cite{RI-III} for the sake of room.

Section~\ref{sec:toolbox} introduces the models that will be used to approximate $\mathcal{V}^u$ and gathers their basic properties. The homogenous models $\mathcal{V}^{u,L}=\Z^d \setminus \mathcal{I}^{u,L}$, corresponding to  \eqref{eq:V-L-informal}, are defined in \S\ref{subsec:fr-models}, their more elaborate (inhomogenous) counterparts $\mathcal{V}_{\ell}^{u,L}=\Z^d \setminus \mathcal{I}^{u,L}_\ell$, where $\ell$ is a half-integer, in \S\ref{Sec:mixedmodelsdef}. The latter come in two variants, $\overline{\mathcal V}_\ell^{u, L}$ and $\widetilde{\mathcal V}_\ell^{u, L}$, but most subsequent arguments apply equally to both, in which case we simply write ${\mathcal V}_\ell^{u, L}$ (see the convention \eqref{eq:barequalstilde}), as for the remainder of this outline. The section culminates in Proposition~\ref{prop:couple_global} and its proof (\S4.3), which relates the models ${\mathcal V}_\ell^{u, L}$ as $\ell$ varies.

Relying on this preparatory work, Section~\ref{subsec:first_reduct} begins the proof of Theorem~\ref{thm:hh1}, which is deduced in \S\ref{subsec:compa} from two comparison inequalities, stated in Corollary~\ref{prop:comparison}, which subsumes~Proposition~\ref{prop:L to infinity} and includes the (harder) estimate in the complementary regime of scales $R$ and $L$. These comparison inequalities are reduced in \S\ref{subsec:reduction} to the difference estimate stated informally in \eqref{eq:intro-f-ineq}, see Proposition~\ref{P:compa_a}. As detailed in \S\ref{subsec:2lemmas}, Proposition~\ref{P:compa_a} follows from two difficult lemmas, namely Lemma~\ref{L:closed_piv2piv} and Lemma~\ref{lem:reduce_distance} below, that separately deal with the shrinking of the pivotal region at super- and (near-)+(sub-)diffusive scales, respectively.

Section~\ref{sec:superdiff} contains the proof of Lemma~\ref{L:closed_piv2piv}. This requires two additional inputs. We first gather (\S\ref{subsec:connec}) important preliminary connection and disconnection estimates for ${\mathcal V}_\ell^{u, L}$, which are inherited from $\mathcal{V}^u$ in suitable regimes of $u$ by virtue of the couplings of Section~\ref{sec:toolbox}. We then present in \S\ref{subsec:pivotal-mixed} some estimates relating pivotality events for ${\mathcal V}_\ell^{u, L}$ at different scales. The ingredients are put together in \S\ref{subsec:pf-supercrit}, where the proof of Lemma~\ref{subsec:pf-supercrit} is presented.

Sections~\ref{A:superhard} and~\ref{sec:subdiff} are devoted to the proof of Lemma~\ref{lem:reduce_distance} that operates at near-diffusive scales and below. As explained above, this requires extending the toolbox, as the models ${\mathcal V}_\ell^{u, L}$ are no longer functional at these scales.

To this effect, Section~\ref{A:superhard}, which is organized in a similar fashion as Section~\ref{sec:toolbox}, introduces a new model, $\mathcal{V}_T$, described around \eqref{eq:V-T-informal} (\S\ref{surgery-1}). After deriving relevant random walk estimates related to the cylinder $T$ (\S\ref{sec:decouple}), we prove in \S\ref{sec:caio} an important conditional decoupling property for $\mathcal{V}_T$ (Lemma~\ref{lem:caio_T}), to which the specifics of $\mathcal{V}_T$ are tailored. The key is then to show that $\mathcal{V}_T$ really interpolates between $ \overline{\mathcal V}_{k+1/2}^{u, L}$ and $ \overline{\mathcal V}_{k+1}^{u, L}$ with very high probability; cf.~\eqref{eq:coupling-V-T} and Lemma~\ref{L:superhardcoupling}. Keeping the analogy with Section~\ref{sec:toolbox}, one can view this as refining Proposition~\ref{prop:couple_global}. The proof of Lemma~\ref{L:superhardcoupling} is given in \S\ref{subsubsec:superhard}. This is the most involved coupling we will work with. Together with the pivotality estimates of \S\ref{subsec_piv_V_T}, corresponding to those of \S\ref{subsec:pivotal-mixed} but now involving $\mathcal{V}_T$, these results constitute the extended toolbox for Lemma~\ref{lem:reduce_distance}.

The proof of Lemma~\ref{lem:reduce_distance} unfolds over Section~\ref{sec:subdiff}, which is organized similarly as Section~\ref{sec:superdiff}. In \S\ref{surgery-2} we first construct an `almost-path' with holes at the bottom scales, relying on a delicate bridge construction exhibited in \cite{RI-II}, where it is used for a related purpose, but for the full model $\mathcal{V}^u$ rather than the more complicated $\mathcal{V}_T$. 
The holes are `plugged' separately in~\S\ref{surgery-3}. Finally \S\ref{sbusec:denouement} wraps things up and concludes by giving the proof of Lemma~\ref{lem:reduce_distance}, which assembles the various elements.

\medskip

Our convention regarding constants is as follows. Throughout the article $c,c',C,C',\dots$ denote generic constants with values in $(0,\infty)$ which are allowed to change from place to place. All constants may implicitly depend on the dimension $d\geq3$. Their dependence on other parameters will  
be made explicit. Numbered constants are fixed when first appearing within the text. To keep notations 
reasonable, a separate protocol is valid from \S\ref{subsec:2lemmas} onwards (see below \eqref{eq:cond-2lemmas}), allowing constants to depend on a larger set of parameters.

\section{Notation and useful facts}
\label{s:not}

We write $\mathbb{N}=\{0,1,2,\dots\}$ for the set of nonnegative integers, $\mathbb{N}^\ast=\mathbb{N} \setminus \{0\}$, $\R$ for the set of reals and $\mathbb{R}_+ = \{x \in \mathbb{R}: x \geq 0\}$. We consider the lattice $\Z^d$, $d \geq3$ and denote by $|\cdot|$ and $|\cdot|_{\infty}$ the $\ell^2$ and $\ell^{\infty}$-norms on $\Z^d$, respectively. We use $e_j$, $1\leq j \leq d$, to denote the standard unit vectors in the $j$-th coordinate direction and frequently write $x\sim y$ when $|x-y|=1$ for $x,y \in \Z^d$. For a set $U \subset \Z^d$, we write $U^c =\Z^d 
\setminus U$ for its complement (in $\Z^d$), $\partial U$ for the 
interior vertex boundary of $U$, i.e.~$\partial U = \{x \in U: \exists y\notin U \text{ s.t. } y \sim x\}$. 
We write $\partial_{\text{out}} U= \partial (U^c)$ for its outer vertex boundary, $\overline{U}=U \cup \partial_{\text{out}} U$. 
The set $B(U, r) = U_{r} = \bigcup_{x \in U} B(x,r)$ denotes the $r$-neighborhood of $U$, for $r>0$, and $U \subset \subset \Z^d$ means that $U\subset \Z^d$ has finite cardinality. We use the notations $B_r(x)=B(x,r)$ interchangeably to denote balls with radius $r> 
0$ around $x \in \Z^d$ with respect to the $\ell^{\infty}$-norm and abbreviate $B_r=B_r(0)$. We write $d(\cdot, \cdot)$ to refer to the $\ell^{\infty}$-distance between subsets of $\Z^d$ and $d_2(\cdot,\cdot)$ for the Euclidean one.

\subsection{Basic properties of random walk}\label{sec:RW}
We now introduce the random walks of interest and recall a few elements of potential theory.
We endow $\Z^d$, $d \geq 3$, with the symmetric weight function $a:\Z^d \times \Z^d \to [0,\infty)$ defined as $a_{x,y} =a_{y,x}=1$ if $x \sim y$, $a_{x,x}=2d$ and $a_{x,y}=0$ otherwise, and write $a_x=\sum_{y \in \Z^d}a_{x,y} \, (=4d)$. We consider the discrete-time Markov chain on $\Z^d$ with transition probabilities $p(x,y)= \frac{a_{x,y}}{a_x}$, $x,y \in \Z^d$.
We write $P_x$ for the canonical law of this chain when started at $x \in \Z^d$ and $X = (X_n)_{n \geq 0}$ for the corresponding canonical process. We often abbreviate $X_{[s,t]}=\{ X_n : s \leq n \leq t\}$ for $0 \leq s \leq t$.
For a positive measure $\mu$ on $\Z^d$ we write $P_\mu = \sum_{x\in \Z^d} \mu(x)P_x$. We refer to $X$ as the \textit{random walk}.
Let $p_n(x,y)=P_x[X_n=y]$, $x,y, \in \Z^d$, $n \geq 0$, with $p_1=p$, denote the transition probabilities of $X$. Note that $p_n(x,y)=p_n (0,y-x)$ by translation invariance. We denote by $P_n$ the corresponding transition operators, i.e.~ 
\begin{equation}\label{eq:P_n-def}
P_nf(x) \stackrel{\text{def.}}{=} E_x[f(X_n)] = \sum_y p_n(x,y)
f(y), \quad x \in \Z^d 
\end{equation}
for \textit{any} $f:\Z^d \to \R$ (note that $P_n$ effectively has finite range by time-discreteness so there is no convergence issue in \eqref{eq:P_n-def}). The family $(P_n)_{n \geq 0}$ forms a semigroup, i.e.~$P_{n+m}=P_{n}P_m$ for all integers $n,m \geq 0$. We now recall a few elements of potential theory for $X$ that will be used throughout. We write 
 \begin{equation}
\label{eq:Green}
	g(x,y) = \sum_{n \geq 0} a_y^{-1}P_x[X_n =y], \quad \text{ for }x,y \in \Z^d
\end{equation}
for the Green's function of $X$ (more precisely its density with respect to $a_{\cdot}$). By translation invariance $g(x,y)= g(x+z ,y +z)$ for all $x,y,z \in \Z^d$ and by \cite[Theorem~1.5.4]{Law91}, one has that
\begin{equation}\label{eq:Greenasympt}
g(x)\stackrel{\text{def.}}{=} g(0,x) \sim  \Cl[c]{c:green} |x|^{2-d}, \quad \text{as } |x| \to \infty,  
\end{equation}
(where $\sim$ means that the ratio of both sides tends to $1$ in the given limit), for an explicit constant 
$\Cr{c:green} \in (0,\infty)$.  We further define, for $K \subset\subset \Z^d$,
\begin{equation}\label{eq:equilib_K} 
e_K(x)=a_xP_x[\widetilde{H}_K=\infty]1\{x \in K\},
\end{equation}
the equilibrium measure of $K$, which is supported on $\partial K$. We denote by
\begin{equation}\label{eq:cap_K} 
\text{cap}(K) = \sum_x e_{K}(x)
\end{equation}
its total mass, the (electrostatic) capacity of $K$ and by $\bar{e}_K= \frac{e_K}{ \text{cap}(K)}$ the normalized 
equilibrium measure. By \cite[Prop.~2.2.1(a)]{Law91}, one knows that $\mathrm{cap}(\cdot)$ is increasing, i.e.
\begin{equation}
\label{eq:cap-K-increasing}
\text{cap}(K) \subset \text{cap}(K'), \text{ for all }K\subset K'\subset \subset \Z^d.
\end{equation}
One further has the last-exit decomposition, see e.g.~\cite[Lemma~2.1.1]{Law91} for a proof,
\begin{equation}
	\label{eq:lastexit}
	P_x[H_K < \infty] = \sum_y g(x,y) e_{K}(y), \text{ for all } x \in \Z^d,
\end{equation}
valid for all $K \subset\subset \Z^d$. Summing \eqref{eq:lastexit} over $x\in K$, one immediately sees that
\begin{equation}\label{eq:cap_bound_Green}
\big(\max_{x\in K}\sum_{y\in K} g(x,y)\big)^{-1}\leq  \frac{ \mathrm{cap}(K)}{|K|}\leq \big({\min_{x\in K}\sum_{y\in K} 
g(x,y)} \big)^{-1},
\end{equation}
where $|K|$ denotes the cardinality of $K$, which along with \eqref{eq:Greenasympt}, readily gives, for all $L > 0$,
\begin{equation}
\label{e:cap-box}
cL^{d-2} \leq \text{cap}(B_L) \leq C L^{d-2}.
\end{equation}

\subsection{Random interlacements}
\label{subsec:RI}

We now introduce the interlacement point process, defined on its canonical space $(\Omega, \mathcal{A}, \P)$, the construction of which we  briefly review. 
We write $W$ for the set of doubly-infinite, nearest-neighbor transient trajectories in $\mathbb{Z}^d$ (by which, slightly departing from usual conventions, we include the possibility to stay put), that is
\begin{equation}
\label{eq:Wdef}
  W \stackrel{\text{def.}}{=} \big\{ (w_i)_{i \in \mathbb{Z}} \in (\Z^d)^{\Z}:\, |w_i - w_{i + 1}| \leq 1,\, i \in \mathbb{Z}, \text{ and } w^{-1}(\{x\}), \, x  \in \mathbb{Z}^d, \text{ is finite}  \big\},
\end{equation}
 endowed with its canonical $\sigma$-algebra $\mathcal{W}$. The corresponding canonical shifts are denoted by $\theta_n: W\to W$, $n \in \Z$, with $\theta_n (w)(\cdot)= w(n+\cdot)$ and the canonical coordinates by $(X_n)_{n\in \Z}$. For later reference, we also introduce $W^+$, the set of one-sided trajectories $(w_i)_{i\in \mathbb{N}}$ with analogous properties to \eqref{eq:Wdef}, and its $\sigma$-algebra $\mathcal{W}^+$. The shifts $\theta_n$, $n \geq 0$, also act on $W^+$. Note that $W^+$ has full measure under the canonical law $P_x$ of $X$ introduced at the beginning of Section~\ref{sec:RW}. We denote by $W^*$ the set of trajectories modulo time shift, i.e.~$W^* = W/\sim$, where $w \sim w'$ if $w = \theta_n (w')$ for some $n\in \mathbb{Z}$, and by $\pi^*: W\to W^*$ the canonical projection. We write $W^*_K \subset W^*$ for the trajectories visiting~$K \subset \mathbb{Z}^d$. We use the shorthands $w[s,t]= \{w(n): s \leq n \leq t \}$ for $s \leq t$ and $w \in W$ and similarly $w[s,t]$ for $w\in W^+$ when $s \geq 0$.

We write $\mathbb{P}$ for the probability measure governing a Poisson point process on $W^* \times \mathbb{R}_+$ with intensity measure $\nu_{\infty}(\mathrm{d}w^*) \mathrm{d}u$, where $\mathrm{d}u$ denotes the Lebesgue measure and for all $K \subset \Z^d$
\begin{equation}\label{e:RI-intensity}
\begin{split}
&1_{W^*_K} \nu_{\infty} = \pi^* \circ Q_K, \text{ where $Q_K$ is a finite measure on $W$, and} \\
&Q_K[(X_{-n})_{n \geq 0} \in {A}, \, X_0 =x , \, (X_{n})_{n \geq 0} \in {A}' ]= P_x[ {A} \, | \, \widetilde{H}_K =\infty]e_K(x)P_x[ {A}'],
\end{split}
\end{equation}
 for all $x \in \Z^d$ and ${A}, {A}' \in \mathcal{W}^+$, with $e_K$ as in \eqref{eq:equilib_K}. The existence of a unique measure $\nu_{\infty}$ satisfying \eqref{e:RI-intensity} was shown in \cite[Theorem 1.1]{MR2680403}, see also \cite[Theorem 2.1]{MR2525105} for its generalization to arbitrary transient weighted graphs (where $\nu_{\infty}$ is denoted by $\nu$ -- the hopefully suggestive notation $\nu_{\infty}$ is not without reason, cf.~\eqref{eq:prelim1} and \eqref{eq:prelim3} below). The `lazyness' inherent to $X$, manifested by the presence of a non-vanishing conductance $a_{x,x}$ (cf.~the beginning Section~\ref{sec:RW}) presents the benefit of avoiding certain parity issues, which is technically convenient. The existence of $\nu_{\infty}$ as in \eqref{e:RI-intensity} falls into the realm of \cite{MR2525105} if one includes an unoriented loop (self-edge) of weight $a_{x,x}$ at every vertex. In fact, due to the transitivity of the weight function, our setup corresponds exactly to that of \cite{MR2680403} up to a global rescaling of $u$.

Given a sample $\omega  \in \Omega $ under $\P$, one defines the interlacement set
\begin{equation}
\label{e:def-I-u}
\mathcal{I}^u= \mathcal{I}^u(\omega)=\bigcup_{(w^*,v) \in \omega, \, v \leq u} \text{range}(w^*),
\end{equation}
where, with a slight abuse of notation, in writing $(w^*,v) \in \omega$ we tacitly identify the point measure $\omega$ with its support, a collection of points. For $K\subset \Z^d$, let $\mu^K_u$ denote the point measure on $W^+$ defined as the push-forward of $\omega$ obtained by retaining only the points $(w^*, v)\in \omega$ such that $v\leq u$ and $w^* \in W_K^*$ and mapping them to the forward trajectory $(\in W^+)$ generated by $w^*$ upon first entering $K$. By \eqref{e:RI-intensity}, it follows that $\mu^K_u$ is a Poisson process on $W^+$ with intensity $\nu^K_u= uP_{e_K}[\cdot]$ and by \eqref{e:def-I-u},
\begin{equation}
\label{e:def-I-u-K}
\mathcal{I}^u\cap K =\bigcup_{w \in \mu_u^K} \text{range}(w) \cap K,
\end{equation}
from which \eqref{eq:I_u_intro} readily follows. We abbreviate $\mu^x_u\equiv \mu^{\{x\}}_u$ in the sequel.

The parameter $u$ entering multiplicatively in the intensity measure $\nu^K_u$ governs the number of trajectories entering the picture, and thus controls the density of $\mathcal{I}^u$ (and $\mathcal{V}^u$). It can be more precisely characterised as follows. Let $\ell^u= (\ell_x^u)_{x\in \Z^d}$ denote the field of (discrete) occupation times under $\P$, defined for $x \in \Z^d$ as
\begin{equation}
\label{e:ell-u}
\ell_x^u(\omega) = a_x^{-1} \sum_{(w,v) \in \omega} \sum_{n \in \Z} 1\{ w(n)=x, \, v \leq u\} \,\Big( = a_x^{-1} \sum_{(w,v) \in \mu_{u}^{x }} \sum_{n \geq 0} 1\{ w(n)=x\} \Big).
\end{equation}
One readily finds using \eqref{e:ell-u} and observing that $\mu_{u}^{x }$ has intensity $\frac{u}{g(0)}P_x$ that
\begin{equation}
\label{e:ell-u-mean}
\E[\ell_x^u] =u, \, \text{ for all } x \in \Z^d;
\end{equation}
indeed, with $N$ a Poisson variable having parameter $\frac{u}{g(0)}$ and $X^i$, $i \geq 0$, i.i.d.~independent of $N$ and having the law of $X$ under $P_x$, the expectation on the left-hand side of \eqref{e:ell-u-mean} is seen to equal $a_x^{-1}E[\sum_{1\leq i \leq N}\sum_{n \geq 0}1\{X^i_n=x\}]= E[N] g(x,x)=u$; cf.~\eqref{eq:Green}. In words, \eqref{e:ell-u-mean} asserts that $u$ corresponds to the average number of visits at $x$ by any of the trajectories in the interlacement process at level $u$ (i.e.~comprising the points with label at most $u$). 

\subsection{Random walk on $\mathbb{T}$}\label{sec:denouement}

We conclude this section by deducing our main first main result, Theorem~\ref{t:main1}, from Theorem~\ref{thm:hh1}. The interlacement point process was introduced in \S\ref{subsec:RI} in its lazy version, which is convenient for later purposes. For the sake of proving Theorem~\ref{t:main1}, and within \S\ref{sec:denouement} \textit{only}, we tacitly modify the definition of $\mathcal{V}^u= \Z^d \setminus \mathcal{I}^u$ in \eqref{e:def-I-u} to include all points with label $v \leq \frac{u}{4d}$ (rather than $v \leq u$), which amounts to an inconsequential rescaling. 

 We now recall the following link between $\mathcal{V}^u$ and $\mathcal{V}_N^u$ from \eqref{e:V-RW}. For all $N \geq1$, $0< u_0 < u_1$, $\delta,\varepsilon \in (0,1)$, writing $Q_N= ([0, (1-\delta)N) \cap \mathbb{Z})^d$, there exists a measure $\mathbb Q$ governing the joint law of $(\mathcal{V}^{u(1-\varepsilon)}, \mathcal{V}_N^{u}, \mathcal{V}^{u(1+\varepsilon)} )_{u \in [u_0, u_1]}$ with the property that
\begin{equation}\label{eq:coup-RW-RI}
\mathbb Q\big[  \big(\mathcal{V}^{u(1+\varepsilon)} \cap Q_N \big)\subset \big(\mathcal{V}_N^{u} \cap \pi (Q_N) \big) \subset \big(\mathcal{V}^{u(1-\varepsilon)} \cap Q_N \big), \text { for all } u \in [u_0, u_1] \big] \geq 1 - Ce^{-cN^c},
\end{equation}
for some $c,C \in (0,\infty)$ depending on $u_0, u_1, \delta, \varepsilon $ and $d$ only;
see \cite[Theorem 4.1]{MR3563197} and \cite{PRS23} for a streamlined proof; see also \cite{zbMATH06247831, MR2838338, MR2386070} for earlier results of this kind. In fact the error term on the right-hand side of \eqref{eq:coup-RW-RI} can be made quantitative in all the parameters but its present form will be sufficient for our purposes.

\begin{proof}[Proof of Theorem~\ref{t:main1}]
We start with \eqref{t:main1-sub}. Thus, let $u> u_*$ and pick $\varepsilon = \varepsilon(u) \in (0,1)$ such that $u(1-\varepsilon)> u_*$. From \eqref{eq:subcrit} in Theorem~\ref{thm:hh1},\textit{i)}, it follows that 
\begin{equation}\label{eq:main-2.1}
\P[\lr{}{\mathcal{V}^{u(1-\varepsilon)}}{0}{\partial B_{R}} ] \leq C(u)e^{- R^{\Cl[c]{c**}}}, \text{ for all $R \geq 1$},
\end{equation}
where $\Cr{c**}= \Cr{c**} (u)$. Let $\mathcal{C}_x^u \subset \mathbb{T}$ denote the cluster of $x \in \mathbb{T}$ in $\mathcal{V}_N^u$, cf.~\eqref{e:V-RW}, and fix a reference point $0= \pi (0) $. It follows in turn from \eqref{eq:main-2.1} and \eqref{eq:coup-RW-RI} (applied with $\delta=\frac12$ say) and the isoperimetric inequality on $\mathbb{Z}^d$ that for all $t, N \geq 1$,
\begin{multline}\label{eq:main-2.2}
P[ |\mathcal{C}_0^u| \geq t] \\ \leq P \big[  \text{diam}(\mathcal{C}_0^u) \geq  c (t \wedge N)^{\frac 1d} \big] 
\leq P \big[ \lr{}{\mathcal{V}^{u}_N}{0}{\partial B_{c' (t \wedge N)^{ 1/d}}} \big] \leq C(u)\exp\big\{-  c (t \wedge N)^{\frac{\Cr{c**}}d}\big\}.
\end{multline}
Let $\lambda= \frac{2d}{\Cr{c**}}$ and $t_N = (\log N)^{\lambda}$. Since the event of interest in \eqref{t:main1-sub} can be expressed as $\{ |\mathcal{C}^u_{\textnormal{max}}| \geq t\} = \bigcup_{x \in \mathbb{T}} \{ |\mathcal{C}^u_{x}| \geq t\} $, by a union bound, translation invariance of $P$ and \eqref{eq:main-2.2}, it readily follows that $P[ |\mathcal{C}_0^u| \geq t]$ is bounded for all $N \geq 1$ and $t \geq t_N$ by the right-hand side of \eqref{eq:main-2.2}, up to possibly adapting the constants $C(u)$ and $c$ in the exponent. From this, \eqref{t:main1-sub} is immediate since $P[ |\mathcal{C}_0^u| \geq t]$ vanishes for $t > |\mathbb{T}|=N^d$. 

We now show \eqref{t:main1-super} and begin with the item in the second line. With hopefully obvious notation, for $R \geq1 $, $u>0$ and $z \in \mathbb{T}^d$, we write $\text{E}_z(R,u)$ for the analogue of $\text{Exist}(R,u)$ in \eqref{eq:EXIST} obtained by replacing $\mathcal{V}^u$ by $\mathcal{V}_N^u$ and $B_R$ by $ B_{\mathbb{T}} (z,R) \stackrel{\text{def.}}{=}\pi (B(\hat{z},R))$ with $\hat{z} \in \Z^d$ any point in $\pi^{-1}(\{z\})$, and similarly $\text{U}_z(R,u,v)$, cf.~\eqref{eq:UNIQUE1}, and omit the subscript $z$ when $z=0$. Combining \eqref{eq:barhEXIST} and \eqref{eq:barhUNIQUE} from Theorem~\ref{thm:hh1}\textit{ii)} and the coupling \eqref{eq:coup-RW-RI}, one finds that for all $0< v<u< u_*$ and all $1 \leq R  \leq \frac N{10}$,
\begin{equation}
\label{eq:main-2.3}
 P[\text{E}(R,u)^c] + P\left[\text{U}(R,u,v)^c\right ] \leq  C{ \e}^{-R^c},
\end{equation}
with constants $C,c$ depending on $u,v$; in applying \eqref{eq:coup-RW-RI}, one notes that $\text{U}(R,u,v)$ is increasing in $u$ and decreasing in $v$. For $n \geq 1$, define $\mathbb{T}_n \subset \mathbb{T}$ 
as $\mathbb{T}_n = \pi (K)$ where $K= \{x \in n\mathbb Z^d: B(x,n) \cap \big([0,N) \cap \mathbb Z \big)^d \neq \emptyset\}$. Now, for all  $u< u_*$ and $\varepsilon \in (0,1)$ and $100 \leq t \leq \frac{N}{10}$, as we now explain,
 \begin{equation}
\label{eq:main-2.4}
\bigcap_{z\in \mathbb{T}_{\lfloor  t/{8} \rfloor} } \textstyle \text{E}_z(\frac t{8},u) \cap \text{U}_{z}(\frac t4,u,u(1-\varepsilon)) \subset \left\{\begin{array}{c}\text{all clusters of $\mathcal{C}_{\geqslant t}^u$ are part of}\\ \text{the same cluster of $\mathcal{V}_N^{u(1-\varepsilon)}$ } \end{array}\right\}.
\end{equation}
Indeed, let $C \in \mathcal{C}_{\geqslant t}^u$ be any cluster of $\mathcal{V}_N^u$ having diameter at least $t$. Picking $z_0 \in C$ any of two points of $C$ at distance at least $t$, and $z \in  \mathbb{T}_{\lfloor  t/{8} \rfloor} $ nearest to $z_0$, it follows that $B_{\mathbb{T}}({z}, \frac t{4}) \cap C$ contains a connected component of diameter at least $ \frac{t}{8}$, and $\text{U}_{z}(\frac t4,u,u(1-\varepsilon))$ implies that this component is connected to all clusters of $ B_{\mathbb{T}}(\hat{z}, \frac t{4}) \cap \mathcal{V}_N^u$ of diameter at least $\frac t{20}$ inside $\mathcal{V}_N^{u(1-\varepsilon)}$ (the existence of at least one such cluster is guaranteed by $ \text{E}_z(\frac t{8},u) $). The latter clusters are all connected as $z$ varies on the event on the left of \eqref{eq:main-2.4}, and the inclusion in \eqref{eq:main-2.4} follows. Taking complements, applying a union bound over $z$ and using \eqref{eq:main-2.3}, the second line of~\eqref{t:main1-super} follows upon choosing $t_N = ( \log N)^{\lambda}$ for large enough $\lambda$, first for $t_N \leq t \leq \frac{N}{10}$, and with it for all larger $t$ by monotonicity and since the event in question is empty for $t> \frac N2$.

We now show the first item in \eqref{t:main1-super}. For $u>0$, consider the random set
\begin{equation}
\label{eq:main-2.5}
\tilde{\mathcal{C}}^u \stackrel{\text{def.}}{=} \big\{x \in \mathbb{T} : \lr{}{\mathcal{V}_N^{u}}{x}{\partial B_{\mathbb{T}}(x,{{N}^{\frac12}}})   \big\}.
\end{equation}
Let $ \varepsilon > 0$ and $u < u_*$. Applying \cite[Proposition 2.3]{MR2838338} with $\delta=\frac12$ to the function $f: 2^{\Z^d} \to [0,1]$, $K \mapsto 1\{ \lr{}{K}{0}{ \infty} \}$, and using continuity of $\theta_{\infty} : [0,u_*) \to [0,1]$, cf.~below \eqref{eq:u_*}, one finds $u' = u'(u, \varepsilon) \in (u, u_*)$ such that
\begin{equation}
\label{eq:main-2.6}
\lim_N P\Big[ \Big| \textstyle \frac{|\tilde{\mathcal{C}}^{u'}| }{N^d} -  \theta_{\infty}(u) \Big| > \varepsilon \Big]=0.
\end{equation}
By \eqref{eq:main-2.5}, for each point $x \in \tilde{\mathcal{C}}^{u'}$, one has that $\text{diam}(\mathcal{C}_x^{u'})\geq N^{1/2}$, i.e.~$\mathcal{C}_x^{u'} \in \mathcal{C}^{u'}_{\geqslant \sqrt{N}}$. But on the event $A_N$ given by the right-hand side of \eqref{eq:main-2.4} with $t=\sqrt{N}$, $u'$ in place of $u$ and $\varepsilon= 1-\frac{u}{u'}$, which has probability tending (rapidly) to $1$ as $N \to \infty$ by what we just proved, each of the elements of $\mathcal{C}^{u'}_{\geqslant \sqrt{N}}$ belong to the same cluster of $\mathcal{V}_N^u$. In particular, this applies to $\mathcal{C}_x^{u'}$ for any $x \in \tilde{\mathcal{C}}^{u'}$, whence $\tilde{\mathcal{C}}^u$ is part of the same cluster of $\mathcal{V}_N^u$. Thus, one obtains that $|{\mathcal{C}}^u_{\textnormal{max}}| \geq |\tilde{\mathcal{C}}^{u'}|$ on $A_N$ and the claim follows with \eqref{eq:main-2.6}.
\end{proof}

\section{The models $\mathcal{I}^{\rho}$ and couplings}
\label{sec:truncation}
We now introduce a framework of interlacement processes with trajectories of varying (finite or infinite) length and intensities that will account for all our needs. These include the (local version of) the usual interlacement set $\mathcal{I}^u$, see \eqref{e:def-I-u}, as well as the relevant homogenous finite-range models $\mathcal{J}^{u,L}$ prior to noising, see \eqref{eq:J} and \eqref{eq:I^Lu} below, but more flexibility will be required in due course (cf.~Sections~\ref{sec:toolbox} and~\ref{A:superhard}). The models in the class $\{\mathcal{I}^{\rho}\}$ are parametrized by an intensity measure $\rho$, see \eqref{e:this-is-mu} below, which  in principle allows for (forward) trajectories of any length started anywhere in space. After introducing the framework \S\ref{subsec:setup}, 
we focus in \S\ref{subsec-coup-gen} on one essential tool in relation with these models, consisting of a pair of couplings stated in Theorems~\ref{thm:short_long} and~\ref{thm:long_short} below, which are special cases of the general coupling results developed in the companion article~\cite{RI-III}. These couplings will feature prominently in the remainder of this work, cf.~in particular the proofs of Proposition~\ref{prop:couple_global} and Lemma~\ref{L:superhardcoupling} below.

\subsection{Basic setup}\label{subsec:setup}
We start by introducing the framework and consider a density  
\begin{equation}\label{e:this-is-mu}
\rho:(\mathbb{N}^\ast\cup\{\infty \}) \times \Z^d \to \mathbb{R}_{+}
\end{equation} 
(recall that $\mathbb{N}^\ast = \{1,2\dots\}$). Since the domain of $\rho$ is discrete, we will frequently think of $\rho$ as a measure on $(\mathbb{N}^\ast\cup\{\infty \}) \times \Z^d$ (or on  any of its factors), and not distinguish between the two. For instance, this means that we routinely write $\rho (A, 
x)=\sum_{\ell \in A } \rho(\ell,x)$, for $A \subset \mathbb{N}^\ast$ etc.
Intuitively, $\rho(\ell, x)$ gives the intensity of 
trajectories that have length $\ell$ and start at $x$. 

 Recall the measurable space $(W^+, \mathcal{W}^+)$, where $W^+$ denotes the set of transient nearest-neighbor trajectories, see below \eqref{eq:Wdef}. With $\rho$ as above, we introduce a Poisson point process $\eta$ on the space $(\mathbb{N}^\ast\cup\{\infty \}) 
 \times W^+$ with intensity measure $\nu_\rho$ given by
\begin{equation}
\label{eq:prelim1}
\nu_\rho (\ell, A ) \stackrel{\text{def.}}{=} \sum_{x\in\Z^d}\rho(\ell, x)P_x[X \in A], \text{ for $A \in \mathcal{W}^{+}$}
\end{equation}
and
\begin{equation}
\label{eq:prelim2}
\mathcal{I}^{\rho} \stackrel{\text{def.}}{=} \bigcup_{(\ell,w)\in \eta}w[0,\ell-1].
\end{equation}
In view of \eqref{eq:prelim2}, the label $\ell$ indeed corresponds to the (time-)length of a trajectory in $\mathcal{I}^{\rho}$, as indicated above.

We denote by $\mathbb{P}_{\rho}$ the canonical law of $\eta$. Notice that, for any finite $K \subset \mathbb{Z}^d$,
\begin{equation}
\label{eq:prelim3}
\begin{split}
&\mathcal{I}^u\cap K \stackrel{\text{law}}{=} \mathcal{I}^{\rho_u} \cap K, \text{ for } \rho_u(\ell,x)=u 1_{\infty}(\ell) \, e_K(x),
\end{split}
\end{equation}
for, $\nu_{\rho_u}(\infty,\cdot)= \nu_u^K(\cdot) $, cf.~above \eqref{e:def-I-u-K}.
Similarly, the set $\mathcal{J}^{f, L}$ from \eqref{eq:J} 
is in the realm of \eqref{e:this-is-mu}. Indeed, in view of \eqref{eq:J} and \eqref{eq:prelim1}-\eqref{eq:prelim2}, for any positive measure $\mu$ on $\Z^d$, one has
\begin{equation}
\label{eq:prelim3.1}
\mathcal{J}^{f,L} \stackrel{\text{law}}{=} \mathcal{I}^{\rho}, \, \text{with }\rho (\ell,x)= \textstyle\frac{a_x f(x)}L 1_{L}(\ell)=4d \textstyle\frac{ f(x)}L 1_{L}(\ell), \, x \in \Z^d,
\end{equation}
which specialises to $\mathcal{J}^{u,L}$ with $\mu(x)=u$, $x\in\Z^d$.

Returning to general $\rho$ as in \eqref{e:this-is-mu}, one has the following alternative description of the law of $\mathcal{I}^\rho$ when restricted to a finite set $K \subset \Z^d$, which often comes in handy in practice. For a measure $\rho$ supported on $\mathbb{N}^\ast  \times \Z^d$ and finite $K \subset \Z^d$, defining
\begin{equation}
\label{eq:prelim4}
{\rho_K(\ell,x)}=  \sum_{\ell' \ge 0} E_x \big[ \rho(\ell+\ell', X_{\ell'}) 1_{\{\widetilde{H}_K > \ell'\}} \big] 1_{x\in K},
\end{equation}
one has, with $\geq_{\textnormal{st}} $ denoting stochastic domination,
\begin{equation}
\label{eq:prelim5}
\mathcal{I}^{\rho}\cap K \stackrel{\textnormal{law}}{=}  \mathcal{I}^{\rho_K}\cap K\quad \text{and}\quad \mathcal{I}^{\rho} \geq_{\textnormal{st}}  \mathcal{I}^{\rho_K};
\end{equation}
see \cite[Lemma 3.1]{RI-III} for a proof. In words, \eqref{eq:prelim5} roughly asserts that, if one is only interested in $\mathcal{I}^\rho \cap K$, then one can replace the intensity $\rho$ with a modified version $\rho_K$ that fast forwards the walk until the first time it hits $K$. 

Lastly, an important quantity `average occupation time density' field $\bar \ell_x^\rho= \bar{\ell}_x$, $x \in 
\Z^d$, for $\mathcal{I}^\rho$, which acts as a surrogate for the scalar parameter $u$ in view of \eqref{e:ell-u-mean}. It is defined as  
\begin{multline}
\label{eq:occtime}
 \bar \ell_x  =a_x^{-1}\sum_{k > 0} \int d\nu_\rho(k,\cdot)  \sum_{0 \le \ell < k}1_{\{ X_{\ell} 
=x\}}\\
\stackrel{\eqref{eq:prelim1}, \eqref{eq:P_n-def}}{=} a_x^{-1}\sum_y \sum_{k > 0} \rho(k,y) \sum_{0\leq \ell < k }p_{\ell}(y,x)= \sum_{\ell \geq 0} E_x\big[\textstyle\frac{\rho(\ell+\mathbb{N}^\ast, X_\ell)}{a_{X_\ell}}\big]=\displaystyle \frac1{4d}\sum_{\ell \geq 0} E_x\big[{\rho(\ell+\mathbb{N}^\ast, X_\ell)}\big].
\end{multline}
 In the case of $\mathcal{J}^{u,L}$, i.e.~with $\rho$ the measure appearing in \eqref{eq:prelim3.1}, it is instructive to observe that for all $x \in \Z^d$,
\begin{equation}
\label{eq:loc-time-uL}
\bar{\ell}_x \stackrel{\eqref{eq:occtime}}{=}\frac1{4d}  \sum_{\ell \geq 0} E_x[\rho(\ell+\mathbb{N}^\ast, X_\ell)]= \sum_{\ell \geq 0} \frac{u}{L} 1\{ \ell <L\} =u.
\end{equation}
With a view towards \eqref{e:ell-u-mean}, \eqref{eq:loc-time-uL} suggests that $\mathcal{J}^{u,L}$ is a good local approximation for $\mathcal{I}^u$. Indeed, by \cite[Proposition 3.6]{RI-III}, one knows that for all $K \subset \subset \Z^d$ and $u>0$,
\begin{equation}
\label{e:loc-limit-J-u-L1}
\textstyle \lim_L\P[\mathcal{J}^{u,L} \cap K= \emptyset] = \exp\{-u\text{cap}(K)\}.
\end{equation}
\begin{remark}[Localizing with a mass] \label{eq:mass} The process $\mathcal{J}^{u,L}$ forms the core of our finite-range approximation $\mathcal{I}^{u,L}$, as will soon become clear; cf.~\eqref{eq:I^Lu} below. The more involved definition of $\mathcal{I}^{u,L}$ has purely technical reasons. One may instead attempt to use `massive' interlacements, which simply amount to including to the setup of \S\ref{subsec:RI} a uniform killing measure $\kappa_x=\kappa$ at every vertex $x \in \Z^d$, see \cite[Section 2]{zbMATH07529630} for the general framework, with $\kappa \approx L^{-1}$. This process has been studied by several authors under the name `finitary' interlacements, see \cite{MR3962876,MR4198877,10.1214/22-EJP824}. In fact a result similar in spirit to \eqref{e:loc-limit-J-u-L1} was shown in \cite[Theorem 2]{MR3962876} (where $\kappa$ is replaced by another parameter $T=\frac{2d}{\kappa}$) in the limit where $\kappa \downarrow 0$. The benefit of the localizing using $\kappa$ is to retain a Markovian character for the trajectories. However the spatial range of the process is unbounded and rather poorly concentrated around $\kappa^{-1/2}$.
\end{remark}

\subsection{General couplings}\label{subsec-coup-gen} We now come to the afore mentioned couplings, which will play a key role throughout this article. They correspond to two special cases of the results of \cite{RI-III}. We consider two parameters $u$ and $\bar{\gamma}$ (both real-valued) and integer length scales $L,L' ,K$ with
\begin{equation}\label{e:couplings-params}
\text{$u > 0$, $\bar{\gamma} >10$, $L \geq 2 L' > 1$ such that $\textstyle L'\in \big[\frac{L}{(\log L)^{\bar{\gamma}}}, \frac{ L }{ (\log L)^{10}}\big]$ and $K \in[ 0, 10L]$ } 
\end{equation}
and let
\begin{equation}
\label{e:couplings-n-L}
l = \textstyle \frac L{L'}.
\end{equation}
We further tacitly assume from here on and throughout the remainder of this article that the two scales $L$ and $L'$ in \eqref{e:couplings-params} are integer 
powers of $2$, which is a matter of convenience. In particular, this implies that $L'$ divides~$L$ and $L/2$. 

With the exception of $L$, to which we return shortly, all subsequent results are tacitly understood to hold uniformly for all possible choices of parameters appearing in 
\eqref{e:couplings-params}, which is in force from here onwards. To avoid clumsy notation, constants may implicitly depend on all of $u$, $\bar{\gamma}$ and $d \geq 3$. Their dependence on any other quantity will appear explicitly in our notation. We will (often tacitly) assume that statements hold for $L \geq C$ (with $C$ possibly depending on $u$, $\bar{\gamma}$, and $d$, in accordance with afore convention). Such a restriction is already implicit in \eqref{e:couplings-params}, as needed for the allowed range of $L'$ to contain a power of $2$. Finally, we will frequently encounter `profiles' $f:\Z^d \to \R_+$ satisfying
\begin{equation}
\label{e:ass-f-coup}
\text{$u \ge f(x) \ge (\log L)^{-{\bar{\gamma}}}$ 
for all $x \in B_{K+L}$.}
\end{equation}
Assumption \eqref{e:ass-f-coup} will always appear explicitly and is gathered here for later reference.

For reasons that will become clear later on (in a nutshell, to suit the set up of Proposition~\ref{prop:couple_global} in the next section, cf.~in particular~around \eqref{eq:defgh} and \eqref{eq:defghbar}), we work with a slightly more general configuration than $\mathcal J^{f, 
L}$ allowing for a mixture of trajectories of length $L$ and $\frac L2$. Note however that the following theorem 
holds perfectly true with $f_1\equiv 0$. Recall from \eqref{eq:P_n-def} that $P_n$ denotes the $n$-step 
transition operator for the (lazy) random walk for $n \geq 0$. 

\begin{thm}\label{thm:short_long} For any $L \geq C$ and $f_1, 
f_2:\Z^d \to \R_+$ such that $f = f_1 + f_2$ satisfies \eqref{e:ass-f-coup}, there exists a coupling $\mathbb Q$ of two $\{0, 1\}^{\Z^d}$-valued random 
variables $\mathcal J_1$, $\mathcal J_2$, such that
	\begin{equation}
	\label{eq:short_long}
	\begin{split}
	&\mathcal J_1 \stackrel{\textnormal{law}}{=}  \mathcal 
	J^{\frac12(1 + P_{L/2})f_1, \frac L2} \cup \mathcal J^{f_2, L}, \quad \mathcal J_2 \stackrel{\textnormal{law}}{=} 
	\mathcal J^{(1 - \Cl{C:sprinkle_easy} l^{-1/2})(f 1_{B_K})', L'} \text{ and }\\
	&	\mathbb Q \left[ {\mathcal J}_1  \supset  {\mathcal J}_2  \right] \geq 1 - C L^d \, e^{-c l^{1 / 4}},
	\end{split}
	\end{equation}
where $\mathcal J^{\frac12(1 + P_{L/2})f_1, \frac L2}, \mathcal J^{f_2, L}$ are 
	independent; here (and in accordance with our convention regarding constants) $c = c(u, {\bar{\gamma}})$, $C = C(u, {\bar{\gamma}})$, 
	$\Cr{C:sprinkle_easy} = \Cr{C:sprinkle_easy}(u, {\bar{\gamma}})$ and for any $f:\Z^d \to \R_+$,
	\begin{equation}
	\label{eq:f'}
	 f' \stackrel{\textnormal{def.}}{=} \textstyle l^{-1} \sum_{ 0\leq k < l}P_{kL'} (f). 
	 	\end{equation}
\end{thm}

\begin{proof}
This is an immediate consequence of the more general result \cite[Theorem 4.1]{RI-III}.
\end{proof}

In \eqref{eq:short_long}, longer trajectories (of lengths $L/2$ and $L$, inherent 
to $\mathcal{J}_1$) are used to cover shorter ones (of length $L'$, inherent to 
$\mathcal{J}_2$). Roughly speaking, our second result, Theorem~\ref{thm:long_short} below, 
displays complementary features. Intuitively, it is more difficult to cover long trajectories 
by short ones than vice versa. Correspondingly, the next theorem is more elaborate. It involves a certain `environment' process $\mathcal{I}^\rho$ inherent to both sets $\mathcal J_1$ 
and $\mathcal J_2$ to be coupled. This process corresponds to an incarnation of the obstacle set $\mathcal{O}$, which features prominently in \cite{RI-III}. We return to this shortly.

We first introduce the conditions on the underlying intensity function $\rho$, cf.~\eqref{e:this-is-mu} that will generate the `random environment'. These conditions depend on the parameters in \eqref{e:couplings-params} as well as the threshold $\tilde{u}$ introduced in \eqref{eq:tildeu}, for reasons explained below. Recall the average occupation time density $\bar\ell_x=\bar\ell_x^{\rho}$ from~\eqref{eq:occtime}.

\begin{defn}\label{def:background} The function $\rho: \mathbb{N}^* \times \Z^d  \to \R_+$ is said to 
	satisfy~$(\textnormal{C}_{\textnormal{obst}})$ (with parameters $(u, \bar{\gamma},
	\kappa, L,K)$) if for some dyadic $\hat{L} \in [\frac L8, 8L]$ of the form $\hat L= \hat{L}_1 - \hat{L}_2$, where $\hat{L}_1, \hat{L}_2$ are dyadic integers such that $\hat{L}_2 \ge 8L(\log \frac{L}{8})^{-4\gamma}$, one has that:
\begin{align}
&\:\:\tilde{u}(1+ \kappa) \leq \bar\ell_x^{\rho} \leq u \text{ and } \rho(\mathbb{N}^*, x) \leq \textstyle \frac{4du}{\hat{L}} \text{ for all $x \in B_{K + 5(L\vee \hat{L})}$;} 
\label{eq:disconnect_background0} 
\\[0.6em]
&\begin{array}{l}\text{$\rho(\ell, \cdot) = \rho(\ell, \cdot)1_{\ell \in \{\hat{L}, \hat{L}/2\}}$, $\frac{\hat{L}}{ 8d} \rho(\frac{\hat{L}}{2}, \cdot) =
\big(\frac{1 + P_{\hat{L}/2}}{2} \big)f_1$ and $\frac{\hat{L}}{4d} \rho(\hat{L}, \cdot) = f_2$,  where} \\
\text{$f_1,f_2: \Z^d \to [0,\infty)$ satisfy $u \ge f_1 + f_2 \ge \kappa (\log L)^{-{\bar{\gamma}}}$ for all $x \in B_{K + 5(L\vee \hat{L})}$.
} \end{array}\label{eq:disconnect_background2} 
\end{align}
\end{defn}

 A simple (but for our later purposes insufficient) example of admissible profile $\rho$ for the next definition, which is good to keep in mind, is that of uniform trajectories of length $\hat{L}=L$ (i.e.~$\rho(\ell,x)=\rho_{u,L}(\ell,x)= \frac{u}L 1_{L}(\ell), \, x \in \Z^d$, cf.~\eqref{eq:prelim3.1}) at intensity $u > \tilde{u}$. In particular, Definition~\ref{def:background} implicitly requires that $u> \tilde u$, see \eqref{eq:disconnect_background0}. That is, the environment parametrized by the intensity profile $\rho$ effectively operates at an intensity at least $\tilde{u}$, which is key for the next coupling to work. The reason for this necessity is explained in \cite{RI-III} and briefly reviewed below the next theorem.  Hereinafter we use $\mathscr{C}^{\partial}_S(\mathcal V)$, for $S \subset \Z^d$, to 
denote the connected component of $\partial S$ in $\mathcal{V} \cap S$ where $\mathcal V= \mathcal{V}(\mathcal{J})= \Z^d \setminus \mathcal{J}$ is the vacant set of a configuration $\mathcal J \subset \Z^d$. We further write $(\Omega_{\rho}, \mathcal{A}_{\rho}, \mathbb{P}_{\rho})$ for the canonical space of $\rho$-interlacements.

\begin{thm}[under~\eqref{e:couplings-params}]
\label{thm:long_short}
For any $B = B_N(x)$, 
$f:\Z^d \to \R_+$ satisfying \eqref{e:ass-f-coup} and $\rho$ satisfying 
$(\textnormal{C}_{\textnormal{obst}})$, there exists for each $\omega \in \Omega_{\rho}$ a coupling $\mathbb Q_{\omega}'$ of $\mathcal J^{f1_{B_{K}},L}$ and $ \mathcal J^{(1+\varepsilon)f',L'} $, where $\varepsilon 
=
 \Cr{C:sprinkle_hard}(u, {\bar{\gamma}} ) l^{-\frac12}$ and $f'$ is as in \eqref{eq:f'}, such that, with $\alpha =  1 - e^{-c' l^{1 / 4}} $ and $c' = c'(u, \bar{\gamma})$,
 \begin{multline}\label{e:hard-coup-final3-quenched}
\P_{\rho}\Big[ \,\omega : \,  \mathbb Q_{\omega}' \big[ \mathscr{C}^{\partial}_{S}\big(\mathcal V\big( \mathcal J^{f1_{B_{K}},L} \cup \mathcal{I}^{\rho}(\omega)\big)  \big)  \supset  \mathscr{C}^{\partial}_{S}\big(\mathcal V\big(  \mathcal J^{(1+\varepsilon)f',L'}\cup \mathcal{I}^{\rho}(\omega) \big)  
\big)   \big] \geq \alpha \Big] 
\geq \alpha,
\end{multline}
 provided $L \ge C(u , \bar{\gamma})$, where $S \in \{ B' \text{ a box}: B' = B \text{ or } B' \supset B_{K + 5(L \vee \hat L)}\}$. In particular, there exists a coupling $\mathbb Q'$ of $(\mathcal J^{f1_{B_{K}},L}, \mathcal{I}^\rho, \mathcal J^{(1 + \varepsilon)f', L'})$ such that $\mathcal I^{\rho}$, $\mathcal J^{f1_{B_K}, L}$ are independent, $\mathcal I^{\rho}$, $\mathcal J^{(1 + \varepsilon)f', L'}$ are independent and, defining (under $\mathbb Q'$)
\begin{equation}\label{eq:long-j-i}
\mathcal J_1 
= \mathcal J^{f1_{B_{K}},L} \cup \mathcal{I}^\rho, \quad \mathcal J_2 
= \mathcal J^{(1+\varepsilon)f',L'} \cup \mathcal{I}^\rho, \text{ with } \varepsilon 
=
 \Cl{C:sprinkle_hard}(u, {\bar{\gamma}} ) l^{-\frac12},
\end{equation}
one has, for all $L \ge C(u , \bar{\gamma})$ and $S$ as above,
\begin{equation}\label{e:hard-coup-final3}
\mathbb Q' \big[ \mathscr{C}^{\partial}_{S}\big(\mathcal V(\mathcal J_{1})  \big)  \supset  \mathscr{C}^{\partial}_{S}\big(\mathcal V(\mathcal J_{2})  
\big)   \big] \geq 1 - 
e^{-c' l^{1 / 4}}. 
\end{equation}
\end{thm}

The obstacle set $\mathcal{O}$ mentioned above corresponds to an arrangement of well-separated boxes inside $B$ around which disconnection occurs in $\mathcal{I}^{\rho}$.
This obstacle set $\mathcal{O}$, which lurks behind the coupling \eqref{e:hard-coup-final3}, delimits a region which is out of reach for the boundary clusters $\mathscr{C}^{\partial}_{S}(\cdot \cup \mathcal{I}^{\rho})$. The very fact that $\mathcal{O}$ is typically seen is guaranteed by the condition $(\textnormal{C}_{\textnormal{obst}})$; in particular the fact that $\mathcal{O}$ is defined in terms of disconnection events by  $\mathcal{I}^{\rho}$ (at certain scales) accounts for the pertinence of $\tilde{u}$ in the condition \eqref{eq:disconnect_background0}. 

Finally, the reason for the flexibility in the choice of $S$ in $\mathscr{C}^{\partial}_{S}$ above (rather than just stating \eqref{e:hard-coup-final3} for $S=B$) is technical, and has to do with the possible effect of the noise operator $\mathsf N^L$ (present later, cf.~\eqref{eq:I^Lu}) on boundary clusters; see~the proof of \eqref{eq:coupling_smallbig} 
at the end of \S\ref{subsec:couple-mixed}. As with Theorem~\ref{thm:short_long}, Theorem~\ref{thm:long_short} is a consequence of the results of \cite{RI-III}. 

\begin{proof}
As we now explain, Theorem~\ref{thm:short_long} follows from \cite[Theorem 7.4]{RI-III}. Assuming the conditions of the latter to hold for a moment, the conclusion \eqref{e:hard-coup-final3-quenched} follows immediately from \cite[(7.7)]{RI-III} (recall to this effect that $K \leq 10L$ by assumption in \eqref{e:couplings-params}), with $\bar{\gamma}$ above playing the role of $\gamma$ in \cite{RI-III}. The annealed result \eqref{e:hard-coup-final3} is an immediate consequence of \eqref{e:hard-coup-final3-quenched} upon integrating. 

Thus, the only thing that requires an explanation is the fact that \cite[Theorem 7.4]{RI-III} indeed applies. In comparison with the condition $(\textnormal{C}_{\textnormal{obst}})$ appearing in \cite[Definition 7.3]{RI-III}, the present Definition 3.3 differs in two respects: first the allowed range for $\hat{L}$, whose slightly circumvoluted form has technical reasons, is larger. This extended range is allowed in view of \cite[Remark 7.5,2)]{RI-III}. Second, the lower bound on the mean occupation time in \eqref{eq:disconnect_background0}, which corresponds to $u'$ in \cite[(7.3)]{RI-III}, is parametrized as $u'=u(1+\kappa)$. 
In view of \eqref{eq:tildeu} and by monotonicity of the relevant disconnection event, this implies in particular that the bound \cite[(7.1)]{RI-III} holds with $u=u'(1-(\log L)^{-4})$ whenever $L \geq C(\kappa)$, as required for \cite[Theorem 7.4]{RI-III} to be in force. The remaining conditions appearing in \cite[Theorem 7.4]{RI-III} are plainly satisfied.
\end{proof}

\section{Interpolation}\label{sec:toolbox}

We now prepare the ground for the proof of Theorem~\ref{thm:hh1}. To this effect we introduce in \S\ref{subsec:fr-models} a family $\mathcal{I}^{u,L}$ for $u\in (0, \infty)$ and integer $L > 1$ (cf.~\eqref{e:couplings-params}) of spatially homogenous models
and in \S\ref{Sec:mixedmodelsdef} a further (inhomogenous) family $\mathcal{I}^{u,L}_{\ell}$ for $\ell \in \mathbb{N}/2$. These correspond to precise versions of  \eqref{eq:V-L-informal} and  \eqref{eq:V-L-k-informal}. All models are within the realm of the class $\{\mathcal{I}^{\rho} \}$ introduced in Section~\ref{sec:truncation}.

The models $\mathcal{V}^{u,L}_{\ell}= \Z^d \setminus \mathcal{I}^{u,L}_{\ell}$ will be instrumental in our proof. Roughly speaking, $\mathcal{V}^{u,L}$ are the truncated (finite-range) models that will be used to approximate $\mathcal{V}^u$, and the models $\mathcal{V}^{u,L}_{\ell}$ interpolate between two homogenous models at 
scales $L$ and $2L$ (synonymous of $\ell=0$ and $\ell=\infty$). The models $\mathcal{V}^{u,L}_{\ell}$ come in two variants, $\widetilde{\mathcal{V}}^{u,L}_{\ell}$ and $\overline{\mathcal{V}}^{u,L}_{\ell}$, see \eqref{eq:tildeI_k_noised} and \eqref{eq:barI_k_noised}, corresponding to each 
of two possible comparison inequalities we aim to show. In spite of their specific features, tailored to their later use, it turns out that most statements hold uniformly for both models. We use the notation $\mathcal{V}^{u,L}_{\ell}$ in the sequel, see \eqref{eq:Ibarinclusions}, when referring to either choice. 

After gathering basic properties of the models $\mathcal{V}^{u,L}$ and
$\mathcal{V}^{u,L}_{\ell}$ that will be needed later on, we derive in Proposition~\ref{prop:couple_global} a coupling between $\mathcal{V}^{u,L}_{k}$ and 
$\mathcal{V}^{u,L}_{k+1/2}$ for all $k \in \mathbb{N}$ (for  improved readability, we use $k$ rather than $\ell$ below when the index in question is an integer).
Proposition~\ref{prop:couple_global}, which appears in \S\ref{subsec:couple-mixed}, can be viewed as the main result of the present section. In rough terms, it asserts that $\ell \mapsto \mathcal{V}^{u,L}_{\ell}$ is decreasing with high probability. Its proof relies on a (skillful) application of Theorems~\ref{thm:short_long} and~\ref{thm:long_short}.

\subsection{The homogenous models $\mathcal{I}^{u,L}$} 
\label{subsec:fr-models}
We start by defining and gathering the main properties of the homogenous length-$L$ models $\mathcal{I}^{u,L}$ used in approximating $\mathcal{I}^u$, and first recall their `pure' version
$\mathcal{J}^{u,L}$, cf.~below \eqref{eq:J}. These models are of class $\mathcal{I}^\rho$, as noted in \eqref{eq:prelim3.1}. We now proceed with the model $\mathcal{I}^{u,L}$ alluded to in~\eqref{eq:V-L-informal}. Its definition involves several sources of randomness. We first introduce a certain noise operator. Let $\P^{\mathsf{U}}$ be an (auxiliary) probability measure carrying~i.i.d.~uniform random variables $\mathsf{U}= \{{\mathsf U}_x: x \in \Z^d\}$.
For $\delta \in (0,1)$ and a set $\mathcal{I}\subset \Z^d$, let $\mathsf{N}_{\delta}(\mathcal{I}) \subset \Z^d$ denote the set whose complement has occupation variables
\begin{equation}
\label{eq:noise1}
1_{\{x \notin \mathsf{N}_{\delta}(\mathcal{I})\}} = \begin{cases}
0 &  \mbox{if } {\mathsf U}_x \leq \textstyle \frac{\delta}2\,,\\
1_{\{x \notin\mathcal{I}\}} & \mbox{if } {\mathsf U}_x \in (\textstyle \frac{\delta}2, 1 - \textstyle \frac{\delta}2)\,,\\
1 & \mbox{if } {\mathsf U}_x \geq 1 - \textstyle \frac{\delta}2,\,
\end{cases}
\end{equation}
and for $L \geq 1$, set
\begin{equation}
\label{eq:noise2}
\mathsf{N}^L(\mathcal{I})= \mathsf{N}_{\delta}(\mathcal{I})\big|_{\delta=e^{-L}}.
\end{equation}
For later reference, we note that
\begin{equation}
\label{e:N-monotonicity}
\mathsf{N}^L(\mathcal{I}) \text{ is increasing in $\mathcal{I}$ and decreasing in $\mathsf{U}$,}
\end{equation}
as follows plainly from \eqref{eq:noise1}-\eqref{eq:noise2}. 

Next, for integer $L \geq 1 $, let $\mathcal{B}_L = \{ B(z,L): z \in (2L+1) \Z^d \}$, which forms a set of $L$-boxes tiling $\Z^d$.
We introduce the random field $(\sigma_L(x))_{x \in \mathbb{Z}^d}$ given by
\begin{equation}
  \label{eq:sigma_L}
  \sigma_L(x) = 1+  \sigma_B, \text{ for the unique $B \in \mathcal{B}_L$ such that $x \in B$,}
\end{equation}
 where $(\sigma_B)_{B \in \mathcal{B}_L}$ denotes a family of independent integer-valued random variables having Poisson distribution with intensity one. Adding one in \eqref{eq:sigma_L} will guarantee a certain `ellipticity' for the random sprinkling to be introduced below, which is proportional to the field $ \sigma_L(\cdot)$, cf.~\eqref{eq:I^Lu}. We assume throughout the remainder of \S\ref{subsec:fr-models} that
$\P$ carries the independent fields $(\omega, \tilde{\omega}, \sigma_L,\mathsf U )$, where $\omega, \tilde{\omega}$ are two (independent) Poisson point processes on $\R_+ \times W_+$ with intensity $\nu$ each (see \eqref{eq:mu_intensity}), $\sigma_L(\cdot)$ has distribution specified by \eqref{eq:sigma_L}, and $\mathsf{U}= \{{\mathsf U}_x: x \in \Z^d\}$ are i.i.d.~uniform. With this we define
\begin{equation}
\label{eq:I^Lu}
\mathcal{I}^{u,L} = \mathcal{I}^{u,L} (\omega, \tilde \omega, \sigma_L, \mathsf{U}) =\mathsf{N}^L \big( \mathcal{J}^{u  ,L} (\omega) \cup \mathcal{J}^{\varepsilon \sigma_L ,L} (\tilde \omega) \big),
\end{equation}
where, for some (large) parameter $\gamma> 1$,
\begin{equation}
\label{eq:varepsilon_L}
\varepsilon = \varepsilon_L \stackrel{\text{def.}}{=} (\log L)^{-(\gamma + 5)}
\end{equation}
and $\mathsf{N}^L$ is the noise operator introduced in \eqref{eq:noise1}-\eqref{eq:noise2}, which resamples each occupation variable independently with exponentially small probability in $L$.

We now collect a few fundamental properties of the random sets $\mathcal{I}^{u,L}$ and its complement $\mathcal{V}^{u,L}=\Z^d \setminus \mathcal{I}^{u,L}$. By restricting to an event involving $\mathsf U$ and $\sigma_L$ of  high probability and using~\eqref{e:loc-limit-J-u-L1}, one readily deduces that for all $u > 0$ and any sequence $(u_L)$ with $\lim_L u_L = u$,
\begin{equation}
\label{e:loc-limit-I-u-L}
\text{$\mathcal{I}^{u_L,L}$ under $\P$ converges in distribution to $\mathcal{I}^u$ as $L \to \infty$.}
\end{equation}
One further infers using~\eqref{eq:I^Lu},~\eqref{eq:prelim2}, \eqref{eq:prelim3.1} and \eqref{eq:noise2} that
\begin{equation}
\label{e:finite-range-I-u-L}
\begin{array}{l}
\text{for all $U,V \subset \Z^d$ with $d(U,V)>2L$, the $\sigma$-algebras}\\
\text{$\sigma\big(1_{\{ x\in \mathcal{I}^{u,L} \}} :  x\in U\big)$ and $\sigma\big(1_{\{ x\in \mathcal{I}^{u,L} \}} :  x\in V\big)$ are independent.}\end{array}
\end{equation}
Moreover, as we now explain, for any suitable pair of increasing functions $f,g:\{ 0,1\}^{\Z^d} \to \R$ (such that the following integrals are all well-defined and finite), one has with $\mathcal{V}^{u,L}= \Z^d \setminus \mathcal{I}^{u,L}$,
\begin{equation}\label{eq:FKG_homog}
\E [ f(\mathcal{V}^{u,L}) g(\mathcal{V}^{u,L}) ]
\ge \E [ f(\mathcal{V}^{u,L}) ] \E [ g(\mathcal{V}^{u,L}) ].
\end{equation}
 for all $u> 0$ and $L \geq1$. In particular, \eqref{eq:FKG_homog} applies when $f,g$ are bounded measurable depending on finitely many coordinates (and increasing), which will be sufficient for our purposes.

To deduce \eqref{eq:FKG_homog}, one first conditions on $(\sigma_L, \mathsf{U})$ under $\P$, and writing $\E_{(\sigma_L, \mathsf{U})} $ for the corresponding conditional expectation, one applies \cite[Theorem~20.4]{LastPen2018} to infer that \eqref{eq:FKG_homog} holds with $\E_{(\sigma_L, \mathsf{U})} $ in place of $\E$ everywhere. Upon integrating the resulting inequality over $(\sigma_L, \mathsf{U})$, one applies the FKG-inequality for independent random variables to the right-hand side to obtain \eqref{eq:FKG_homog}, noting that the relevant quantities $\E_{(\sigma_L, \mathsf{U})} [ f(\mathcal{V}^{u,L}) ]$ and $\E_{(\sigma_L, \mathsf{U})} [ g(\mathcal{V}^{u,L}) ]$ are decreasing functions of $(\sigma_L, \mathsf{U})$, which follows on account of \eqref{eq:I^Lu} and \eqref{e:N-monotonicity}.

Lastly, by definition, the law of $\mathcal{I}^{u,L}$ is translation invariant and ergodic with respect to lattice shifts on $(2L+1)\Z^d$. One can then introduce critical parameters $u_*^L, u_{**}^L$ and $\tilde{u}^L$ akin to \eqref{eq:u_*}, \eqref{eq:u_**} and \eqref{eq:tildeu}, with $\mathcal{V}^{u,L}= \Z^d \setminus \mathcal{I}^{u,L}$ in place of $\mathcal{V}^u$ everywhere. The following  can then be regarded as a consequence of the combined results of \cite{DumRaoTas17b} and \cite{GriMar90} (generalized in the latter case to percolation models with finite-range dependence).

\begin{prop}
\label{prop:sharptruncated}
For all $d\geq3$ and $L \geq 1$,
\begin{equation}\label{e:I-u-L-sharp}\tilde{u}^{ L} = u^{ L}_* =
u^{L}_{**}.
\end{equation}
\end{prop}
\begin{proof}
Referring to Section~6 of \cite{DCGRS20}, as explained therein starting with the paragraph above (6.7) until the end of that section, the claim \eqref{e:I-u-L-sharp} follows at once if the properties listed as (a)-(e) at the beginning of that section can be verified, with the occupation field $(1_{\{x\in \mathcal{V}^{u,L}\}})_{x\in \Z^d, u > 0}$ in place of $(\omega_h(x))_{x\in \Z^d, h \in \R}$. Inspection of the argument in \cite{DCGRS20} reveals that the translation invariance inherent to (a) (lattice symmetry) can be replaced by the coarse one noted above, the invariance of the law of $\mathcal{V}^{u,L}$ under lattice rotations and coordinate reflections is also plain. Property (b) (positive association) is precisely \eqref{eq:FKG_homog}. Similarly, \eqref{eq:noise1} and \eqref{eq:noise2} yield that $\frac{e^{-L}}{2} \leq \P[x\notin \mathcal{V}^{u,L}| \sigma (1_{\{x\in \mathcal{V}^{u,L}\}}, y \neq x)] \leq 1- \frac{e^{-L}}{2}$, whence~(c) (finite energy).
Property (d), which refines \eqref{e:finite-range-I-u-L}, requires a more detailed explanation and is postponed for a few lines. Finally, inspection of the proof of Lemma 6.1 in \cite{DCGRS20} reveals that the final sprinkling property (e) follows at once if one shows that for all $u<u'$ and $x \in \Z^d$,
\begin{equation}\label{e:sprinkling-I-u-L}
\P[x\in \mathcal{V}^{u,L}| x\notin \mathcal{V}^{u',L}] \geq c(u,u', \gamma, L)
\end{equation}
(indeed \eqref{e:sprinkling-I-u-L} is the only model-specific input, which appears towards the end of the proof of Lemma 6.1 in \cite{DCGRS20}; the rest of the proof shows how to deduce Property (e) from it). In the present case \eqref{e:sprinkling-I-u-L} is readily obtained by splitting $\mathcal{J}^{u',L}= \mathcal{J}_1 \cup \mathcal{J}_2$ into the sum of independent processes $ \mathcal{J}_1 \stackrel{\text{law}}{=}\mathcal{J}^{u,L}$ and $ \mathcal{J}_2 \stackrel{\text{law}}{=} \mathcal{J}^{u'-u,L}$, cf.~\eqref{eq:I^Lu}, while requiring that ${\mathsf U}_x \in (\textstyle \frac{\delta}2, 1 - \textstyle \frac{\delta}2)$.

Finally, the existence of a bounded-range i.i.d.~encoding postulated by Property~(d) touches on all of the randomness $(\omega, \tilde{\omega}, \sigma_L, \mathsf{U})$ entering the definition of $\mathcal{V}^{u,L}$ under $\P$, and can be obtained as follows. One conveniently generates $\mathcal{V}^{u,L}$ using an independent~family of Poisson processes $(\omega_{x,k}^L, \tilde{\omega}_{x,k}^L)_{x\in \Z^d, \,k \in \N}$ (along with $\sigma_L, \mathsf{U}$) rather $(\omega, \tilde{\omega})$. The process $\omega_{x,k}$ (and $\tilde{\omega}_{x,k}$) has values in $[0,1] \times W^+_f$ (recall that $W^+_f$ refers to the space of finite $\Z^d$-valued trajectories) and finite intensity $ \nu_{x,k}^L= du \times P_x[X_{[0,L-1]} \in \cdot]$, cf.~\eqref{eq:prelim2} (which, by defining properties of a Poisson process can be sampled from a Poisson random variable of parameter $1$ and an i.i.d.~collection of $([0,1] \times W^+_f)$-valued random variables having law $\nu_{x,k}^L$). In view of \eqref{eq:sigma_L}, \eqref{eq:I^Lu} and \eqref{eq:noise1}, the required locality in the way $(\sigma_L, \mathsf{U})$ enter the construction of $\mathcal{V}^{u,L}$ is plain, leading overall to an i.i.d.~encoding having range $2L$.
\end{proof}

\subsection{The inhomogenous models $\mathcal{I}^{u,L}_{\ell}$}
\label{Sec:mixedmodelsdef}

We now introduce the mixed (inhomogenous) models that interpolate between homogenous ones. We will mostly work with these mixed models, which will permeate our proofs. Slight care is needed due to the presence of several sources of randomness, as we now detail. We first revisit the random sprinkling parameter $\sigma_L$ in \eqref{eq:sigma_L} inherent to the homogenous model $\mathcal{I}^{u,L} 
$ in \eqref{eq:I^Lu}, and start by adding spatial structure to it. Recall the definition of the paving $\mathcal{B}_L$ of $\Z^d$ by boxes of radius $L$, see above \eqref{eq:sigma_L}.
Let $\Sigma = \{ \sigma^{B'}_{B}: B, B' \in \mathcal{B}_L, L > 1\}$ be a family of independent integer-valued random variables having the following distribution: if $B=B_L(z)$ and $B'= B_L(z')$, and with $\text{Poi}(\lambda)$ denoting the Poisson distribution with mean $\lambda>0$,
    \begin{equation}
      \label{eq:sigma_bb'}
      \sigma_B^{B'} \stackrel{\text{law}}{=}
      \begin{cases}
        1 & \text{if $z = z'$,}\\
        \text{Poi} \Big(\frac{1}{\Cl[c]{c:sum_sprinkling}} \big( \frac{2L + 1}{|z-z'|_{\infty}}\big)^{d+1} \Big) & \text{if $z \neq z'$},
      \end{cases}
    \end{equation}
    where $\Cr{c:sum_sprinkling} = \sum_{z \in \mathbb{Z}^d \setminus \{0\}} |z|_{\infty}^{-(d + 1)}$, so that the parameters of the Poisson random variables sum up to one as $z$ ranges over $(2L + 1) \mathbb{Z}^d$. In words, when $z=z'$, \eqref{eq:sigma_bb'} simply means that $ \sigma_B^{B'}$ is constant and equal to one.
    We then define the random fields
    \begin{equation}
      \label{eq:sigma}
      \sigma^{B'}(x) = \sigma^{B'}_{B} \quad \text{ and } \quad \sigma_L(x) = \sum_{B' \in \mathcal{B}_L} \sigma^{B'}(x) \quad \text{ for } x \in B \, (\in \mathcal{B}_L).
    \end{equation}
    Observe that the random field $\sigma_L$ defined by \eqref{eq:sigma} has the same distribution as the one previously defined in \eqref{eq:sigma_L}.
    However, the explicit spatial decomposition in \eqref{eq:sigma} will allow us to regard the sprinkling $\sigma_L(x)$ as being added/removed in steps, one for each box $B' \in \mathcal{B}_L$, with decreasing intensities as $B'$ moves away from $x$. When referring to $\sigma_L(\cdot)$ from here on, we always mean the random field declared by \eqref{eq:sigma} (rather than \eqref{eq:sigma_L}, which is equal in law).
    
Throughout the remainder of this article, we assume that $\P$ carries a Poisson process $\widetilde{\omega}$ on $\{1,2,3,4\} \times (\R_+ \times W^+)$ having intensity $c \times \nu$, with $c$ denoting counting measure on $\{1,2,3,4\}$ and $\nu$ as in \eqref{eq:mu_intensity}. The (big) process $\widetilde{\omega}$ gives rise to the processes $\omega_i$, $1\leq i \leq 4$ on $\R_+ \times W^+$, obtained by retaining all points in $\widetilde{\omega}$ whose first label is $i$, and forgetting this label. Thus, $\omega_i$, $1\leq i \leq 4$, are independent Poisson processes with intensity $\nu$ each, i.e.~each $\omega_i$ is a copy under $\P$ of the process $\omega$ introduced around \eqref{eq:mu_intensity}.

Along with $\widetilde{\omega}$, the measure $\P$ is assumed to carry the family $\Sigma$ introduced above \eqref{eq:sigma_bb'} and the i.i.d. family $\mathsf{U}$, see above \eqref{eq:noise1}. All  fields $\widetilde{\omega}$, $\Sigma$, $\mathsf{U}$ are independent under $\P$, and we will frequently abbreviate by
\begin{equation}\label{e:disorder-mixed}
\tilde{\sigma}= (\Sigma,\mathsf{U})
\end{equation}
the `disorder' variables (under $\P$). We write $\mathcal{F}_{\tilde{\sigma}}$ for the sigma-algebra generated by these random variables and $\P_{\tilde{\sigma}}$ for the corresponding quenched law, so $\P[\, \cdot \,]= E^{\tilde \sigma}[\P_{\tilde \sigma}[\, \cdot \,]]$ with $E^{\tilde \sigma}$ denoting 
averages with respect to $\tilde \sigma$. 

Without further ado, we now define under the above measure $\P$ two sequences of models, $(\widetilde{\mathcal{I}}^{u,L}_{\ell})_{\ell \in \N/2}$ and $(\overline{\mathcal{I}}^{u,L}_{\ell})_{\ell \in \N/2}$, see \eqref{eq:tildeI_k_noised} and \eqref{eq:barI_k_noised} below, which will interpolate between the 
two homogenous models, see 
\eqref{eq:Itildeinclusions} and \eqref{eq:Ibarinclusions}. We first introduce, for any $L > 1$, and $\delta_2 > 0$ to be chosen soon (see \eqref{eq:delta-2-mixed}), recalling the transition operator $P_n$ from \eqref{eq:P_n-def}, two functions $\tilde g, \tilde h: \Z^d\to \R_+$ with
\begin{equation}
\label{eq:defgh}
\textstyle \tilde g \stackrel{{\rm def}.}{=} \big(\frac{1+P_L}{2} \big) 1_{B_{2L}} , \quad
\tilde h \stackrel{{\rm def}.}{=}    1_{B_{2L}} + \delta_2  1_{B_{6L}}, 
\end{equation}
Let $\{\tilde{x}_0,\tilde{x}_1,\ldots \}$ denote an arbitrary enumeration of $(4L+1)\,\Z^d$ and define $\widetilde{B}_j = B_{2L}(\tilde{x}_{j})$ for $j \ge 0$ as well as $\widetilde{D}_k = \bigcup_{j<k} \widetilde{B}_j$. Note that $\widetilde{\mathcal{B}}= \{  
\widetilde{B}_j : j \geq 0\}$ is an enumeration of $\mathcal B_{2L}$, cf.~the paragraph 
preceding \eqref{eq:sigma_L}, and consists of boxes which pave $\Z^d$. With 
$\tilde g$, $\tilde h$ as in  \eqref{eq:defgh}, we set
\begin{equation}
\label{eq:defgh^k}
\tilde g_k(\cdot) \stackrel{{\rm def}.}{=} \sum_{j \geq k } \tilde g(\cdot -\tilde{x}_{j}),\quad \tilde h_k(\cdot) \stackrel{{\rm def}.}{=} 
\sum_{j < k } \tilde h(\cdot-\tilde{x}_{j}),
\end{equation}
and define, with $\sigma^{\widetilde{B}_j}$ and $\sigma_L$ as in \eqref{eq:sigma} and $\varepsilon_L$ as in \eqref{eq:varepsilon_L},
\begin{equation}
\label{eq:def_fs}
\tilde r_k  \stackrel{{\rm def}.}{=} \varepsilon_L \sigma_L  1_{\widetilde{D}_k^c},\quad \tilde s_k \stackrel{{\rm def}.}{=} \varepsilon_{2L} \sum_{j < k}  \sigma^{\widetilde{B}_j}.
\end{equation}
As opposed to $\tilde g_k$, $\tilde h_k$, the functions $\tilde r_k$ and $\tilde s_k$ are random and declared under $\P$. Recalling that $\mathbb P$ further carries $\omega_1$, 
$\omega_2$, $\omega_3$ and $\omega_4$, independent copies of $\omega$, independent from $\Sigma$, we then let (see \eqref{eq:J} for notation)
\begin{equation}
\label{eq:tildeI_k}
\begin{split}
&\widetilde{\mathcal{J}}^{u,L}_k (\omega_1, \omega_2, \omega_3,  \omega_4, \Sigma) \stackrel{{\rm def}.}{=} {\mathcal{J}}^{\, u \tilde g_k,L}(\omega_1)  \cup  {\mathcal{J}}^{\,  u \tilde r_k,L}(\omega_2)\cup  {\mathcal{J}}^{\, u \tilde h_k,2L} (\omega_3) \cup   {\mathcal{J}}^{\, u \tilde s_k,2L}(\omega_4)\\
&\widetilde{\mathcal{J}}^{u,L}_{k+\frac12}  (\omega_1, \omega_2, \omega_3,  \omega_4, \Sigma)  \stackrel{{\rm def}.}{=} {\mathcal{J}}^{\, u \tilde g_{k+1},L}  (\omega_1)  \cup  
{\mathcal{J}}^{\, u \tilde r_{k+1},L}  ( \omega_2)  \cup  {\mathcal{J}}^{\, u \tilde h_{k+1},2L} (\omega_3)  \cup   {\mathcal{J}}^{\, u \tilde s_k,2L} (\omega_4).  
\end{split}
\end{equation}
In words, on account of \eqref{eq:defgh^k}, when passing from $\widetilde{\mathcal{J}}^{u,L}_k$ to $\widetilde{\mathcal{J}}^{u,L}_{k+\frac12}$, the combined effect of $$({\mathcal{J}}^{\, u \tilde g_k,L}(\omega_1), {\mathcal{J}}^{\, u \tilde h_k,2L}(\omega_3)) \to ({\mathcal{J}}^{\, u \tilde g_{k+1},L}(\omega_1), {\mathcal{J}}^{\, u \tilde h_{k+1},2L}(\omega_3))$$ is to replace the relevant length-$L$ trajectories starting in $\widetilde{B}_k$ by trajectories of length $2L$, with slightly higher intensity, cf.~\eqref{eq:defgh}. In view of \eqref{eq:def_fs}, a similar fate occurs to the randomly sprinkled trajectories, parametrized by $\tilde{r}_k$ and $\tilde{s}_k$, but whereas the `removal' inherent to $\tilde{r}_k \to \tilde{r}_{k+1}$ happens during step $k \to (k+\frac12)$, the `addition'  $\tilde{s}_k \to \tilde{s}_{k+1}$ is performed separately as $(k+\frac12)\to (k+1)$.

We are only one step away from defining $\widetilde{\mathcal I}^{u, L}_{\,\cdot}$, which is just a noised 
version of $\widetilde{\mathcal J}^{u, L}_{\,\cdot}$. Recall the noise operator 
$\mathsf{N}^L(\cdot)$ from \eqref{eq:noise1}--\eqref{eq:noise2} in Section~\ref{subsec:fr-models}, which involves an independent family $\mathsf{U}$ of i.i.d.~uniform random variables, carried by $\P$ within our setup; see above \eqref{e:disorder-mixed}. Now let
\begin{equation}
	\label{eq:tildeI_k_noised}
\widetilde{\mathcal{I}}^{u,L}_{\ell} \cap \widetilde B_j \stackrel{{\rm def.}}{=} \begin{cases}
	\mathsf{N}^{2L}(\widetilde{\mathcal{J}}^{u,L}_{\ell} \cap \widetilde B_j) &  \mbox{if } j < \lceil \ell \rceil \text{, and}\\
\mathsf{N}^{L}(\widetilde{\mathcal{J}}^{u,L}_{\ell} \cap \widetilde B_j) & \mbox{if }  j \ge \lceil \ell \rceil.
\end{cases}
\end{equation}
We also extend the definitions \eqref{eq:tildeI_k} and \eqref{eq:tildeI_k_noised} when $k=\infty$ (whereby $\tilde{g}_\infty = \tilde{r}_{\infty}=0$).
As we now explain, our goal is to eventually couple $\widetilde{\mathcal{I}}^{u,L}_{k}$ and $\widetilde{\mathcal{I}}^{u,L}_{k+1}$ 
in such a way that $\widetilde{\mathcal{I}}^{u,L}_{k} \subset \widetilde{\mathcal{I}}^{u,L}_{k+1}$ with 
sufficiently high probability; see \S\ref{subsec:couple-mixed} below. The following notation will be useful: for any pair of functions $f_1, f_2: 
\mathbb Z^d \to \mathbb R_{+}$ with $f_1(x) \leq f_2(x)$ for all $x \in \Z^d$, define, for $\omega\in \Omega_+$ (cf.~\eqref{eq:J}),
\begin{equation}
\label{eq:two_fnc}
\mathcal J^{[f_1, f_2], L}(\omega) \stackrel{{\rm def}.}{=} \bigcup_{ \substack{ (v, w)\in \omega:\\ \frac {4d}Lf_1(w(0)) < v \le \frac {4d}L f_2(w(0))}}w[0, L-1], 
\end{equation}
so that, with a view towards \eqref{eq:J}, one has $ \mathcal{J}^{f,L}=\mathcal J^{[0, f], L}$.
With this notation, by \eqref{eq:tildeI_k} \eqref{eq:tildeI_k_noised} and \eqref{eq:I^Lu}, and for a suitable constant $\Cl{sprinkling1}=\Cr{sprinkling1}(d)$, one obtains that
\begin{equation}
\label{eq:Itildeinclusions}
\begin{split}
&\widetilde{\mathcal{I}}^{u,L}_0  \stackrel{\text{law}}{=}  {\mathcal{I}}^{u,L}  , \quad  \widetilde{\mathcal{I}}^{u,L}_{\infty} \leq_{\textrm{st.}}{\mathcal{I}}^{u(1 + {\Cr{sprinkling1}\delta_2}) ,2L}\\
&\widetilde{\mathcal{J}}^{u,L}_{k+1} = \tilde{\mathcal{J}}^{u,L}_{k+\frac12} \cup  {\mathcal{J}}^{\, [u\tilde s_{k} , u\tilde s_{k+1}],2L} (\omega_4)
\ \big( \supset  \widetilde{\mathcal{J}}^{u,L}_{k+\frac12} \big),\,\, \widetilde{\mathcal{I}}^{u,L}_{k+1} \supset    \widetilde{\mathcal{I}}^{u,L}_{k+\frac12};
\end{split}
\end{equation}
the second inclusion (involving $\widetilde{\mathcal{I}}$'s) in the second line follows from the  first one and the monotonicity property \eqref{e:N-monotonicity} of $\mathsf{N}^L(\mathcal I)$ in 
$\mathcal I$. We now choose 
\begin{equation}\label{eq:delta-2-mixed}
\delta_2 =  {\delta_1}/{ \Cr{sprinkling1}},  \text{ with } \delta_1= \delta_1(L) = (\log L)^{-4},
\end{equation}
so that the total sprinkling in the first line of \eqref{eq:Itildeinclusions} amounts to $\delta_1$.

We now define a second sequence $(\overline{\mathcal{I}}^{u,L}_{\ell})_{\ell \in \N/2}$. Akin to \eqref{eq:defgh}, we introduce two functions
\begin{equation}\label{eq:defghbar}
\textstyle \bar g =  1_{B_L}, \quad \bar h = \big( \frac{1+P_L}{2} \big) 1_{B_L}  + \delta_2 1_{B_{6L}}
\end{equation}
(with domain $\Z^d$). Let $\{ \bar{x}_0,\bar{x}_1,\ldots\}$ denote an enumeration of $(2L+1)\Z^d$, 
 consider the boxes
$\overline{B}_j = B_L(x_j)$ for $j \geq 0$ and let $\overline{D}_k =
\bigcup_{j<k} \overline{B}_j$.  Likewise, $\overline{\mathcal{B}} = \{\overline{B}_j 
: j \geq0 \}$ is an enumeration of $\mathcal{B}_{L}$, which paves $\Z^d$. In the same vein 
as in \eqref{eq:defgh^k}, we then define 
\begin{equation}
\label{eq:defghbar^k}
\bar g_k(\cdot) = \sum_{j \geq k } \bar g(\cdot - \bar{x}_j), \quad \bar h_k(\cdot) = \sum_{j < k } \bar h(\cdot-\bar{x}_j).
\end{equation}
Furthermore, let
\begin{align}
\label{eq:def_fsbar}
\bar r_k =\varepsilon_{2L} \sigma_{2L} 1_{\overline{D}_k^c},  \quad \bar s_k = \varepsilon_{L} \sum_{j < k}  \sigma^{\overline{B}_j}. 
\end{align}
Now under $\mathbb P$, define $ \overline{\mathcal{J}}^{u,L}_{\ell} = \overline{\mathcal{J}}^{u,L}_\ell(\omega_1, 
\omega_2, \omega_3, \omega_4, \Sigma)$, $ \ell \in \mathbb{N}/2$, by setting
\begin{equation}
\label{eq:barI_k}
\begin{split}
\overline{\mathcal{J}}^{u,L}_k &=  \mathcal{J}^{\, u\bar g_k,2L}(\omega_1) \cup  
\mathcal{J}^{\,  u \bar r_k,2L}(\omega_2) \cup  \mathcal{J}^{\, u\bar 
h_k,L}(\omega_3)  \cup    \mathcal{J}^{\,  u \bar s_k,L}(\omega_4)  \\
\overline{\mathcal{J}}^{u,L}_{k+\frac12} &= \mathcal{J}^{\, u\bar g_{k+1},2L}(\omega_1) 
\cup  \mathcal{J}^{\,  u \bar r_{k+1},2L}(\omega_2) \cup  \mathcal{J}^{\, u\bar 
h_{k+1},L}(\omega_3) \cup    \mathcal{J}^{\,  u \bar s_k,L}(\omega_4),
\end{split}
\end{equation}
for $k \in \mathbb{N}$, which naturally extends to $k = \infty$.
Similarly as in \eqref{eq:tildeI_k_noised}, we then set
\begin{equation}
	\label{eq:barI_k_noised}
	\overline{\mathcal{I}}^{u,L}_{\ell} \cap \overline B_j \stackrel{{\rm def.}}{=} \begin{cases}
		\mathsf{N}^{L}(\overline{\mathcal{J}}^{u,L}_{\ell} \cap \overline B_j) &  \mbox{if } j < \lceil \ell \rceil \text{, and}\\
		\mathsf{N}^{2L}(\overline{\mathcal{J}}^{u,L}_{\ell} \cap \overline B_j) & \mbox{if }  j \ge \lceil \ell \rceil,
	\end{cases}
\end{equation}
for $\ell \in \N/2$. Note that, with these definitions (possibly enlarging the value of $\Cr{sprinkling1}$ in \eqref{eq:delta-2-mixed}),
\begin{equation}
\label{eq:Ibarinclusions}
\begin{split}
&\overline{\mathcal{I}}^{u,L}_0 \stackrel{\text{law}}{=} 
{\mathcal{I}}^{u,2L}, \quad  \overline{\mathcal{I}}^{u,L}_{\infty} \leq_{\text{st.}} {\mathcal{I}}^{u(1 + \delta_1) ,L} \\
&\overline{\mathcal{J}}^{u,L}_{k+1} = \overline{\mathcal{J}}^{u,L}_{k+\frac12} \cup \mathcal{J}^{\, [u \bar s_k, u \bar s_{k+1}],L}(\omega_4) \ \big( \supset  \overline{\mathcal{J}}^{u,L}_{k+\frac12} \big),\,\, 
\overline{\mathcal{I}}^{u,L}_{k+1} \supset  \overline{\mathcal{I}}^{u,L}_{k+\frac12}.
\end{split}
\end{equation}
As in \eqref{eq:Itildeinclusions}, the two sets on the right-hand side in the first equality of the second line are independent. Most of the arguments in the sequel can be performed in a unified manner for both 
sequences $(\widetilde{\mathcal{I}}^{u,L}_{\ell})_{\ell \in \N/2}$ and 
$(\overline{\mathcal{I}}^{u,L}_{\ell})_{\ell \in \N/2}$. Accordingly we often write
\begin{equation}
\label{eq:barequalstilde}
{\mathcal{I}}_{\ell}^{u,L}= \,  \widetilde{\mathcal{I}}_{\ell}^{u,L} \text{ or } \overline{\mathcal{I}}_{\ell}^{u,L}\,\, \text{(similarly for ${\mathcal{J}}_{\ell}^{u,L}$)}
\end{equation}
and, correspondingly, drop tildes and bars from all notations, e.g.~writing $g_k,h_k,B_k$ etc. While doing so we tacitly agree 
that the statements are true in either case. We write $\mathcal V_{\ell}^{u, L}$ for the 
vacant set $\Z^d \setminus \mathcal I_{\ell}^{u, L}$  and often omit the superscripts $u,L$ altogether, so that ${\mathcal{I}}_{\ell}={\mathcal{I}}_{\ell}^{u,L}$, ${\mathcal{V}}_{\ell}={\mathcal{V}}_{\ell}^{u,L}$. Finally we let $\mathbb{L}= \{ 
x_0,x_1,\dots \}$ correspond to the enumeration of the sublattice $(4L+1)\Z^d$, 
resp.~$(2L+1)\Z^d$ depending on the choice of ${\mathcal{I}}_{\,\cdot}^{u,L}$, so $x_j$ denotes the center $B_j$ for $j \geq 0$. 

\medskip
We now collect a few consequences of the above setup that will be used repeatedly in the sequel. 
These includes basic measurability and independence properties as well as an
FKG-type inequality due to the positive association inherent to the above models, which follows a similar pattern as in the homogenous case; cf.~\S\ref{subsec:fr-models}. Recall that $\P_{\tilde \sigma}$ refers to the quenched law given the realization of the `disorder' $\tilde 
\sigma =(\Sigma, \mathsf U)$, see \eqref{e:disorder-mixed}, which is simply the law of a Poisson process.

 We then  write $\widetilde{\omega}^L$, $L 
\geq 1$, for the point measure induced by $\widetilde{\omega}$ (introduced above \eqref{e:disorder-mixed}) which only retains $(j, v, (w(n))_{0 \leq n  \leq 2L-1})$ for any point 
$(j,v,w) \in \widetilde{\omega}$; here, with hopefully obvious notation, $j \in \{1,\dots, 4\}$, $v \geq 0$ and $w \in W^+$. For $K \subset \Z^d$, let $\widetilde{\omega}_K^L$ refer to the 
process of points $(j, v, w') \in \text{supp}(\widetilde{\omega}^L)$ with $\text{range}(w') \cap K \neq 
\emptyset$ and define $(\widetilde{\omega}^L_{K})^c = \widetilde{\omega}^L - \widetilde{\omega}^L_K 
$. Thus, $\widetilde{\omega}^L_K$,  $ (\widetilde{\omega}^L_{K})^c$ form independent Poisson processes under $\P_{\tilde{\sigma}}$. Similarly we write $\Sigma^L \stackrel{{\rm def}.}{=} \{\sigma^{B'}_B: B, B' \in \widetilde{\mathcal{B}} \mbox{ or } \overline{\mathcal{B}}\} = \Sigma_{K}^L \cup  (\Sigma_{K}^L)^c$ (recall \eqref{eq:sigma_bb'}) where $\Sigma_{K}^L \stackrel{{\rm def}.}{=} \{\sigma^{B'}_B: B, B' \in \widetilde{\mathcal{B}} \mbox{ or } 
\overline{\mathcal{B}}, B \cap K \ne \emptyset\}$ and $(\Sigma_{K}^L)^c \stackrel{{\rm def}.}{=} \Sigma^L 
\setminus \Sigma_{K}^L $ are independent families of random variables. 
With regards to the family $\mathsf U$ of independent random variables, which is indexed by points in $\Z^d$, we partition $\mathsf U = \mathsf 
U_K \cup \mathsf U_{K^c}$ where $\mathsf U_K \stackrel{{\rm def}.}{=} \mathsf U_{|K}$ for $K \subset \Z^d$.

It follows from the previous definitions that for any $\ell \in \mathbb{N}/2$, the set $\mathcal V_{\ell}^{u, L} \cap K = \mathcal V_{\ell} \cap K$ (i.e.~$\sigma(1\{ x \in \mathcal V_{\ell}\}, x\in K)$) is 
measurable under $\P_{\tilde \sigma}$ relative to $\widetilde{\omega}_K^{L}$. Hence, in particular, it is independent of 
$(\widetilde{\omega}_K^{L})^c$. Since the length of any trajectory corresponding to a 
point in ${\rm supp}(\widetilde{\omega}^L)$ is at most $2L$, we similarly get that $\mathcal V_{\ell} \cap K$ is measurable under $\mathbb P$ with respect to $(\widetilde{\omega}_K^{L}, \Sigma_{K_{2L}}^L, \mathsf U_K)$, where  
$K_{r}$ denotes the $r$-neighborhhood of $K$, see \S\ref{s:not} for notation. Now recall that the diameter of 
any box in $\widetilde{\mathcal B} \cup \overline{\mathcal B}$ is at most $4L$. Hence for any 
two sets $U, V \subset \Z^d$ with $d(U, V) \ge 10 L$, 
the corresponding collections of random processes 
$(\widetilde{\omega}_U^{L}, \Sigma_{U_{2L}}^L, \mathsf U_U)$ and $(\widetilde{\omega}_V^{L}, \Sigma_{V_{2L}}^L, \mathsf U_V)$ are in fact independent. As a consequence of this observation, the set ${\mathcal{V}}_\ell$ has the following finite-range property. 
Under both $\P$ and $\P_{\tilde \sigma}$, for any $\ell \in \mathbb{N}/2$,
\begin{equation}
\label{eq:barequalstildeRANGE}
\text{${\mathcal{V}}_\ell \cap U$ and ${\mathcal{V}}_\ell \cap V$ are independent for any 
$U,V \subset \Z^d$ with $d(U,V) \geq 10 L$},
\end{equation}
by which we mean that the $\sigma$-algebras generated by $\{ 1\{ x \in \mathcal V_{\ell}\}: x\in U\}$, and $\{1\{ x \in \mathcal V_{\ell}\} : x\in V\}$ are independent.
 
We call an event $A \in \mathcal F^L =\sigma(\widetilde{\omega}^L)$ \textit{increasing in 
$K\subset \Z^d$} under $\P_{\tilde \sigma}$ if $\widetilde{\omega} \in A^c$ (the complement of $A$) implies $\widetilde{\omega}' \in A^c$, where 
$\widetilde \omega' \geq \widetilde \omega$ is obtained by addition of points $(j, v, w)$ with $\text{range}(w) \cap K \neq \emptyset$. In particular, any event $A$ that is increasing w.r.t.~$\mathcal V_{\ell} \cap K$, i.e.~that satisfies $A \in \sigma(1\{ x \in \mathcal V_{\ell}\} : x\in \Z^d)$ and increasing in the variables $1\{ x \in \mathcal V_{\ell}\}$, $x \in K$, is also increasing in $K$.  

\begin{lemma}[$L \geq 1, \, K\subset \Z^d$] 
\label{lem:FKG}
If both $A \in \mathcal{F}_K^L \stackrel{{\rm def}.}{=} \sigma(\widetilde{\omega}^L_K)$ and $B \in 
\mathcal{F}^{L}$ are increasing in $K$ under $\P_{\tilde \sigma}$, then
\begin{equation}
\label{eq:FKG}
\P_{\tilde \sigma}[A\cap B] \geq \P_{\tilde \sigma}[A] \cdot \P_{\tilde \sigma}[B].
\end{equation}
\end{lemma}

\begin{proof}
First notice that $\widetilde{\omega}_K^L$ is a Poisson process under 
$\P_{\tilde \sigma} [\, \cdot \, |\, (\widetilde{\omega}_K^{L})^c]$ and both $A$ and $B$ are increasing events 
under this law. Hence, applying \cite[Theorem~20.4]{LastPen2018}, it follows that 
\begin{equation*}
\P_{\tilde \sigma}[A \cap B \,|\, (\widetilde{\omega}_K^{L})^c ] \ge \P_{\tilde \sigma}[A \,|\, (\widetilde{\omega}_K^{L})^c ] \cdot \P_{\tilde \sigma}[B \,|\, (\widetilde{\omega}_K^{L})^c ] = \P_{\tilde \sigma}[A] \cdot \P_{\tilde{\sigma}}[B \,|\, (\widetilde{\omega}_K^{L})^c ]
\end{equation*}
where in the second step we used the fact that $A \in \mathcal F_K^L$. The lemma now follows 
by taking expectations with respect to $\P_{\tilde{\sigma}}$ on both sides.
\end{proof}

\subsection{Couplings between $\mathcal{V}^{u,L}_{\ell}$} \label{subsec:couple-mixed}
In this subsection we further develop the toolbox for the models $\mathcal{V}^{u,L}_{\ell}= \Z^d \setminus \mathcal{I}^{u,L}_{\ell}$ 
we will be working with. As opposed to the properties gathered above, the more elaborate features of the models $\mathcal{V}^{u,L}_{\ell}$ used later (such as suitable connectivity estimates) all hinge on a central coupling statement, see Proposition~\ref{prop:couple_global} below, which roughly 
speaking establishes that the models $\widetilde{\mathcal{I}}^{u,L}_{k}$ and 
$\overline{\mathcal{I}}^{u,L}_{k}$ are both `increasing in $k$' (with high probability). In accordance with the 
convention from \eqref{eq:barequalstilde}, the following statement(s) apply to both 
$\widetilde{\mathcal{I}}^{u,L}_{k}$ and $\overline{\mathcal{I}}^{u,L}_{k}$.
Recall $\tilde{u}$ from \eqref{eq:tildeu}, $\mathscr{C}^{\partial}_S(\cdot)$ for $S \subset \Z^d$ from above Theorem~\ref{thm:long_short} and $\delta_2$ from \eqref{eq:delta-2-mixed}.

\begin{prop}[Coupling between $\mathcal I_{k}^{u, L}$ and $\mathcal I_{k+1/2}^{u, L}$]
	\label{prop:couple_global}
Let $k \in \mathbb{N}$, $L > 2$ be a dyadic integer, $\delta \in(0,\frac12)$, $u \in  [\tilde{u}(1 + \delta), \frac{\tilde{u}}{\delta}]$ and $\bar \gamma > 10$; let $B = B_N(x)$ for some 
$N \ge 1$ 
and $x \in \Z^d$. There exists a coupling 
	$\mathbb{Q}_{\tilde \sigma}$ 
of two $\{0, 1\}^{\Z^d}$-valued random variables $(\widehat{\mathcal{I}}_k, \widehat{\mathcal{I}}_{k + \frac12})$ with the following properties. Letting $U_k = B(x_{k}, {\the\numexpr\couprad\relax} L)$ and $V_k =
B(x_{k}, {\the\numexpr\couprad+\rangeofdep\relax}L)$, one has that
	\begin{align}
		&\text{$\widehat{\mathcal{I}}_k$ (resp.~$\widehat{\mathcal{I}}_{k+\frac 12}$) has the same law as ${\mathcal{I}}_k^{u,L}$ 
			(resp.~${\mathcal{I}}_{k + \frac 12}^{u,L}$) under $\P_{\tilde \sigma}$;}\label{eq:marginal0}\\
		& \widehat{\mathcal{I}}_k \cap  U_k^c = \widehat{\mathcal{I}}_{k+\frac12} \cap  U_k^c; \label{eq:easy1}\\
		&\textstyle\sigma\big(1\{x \in \widehat{\mathcal{I}}_j \}, x\in  V_k^c, j=k,k+\frac12\big) \text{ is indep.~from } \sigma\big(1\{x \in \widehat{\mathcal{I}}_j \}, x\in   U_k, j=k,k+\frac12\big). \label{eq:easy2}
	\end{align}
	Furthermore, for suitable $\Cr{c:coupling100}=\Cr{c:coupling100}(\delta, \bar \gamma) > 0$,
	\begin{equation}
	\mathbb Q_{\tilde \sigma} \big[ \mathscr{C}^{\partial}_{B}\big(\widehat{\mathcal V}_k \big)  \supset  
	\mathscr{C}^{\partial}_{B}\big(\widehat{\mathcal V}_{k+\frac12})  \big)  \big] \geq (1 -  e^{ - 
		\Cl[c]{c:coupling100}(\log L)^{\bar \gamma}})  1_{G(x_k)},	
	\label{eq:coupling_smallbig}
	\end{equation}
	where \begin{equation}\label{def:Gxk}
	\textstyle G(x_k) = G^{u, L}(x_k) \stackrel{{\rm def}.}{=} 
\big\{ u( r_k \vee  s_k)\vert_{V_k} \leq \frac{\delta_2(2L)}{100},\, \mathsf U_{V_k} \in [\frac{e^{-L}}{2}, 1 - \frac{e^{-L}}{2}] \big\} 	 \, (\in \mathcal F_{\tilde \sigma}).\end{equation} 
\end{prop}

\begin{remark}
	\begin{enumerate}
		\label{remark:annealed_coupling}
		\item[1)] (Annealed coupling). By \eqref{eq:marginal0}, it is clear that the measure $\mathbb Q[\cdot] 
		\stackrel{{\rm def}.}{=} E^{\tilde \sigma}[\mathbb Q_{\tilde \sigma}[\cdot]]$ constitutes a coupling between $\mathcal I_k^{u, L}$ and $\mathcal I_{k + 1/2}^{u, L}$ under $\P$. 
		We now aim to lift the (conditional) independence property~\eqref{eq:easy2} to an independence property under $\mathbb Q$. To this end, note that 
		for any event $A$ measurable under $\P_{\tilde \sigma}$ relative to $\sigma(1\{x \in \widehat{\mathcal{I}}_j \}, x\in   U, j=k,k+\frac12)$ with $U = U_k$, the quantity $\mathbb Q^{\tilde \sigma}[A]$ is 
		measurable with respect to $(\Sigma_{U_{2L}}^L, \mathsf U_U)$. This essentially follows from the discussion leading up to \eqref{eq:barequalstildeRANGE} in view of \eqref{eq:marginal0}. Similarly, for 
		any event $A$ measurable relative to $\sigma(1\{x \in \widehat{\mathcal{I}}_j \}, x\in   V^c, 
		j=k,k+\frac12)$, where $V=V_k$, the quantity $\mathbb Q^{\tilde \sigma}[A]$ is measurable with respect 
		to $(\Sigma_{(V^c)_{2L}}^L, \mathsf U_{V^c})$. However, since $d(U, 
		V^c) \ge \rangeofdep L$, the collections $(\Sigma_{U_{2L}}^L, \mathsf U_S)$ and 
		$(\Sigma_{(V^c)_{2L}}^L, \mathsf U_{V^c})$ are independent and consequently, under $\mathbb{Q}$,
		\begin{equation} \label{eq:easy2'}\tag{\ref*{eq:easy2}'}
		\begin{array}{l}\sigma\big(1\{x \in \widehat{\mathcal{I}}_j \}, x\in  V_k^c, j=k,k+\tfrac{1}{2}\big) \text{ is indepen-}\\
		\text{dent from } 
		\sigma\big(1\{x \in \widehat{\mathcal{I}}_j \}, x\in   U_k, j=k,k+\tfrac{1}{2}\big).
		\end{array}
		\end{equation}
		We now determine the corresponding annealed coupling error implied by \eqref{eq:coupling_smallbig}.
		First recall (see, e.g. 
		\cite[pp.~97-98]{mitzenmacher2017probability}) that the tails of a Poisson variable $X$ with mean $\lambda$ satisfy
		\begin{equation}\label{eq:Poisson_tailbnd}
		\P[ X \ge \lambda + x] \vee \P[ X \le \lambda - x] \le e^{- c x}, 
		\end{equation}
		valid for all $x \ge \lambda / 2$.	Since the Poisson variables $\sigma^B$'s (see \eqref{eq:sigma_L}) have mean 1 and, on account of \eqref{eq:varepsilon_L} and \eqref{eq:delta-2-mixed}, one has that $(\varepsilon_L^{-1}\wedge \varepsilon_{2L}^{-1}){\delta_2(L)} 
		\ge c(\log L)^{\gamma},$
		it follows combining \eqref{eq:def_fs}/\eqref{eq:def_fsbar}, the tail estimate \eqref{eq:Poisson_tailbnd} and a simple union bound, that \begin{equation}\label{eq: Gxk_bnd}
			P^{\tilde \sigma}[G(x_k)] \le e^{-c(\delta) (\log L)^\gamma} + e^{-L}\, |V| \le e^{-c(\delta) (\log L)^\gamma}.
			\end{equation}
		Together with \eqref{eq:coupling_smallbig} applied with $\bar \gamma = \gamma$, this yields that 
		for some $\Cl[c]{c:coupling101} = \Cr{c:coupling101}(\delta, \gamma) > 0$,
		\begin{equation}
		\mathbb Q \big[\mathscr{C}^{\partial}_{B}\big(\widehat{\mathcal V}_k \big) \supset  \mathscr{C}^{\partial}_{B}\big(\widehat{\mathcal V}_{k+\frac12})\big] \geq 1 -  e^{- \Cr{c:coupling101}(\log L)^{\gamma}} \label{eq:coupling_smallbig_annealed}.
		\end{equation}
		\item[2)] (Perfect coupling in distant regions). Observe for later use that \eqref{eq:easy1} readily implies that the event in 
		\eqref{eq:coupling_smallbig} (as well as in \eqref{eq:coupling_smallbig_annealed}) has full 
		measure whenever $B \subset U_k^c$.
	\end{enumerate}
\end{remark}

In the next two subsections, we harvest consequences of the above couplings 
which are easily isolated and will be used throughout. We conclude this subsection 
with the proof of Proposition~\ref{prop:couple_global}. 
The proof combines Theorems~\ref{thm:short_long} and~\ref{thm:long_short}, which will be applied repeatedly.

\begin{proof}[Proof of Proposition~\ref{prop:couple_global}]	
	We only give the proof for $\mathcal I ^{u, L}_{\,\cdot} = \widetilde{\mathcal{I}}^{u,L}_{\,\cdot}$ as the proof 
	for $\overline{\mathcal{I}}^{u,L}_{\,\cdot}$ is similar. It is further sufficient to show the conclusions for $L \geq C( \delta, \bar \gamma)$, which will often tacitly be assumed in the sequel. Indeed, the remaining cases for $L$ follow by choosing $\mathbb Q_{\tilde\sigma} = \mathbb P_{\tilde\sigma}$ (cf.~around~\eqref{e:disorder-mixed}), which is readily seen to satisfy \eqref{eq:marginal0}-\eqref{eq:easy2}, as well as
\eqref{eq:coupling_smallbig} upon possibly adjusting $\Cr{c:coupling100}$.	
	
	Recall the configurations 
	$\widetilde{\mathcal{J}}^{u,L}_{\ell}$'s from \eqref{eq:tildeI_k}, which are in fact identical to 
	$\widetilde{\mathcal{I}}^{u,L}_{\ell}$ inside $V_k$ on the event $G(x_k)$ by \eqref{eq:tildeI_k_noised}, \eqref{def:Gxk} and \eqref{eq:noise1}-\eqref{eq:noise2}. 
	In the course of the proof, we will define a sequence of configurations $\{\widetilde {\mathcal J}_{k, a}; 0 \le a \le A\}$ 
	interpolating between $\widetilde{\mathcal{J}}_{k}$  
	and $\widetilde{\mathcal J}_{k + 1/2}$ where $\widetilde {\mathcal J}_{k, a+1}$ is obtained from 
	$\widetilde {\mathcal J}_{k, a}$ by replacing a small fraction of the $L$-trajectories with 
	$2L$-trajectories; see \eqref{eq:intermed_decomp} below. For each 
	$0 \leq a < A$, we will then produce a coupling $\mathbb Q_a$ between the laws of $\widetilde {\mathcal 
		J}_{k, a}$ and $\widetilde {\mathcal J}_{k, a+1}$ by means of Theorems~\ref{thm:short_long} 
	and \ref{thm:long_short}. We will eventually arrive at our final coupling $\mathbb Q_{\tilde{\sigma}}$ by concatenating the 
	couplings $\mathbb Q_a$'s, in the manner of \cite[Section~2.3]{RI-III}.
	
We now introduce the relevant intermediate configurations (in law). Let $A \stackrel{{\rm def}.}{=} \lceil \frac{2}{\delta} \rceil$ so that $\frac{u}{A} \le 
	\frac{u\delta}2$. For integer $a$ with $0 \leq a \leq A$, with $\tilde g, 
	\tilde h$ as in  \eqref{eq:defgh} and $\tilde g_k, \tilde h_k$ as in \eqref{eq:defgh^k}, let
	\begin{equation}
	\label{eq:defgh^ka}
\textstyle	\tilde g_k^a(\cdot) = \tilde g_{k+1}(\cdot) + \big(1 - \frac aA\big) \, \tilde g(\cdot -\tilde{x}_{k}),\quad 
	\tilde h_k^a(\cdot)= \tilde h_{k} + \frac aA\, \tilde h(\cdot-\tilde{x}_{k}),
	\end{equation}
so that $\tilde g_k^0= \tilde g_k$, $\tilde g_k^A=\tilde g_{k+1}$ and similarly for $\tilde h_k^a$. In the same vein, for $\tilde{r}_k$ as in \eqref{eq:def_fs}, define $\tilde r_k^a  = \tilde r_{k+1} + (1 - \frac{a}{A}) \varepsilon_L \sigma_L  1_{\widetilde{B}_k}.$ We now introduce under $\P_{\tilde{\sigma}}$ the configurations (cf.~\eqref{eq:tildeI_k})
	\begin{align}\label{eq:intermediate_config}
		\widetilde{\mathcal{J}}_{k}^{a} (\omega_1, \omega_2, \omega_3,  \omega_4) \stackrel{{\rm def}.}{=} {\mathcal{J}}^{\, u \tilde g_k^a,L}(\omega_1)  \cup  {\mathcal{J}}^{\,  u \tilde 
			r_k^a, L}(\omega_2)\cup  {\mathcal{J}}^{\, u \tilde h_k^a,2L} (\omega_3) \cup   {\mathcal{J}}^{\, u \tilde s_k,2L}(\omega_4),
	\end{align}
so that $\widetilde{\mathcal{J}}_{k}^0 = \widetilde{\mathcal{J}}_{k}$ and $\widetilde{\mathcal{J}}_{k}^{A} = \widetilde{\mathcal{J}}_{k + 1/2}$. For any $0 \le a < A$, one extracts from both $\widetilde{\mathcal{J}}_{k}^a$ and $\widetilde{\mathcal{J}}_{k}^{a+1}$  a joint `bulk' contribution $\widetilde{\mathcal{J}}_{k}^{a,1}$ by decomposing
\begin{align}\label{eq:intermed_decomp}
\widetilde{\mathcal{J}}_{k}^{a} = \widetilde{\mathcal{J}}_{k}^{a,1} \cup \widetilde{\mathcal{J}}_{k}^{a,2}, \:\:\: 
\widetilde{\mathcal{J}}_{k}^{a+1} = \widetilde{\mathcal{J}}_{k}^{a,1} \cup \widetilde{\mathcal{J}}_{k}^{a,3},
\end{align}
	where (under $\P_{\tilde{\sigma}}$)
	\begin{equation}\label{def:I'ka}
	\begin{split}
	&\widetilde{\mathcal{J}}_{k}^{a,1}  \stackrel{{\rm def}.}{=} {\mathcal{J}}^{\, u \tilde g_k^{a+1},L}(\omega_1) 
	\cup {\mathcal{J}}^{\,  u \tilde r_k^{a+1}, L}(\omega_2)\cup  {\mathcal{J}}^{\, u 
		\tilde h_k^a,2L} (\omega_3) \cup   {\mathcal{J}}^{\, u \tilde s_k,2L}(\omega_4), \\
	&\widetilde{\mathcal{J}}_{k}^{a,2} \stackrel{{\rm def}.}{=} \mathcal J^{[u \tilde g_k^{a+1},\, u \tilde g_k^a], L}(\omega_1) \cup \mathcal J^{[u \tilde r_k^{a+1},\, u \tilde 
		r_k^a], L}(\omega_2),\\
		&\widetilde{\mathcal{J}}_{k}^{a,3} \stackrel{{\rm def}.}{=} \mathcal J^{[u \tilde h_k^{a},\, u \tilde h_k^{a+1}], 2L}(\omega_3). \end{split}
	\end{equation}
 The bulk contribution $\widetilde{\mathcal{J}}_{k}^{a,1}$ is further split as
	\begin{equation}\label{eq:I'ka_decomp}
		\widetilde{\mathcal{J}}_{k}^{a,1} =  \widetilde{\mathcal{J}}_{k}^{a, 1,1} \cup \widetilde{\mathcal{J}}_{k}^{a, 1,2},
	\end{equation}
	where, with $\tilde U = B(\tilde x_k, 15L)$ and using that $(\tilde g_{k}^{a+1})_{| \tilde U^c}= (\tilde g_{k+1})_{| \tilde U^c}$ and $(\tilde h_{k}^{a})_{| \tilde U^c}= (\tilde h_{k})_{| \tilde U^c}$,
	\begin{equation}\label{eq:I'kaj}
	\begin{split}
	&\widetilde{\mathcal{J}}_{k}^{a, 1,1}  \stackrel{{\rm def}.}{=} {\mathcal{J}}^{\, u (\tilde g_{k}^{a+1})_{| \tilde U}, \,L}(\omega_1) \cup  {\mathcal{J}}^{\, u 
		(\tilde h_{k}^a)_{| \tilde U},\, 2L} (\omega_3), \text{ and} \\
		&\widetilde{\mathcal{J}}_{k}^{a, 1,2} \stackrel{{\rm def}.}{=} {\mathcal{J}}^{\, u (\tilde g_{k+1})_{| \tilde U^c}, \,L}(\omega_1) 
	\cup {\mathcal{J}}^{\,  u \tilde r_k^{a+1}, \, L}(\omega_2)\cup  {\mathcal{J}}^{\, u 
		(\tilde h_{k})_{| \tilde U^c}, \,2L} (\omega_3) \cup   {\mathcal{J}}^{\, u \tilde s_k,2L}(\omega_4).
	\end{split}
	\end{equation}
By construction, the random sets $\widetilde{\mathcal{J}}_{k}^{a, 1,1}, \widetilde{\mathcal{J}}_{k}^{a, 1,2}$, $\widetilde{\mathcal{J}}_{k}^{a,2}$ and $\widetilde{\mathcal{J}}_{k}^{a,3}$  are  independent under $\mathbb P_{\tilde \sigma}$. 

The set $\widetilde{\mathcal{J}}_{k}^{a, 1,1}$ is destined to play the role of `background' configuration $\mathcal I^\rho$ appearing in the context of Theorem~\ref{thm:long_short}.
Recall the definition of $\mathcal I^\rho$ from \eqref{eq:prelim2}. Clearly, $\widetilde{\mathcal{J}}_{k}^{a, 1,1}$ has the same law (under $\mathbb P_{\tilde \sigma}$) as 
	$\mathcal I^\rho$ (under $\P_{\rho}$), where $\rho: \mathbb{N}^\ast \times \Z^d  \to \R_+$ is given by
	\begin{equation}\label{def:v} \textstyle
	\rho(\ell, x) = 
	\frac{4d u}{L}\big( 1_L(\ell) \tilde g_k^{a+1}(x) + \frac12 1_{2L}(\ell) \, \tilde 
	h_k^a(x)\big)1_{\tilde{U}}(x).\end{equation}
Recall the obstacle condition $(\textnormal{C}_{\textnormal{obst}})$ from Definition~\ref{def:background}. In what follows, $\rho$ is said to satisfy $(\textnormal{C}_{\textnormal{obst}})(x)$ for some $x \in \Z^d$ if $ \rho_0$ given by $\rho_0(\ell,y)= \rho (\ell,x+y)$ for all $\ell \in \mathbb{N}^\ast$ and $y \in \Z^d$ satisfies $(\textnormal{C}_{\textnormal{obst}})$. We first isolate the following result.
	
\begin{lemma} \label{L:Cobst-ver} For all $\delta \in(0,\frac12)$, $u \in  [\tilde{u}(1 + \delta), \frac{\tilde{u}}{\delta}]$, $\bar{\gamma} > 1$ and $L \geq C(\delta , \bar{\gamma})$, the density $\rho$ in \eqref{def:v} satisfies $(\textnormal{C}_{\textnormal{obst}})(\tilde x_k)$ with parameters $(4 \tilde{u}\delta^{-1}, \bar{\gamma},
	\kappa= \frac\delta4, L ,K =3L)$.
\end{lemma}

\begin{proof}[Proof of Lemma~\ref{L:Cobst-ver}] Verifying $(\textnormal{C}_{\textnormal{obst}})(\tilde x_k)$ amounts to checking conditions ~\eqref{eq:disconnect_background0} and \eqref{eq:disconnect_background2} inside the box $B_{K + 5(L\vee \hat{L})}(\tilde x_k)$ in place of $B_{K + 5(L\vee \hat{L})}$. We choose $\hat L = 2L$. Since $K = 3L$ we thus need \eqref{eq:disconnect_background0}-\eqref{eq:disconnect_background2} to hold in $B_{3 L + 5 \times 2L}(\tilde x_k) = B_{13L}(\tilde x_k)$. By \eqref{eq:occtime}, $\bar \ell_x^\rho$ only depends on $\rho(k,y)$ if $|y-x|_1 < k$ and the relevant values of $k$ are $L$ and $2L$, we can safely drop the indicator function $1_{\tilde{U}}$ in \eqref{def:v} as long as we only deal with the quantities $\bar \ell_x^\rho$ and $\rho(\ell, x)$ for $ x\in B_{13 L}(\tilde x_k)$, as for the verification of \eqref{eq:disconnect_background0}-\eqref{eq:disconnect_background2} with the above choices. 
	
We first verify condition \eqref{eq:disconnect_background2} for $\rho$ as in \eqref{def:v}. It follows from the definitions of $(\tilde g, \tilde h)$, $(\tilde g_k, \tilde 
	h_k)$ and $(\tilde g_k^a, \tilde h_k^a)$, see~\eqref{eq:defgh}, \eqref{eq:defgh^k} and 
	\eqref{eq:defgh^ka} respectively, that $\rho$ satisfies~\eqref{eq:disconnect_background2} with $\hat L = 2L$:
	indeed, $2L \rho(2L, x) = 4d\,f_2$ and  $L \rho(L, x) = 4d\,(\frac{1 + P_{L}}{2})f_1$ with 
	$f_1 + f_2 \in \left[u(1 - \frac1A),  u (1 + \delta_1)\right]$ for all $x \in B_{13 L}(\tilde x_k)$; see in particular around 
	\eqref{eq:Itildeinclusions} regarding the upper bound $u (1 + \delta_1)$.  
		
	Let us now verify~\eqref{eq:disconnect_background0}. The previous paragraph also implies that $\rho(\N^\ast, x) \le 4d\frac{4u}{\hat 
	L}$ with $\hat L = 2L$ for all $L \geq C$. We now proceed to verify the required lower bound on $\bar \ell_x^\rho$. It will be convenient to use the notation $ \bar \ell_x(\rho) = \bar \ell_x^\rho$ in the sequel. Let $\tilde{\rho} \leq \rho$ be obtained by replacing $\tilde h_k^a$ in
\eqref{def:v} by $\tilde f_k^a$, defined similarly as $\tilde h_k^a$ but with the function 
	$\tilde f \stackrel{{\rm def}.}{=} 1_{B_{2L}} (\le \tilde h)$ playing the role of $\tilde h$. By monotonicity and linearity of $ \rho \mapsto \bar \ell_{\cdot}(\rho)$ (see~\eqref{eq:occtime}), it follows by definition of $\tilde g_k^{a+1}$ and $\tilde 
	f_k^a$ (see \eqref{eq:defgh^ka}) that for all $x \in B_{13L}(\tilde x_k)$,
	\begin{align*}\bar \ell_x(\rho)  \geq \bar \ell_x(\tilde \rho) &=  \bar \ell_x\big( \textstyle 4d\frac u{2L}1_{2L}(\cdot) \, \tilde f_k^a(\cdot)\big) + \bar \ell_x\big( 4d \frac u L1_L(\cdot)\, \tilde g_k^{a+1}(\cdot)\big)\\
		&= \sum_{n < k } \bar \ell_x(\tilde \rho_{n;  \tilde f}) +  \frac a A \, \bar \ell_x(\tilde \rho_{k;  \tilde f}) +  \Big(1 -  \frac{a+1}{A}\Big)\bar \ell_x(\tilde \rho_{k;  \tilde g}) + \displaystyle \sum_{n > k} \bar \ell_x(\tilde \rho_{n;  \tilde g})
	\end{align*} 
	where $\tilde \rho_{n; \tilde f}(x, \ell) = 4d \frac 
	u{2L}1_{2L}(\ell)\tilde f(x-\tilde{x}_{n})$ and $\tilde \rho_{n; \tilde g}(x, \ell) =  4d \frac u{L}1_{L}(\ell)\tilde 
	g(x-\tilde{x}_{n})$. We will prove shortly that for all $x \in B_{13L}(\tilde x_k)$ and every $n \ge 0$,
	\begin{equation}\label{eq:equality-localtimes}
	\bar \ell_x(\tilde \rho_{n; \tilde g}) = \bar \ell_x(\tilde \rho_{n; \tilde f}),
	\end{equation} 
	 which, together with the previous display, then implies \begin{equation}\label{eq:local_time_lower_bnd}
		\begin{split}
		\bar \ell_x(\rho) &\ge \big(1 - A^{-1}\big) \sum_{n \geq 0} \bar 
		\ell_x(\tilde \rho_{n; \tilde f}) = \big(1 - A^{-1}\big) \sum_{n \geq 0} \bar \ell_x\big( \textstyle4d\frac 
		u{2L}1_{2L}(\cdot)1_{B_{2L}(\tilde{x}_{n})}(\cdot) \big) \\ 
		&= \big(1 - A^{-1}\big) \bar \ell_x\big( \textstyle 4d\frac u{2L}1_{2L}(\cdot)\big) 
		\stackrel{\eqref{eq:loc-time-uL}}{=} \big(1 - A^{-1}\big) u \ge u \big(1 - \frac 
		{\delta}{2}\big) \ge \tilde{u}\big(1 + \frac{\delta}{4} \big), 
		\end{split} \end{equation}
	where in the first step of the second line, we used the fact that $B_{2L}(\tilde x_n)$'s form a partition of $\Z^d$ (see above \eqref{eq:defgh^k}) and in the last two steps we used the definition of $A$ 
	as well as the fact that $u \ge \tilde u(1 + \delta)$; we also used $\delta < 1 / 2$ in 
	deriving the last inequality. 
	Thus $\rho$ satisfies the lower bound condition in~\eqref{eq:disconnect_background0} with $\kappa = \delta/ 4$ provided \eqref{eq:equality-localtimes} holds, which we verify now. Due to translation 
	invariance, it suffices to check this for $n = 0$. Invoking \eqref{eq:occtime} again, we can write
	\begin{equation}\label{eq:occup_time_compute}
		\begin{split}
		\bar \ell_x(\tilde \rho_{0; \tilde g}) &= \bar \ell_x \big(4d \textstyle \frac u{L}1_{L}(\cdot)\tilde 
		g(\cdot)\big) \displaystyle\stackrel{\eqref{eq:occtime}}{=} \frac u{L}\sum_{\ell \geq 0} E_x \Big[ \, \sum_{\ell' > \ell}1_{L}(\ell') \tilde g(X_\ell)\,\Big]	= \frac u{L}\sum_{0 \le \ell < L} E_x [\tilde g(X_\ell)]
		\\ &= \frac u{L}\sum_{0 \le \ell < L} (P_\ell\tilde g)(x) \stackrel{\eqref{eq:defgh}}{=} \frac u{L}\sum_{0 \le \ell < L} P_\ell \Big(\frac{1+P_L}{2}\Big)\tilde f(x) 
		= \frac u{2L}\sum_{0 \le \ell < 2L} P_\ell\tilde f(x) = \bar \ell_x(\tilde \rho_{0;  \tilde f})
		\end{split}
	\end{equation}
	where in the penultimate step we applied the semigroup property and we suppressed in the final step the details which are similar to the first four steps. It remains to verify the upper bound condition on $\bar \ell_x(\rho)$  in \eqref{eq:disconnect_background0}, which is straightforward.  Indeed, since $\bar \ell\big(4d\frac u{L}1_{L}(\ell)\big) = u$ for all $u$ and $L$ by \eqref{eq:loc-time-uL}, and 
	$\max(\tilde g_k(x), \tilde h_k(x)) \le 2$  for all $x \in 
	B_{{\the\numexpr\couprad+\rangeofdep\relax}L}(\tilde x_k)$ and $L \ge C$, we get immediately in view of \eqref{def:v} that $\bar \ell_x(\rho) \le 4 u \le 4\tilde{u}\delta^{-1}$ for all $x \in 
	B_{{\the\numexpr\couprad+\rangeofdep\relax}L}(\tilde x_k)$.
\end{proof}

We resume the proof of Proposition~\ref{prop:couple_global}. We are now ready to define the desired coupling $\mathbb Q_a$ of $\widetilde {\mathcal 
		J}_{k, a}$ and $\widetilde {\mathcal J}_{k, a+1}$, for $0 \leq a < A$, via a combination of 
	Theorems~\ref{thm:long_short} and \ref{thm:short_long}. In view of \eqref{eq:intermed_decomp}, this amounts to replacing $\widetilde{\mathcal{J}}_{k}^{a,2}$ by $\widetilde{\mathcal{J}}_{k}^{a,3}$. This will involve an intermediate (fragmented) configuration $\mathcal{K}$, comprising shorter trajectories of length $L' \ll L$ of appropriate intensity.
	
	 Without any loss of generality, we assume from now on that $\bar \gamma > 10$ so so as to meet the relevant condition in \eqref{e:couplings-params}, inherent to both Theorems~\ref{thm:short_long} and \ref{thm:long_short}.
Instead of working directly with $\widetilde{\mathcal{J}}_{k}^{a,2}$ directly, we will work with a `larger' configuration $\mathcal J^{f, L}$ (in the sense of \eqref{eq:J-k-a-2-dom} below), which renders certain computations more transparent. Let 
	\begin{equation}\label{def:f}
	\textstyle f(x) = \frac 1A\, u\tilde g(x - \tilde x_k) + 
	\frac{\delta_2}{80 A}\,u1_{B_{4L}}(x - \tilde x_k), \, x \in \Z^d,\end{equation}
let $L' = L\big(\frac{1}{(\log L)^{4\bar \gamma}} \wedge (\delta_2)^4\big)$ and $f'$ be given in terms of $f$, $L$ and $L'$ by \eqref{eq:f'}. One readily checks that $f$ is supported on $B_{4L}(\tilde x_k)$ and that $2u \ge f \ge  c(\delta) \tilde{u} \, (\log L)^{-4}$ pointwise on $B_{4L}(\tilde x_k)$ (for the lower bound on $f$, one only considers the second term in \eqref{def:f} and recalls $\delta_2$ from \eqref{eq:delta-2-mixed}). Moreover, since $u(\tilde r_k^a- \tilde r_k^{a+1}) \leq u \tilde{r}_k^0= u\tilde r_k$, which by  \eqref{def:Gxk} is at most $\frac{\delta_2}{100}$ in $B_{4L}(\tilde x_k)$ on the event $G(\tilde x_k)$, one deduces
from \eqref{def:I'ka} and \eqref{def:f} that for all $L \ge 
	C(\delta)$,
\begin{equation}\label{eq:J-k-a-2-dom}\widetilde{\mathcal{J}}_{k}^{a,2} \text{ (under $\P_{\tilde{\sigma}}$) }
	\leq_{\text{st.}} \mathcal J^{f, L}, \text{ for any } \tilde{\sigma} \in G(\tilde x_k).
	\end{equation}
Recalling that $\widetilde{\mathcal{J}}_{k}^{a, 1,1}\stackrel{\text{law}}{=} \mathcal{I}^{\rho}$ with $\rho$ as in \eqref{def:v} and combining \eqref{eq:J-k-a-2-dom} with Theorem~\ref{thm:long_short}, which is in force in the wake of Lemma~\ref{L:Cobst-ver} and the discussion following \eqref{def:f}, it follows that for all $L \geq C( \delta, \bar \gamma)$, with $l = \frac{L}{L'}$ and $L'$ as defined below \eqref{def:f}, there exists a coupling $\mathbb Q_{1}^a$ of $(\widetilde{\mathcal{J}}_{k}^{a,2}\cup 
	\widetilde{\mathcal{J}}_{k}^{a, 1,1})$ and $(\mathcal{K} \cup 
	\widetilde{\mathcal{J}}_{k}^{a, 1,1})$, with the two configurations sampled independently for either pair, where $\mathcal K \stackrel{\text{law}}{=}  \mathcal J^{(1 + \Cr{C:sprinkle_hard} l^{-\frac12})f', L'}$, and such that
\begin{equation}\label{eq:couple_L_by_L'}\mathbb Q_{1}^a \big[ \mathscr{C}^{\partial}_{S}\big(\mathcal V\big(\widetilde{\mathcal{J}}_{k}^{a,2}\cup 
	\widetilde{\mathcal{J}}_{k}^{a, 1,1} \big)  \big)  \supset  \mathscr{C}^{\partial}_{S}\big(\mathcal V \big( \mathcal{K} \cup 
	\widetilde{\mathcal{J}}_{k}^{a, 1,1}\big)  
	\big)   \big] \geq 1 - e^{- c (\log L)^{\bar \gamma}},
	\end{equation} for both $S = B= B_N(x)$ and $S=B_{{\the\numexpr3 + 5*2\relax}L}(\tilde x_{k})$ and some $c = c(\delta, \bar \gamma ) > 0$. Note that $\Cr{C:sprinkle_hard}= \Cr{C:sprinkle_hard}(\delta, \bar \gamma )$ for the constant appearing in the statement of Theorem~\ref{thm:long_short} (and in $\mathcal K$) due to the choice of parameters prescribed by Lemma~\ref{L:Cobst-ver}.
	
	Next we aim to effectively replace $\mathcal{K}$ by $\widetilde{\mathcal{J}}_{k}^{a,3}$. To this end, we apply Theorem~\ref{thm:short_long} with $2L$ playing the role of $L$ (and thus $2l$ replacing $l$), $ 2\tilde{u} \delta^{-1}$ in place of $u$, $K = 4L$ and
$$ \textstyle f = f_2= \frac{1 + \Cr{C:sprinkle_hard} l^{-\frac12}}{1 -\Cr{C:sprinkle_easy} l^{-\frac12}} \bar f, \quad \text{ where } \bar f = \frac 1A\, u 1_{B_{2L}}(\cdot - \tilde x_k) + \frac{\delta_2}{40 
		A}\,u1_{B_{6L}}(\cdot - \tilde x_k)$$ (with $\Cr{C:sprinkle_easy}=\Cr{C:sprinkle_easy}( \delta, \bar \gamma)$). As we now explain, this yields for $L \geq C( \delta, \bar \gamma)$ a coupling of $\mathcal{K}$ and $\widetilde{\mathcal{J}}_{k}^{a,3}$ such that, for $S \in \{ B, B_{{\the\numexpr3 + 5*2\relax}L}(\tilde x_{k})\}$ and some $c = c(\delta, 
	\bar \gamma) > 0$,
\begin{equation}\label{eq:couple_L'_by_2L}\mathbb Q_{2}^a \big[ \mathscr{C}^{\partial}_{S}\big(\mathcal V(\mathcal K)  \big)  \supset  \mathscr{C}^{\partial}_{S}\big(\mathcal V\big(\widetilde{\mathcal{J}}_{k}^{a,3}\big) 
	\big)   \big] \geq 1 - e^{- c (\log L)^{\bar \gamma}}.
\end{equation}
Indeed, one readily verifies that the function $f_2$ above \eqref{eq:couple_L'_by_2L} satisfies $c(\delta, \bar \gamma) (\log L)^{-4} \leq f_2 \leq 2\tilde u \delta^{-1}$ pointwise on $B_{5L}$, so that condition \eqref{e:ass-f-coup} is met for our choice of parameters. Thus, Theorem~\ref{thm:short_long} yields a coupling of $\mathcal{J}_1=\mathcal{J}^{f, 2L}$ (where $f=f_2$) and $\mathcal{J}_2 =\mathcal J^{(1 - \Cr{C:sprinkle_easy} l^{-1/2})(f 1_{B_{K}})', L'} = \mathcal J^{(1 + \Cr{C:sprinkle_hard} l^{-1/2})(\bar f 1_{B_{4L}})', L'}$ such that the inclusion $\mathcal{V}(\mathcal{J}_1)\subset \mathcal{V}(\mathcal{J}_2)$ holds with probability as in \eqref{eq:couple_L'_by_2L} on account of \eqref{eq:short_long}
and by choice of $L'$; here it is important to notice that $(\cdot)'$ refers to the operation in \eqref{eq:f'} with $2L$ in place of $L$, in accordance with our choice of parameters. But now, in view of $\eqref{def:f}$, with $l =\frac L{L'}$, we have, pointwise on $\Z^d$,
	\begin{align*}
		f' &\stackrel{\eqref{eq:f'}}{=} l^{-1} \sum_{0 \le k < l}P_{kL'} (f) = (lA)^{-1} \sum_{0 \le k < l}P_{kL'} (u\tilde g(\cdot - \tilde x_k)) + \frac{\delta_2}{80} (lA)^{-1}\sum_{0 \le k < l}P_{kL'} (u 1_{B_{4L}(\tilde x_k)})\\
		& \stackrel{\eqref{eq:defgh}}{\le}  (lA)^{-1} \sum_{0 \le k < l}P_{kL'}\Big(\frac{1+P_L}{2} \Big) (u\,1_{B_{2L}(\tilde x_k)}) + \frac{\delta_2}{40}(lA)^{-1} \sum_{0 \le k < l} P_{kL'}\Big(\frac{1+P_L}{2} \Big)(u 1_{B_{4L}(\tilde x_k)}) \\
		&\ = (2l)^{-1} \sum_{0 \le k < 2l} P_{kL'} ( \bar f 1_{B_{4L}}),
	\end{align*}
with $\bar f$ as above. Since $\mathcal K $ has the same law as $\mathcal J^{(1 + \Cr{C:sprinkle_hard} l^{-\frac12})f', L'}$, one immediately infers from this that $\mathcal{K} \leq_{\text{st.}} \mathcal{J}_2$. Moreover, by choice of $L'$ and since $l=\frac L{L'}$, keeping in mind that $\delta_2=  c(\log L)^{-4}$ whilst $\bar{\gamma}> 10$, one readily sees that whenever $L \geq C( \delta, \bar \gamma)$,
\begin{multline*} \textstyle f_2 \le (1 + C(  \delta, \bar \gamma, \gamma_M)(\delta_2)^2 ) \bar f \\
\le  \textstyle \frac uA\, 1_{B_{2L}}(\cdot - \tilde x_k) + \frac{u \delta_2}{A} 1_{B_{6L}}(\cdot - \tilde x_k) \stackrel{\eqref{eq:defgh}}{=} \frac u A \tilde h(\cdot - \tilde x_k) \stackrel{\eqref{eq:defgh^ka}}{=} u(\tilde h_k^{a+1}-  \tilde h_k^{a})
	\end{multline*}
whence $\mathcal{J}_1 \leq_{\text{st.}} \widetilde{\mathcal{J}}_{k}^{a,3} $ in view of 
\eqref{def:I'ka}. Thus, applying \cite[Lemma 2.4]{RI-III} (twice) to concatenate the coupling of $\mathcal{J}_1$ and $\mathcal{J}_2$ supplied by Theorem~\ref{thm:short_long} with those implied by the two dominations yields $\mathbb Q_{2}^a$ satisfying \eqref{eq:couple_L'_by_2L}.

	\medskip
	
	Having obtained $\mathbb Q_{1}^a$ and $\mathbb Q_{2}^a$ satisfying \eqref{eq:couple_L_by_L'} and \eqref{eq:couple_L'_by_2L} for each $0 \le a < A$, the remaining task is to extend and concatenate these so as to produce a measure 
	$\mathbb Q_{\tilde \sigma}$ satisfying all of \eqref{eq:marginal0}--\eqref{eq:coupling_smallbig}. We start by reconstructing for individual $a$'s the full configurations $\widetilde{\mathcal{J}}_{k}^{a}$ and $\widetilde{\mathcal{J}}_{k}^{a+1}$. In light of \eqref{eq:intermed_decomp}, \eqref{eq:I'ka_decomp} and the sets coupled under $\mathbb Q_{1}^a$ and $\mathbb Q_{2}^a$, this boils down to defining $\widetilde{\mathcal{J}}_{k}^{a, 1,2}$, which is not involved in either of the two measures. With a view towards the required independence property \eqref{eq:easy2}, which requires slightly more care than the rest, we first refine \eqref{eq:I'ka_decomp} by further decomposing (under $\P_{\tilde{\sigma}}$)
$$\widetilde{\mathcal{J}}_{k}^{a, 1,2}\stackrel{\eqref{eq:I'kaj}}{=} \widetilde{\mathcal J}' \cup \mathcal J^{[u\tilde r_{k+1}, u\tilde r_{k}^{a+1}], 
		L}(\omega_2)
$$
where
\begin{equation}\label{eq:I'kaj2}
	\widetilde{\mathcal J}' \stackrel{{\rm def}.}{=} {\mathcal{J}}^{\, u (\tilde g_{k+1})_{| \tilde U^c}, \,L}(\omega_1) 
	\cup {\mathcal{J}}^{\,  u \tilde r_{k+1}, \, L}(\omega_2)\cup   {\mathcal{J}}^{\, u 
		(\tilde h_{k})_{| \tilde U^c}, \,2L} (\omega_3) \cup   {\mathcal{J}}^{\, u \tilde s_k,2L}(\omega_4).
\end{equation}
Note that $\widetilde{\mathcal J}' $ does not depend on $a$ (and thus evolves trivially as $a\to (a+1)$). Feeding the decomposition of $\widetilde{\mathcal{J}}_{k}^{a, 1,2}$ into \eqref{eq:I'ka_decomp} and subsequently \eqref{eq:intermed_decomp} yields the rewrite (still under $\P_{\tilde{\sigma}}$)
\begin{equation}\label{eq:tildeIdecomp}\begin{split}
\widetilde {\mathcal J}_{k}^a &= \widetilde{\mathcal J}' \cup \widetilde {\mathcal K}_{k}^a, \quad \qquad  \widetilde {\mathcal K}_{k}^a = \widetilde{\mathcal{J}}_{k}^{a,2}
	\cup \widetilde{\mathcal{J}}_{k}^{a, 1,1} \cup \mathcal J^{[u\tilde r_{k+1}, u\tilde r_{k}^{a + 1}]}, \\
 \widetilde{\mathcal{J}}_{k}^{a+1} &= \widetilde{\mathcal J}' \cup \widetilde{\mathcal K}_k^{a+1}, \quad  \widetilde{\mathcal K}_k^{a+1} = \widetilde{\mathcal{J}}_{k}^{a,3} \cup 
	\widetilde{\mathcal{J}}_{k}^{a, 1,1} \cup \mathcal J^{[u\tilde r_{k+1}, u\tilde r_{k}^{a + 1}]}, \end{split}
\end{equation} 
valid	for all $0 \le a < A$, and 
and each of the union is over independent sets. Moreover, it follows from the definitions of $\widetilde{\mathcal J}'$, $\widetilde 
	{\mathcal J}_{k}= \widetilde 
	{\mathcal J}_{k}^0 $ and $\widetilde {\mathcal J}_{k+ 1/2}= \widetilde 
	{\mathcal J}_{k}^A$ that 
	\begin{equation}\label{eq:independent_backgrnd} 
	 \widetilde{\mathcal 
		J}'  \cap B_{17L}(\tilde x_k)^c = \widetilde {\mathcal J}_{k}^0 \cap B_{17L}(\tilde x_k)^c =  \widetilde {\mathcal J}_{k}^A  \cap B_{17L}(\tilde x_k)^c.
	\end{equation}
	This observation will be crucial towards deriving \eqref{eq:easy2}.

Returning to $\mathbb Q_{1}^a$ and $\mathbb Q_{2}^a$, observe that the inclusion in \eqref{eq:couple_L'_by_2L} (and similarly in \eqref{eq:couple_L_by_L'}) remains true if the pair $(\mathcal K, \widetilde{\mathcal{J}}_{k}^{a,3})$ is replaced by $(\mathcal K \cup \mathcal{J}, \widetilde{\mathcal{J}}_{k}^{a,3} \cup \mathcal{J})$, for arbitrary $\mathcal{J} \subset \Z^d$. We apply this observation to both $\mathbb Q_{1}^a$ and $\mathbb Q_{2}^a$ separately with the choice $\mathcal{J}\stackrel{\text{law}}{=} \mathcal J^{[u\tilde r_{k+1}, u\tilde r_{k}^{a+1}]}$ in the former and $\mathcal{J}\stackrel{\text{law}}{=} \mathcal J^{[u\tilde r_{k+1}, u\tilde r_{k}^{a+1}]} \cup \widetilde{\mathcal{J}}_{k}^{a, 1,1}$ in the latter case, sampled independently and incorporated into $\mathbb Q_{1}^a$ and $\mathbb Q_{2}^a$ by suitable extension. In view of \eqref{eq:tildeIdecomp}, this yields with \eqref{eq:couple_L_by_L'} that the inclusion $\mathscr{C}^{\partial}_{S}(\mathcal V( \widetilde {\mathcal K}_{k}^a )  )  \supset  \mathscr{C}^{\partial}_{S}(\mathcal V ( \mathcal{K}' ))$, where $\mathcal{K}'\stackrel{\text{law}}{=} \mathcal{K} \cup 
	\widetilde{\mathcal{J}}_{k}^{a, 1,1} \cup \mathcal J^{[u\tilde r_{k+1}, u\tilde r_{k}^{a+1}]}$, holds with $\mathbb Q_{1}^a$-probability $1 - e^{- c (\log L)^{\bar \gamma}}$. In the same vein, \eqref{eq:couple_L'_by_2L} lifts to the event $ \mathscr{C}^{\partial}_{S}(\mathcal V ( \mathcal{K}' )) \supset \mathscr{C}^{\partial}_{S}(\mathcal V( \widetilde {\mathcal K}_{k}^{a+1} )  )$ under $\mathbb Q_{2}^a$. 
Thus, concatenating $\mathbb Q_{1}^a$ and $\mathbb Q_{2}^a$ by means of \cite[Lemma 2.4]{RI-III} produces a coupling $\mathbb Q_a$ of $( \widetilde {\mathcal K}_{k}^{a} ,  \widetilde {\mathcal K}_{k}^{a+1} )$ satisfying
	\begin{equation}\label{eq:couple_Qa}
		\mathbb Q^{a} \big[ \mathscr{C}^{\partial}_{S}\big(\mathcal 
		V( \widetilde {\mathcal K}_{k}^{a} )  \big)  \supset  \mathscr{C}^{\partial}_{S}\big(\mathcal V( \widetilde {\mathcal K}_{k}^{a+1} ) \big)   
		\big] \geq 1 - 2e^{- c (\log L)^{\bar \gamma}}.\end{equation}
	Since both $\widetilde {\mathcal J}^{a,2}_{k}$ and $\widetilde{\mathcal{J}}_{k}^{a,3}$ are 
	supported on $B_{8L}(\tilde x_k)$ and the remaining sets constituting $\widetilde {\mathcal K}_{k}^{a}$ and $\widetilde {\mathcal K}_{k}^{a+1}$ in \eqref{eq:tildeIdecomp} coincide by construction under $\mathbb Q^{a}$, it follows with $U_k = B_{10 L}(\tilde x_{k})$ (cf.~\eqref{eq:easy1}) that for all $0 \leq a < A$,
	\begin{equation}\label{eq:locality1}\widetilde {\mathcal K}_{k}^{a} \cap U_k^c = \widetilde {\mathcal K}_{k}^{a+1} \cap 
	U_k^c \text{ under $\mathbb Q^{a}$.}\end{equation} We will soon refer 
	to this equality when verifying the property~\eqref{eq:easy1} for $\mathbb Q_{\tilde \sigma}$.

	Further concatenating the couplings $\mathbb Q_a$'s over all $a$ with $0 \leq a < A$ by repeated application of \cite[Lemma 2.4]{RI-III} and extending the resulting measure by an independent sample of $\widetilde{\mathcal J}'$ with law given by \eqref{eq:I'kaj2}, we arrive in view of \eqref{eq:tildeIdecomp} at a coupling $\mathbb Q_{\tilde \sigma}$
of $  \widetilde{\mathcal J}' \cup \widetilde {\mathcal K}_{k}^0\stackrel{\text{law}}{=} \widetilde {\mathcal J}_{k}^0= \widetilde {\mathcal J}_{k}$ and $  \widetilde{\mathcal J}' \cup \widetilde {\mathcal K}_{k}^{A} \stackrel{\text{law}}{=} \widetilde {\mathcal J}_{k}^{A} = \widetilde {\mathcal J}_{k+1/2}$. Combining \eqref{eq:couple_Qa}, the same observation as following \eqref{eq:independent_backgrnd} and a union bound, cf.~\cite[Remark 2.5,2)]{RI-III}, it follows that for $S \in \{ B= B_N(x), B_{{\the\numexpr3 + 5*2\relax}L}(\tilde x_{k})\}$,
\begin{equation}\label{eq:final_couple_Q}
		\mathbb Q_{\tilde \sigma} \big[ \mathscr{C}^{\partial}_{S}\big(\mathcal 
		V( \widetilde{\mathcal J}' \cup \widetilde {\mathcal K}_{k}^{0} )  \big)  \supset  \mathscr{C}^{\partial}_{S}\big(\mathcal V( \widetilde{\mathcal J}' \cup \widetilde {\mathcal K}_{k}^{A} )  \big)   \big] \geq 1 - 2A e^{- c (\log L)^{\bar \gamma}},
	\end{equation}

The measure $\mathbb Q_{\tilde \sigma}$ is our final coupling. Let $\widehat{\mathcal I}_{\ell}$, $\ell=k, k+ \frac12$, be defined as $\widetilde{\mathcal{I}}^{u,L}_{\ell}$ in \eqref{eq:tildeI_k_noised} but with $ \widetilde{\mathcal J}' \cup \widetilde {\mathcal K}_{k}^{0} $ resp.~$ \widetilde{\mathcal J}' \cup \widetilde {\mathcal K}_{k}^{A} $ in place of $\widetilde{\mathcal{J}}^{u,L}_{k}$ resp.~$\widetilde{\mathcal{J}}^{u,L}_{k+1/2}$ on the right-hand side. The noise $\mathsf{N}^{\cdot}$ present in \eqref{eq:tildeI_k_noised} thereby acts deterministically under $\mathbb Q_{\tilde \sigma}$. We proceed to verify \eqref{eq:marginal0}-\eqref{eq:coupling_smallbig} with these choices. 

Property~\eqref{eq:marginal0} is immediate since $ \widetilde{\mathcal J}' \cup \widetilde {\mathcal K}_{k}^{0} \stackrel{\text{law}}{=}   \widetilde {\mathcal J}_{k}^{u,L}$ and $ \widetilde{\mathcal J}' \cup \widetilde {\mathcal K}_{k}^{A} \stackrel{\text{law}}{=}   \widetilde {\mathcal J}_{k+1/2}^{u,L}$. Property~\eqref{eq:easy1} follows directly from our construction as each measure 
	$\mathbb Q_a$ satisfies \eqref{eq:locality1} and the {\em same} noise operator is applied to both configurations outside the box $\widetilde B_k \subset U_k$ according to \eqref{eq:tildeI_k_noised}. 
	Property~\eqref{eq:easy2} is a consequence 
	of \eqref{eq:independent_backgrnd}, the independence between ${\mathcal{J}}'$ and $(\widetilde {\mathcal K}_{k}^{0} , \widetilde {\mathcal K}_{k}^{A} )$ and the fact that $\sigma(1\{x \in {\mathcal{J}}' \}, x\in  V_k^c)$ is independent from $\sigma(1\{x \in {\mathcal{J}}' \}, x\in   U_k)$, as can be seen by inspection of \eqref{eq:I'kaj2}. 
 
Last but not least,~\eqref{eq:coupling_smallbig} does not immediately follow from \eqref{eq:final_couple_Q} with $S=B$. For, the (deterministic, under $\mathbb Q_{\tilde \sigma}$) noise $\mathsf{N}$ acting through \eqref{eq:tildeI_k_noised} in the definition of both $\widehat{\mathcal I}_{\ell}$, $\ell=k, k+ \frac12$, varies as $\ell$ changes and may in principle completely spoil the inclusion (ensured
by \eqref{eq:final_couple_Q}) of the boundary clusters of $B$ prior to noising. This is where we leverage the flexibility in $S$. Thus, let $\gamma \subset B \cap \widehat{\mathcal V}_{k+1/2}$ be a path starting on $\partial B$. Note that the location of $B=B_N(x) \subset \Z^d$ is arbitrary. We claim that on the intersection of both events in \eqref{eq:final_couple_Q} as $S$ ranges over the two allowed sets, and if $\tilde \sigma \in G(x_k)$, one has $\gamma \subset \widehat{\mathcal V}_{k}$. From this, \eqref{eq:coupling_smallbig} follows. 

To see that $\gamma \subset \widehat{\mathcal V}_{k}$, we decompose $\gamma = (\gamma_n)$ into disjoint segments (sub-paths) as follows. Let  $ \tilde U_k = B_{{\the\numexpr3 + 5*2\relax}L}(\tilde x_{k})$. We split $\gamma$ into $(\gamma_n)_{0 \leq n < T_{ \tilde U_k}}$ and, in case $T_{ \tilde U_k}< \infty$, decompose $\gamma \circ T_{ \tilde U_k}$ into its excursions in $\tilde U_k$
 and their complements (contained inside $ \tilde U_k^c$). 
 
 If $T_{ \tilde U_k} = \infty$, then $\gamma\subset  \tilde U_k$ and hence $\gamma\subset \widehat{\mathcal V}_{k}$ automatically by \eqref{eq:final_couple_Q} for $S=B$ since $$
 \text{$\widehat{\mathcal V}_{k+1/2} \cap \tilde U_k = \mathcal 
		V( \widetilde{\mathcal J}' \cup \widetilde {\mathcal K}_{k}^{A} ) \cap \tilde U_k$ and 
		$\widehat{\mathcal V}_{k} \cap \tilde U_k = \mathcal 
		V( \widetilde{\mathcal J}' \cup \widetilde {\mathcal K}_{k}^{0} ) \cap \tilde U_k$}
$$
 if $\tilde \sigma \in G(x_k)$ on account of \eqref{def:Gxk}.

If $T_{ \tilde U_k} < \infty$, then all of $(\gamma_n)_{0 \leq n < T_{ \tilde U_k}}$ and the excursions of $\gamma \circ T_{ \tilde U_k}$ inside $\tilde U_k$ are each contained in $\tilde U_k$ and part of boundary clusters of $\mathscr{C}^{\partial}_{\tilde U_k}(\widehat{\mathcal V}_{k+1/2}) $, and thus part of $\widehat{\mathcal V}_{k}$ for the same reason as above, but using \eqref{eq:final_couple_Q} for $S=\tilde U_k$ instead. Finally, each segment $\gamma' \subset (\tilde{U}_k^c \cap \widehat{\mathcal V}_{k+1/2})$ is also in $\widehat{\mathcal V}_{k}$ by \eqref{eq:easy2}. This completes the verification of \eqref{eq:coupling_smallbig}.
\end{proof}

\section{Reduction to super- and (near-/sub-)diffusive estimates}\label{subsec:first_reduct} 

We now commence with the proof of Theorem~\ref{thm:hh1}. The purpose of the present section is to progressively reduce it to two individual estimates, stated in Lemmas~\ref{L:closed_piv2piv} and~\ref{lem:reduce_distance} below, which deal with complementary sets of scales (super- vs.~(sub-)diffusive for random walks of time length roughly $L$). The two lemmas will be proved separately in Sections~\ref{sec:superdiff} and~\ref{sec:subdiff}. These two estimates represent the technical core of certain comparison inequalities for connection events, stated in Proposition~\ref{prop:comparisonLk} and Corollary~\ref{prop:comparison} below, which imply Theorem~\ref{thm:hh1} rather straightforwardly. We present this latter implication in \S\ref{subsec:compa}. The remainder of this section is devoted to reducing Proposition~\ref{prop:comparisonLk} progressively to the two aforementioned lemmas. They encapsulate the cost of performing surgery on certain coarse pivotal events (alluded to in \S\ref{subsec-pf-outline}) in two distinct regimes of a multi-scale analysis. They correspond in a sense to `off-critical' vs.~`critical' scales. 

\subsection{Comparison inequalities} \label{subsec:compa}

Recall the models $\mathcal{V}_{\ell}^{u,L}= \Z^d \setminus \mathcal I_{\ell}^{u, L}$ from \eqref{eq:barequalstilde}, \eqref{eq:Itildeinclusions} and \eqref{eq:Ibarinclusions} in the previous section. The following result will be key.

\begin{prop}
	\label{prop:comparisonLk}
For all  $\delta \in (0, \frac 13)$ and $\gamma  \ge \Cl{c:gammaLB}$ (recall \eqref{eq:varepsilon_L}), there exists $L_0(\delta, \gamma) \geq 1$ such that for all dyadic $L \geq 
L_0$, $k\in \mathbb N$ and every $r \geq 1$, $R \geq 2(r \vee M_0(L))$ (recall \eqref{eq:def_M_0}) and $u \in (\tilde{u}(1 + 3\delta), u_{**}(1 - 3\delta))$, 
	\begin{equation}\label{eq:compa200}
	 \P[\lr{}{ {\mathcal V}_k}{B_{r}}{\partial B_{R}}]\geq \P[\lr{}{ {\mathcal V}_{k+1}}{B_{r}}{\partial B_{R}}], \textrm{ where $\mathcal V_{\cdot} \in\{ \overline{\mathcal V}_{\cdot}^{u, L}, \widetilde{\mathcal V}_{\cdot}^{u, L}\}$ (see \eqref{eq:barequalstilde}).}
	 	\end{equation}
\end{prop}

We admit Proposition~\ref{prop:comparisonLk} for the time being and return its proof in \S\ref{subsec:reduction}. Together with Proposition~\ref{prop:couple_global}, which can be used directly in the complementary regime of $R$, Proposition~\ref{prop:comparisonLk} has the following important consequence. Note that the following statement is uniform with respect to the radii $r$ and $R$, and that the comparison between $\mathcal V^u$ and the homogenous model $\mathcal{V}^{u',L}$ (see \S\ref{subsec:fr-models}) it entails, requires a small amount of sprinkling, from $u$ to $u'=u(1 \pm C(\log L)^{-4})$.
\begin{corollary}
  \label{prop:comparison}
  For all $\gamma\geq C$ (cf.~\eqref{eq:varepsilon_L}) and $\delta \in (0, 1)$, there exists $L_0= L_0(\delta, \gamma) > 1$ such that for all $L \geq L_0$ integer power of 2, $u \in (\tilde{u}(1 + 
  \delta), u_{**}(1 - \delta))$ and $R \geq 2r \geq 1$,
  \begin{align}
    &\P[\lr{}{\mathcal V^u}{B_{r}}{\partial B_{R}}]\geq \P[\lr{}{\mathcal{V}^{u(1+ C(\log L)^{-4}),L}}{B_{r}}{\partial B_{R}}] - \exp\big\{-(\log R)^{\Cl[c]{c:add}\gamma}\big\},\, \label{eq:comparison1} \\
    &\P[\lr{}{\mathcal{V}^u}{B_{r}}{\partial B_{R}}] \leq \P[\lr{}{\mathcal{V}^{u(1-C(\log L)^{-4}),L}}{B_{r}}{\partial B_{R}}] + \exp\big\{-(\log R)^{\Cr{c:add}\gamma}\big\}. \label{eq:comparison2}
  \end{align}
 Moreover, the requirement that $u \leq u_{**}(1 - \delta) $ can be replaced by $u \leq \frac{\tilde{u}}{\delta}$ when $R \leq 2 M_0(L)$.
\end{corollary}

\begin{remark}\label{R:Prop-intro}
In view of the last sentence, Proposition~\ref{prop:L to infinity} thus corresponds to a special case of Corollary~\ref{prop:comparison}. 
\end{remark}

We proceed to give the proof of Corollary~\ref{prop:comparison}, which combines Propositions~\ref{prop:comparisonLk} and~\ref{prop:couple_global}. From Corollary~\ref{prop:comparison}, Theorem~\ref{thm:hh1} will readily follow. The proof is given at the end of this paragraph.

\begin{proof}[Proof of Corollary \ref{prop:comparison}]
Let $\gamma \geq 4 \gamma_M^2 \vee \Cr{c:gammaLB}\vee 11$, cf.~\eqref{eq:def_M} and Proposition~\ref{prop:comparisonLk}. We first establish the following analogue of \eqref{eq:comparison1}-\eqref{eq:comparison2} between homogenous models at scales $L$ and $2L$, which will then be iterated over dyadic scales.
Namely, as we now explain, for every $\delta'\in(0,1)$ there exists $L_0(\delta', \gamma) > 1$ such that, for all $u 
	\in (\tilde{u}(1 + \delta'), u_{**}(1 - \delta'))$, dyadic $L \geq L_0$ and all $r \geq 1$, $R \geq 2r $, \begin{align}
	 \P[\lr{}{ {\mathcal V}^{u,2L}}{B_{r}}{\partial B_{R}}]&\geq \P[\lr{}{ {\mathcal V}^{ u(1+\delta_1),L}}{B_{r}}{\partial B_{R}}]- \eta,\,
	  \label{eq:comparison100}\\
	 \P[\lr{}{ {\mathcal V}^{u,L}}{B_{r}}{\partial B_{R}}]&\geq \P[\lr{}{ {\mathcal V}^{u(1+\delta_1), 2L }}{B_{r}}{\partial B_{R}}]- \eta,\, \label{eq:comparison101} 
	\end{align}
 where $\delta_1= \delta_1(L) = (\log L)^{-4}$ as in \eqref{eq:delta-2-mixed} and
	\begin{align*} &\eta=\eta(L) \stackrel{\textnormal{def.}}{=} \exp\big\{-c(\log L)^{{\gamma}{}}\big\} 1_{R <  2M_0(L)},  
	\end{align*}
with $M_0(L)= 10^3M(10^3L)$ as in \eqref{eq:def_M_0}. Moreover, we claim that the conclusions \eqref{eq:comparison100}-\eqref{eq:comparison101} remain true when requiring $u \leq \frac{\tilde{u}}{\delta}$ instead of $u \leq u_{**}(1 - \delta)$ whenever $R \leq 2 M_0(L)$. To see all of this, pick $\delta= \frac{ \delta' }{ 18}$ in Proposition~\ref{prop:comparisonLk}, which is in force when $R\geq 2 M_0$. The inequalities \eqref{eq:comparison100}-\eqref{eq:comparison101} immediately follow in the regime $R\geq 2 M_0$, simply by iterating~\eqref{eq:compa200} and letting $k \to \infty$ whilst keeping in mind \eqref{eq:Itildeinclusions}, \eqref{eq:delta-2-mixed} and \eqref{eq:Ibarinclusions}, which account for the sprinkling $\delta_1$. Note that each of the two sequences of inhomogenous models thereby yields precisely one of the two inequalities among \eqref{eq:comparison100} and \eqref{eq:comparison101}. If instead $R < 2M_0$, we choose an enumeration 
$\{x_0, x_1, \ldots\}$ of the lattice $\mathbb L \subset \Z^d$ underlying the definition of $(\mathcal{V}_k)_{k \geq 0}$ (cf.~below~\eqref{eq:barequalstilde}) in such a way that $\mathcal V_K \cap B_R = \mathcal V_{\infty} \cap B_R$ for some $K \le CM_0^d$. In view of \eqref{eq:Itildeinclusions} and \eqref{eq:Ibarinclusions}, we now apply Proposition~\ref{prop:couple_global} 
(see \eqref{eq:coupling_smallbig_annealed} in the subsequent remark) with $B=B_R$ and $\bar{\gamma} = \gamma$ to deduce 
\eqref{eq:comparison100}-\eqref{eq:comparison101}. The application of Proposition~\ref{prop:couple_global} underlying the case $R < 2M_0$ does not require an upper bound on $u$ in terms of $u_{**}$. This completes the verification of \eqref{eq:comparison100}-\eqref{eq:comparison101}.

We now show how to deduce \eqref{eq:comparison1} from \eqref{eq:comparison100} and the convergence 
\eqref{e:loc-limit-I-u-L}. Fix $\delta\in (0,1)$ as in Corollary~\ref{prop:comparison} 
and $u \in (\tilde{u}(1 + \delta), u_{**}(1 - \delta))$. 
Then, letting $u_0 = u(1+5\delta')$ with $\delta' = \delta/15$ and 
$$L_{k+1}= 2L_k \text{ and } u_{k+1}={u_k}({1+\delta_1(L_k)})^{-1} , \, k \geq 1, $$
we have for all $L_0 \geq C(\delta)$ that 
$u_k \in (u(1 +  \delta'), u_{**}(1 - \delta'))$ for all $k \geq 0$. Thus, \eqref{eq:comparison100} applies 
between all pairs $(u_{k+1},L_{k+1})$ and $(u_k,L_k)$ yielding, for all $L_0 \geq 
C(\delta, \gamma)$,
\begin{equation}
\label{eq:eq:comparison100iterate1}
\P[\lr{}{ {\mathcal V}^{u_n,L_n}}{B_{r}}{\partial B_{R}}]\geq \P[\lr{}{ {\mathcal V}^{u_0,L_0}}{B_{r}}{\partial B_{R}}]- \sum_{k \ge 0} \eta(L_k), \quad  n \geq 0.
\end{equation}
By definition of $\eta(\cdot)$, the sum in \eqref{eq:eq:comparison100iterate1} only runs over $k$ such 
that $2M_0(L_k) > R$. Moreover, since $\gamma \geq  4\gamma_M^2$ and by definition of $M_0$ (recall \eqref{eq:def_M_0}), we see that
\begin{equation}
\label{eq:eq:comparison100iterate2}
\sum_{k \ge 0} \eta(L_k) \leq \exp\{-(\log R)^{c\gamma}\}
\end{equation}
 whenever $L_0 \geq C(\gamma)$. The desired inequality \eqref{eq:comparison1} now follows from \eqref{eq:eq:comparison100iterate1} and \eqref{eq:eq:comparison100iterate2} upon letting $n \to \infty$, thereby applying \eqref{e:loc-limit-I-u-L}, using the fact that $\lim_n u_n \geq u$ and monotonicity of all vacant sets involved in order to conclude. The complementary estimate \eqref{eq:comparison2} is obtained similarly by combining 
 \eqref{eq:comparison101} and \eqref{e:loc-limit-I-u-L} instead.
 \end{proof}

We conclude with the:

\begin{proof}[Proof of Theorem~\ref{thm:hh1}]
Recall $\tilde{u}$ from \eqref{eq:tildeu}. By \cite[Corollary 1.2]{RI-II}, we know that
$\tilde 
u = \bar u$. On account of this, and since $\bar{u} \leq u_* \leq u_{**}$ by definition, in order to complete the proof it is enough to argue that $u_{**} \leq \tilde{u}$.  Suppose on the contrary that 
 $\tilde{u} < u_{**}$. Pick $\delta = \frac1{10 u_{**}}(u_{**}- \tilde{u}) > 0$, $\gamma \geq C$ 
  large enough so that $\Cr{c:add} \gamma \geq 2 \gamma_M$ and let $L = 2^{\Cl{C:choice-L} \lceil \log_2 L_0 \rceil}$ with $L_0= L_0 (\delta, \gamma)$ the length scale supplied by Corollary~\ref{prop:comparison} for these choices of 
 $\delta$ and $\gamma$ and $\Cr{C:choice-L} =\Cr{C:choice-L}(\delta) \geq 1$ an integer chosen large enough so that the sprinkling appearing in \eqref{eq:comparison1} and \eqref{eq:comparison2} satisfies $C(\log L)^{-4} \leq \frac\delta{3}$. 
 Applying \eqref{eq:comparison1} with $u = \tilde u(1 + 2\delta) < u_{**}(1 - \delta)$ and $R=M(r)$, it follows by 
\eqref{eq:tildeu}-\eqref{eq:def_M} with $u'= u(1+ \tfrac\delta3)$ that 
  $$\inf_{r}M(r)^d\,\P\big[\nlr{}{\mathcal{V}^{u', L}}{B_{r}}{\partial B_{M(r)}}  \big] > 0,$$ whence 
 \begin{equation}\label{e:contra-1}
{ \tilde u(1+ 3\delta) \geq \tilde u^L}
 \end{equation}
{(see around Proposition~\ref{prop:sharptruncated} regarding $\tilde u^L$ and $u_{**}^L$).} In a similar vein, applying \eqref{eq:comparison2} 
 with {$u = u_{**}(1 - 2 \delta)$ and $R = 2r$, and combining with the bound implied by 
 \eqref{eq:equivalent_**}, it follows with $u'= u_{**}(1-3\delta)$ that $$\inf_{r}\P\big[\lr{}{ \mathcal{V}^{u',L}}{B_{r}}{\partial B_{2r}}\big] > 0,$$ and therefore
  \begin{equation}\label{e:contra-2}
u^L_{**} \ge  u_{**}(1- 3\delta).
\end{equation}}
But by choice of $\delta$, one has  {$u_{**}(1 - 3\delta) > \tilde u(1 + 3\delta)$.} Together with \eqref{e:contra-1}-\eqref{e:contra-2}, this yields {$u_{**}^L > \tilde u^L$}, which violates Proposition~\ref{prop:sharptruncated}. Thus, $u_{**} \leq \tilde{u}$ 
and Theorem~\ref{thm:hh1} follows.
\end{proof}

\subsection{Surgery: difference estimate}\label{subsec:reduction}
With \S\ref{subsec:compa} completed, the task of proving Theorem~\ref{thm:hh1} is reduced to that of proving the comparison inequality between $\mathcal{V}_k$ and $\mathcal{V}_{k+1}$ (omitting the superscripts $u$ and $L$) stated in Proposition~\ref{prop:comparisonLk}. The proof of the latter occupies the remainder of this article.
Our approach involves a `discretized differential calculus' which will make certain (coarse) pivotal events already met in \S\ref{subsec-pf-outline}, see \eqref{def:coarse_piv}, naturally appear. In the present section we derive Proposition~\ref{prop:comparisonLk} from an inequality involving these pivotal events, see Proposition~\ref{P:compa_a} and in particular \eqref{eq:compa_a_state} below, which in a rather loose sense can be viewed as a differential version of \eqref{eq:compa200}. We explain the meaning of the bound \eqref{eq:compa_a_state} in
greater detail following the statement of Proposition~\ref{P:compa_a}. This proposition will then be derived in \S\ref{subsec:2lemmas} from the two key lemmas mentioned at the beginning of this section.

\begin{proof}[Proof of Proposition \ref{prop:comparisonLk}] Let $U = B_r$ and $V=\partial B_R$. Throughout 
the proof, we always tacitly assume that $k \in \mathbb{N}$, $\delta \in (0, \frac 13)$, $u \in (\tilde{u}(1 + 3\delta), u_{**}(1 - 3\delta))$, $r \geq 1$ and $R \geq 2(r \vee M_0)$. Recalling $\mathbb{L}$ from below \eqref{eq:barequalstilde}, $\text{Piv}_K(\mathcal{V})$ from \eqref{def:coarse_piv} and writing $\piv_{x, N}(\mathcal V) = \piv_{B_{N}(x)}(\mathcal V)$, let
\begin{equation}
\label{def:fx}
f(x) \stackrel{{\rm def}.}{=} \P\big [\, \text{Piv}_{x, {\the\numexpr\couprad+\rangeofdep\relax}L}(\mathcal {V}_{k + \frac 1 2}) \, \big],\quad x \in \mathbb{L}.
\end{equation}
Introduce (cf.~\eqref{eq:Itildeinclusions} and \eqref{eq:Ibarinclusions} regarding the second equality)
\begin{align}
\label{eq:def_b}
b \stackrel{{\rm def}.}{=} \P[\lr{}{ {\mathcal V}_{k+ \frac 1 2}}{U}{V}]- \P[\lr{}{ {\mathcal V}_{k+1}}{U}{V}] = \P[\lr{}{ {\mathcal V}_{k+ \frac 1 2}}{U}{V}, \nlr{}{ {\mathcal V}_{k+1}}{U}{V}].
\end{align}
We will show that for suitable $L_0(\delta, \gamma)$, all $L \geq L_0$ and $\gamma \geq C$,
\begin{equation}
\label{eq:compa201}
e^{-\Cr{c:coupling101}(\log L)^{\gamma}} f(x_k) \leq  b
\end{equation}
(see~\eqref{eq:coupling_smallbig_annealed} regarding the definition of $\Cr{c:coupling101}$). We first explain how \eqref{eq:compa201} yields \eqref{eq:compa200}. 
Consider the configurations $\mathcal {V}_{k+1/2}$ and $\mathcal {V}_{k}$. Applying Proposition~\ref{prop:couple_global} with $B=B(x_k, \the\numexpr\couprad\relax L)$ together with Remark~\ref{remark:annealed_coupling},1) (see \eqref{eq:coupling_smallbig_annealed} in particular), 
denoting by $\text{Coup}$ the event appearing on the left-hand side of \eqref{eq:coupling_smallbig} with $x=x_k$ and $N=\couprad L$ and recalling the standing assumption on $r,R$, we first claim that under the annealed coupling measure $\mathbb{Q}$,
\begin{equation}
\label{eq:compa202}
\{\lr{}{ \widehat{\mathcal V}_{k+ \frac 12}}{U}{V} , \text{Coup}\}\subset  \{\lr{}{ \widehat{\mathcal V}_{k}}{U}{V}\},
\end{equation}
where $\widehat{\mathcal{V}}_j\stackrel{{\rm law}}{=} {\mathcal{V}}_j$, $j=k, k+ \frac12$. 
Indeed, \eqref{eq:compa202} can be seen as follows. Since $ B=
B(x_{k },{\the\numexpr\couprad\relax}L)$ cannot intersect both $U$ and $V$ simultaneously 
owing to the hypothesis $R \ge 2(r \vee M_0)$ where $M_0=M_0(L)$, we get for any $\omega \in \{0, 
1\}^{\Z^d}$ 
that
	\begin{align}
		\label{eq:compa202.a}
		1_{\{\lr{}{ \omega }{U}{V}\}} \mbox{ is an increasing measurable function of } \big(\{\omega \cap B^c, \mathscr{C}^{\partial}_{B}(\omega)\big),
	\end{align}
where $\omega$ is tacitly identified with its open set $\mathcal{V}= \{\omega =1\} $ and the underlying partial order $\preceq$ is the usual inclusion as 
subsets of $\Z^d$. Now $\widehat{\mathcal V}_{k+ 1/2} \cap B^c = \widehat{\mathcal V}_{k} \cap  B^c$ by \eqref{eq:easy1} and on the event ${\rm Coup}$, we have $\mathscr{C}^{\partial}_{ 
		B}(\widehat{\mathcal V}_{k+ 1/2}) \subset \mathscr{C}^{\partial}_{B}(\widehat{\mathcal V}_{k})$. Therefore 	$(\widehat{\mathcal V}_{k + 1/2} \cap B^c, \mathscr{C}^{\partial}_{B}(\widehat{\mathcal V}_{k + 1/2})) 
	\preceq (\widehat{\mathcal V}_{k} \cap B^c, \mathscr{C}^{\partial}_{B}( 
	\widehat{\mathcal V}_{k}))$ and consequently, \eqref{eq:compa202} follows by \eqref{eq:compa202.a}. In turn, \eqref{eq:compa202} yields that
\begin{equation}
\label{eq:compa1}
\begin{split}
\P&[\lr{}{ {\mathcal V}_{k+ \frac 12}}{U}{V}]- \P[\lr{}{ {\mathcal V}_{k}}{U}{V}] \\
&= \mathbb{Q}[\lr{}{ \widehat{{\mathcal V}}_{k+ \frac 12}}{U}{V}]- \mathbb{Q}[\lr{}{ \widehat{{\mathcal V}}_{k}}{U}{V}] \le 
	\mathbb{Q}[\lr{}{ \widehat{\mathcal V}_{k+ \frac 12}}{U}{V}, \nlr{}{ \widehat{\mathcal V}_{k}}{U}{V}] \\
&\stackrel{\eqref{eq:compa202}}{=} \mathbb Q[\text{Coup}^c, \, \lr{}{ \widehat{\mathcal V}_{k+ \frac 
12}}{U}{V}, \nlr{}{ \widehat{\mathcal V}_{k}}{U}{V}] \stackrel{\eqref{eq:easy1}}{\leq} \mathbb Q[\text{Coup}^c, 
\, \text{Piv}_{x_k, {\the\numexpr\couprad+\rangeofdep\relax}L}(\widehat{\mathcal {V}}_{k + \frac 1 2})].
 \end{split}
\end{equation}
Finally, observe that $\text{Coup}$ is measurable with respect to $\sigma(1\{x \in 
\widehat{\mathcal{I}}_j \}, x\in   B, j=k,k+1/2)$ whereas $\text{Piv}_{x_k, 
{\the\numexpr\couprad+\rangeofdep\relax}L}(\widehat{\mathcal {V}}_{k + 1/2}) \in \sigma\big(1\{x 
\in \widehat{\mathcal{I}}_{k+1/2} \}, x\in  \bar B^c)$ for $\bar B \stackrel{{\rm def}.}{=} 
B(x_k, {\the\numexpr\couprad+\rangeofdep\relax}L)$ and therefore these two events 
independent by virtue of \eqref{eq:easy2'}. Thus, the bound obtained in 
\eqref{eq:compa1} factorizes, and applying \eqref{eq:coupling_smallbig_annealed} to bound $\mathbb 
Q[\text{Coup}^c]$ followed by \eqref{eq:compa201} gives
$$\P[\lr{}{ {\mathcal V}_{k + \frac 12}}{U}{V}]- \P[\lr{}{ {\mathcal V}_{k}}{U}{V}] \leq b,$$
for $L \geq L_0$ and $\gamma \geq C$. By definition of $b$ in \eqref{eq:def_b}, the claim \eqref{eq:compa200} follows.

It thus remains to show \eqref{eq:compa201}. Recall our standing assumptions on $\delta,u, r$ and $R$ from the statement of 
Proposition~\ref{prop:comparisonLk}, which are implicitly in force in the statement of the next proposition. For $\alpha, \beta >0$, define $g_{\alpha,\beta}: \mathbb{L} \mapsto [0,\infty)$ given by
	\begin{equation}
		\label{eq:def_g}
		g_{\alpha, \beta}(x) = \begin{cases}
			e^{(\log L)^{\alpha\gamma}}, & \text{ if } |x-x_k|_{\infty} \leq M_1^{2},\mbox{}\\
			e^{(\log L)^{\beta \gamma}}|x-x_k|_{\infty}^{\beta(\log 
				L)^{\alpha\gamma}},  & \text{ if } |x-x_k|_{\infty} >  M_1^{2},
		\end{cases}
	\end{equation}
	where $M_1 \stackrel{{\rm def}.}{=} M (\the\numexpr5*(\couprad+3*\rangeofdep)\relax L)$ (so that, in particular, $M_1 < M_0$, cf.~\eqref{eq:def_M_0}). We will derive the following inequality for $f(\cdot)$ and $b$, see \eqref{def:fx} and 
\eqref{eq:def_b}, which is at the heart of our proof.
\begin{prop}[Surgery]  \label{P:compa_a}
	There exist $\Cl[c]{c:1step}(\delta)\in (\frac 12, 1)$, $\Cl[c]{c:1stepop}(\delta) \in (0,1)$, $\Cl{c:gammapre}(\delta) 
	> 10$ and  $\Cl{C:largestep}(\delta) > 0$ such that for all $\gamma \geq \Cr{c:gammapre} \gamma_M^2$, all dyadic $L \ge C(\delta,\gamma)$, $k  \in \mathbb{N}$ and $x \in \mathbb{L}$, one has
	\begin{equation}
		\label{eq:compa_a_state}
		f(x) \leq 
		b\,g(x) + e^{-c(\log L)^{\Cr{c:1stepop}\gamma} }\bar A f(x),
	\end{equation}
	where $g=g_{\alpha=  \Cr{c:1step}, \beta=\Cr{C:largestep}  }$ and $\bar 
	A$ is the local averaging operator
	\begin{equation}\label{eq:Abar}
		\bar Af(x) \stackrel{{\rm def}.}{=} \frac{1}{|B^{\mathbb{L}}|}\, \sum_{y \in B^{ \mathbb{L}}} f(y),
	\end{equation}
	with $B^{\mathbb{L}}= B^{\mathbb{L}}(x, M_1) = B(x,M_1)  \cap \mathbb{L}$.
\end{prop}

Roughly speaking, \eqref{eq:compa_a_state} should be read as follows: at a cost controlled by $g$ in \eqref{eq:def_g}, one can typically reconstruct the (discrete, see \eqref{eq:def_b}) gradient $b$ from the coarse pivotal event $f$. This may however fail and the resulting `error' is conveniently expressed in terms of 
a suitable average of $f$ itself, attenuated by the pre-factor $\exp\{-c(\log L)^{\Cr{c:1stepop}\gamma} \}$. Due to the form of the error term, the resulting bound lends itself to iteration, as performed below in order to conclude the proof of Proposition~\ref{prop:comparisonLk}. The two cases distinguished in  \eqref{eq:def_g} reflect the fact that the contraction on the error term occasioned by iterating \eqref{eq:compa_a_state} only becomes effective after a certain number of iterations, as will become clear momentarily, until which point the exponent $\gamma$ from \eqref{eq:compa201} takes precedence over the exponent $\alpha \gamma$ (with $\alpha= \Cr{c:1step} < 1$) in the first line of \eqref{eq:def_g}.

\smallskip
Before delving into the proof of Proposition~\ref{P:compa_a}, which is rather involved, let us conclude the proof of Proposition~\ref{prop:comparisonLk}, which we reduced 
to proving the bound \eqref{eq:compa201}. We now explain how Proposition~\ref{P:compa_a} 
implies this bound. Pick $\gamma$ large enough as for the conclusions of Proposition~\ref{P:compa_a} to hold.
Notice that $\bar A$ is a linear operator on the space 
$L^{\infty}(\mathbb{L})$ with unit norm which preserves non-negativity, i.e.~$Af \geq 0$ whenever $f \geq 0$. Consequently, we can iterate \eqref{eq:compa_a_state} to arrive at
\begin{equation}\label{eq:recursion_iterate}
	f(x_k) \leq b \sum_{j = 0}^{m-1}e^{-c\,j(\log L)^{\Cr{c:1stepop}\gamma} } (\bar A^j g) (x_k) + e^{-c\,m(\log L)^{\Cr{c:1stepop}\gamma} }(\bar A^{m}f)(x_k),
\end{equation}
valid for any integer $m \geq 1$. Indeed, \eqref{eq:recursion_iterate} can be verified by a 
straightforward induction argument. Moreover, $\bar A$ acts on $\R^{\mathbb{L}}$ since it has finite range, cf.~\eqref{eq:Abar}. In particular, $\bar A^{j} g$ with $g$ as in \eqref{eq:def_g} (which is not in $L^\infty(\mathbb L)$) is well-defined for every integer $j \geq 0$. 
Now since $\|A\|_{\infty} = 1$ and $|f| \leq 1$, cf.~\eqref{def:fx}, we immediately obtain that
\begin{equation*}
	(\bar A^{m}f)(x_k) \le \|f\|_{\infty} \le 1.
\end{equation*}
Hence, taking the limit $m\to \infty$ in \eqref{eq:recursion_iterate}, we deduce that
\begin{equation}\label{eq:recursion_iterate2}
	f(x_k) \leq b \sum_{j = 0}^{\infty}e^{-c\,j(\log L)^{\Cr{c:1stepop}\gamma} } (\bar A^j g) (x_k). 
\end{equation}
With a view towards \eqref{eq:compa201}, we need to obtain a suitable upper bound on the series \eqref{eq:recursion_iterate2}. To this end, we distinguish two cases in \eqref{eq:recursion_iterate2} based on the value of $j$. 
For $j \leq M_1$ 
and since $A$ has range $M_1$, one has that $\bar A^j g(x_k) = \bar A^j (g1_{x \in B(x_k, M_1^2)})(x_k)$. Hence, for such $j$, using that $\Cr{c:1step} \gamma >\frac{\gamma}{2} > \gamma_M$ (recall that $\gamma \ge 10 \gamma_M^2$) and that
\begin{equation*}
	(\bar A^jg) (x_k) \le \|g1_{x \in B(x_k, M_1^2)}\|_{\infty} \stackrel{\eqref{eq:def_g}}{\le} e^{C(\log L)^{\Cr{c:1step}\gamma}},
\end{equation*}
we readily infer that
\begin{equation}\label{eq:recursion_iterate3}
	\sum_{0 \leq j \leq M_1}e^{-c\,j(\log L)^{\Cr{c:1stepop}\gamma} } (\bar A^j g) (x_k) \leq \sum_{0 \leq j \leq M_1} (\bar A^j g) (x_k) \leq 2^{-1} e^{\Cr{c:coupling101}(\log L)^{\gamma}} ,
\end{equation}
for all $L \geq C(\delta,\gamma)$. In a similar vein, we get for $j \ge M_1$,
\begin{equation*}
	(\bar A^jg) (x_k) \le \|g1_{x \in B(x_k,j M_1)}\|_{\infty} \stackrel{\eqref{eq:def_g}}{\le} e^{C(\log L)^{\Cr{C:largestep} \gamma}}(jM_1)^{\Cr{C:largestep}(\log L)^{\Cr{c:1step}\gamma}},
\end{equation*}
which immediately yields, for $L \geq C(\delta,\gamma)$, 
\begin{equation}\label{eq:recursion_iterate4}
	\sum_{ j > M_1}e^{-c\,j(\log L)^{\Cr{c:1stepop}\gamma} } (\bar A^j g) (x_k) \leq e^{-c'\,M_1(\log L)^{\Cr{c:1stepop}\gamma} }.
\end{equation}
Together, \eqref{eq:recursion_iterate2}, \eqref{eq:recursion_iterate3} and 
\eqref{eq:recursion_iterate4} imply \eqref{eq:compa201} for $L \ge C(\delta,\gamma)$, which concludes the proof of 
Proposition~\ref{prop:comparisonLk} conditionally on the validity of Proposition~\ref{P:compa_a}.
\end{proof}

\subsection{Reduction to two lemmas}\label{subsec:2lemmas}

The proof of Proposition~\ref{P:compa_a} will follow from two intermediate results, stated separately in 
Lemmas~\ref{L:closed_piv2piv} and~\ref{lem:reduce_distance} below. The first of these two 
lemmas is proved in Section~\ref{sec:superdiff}, the second in Section~\ref{sec:penelope}. 
Combining the two results, Proposition~\ref{P:compa_a} quickly follows. Its short proof is 
given at the end of the present section.

Lemma~\ref{L:closed_piv2piv} essentially converts the pivotality event $f$, cf.~\eqref{def:fx}, at the cost of similar additive errors as those appearing in 
\eqref{eq:compa_a_state} into a (nearby) closed pivotal event (recall \eqref{def:closed_piv}) 
for the modified configuration $\mathcal{V}_{k+1}$ at an intermediate, near-diffusive range $R_T$, slightly larger than $\sqrt{L}$, see \eqref{eq:defR_T} and \eqref{eq:def_q1} below. Lemma~\ref{lem:reduce_distance} then further reduces the cluster separation to sub-diffusive scales $R_{T,m}$, eventually removing it entirely, in order to re-construct a `discrete gradient' configuration in $b$, cf.~\eqref{eq:def_b}. Before carrying on, we introduce the notation (cf.~\eqref{eq:barequalstilde} and \eqref{eq:compa200})
\begin{equation}\label{eq:defL_*}
L^*= \begin{cases}
2L , & \text{if $\mathcal V_{\cdot} = \widetilde{\mathcal{V}}^{u, L}_{\cdot}$,} \\
L , & \text{if $\mathcal V_{\cdot} = \overline{\mathcal{V}}^{u, L}_{\cdot}$,}
\end{cases}
\end{equation}
which corresponds to the opposite roles played by $L$ and $2L$ in \eqref{eq:tildeI_k} and \eqref{eq:barI_k}, respectively. The  convenient notation $L^*$ will be freely used throughout the text.

For the remainder of \S\ref{subsec:2lemmas} (as in the statement of 
Proposition~\ref{P:compa_a}), we are always tacitly assuming that 
\begin{equation}\label{eq:cond-2lemmas}
\textstyle\text{$\gamma \geq 10\gamma_M^2$, $\delta \in 
(0, \frac13)$, $u \in (\tilde{u}(1 + 3\delta), u_{**}(1 - 3\delta))$, and $r \geq 1$, $R \geq 2(r 
\vee M_0)$,}
\end{equation} 
where $M_0=M_0(L)$ depends on the range parameter $L$, which is always (often tacitly) assumed to be a dyadic positive integer. In particular, these assumptions are implicit in the statements of Lemma~\ref{L:closed_piv2piv} and \ref{lem:reduce_distance} below. To lighten the load, un-numbered constants such as $c,C,\dots$ may from here onwards implicitly depend on $\delta$ and $\gamma$, as does any parameter $A$ (like $L$ etc.) whenever we use the phrase `for $A$ large enough.'

\subsubsection{Super-diffusive scales} We begin by introducing a suitable framework for the near-diffusive estimate entailed by Lemma~\ref{L:closed_piv2piv}. Its outcome is a bound on $f$ in terms of the quantity $q$ defined in \eqref{eq:def_q1} below.
We will condition on a certain amount of `random sprinkling' $s$, as follows. Recalling $\varepsilon_L$ from \eqref{eq:varepsilon_L}, let
\begin{equation}\label{def:epsx}
\varepsilon = \varepsilon_{L^*} \lceil (\log {L^{\ast}})^{{\gamma}-\gamma_1} \rceil, \text{ for $\gamma_1 \in (0, \gamma)$,}
\end{equation}
with $L^*$ as in \eqref{eq:defL_*}. Abbreviating 
$s (\cdot)= (s_{k+1}-s_k)(\cdot)$ where $s_k = \tilde s_k,$ or $\bar s_k$ depending 
on $\mathcal{V}$ (and similarly for $r_k$, recall \eqref{eq:def_fs} and \eqref{eq:def_fsbar}), we then set
\begin{equation}\label{def:Px}
 \P_x^{\varepsilon}[\, \cdot\,] \stackrel{{\rm def}.}{=} \P \big[\cdot  \big| \,  s_{| C_x}= \varepsilon \big], \text{ where } C_x= B(x, 
 	\the\numexpr\couprad+2*\rangeofdep\relax L) \mbox{ for } x \in \mathbb{L}.
\end{equation}
We will also abbreviate 
$\pi_{x}^{\varepsilon}= \P \big[  s_{| C_x}= \varepsilon \big]$ and $\lambda_x=1\wedge 
\Cr{c:sum_sprinkling}(\frac{L}{|x-x_k|})^{d+1}$. On account of 
\eqref{def:epsx} and \eqref{eq:sigma_bb'}, \eqref{eq:sigma}, \eqref{eq: Gxk_bnd},  using that $P[\text{Poi}(\lambda)=k]^{-1} \leq e (\frac{k}{\lambda})^k $ for all $\lambda \leq 1$ and $k \geq 1$, it follows that $\pi_{x}^{\varepsilon}$ satisfies (cf.~\eqref{eq:def_g}) 
 \begin{align}
\label{eq:defpi}
 (\pi_{x}^{\varepsilon})^{-1}&\leq   \prod_{y \in B^{\mathbb{L}}(x, \the\numexpr \couprad+2*\rangeofdep\relax L)} P\big[\text{Poi}(\lambda_y) = \lceil(\log {L^{\ast}})^{\gamma - \gamma_1}\rceil\big]^{-1} \nonumber \\ 
 &\leq 
 \begin{cases}
e^{C(\log L)^{\gamma-\gamma_1 +\gamma_M}},  & \text{ if } |x-x_k|_{\infty} \leq 2M_1^{2},\mbox{}\\ 
	e^{(\log L)^{\Cl{C:sprinkling_exp}\gamma}}|x-x_k|_{\infty}^{\Cr{C:sprinkling_exp}(\log L)^{\gamma - \gamma_1}}, & \text{ if } 
|x-x_k|_{\infty} > 2M_1^{2} \mbox{}
\end{cases}
\end{align}
for some $\Cr{C:sprinkling_exp}=\Cr{C:sprinkling_exp}(d) > 1$ and all $L \ge C$. 
Finally, for a parameter $\gamma_2 >0$, let 
\begin{equation}
\label{eq:defR_T}
R_T = \lfloor L^{1/2}(\log L)^{\gamma_2} \rfloor.
\end{equation}
(the reason for the notation $T$ will become apparent in \S\ref{surgery-1}; $R_T$ will eventually correspond to the radius of a certain tubular region $T$). The near-diffusive bound stated below is expressed in terms of the quantity
\begin{equation}
\label{eq:def_q1}
q(y) \stackrel{{\rm def}.}{=} \P_{y}^{\varepsilon}\left[\,\overline{\textnormal{Piv}}_{y,\the\numexpr\couprad+\rangeofdep\relax L}(\mathcal {V}_{k+1}), \, 
d_y(\mathscr{C}_U(\mathcal V_{k+1}), \mathscr{C}_V(\mathcal V_{k+1})) \leq R_T \, \right ],
\end{equation}
with $\P_{y}^{\varepsilon}$ as in \eqref{def:Px}, $\mathscr{C}_S(\mathcal{V})$ denoting the cluster of $S \subset 
\Z^d$ in $ \mathcal{V}$ and $d_y(K,K')= d( K_y , K_y')$, for $K,K' \subset \Z^d$ and 
$K_y = K \cap B(y,20L)$. 
\begin{lemma}[Super-diffusive scales, \eqref{eq:cond-2lemmas}]
\label{L:closed_piv2piv} Let $\beta > 2\Cr{C:sprinkling_exp}$ and $\alpha, \gamma_1$, $\gamma_2$ be such that
\begin{equation}
\label{eq:gammacond100}
\text{$\alpha \in (\tfrac{1}{2},1)$, $\gamma_{1} - 3\gamma_M > (1-\alpha) \gamma$ and $\gamma_2 > 3 \gamma_M$}.
\end{equation}
Then, with 
$M_{2}= 4\widetilde{M}(\the\numexpr\couprad+\rangeofdep\relax L)(< M_1 / 4, \text{ see below }\eqref{eq:boxpiv2closedpiv_gen_state})$ and $M=M(L)$, one has
\begin{equation}
\label{eq:compa_a_state.1}
\begin{split}
f(x) \leq &\,  \frac{g_{\alpha,\beta}(x) }{2} b +  e^{-c(\log L)^{2\gamma_2 \wedge \gamma}}Af(x) + e^{C (\log 
M)^{2}} \,\sum_{y \in B^{ \mathbb{L}}(x,M_{2})}   q(y),
\end{split}   
\end{equation}
for all dyadic $L \geq C$ 
and $x \in \mathbb{L}$, where $Af(x) \stackrel{{\rm def}.}{=} 
\sum_{y \in B^{ \mathbb{L}}(x,\frac{1}{2}M_1) } f(y)$.
\end{lemma}

The proof of Lemma~\ref{L:closed_piv2piv} is the subject of Section~\ref{sec:superdiff}.

\subsubsection{
Near- and sub-diffusive scales} \label{sssec_subdiff}We now improve on the outcome of Lemma~\ref{L:closed_piv2piv} by replacing the quantity $q$ ($=q_{m_0}$ in the notation below) appearing in the bound \eqref{eq:compa_a_state.1} by the quantity $q_{0}$ below (see \eqref{eq:q_m-def}), thus progressively reducing the cluster separation to sub-diffusive scales $R_{T,m}$ for $0 \leq m \leq m_0$ (with $R_{T,m_0}=R_T$), in steps indexed by $m$. The corresponding single-step bound relating $q_m$ to $q_{m-1}$ is the content of the next lemma. This lies at the very heart of the surgery argument. For $\gamma_3 \geq 10$, let 
$\Cl{subdivide} = \Cr{subdivide}( \gamma_3 ) > 0$ denote the smallest integer such that 
\begin{equation}\label{def:lowest_scale}
R_T \big(1 - d^{-1}\big)^{m_0} \leq  (\log L)^{\gamma_3 }, \quad \text{ where $m_0 =  \Cr{subdivide}\lfloor \log L \rfloor$.}
\end{equation} 
and define $R_{T,m} = R_T(1 - d^{-1})^{m_0-m}$, so that $R_{T,m_0}=R_T$. Notice that $\sup_{\gamma_3}\Cr{subdivide}( \gamma_3 ) < \infty$. Further let, with $s (\cdot) = (s_{k+1}-s_k)(\cdot) $ as before, and for $y \in \mathbb{L}$,
\begin{equation}\label{def:s'}
s_k' \stackrel{{\rm def.}}{=} s_{k+1} - s' \stackrel{{\rm def.}}{=} s_{k+1} - s_{|C_y} 
 \, (\geq s_k); 
\end{equation}
this singles out the contribution $s'= s_{|C_y} $ employed in the conditioning, see \eqref{def:Px}.
Now consider for any {\em real} $t \in [0, m_0]$ 
the configuration (recall and compare to \eqref{eq:Itildeinclusions},\eqref{eq:Ibarinclusions})
\begin{equation}\label{eq:intermediate_cofig}
\mathcal{J}_{k,  t} = \mathcal{J}_{k,  t; y} \stackrel{{\rm def}.}{=} 
\begin{cases}
\mathcal J^{u, L}_{k+\frac 12}, & \text{ if } t = 0,\mbox{} \\
\mathcal{J}^{u,L}_{k+\frac12} \cup   {\mathcal{J}}^{[ us_{k}, us_{k}'],\, {L^{\ast}}} ( 
\omega_4) \cup  {\mathcal{J}}^{[ us_{k}' , us_{k}' + \frac{t}{m_0}us' ],\, {L^{\ast}}} ( 
\omega_4),& \text{ if } t > 0,\mbox{}
\end{cases}
\end{equation}
with $L^*$ as in \eqref{eq:defL_*} and let (cf.~\eqref{eq:tildeI_k_noised} and \eqref{eq:barI_k_noised})
\begin{equation}\label{def:Ikm}
\mathcal I_{k, t}=
\begin{cases}
	\mathsf{N}^{{L^{\ast}}}(\mathcal J_{k, t} \cap \widetilde B_j) &  \mbox{if } j \leq k,\\
\mathsf{N}^{3L-{L^{\ast}}}(\mathcal J_{k, t}\cap \widetilde B_j) & \mbox{if }  j > k,
\end{cases}
\quad 
 \mathcal{V}_{k, t} = \Z^d \setminus \mathcal I_{k, t}.
\end{equation}
The decomposition in \eqref{eq:intermediate_cofig} when $t>0$ is meant to highlight that the increment $\mathcal{J}_{k,  t} \setminus \mathcal{J}_{k,  t-1} $ for $t= m$ an integer in $\{1,\dots, m_0\}$ (which will be relevant in the sequel) differs depending on whether $m=1$ or $m >1$. With the above choices, for any $y \in \mathbb{L}$ one has that
\begin{equation}\label{eq:intermediate_config_inclusion}
\begin{split}
&\mathcal V_{k, 0} = \mathcal V_{k + \frac 12}, \quad  \mathcal V_{k,\, m_0} =  \mathcal V_{k + 1},\\
& \mathcal V_{k, t} \cap B(y, \the\numexpr\couprad+3*\rangeofdep\relax L)^c = \mathcal V_{k+1} \cap B(y, \the\numexpr\couprad+3*\rangeofdep\relax L)^c \quad \forall\, 0 < t \le m_0\,\,\, \mbox{(cf.~\eqref{eq:barequalstildeRANGE})}\\
&\mathcal V_{k, t} \subset \mathcal V_{k, s} \quad \forall\, 0 \leq s \leq t \leq m_0.
\end{split}
\end{equation}\color{black}
Now define $q_0 = 0$ and for any integer $m$ with $1\leq m \leq m_0$,
\begin{equation}\label{eq:q_m-def}
q_m= q_m(y) = \P_{y}^{\varepsilon}\left[\, \overline{\textnormal{Piv}}_{y, \the\numexpr\couprad+\rangeofdep\relax L}(\mathcal {V}_{k,m}), \,
d_y(\mathscr{C}_U(\mathcal V_{k, m}), \mathscr{C}_V(\mathcal V_{k, m})) \leq R_{T,m} \, \right],
\end{equation}
so $q_{m_0}(y)=q(y)$
with the same notation as in \eqref{eq:def_q1}. Recall that $\gamma_1$ regulates the amount of sprinkling $\varepsilon$, see \eqref{eq:varepsilon_L}, $\gamma_2$ defines the initial near-diffusive scale $R_T$, see \eqref{eq:defR_T}, and $\gamma_3$ the final scale $R_{T,0}$, see \eqref{def:lowest_scale}. To avoid unwieldy notation, we henceforth abbreviate $\Gamma= (\delta, \gamma, \gamma_1, \gamma_2, \gamma_3)$. Our goal is to prove the following.
\begin{lemma}[Critical scales, \eqref{eq:cond-2lemmas}, $y \in \mathbb L$]\label{lem:reduce_distance} There exist 
$\Cl{C:c11_inv}(\delta),  \Cl{C:old6}(\delta), \Cl{C:Cd} \geq 10$ such that, for all 
$\gamma, \gamma_1, \gamma_2, \gamma_3$ satisfying
\begin{equation}
\label{eq:gammacond_intermed}
\text{$\gamma  / \Cr{C:old6} \ge \gamma_{2} \ge \big( (\gamma_{1} + 5) \vee 3\gamma_M \vee \Cr{C:Cd} \big)$ and $\gamma_3 \ge \Cr{C:c11_inv} \gamma_2$},
\end{equation}
and for all dyadic 
$L \ge L_0( \Gamma ) >  1$ and $1 \leq m \leq m_0$, one has, for some $c=c( \Gamma ) > 0$,
\begin{align}
q_{m} &\leq 
e^{(\log L)^{\Cr{C:Cd}\gamma_3}}(q_{m-1} + (\pi_{y}^{\varepsilon})^{-1}\,b ) + e^{-c(\log L)^{\gamma_2}} A f(y). 
\label{eq:reduce_distance1}
\end{align}
\end{lemma}

The proof of Lemma~\ref{lem:reduce_distance} draws on several distinct ideas, which we won't attempt to summarize here; see the beginning Section~\ref{sec:penelope} (which contains the proof) for an overview.

\subsubsection{Proof of Proposition~\ref{P:compa_a}}
Admitting Lemmas~\ref{L:closed_piv2piv} and~\ref{lem:reduce_distance} for the time being, we conclude this section by giving the short:
\begin{proof}[Proof of Proposition~\ref{P:compa_a} (assuming Lemma~\ref{lem:reduce_distance})] Recalling the value of $m_0$ from \eqref{def:lowest_scale} and applying \eqref{eq:reduce_distance1} iteratively over all $1 \le m \le m_0$, we get
\begin{align}
	q(y) \leq
	e^{(\log L)^{\Cl{Cd2}\gamma_3}} (\pi_{y}^{\varepsilon})^{-1}\,b  + e^{-c(\log L)^{\gamma_2}} A f(y)
	\label{eq:reduce_distance1_added}
\end{align}
for suitable $\Cr{Cd2} > \Cr{C:Cd}$, all 
$\gamma, \gamma_1, \gamma_2, \gamma_3$ satisfying \eqref{eq:gammacond_intermed} and all dyadic $L \geq C(\delta, \gamma, \gamma_2, \gamma_3)$.

With a view towards the statement of Proposition~\ref{P:compa_a}, we now set $\Cr{c:gammapre} = 100 \Cr{C:old6}  (\Cr{C:c11_inv}\,\Cr{Cd2})^3$ and consider any $\gamma \geq \Cr{c:gammapre} \gamma_M^2$. We then pick $\gamma_1 = 
\frac{\gamma}{10\Cr{C:old6} (\Cr{C:c11_inv}\,\Cr{Cd2})^2}$, $\gamma_2 = \max(\gamma_1 + 5, 3\gamma_M,\, 
\Cr{C:Cd})  = \gamma_1 + 5$ ($< \gamma/\Cr{C:old6}$) and $\gamma_3 = \Cr{C:c11_inv}\gamma_2$, so that the conditions in
\eqref{eq:gammacond_intermed} are all satisfied. With these choices, we have for the exponent appearing in the first term on the right-hand side of \eqref{eq:reduce_distance1_added} that
$\Cr{Cd2}\gamma_3 \le \frac{\gamma}{5 \Cr{C:old6}\Cr{C:c11_inv}\,\Cr{Cd2}}$.  Thus, substituting the bound for 
$(\pi_{y}^{\varepsilon})^{-1}$ from \eqref{eq:defpi} into \eqref{eq:reduce_distance1_added} and combining the resulting bound with \eqref{eq:compa_a_state.1}, for $\alpha = 1 - (20 \Cr{C:old6} (\Cr{C:c11_inv}\,\Cr{Cd2})^2)^{-1}$, for which condition \eqref{eq:gammacond100} holds, and $\beta= 10\max(\Cr{C:sprinkling_exp}, 
\Cr{C:c11_inv}, \Cr{Cd2}) $ with $\Cr{C:sprinkling_exp}$ as in \eqref{eq:defpi}, yields that
\begin{equation}\label{eq:compa_a_state2}
	f(x) \leq b\,g_{\alpha, \beta}(x) + e^{-c(\log L)^{\gamma_2} }\bar A f(x)
\end{equation}
for dyadic $L \ge C$.
Thus, we deduce Proposition~\ref{P:compa_a} from the inequality \eqref{eq:compa_a_state2} for $\Cr{c:1step} = \alpha (= 1 - (20 \Cr{C:old6}(\Cr{C:c11_inv}\,\Cr{Cd2})^2)^{-1})$, $\Cr{c:1stepop} = (20\Cr{C:old6} (\Cr{C:c11_inv}\,\Cr{Cd2})^2)^{-1}$, with the above choice of 
$\Cr{c:gammapre}$ and $\Cr{C:largestep} = \beta $. 
\end{proof}

\section{Super-diffusive scales}\label{sec:superdiff}

The remainder of this article is devoted to the proofs of Lemmas~\ref{L:closed_piv2piv} and~\ref{lem:reduce_distance}, to which Theorem~\ref{thm:hh1} has been reduced in the previous section. This will take us on a long and exciting journey (the Odyssey). We will try to be cunning archers but the key inequality \eqref{eq:compa_a_state} (Penelope) the two lemmas allow us to establish seemingly requires our arrow to traverse twelve axe heads. The present section deals with Lemma~\ref{L:closed_piv2piv}, which is conceptually simpler. We first gather in \S\ref{subsec:connec} some a-priori connectivity estimates for the mixed models ${\mathcal{V}}^{u,L}_{\ell}$ introduced in \S\ref{Sec:mixedmodelsdef}. These bounds are inherited from $\mathcal{V}^u$ in a suitable regime of the intensity parameter $u$ by means of Proposition~\ref{prop:couple_global}, which acts as a transfer mechanism. We then extend the toolbox by proving in \S\ref{subsec:pivotal-mixed} certain inequalities allowing to switch between pivotal configurations in different vacant configurations and at different scales. The various ingredients are put to use in \S\ref{subsec:pf-supercrit}, which comprises the proof of Lemma~\ref{L:closed_piv2piv}.

\subsection{Connectivity estimates}\label{subsec:connec}

Our starting point is the following connectivity lower bound for the (full) vacant set $\mathcal{V}^u$. Namely, combining \eqref{eq:equivalent_**} and \cite[Lemma 2.2]{RI-II}, one sees that for all $\delta \in (0, 1)$, $u < u_{**}(1 - \delta)$, every $N\geq 1$ and $x, y\in B_N$, 	\begin{align}
	&\label{eq:twopointsbound}
	\P\big[\lr{}{\mathcal V^u \cap B_{N}}{x}{y}\big]\geq e^{-\Cl{ctwopoints}(\delta)(\log N)^2}.
	\end{align}
The next result translates these lower bounds to the inhomogenous models $\mathcal V_{\ell}^{u, L}$ (recall the notational convention from \eqref{eq:barequalstilde}) under suitable assumptions. 
This is not completely innocent. Crucially, the allowed spatial range (parametrized by $N$ below) relative to $L$ does \textit{not} just cover a  regime of `very long' walks, i.e.~the case $L^{1/2} \gg N$. This warrants the use of Proposition~\ref{prop:couple_global} and with it already the full technology of Section~\ref{sec:truncation}, developed separately in \cite{RI-III}. 

\begin{lem}[Connection lower bound] 
	\label{L:connection}
Let $\delta \in (0, \frac 12)$, $\gamma > 2\gamma_M$ and $u\in (\tilde u 
(1+\delta),u_{**}(1-\delta))$. For any $L > 2$ integer power of $2$, $z \in \Z^d$ and $1 \leq N \leq M_0(L)$ (see \eqref{eq:def_M_0}), letting
\begin{equation}\label{def:GN}
G_N(z) \stackrel{{\rm \, def}.}{=} \bigcap_{\substack {k \geq 0: \\ U_k \cap B_N(z) \ne \emptyset}} G^{u,L}(x_k)
\end{equation}
(cf.~\eqref{def:Gxk}), one has 
for all 
$x \in \partial B_N(z)$, $y \in B_N(z)$ and $\tilde \sigma \in G_N(z) $,
\begin{equation}
\label{eq:connectionLB}
\P_{\tilde \sigma}  \big[\lr{}{\mathcal{V} \cap B_N(z)}{x}{y}  \big] \geq e^{- C(\delta, \gamma) (\log N)^2}, \quad \text{with $\mathcal{V} = \mathcal V_{\ell}^{u, L}$}, 
	\quad \ell \in \mathbb{N}/2.
\end{equation}
Moreover, \eqref{eq:connectionLB} continues to hold under the unconditional measure $\P$ in place of $\P_{\tilde \sigma}$.
\end{lem}
\begin{proof}
	
Throughout the proof, generic constants may implicitly depend on $\delta$ and $\gamma $. We will omit the intersection with $B_N(z)$, as present in \eqref{eq:connectionLB}, from all 
connection events below to avoid unnecessary clutter of notations. We first prove a version of \eqref{eq:connectionLB} with $\P$ instead of $\P_{\tilde \sigma}$ and explain how to deduce its quenched version at the end. By writing 
\begin{equation}
\label{eq:compare_mixed_pure0}
\P\big[\lr{}{\mathcal{V}_{\ell}}{x}{y}\big] = \big( 
\P\big[\lr{}{\mathcal{V}_{\ell}}{x}{y}\big]  - 
\P\big[\lr{}{\mathcal{V}_{\lceil \ell \rceil}}{x}{y}\big]  \big) + \P\big[\lr{}{\mathcal{V}_{\lceil \ell \rceil}}{x}{y}\big] \stackrel{\eqref{eq:Itildeinclusions}, \eqref{eq:Ibarinclusions}}{\ge} \P\big[\lr{}{\mathcal{V}_{\lceil \ell \rceil}}{x}{y}\big],
\end{equation}
it is enough to prove \eqref{eq:connectionLB} for $ \ell \in \mathbb{N}$, which will be tacitly assumed from here on. We bound the probability on the right-hand side in two steps. In the first step, we use Proposition~\ref{prop:couple_global} (more precisely, its annealed version \eqref{eq:coupling_smallbig_annealed}) repeatedly in the region $B=B_N(z)$ to obtain a 
lower bound in terms of the same connection probability for 
$\mathcal V^{u', L'}$ for some $u' > u$ close to $u$ and $L' \geq L$ very large compared to $N$. The coupling errors accumulated along the way will turn out to be negligible due to the upper bound constraint on $N$. 
 In the 
second step, we use \eqref{e:loc-limit-I-u-L} to directly 
compare the resulting connection probability with the corresponding quantity for the full interlacement, for which 
 the required lower bound holds in view of \eqref{eq:twopointsbound}. 

The construction of $\mathcal{V}_{\cdot}^{u,L}$ involves an enumeration  $\mathbb{L}=\{ x_0,  x_1, \ldots\}$ of $(4L + 1)\Z^d$ (resp.~$(2L + 1)\Z^d$), see above \eqref{eq:defgh^k} (resp.~\eqref{eq:defghbar^k}).
Although this is strictly speaking not necessary, we assume for convenience that
the ordering first exhausts the points $x_k$ with $B_{N}(z) \cap B_k \neq \emptyset$ 
where $B_k = B(x_k, 2L)$ (resp.~$B(x_k, L)$). Thus,
 $\{ x_0,  x_1, \ldots\}$ has the property that for some $n \in \mathbb{N}$ with $n \leq |B_N|$, one has 
$B_{N}(z)\subset \bigcup_{k < n} U_k$, where $U_k= B_{\couprad L}( x_{k})$. 
With these choices, it follows that 
\begin{equation}\label{e:conn-LB-1}
{\mathcal V}^{u, L}_{\ell} \cap B_N(z)= {\mathcal V}^{u, L}_{\ell \wedge n} \cap B_N(z), \text{ for all $\ell \in \mathbb{N}/2$.}
\end{equation}
Let $\mathbb Q_{k}= \mathbb Q$ refer to the (annealed) coupling \eqref{eq:coupling_smallbig_annealed} and observe that the connection event in \eqref{eq:connectionLB} implies that $y \in \mathscr{C}^{\partial}_B(\mathcal{V})$ with $B=B_N(z)$. Then by \eqref{e:conn-LB-1}, repeated 
application of \eqref{eq:coupling_smallbig_annealed} as well as 
the relationships in \eqref{eq:Itildeinclusions} or \eqref{eq:Ibarinclusions} (as appropriate) 
 we obtain that for all $\ell \in \mathbb{N}$,
\begin{multline}
\label{eq:compare_pure_pure}
\P\big[\lr{}{\mathcal{V}_{\ell}^{u,L}}{x}{y}\big] = \mathbb Q_{\ell \wedge n} \big[\lr{}{\widehat{\mathcal{V}}^{u,L}_{\ell \wedge n}}{x}{y} \big] \\\ge \mathbb Q_{n}\big[\lr{}{\widehat{\mathcal{V}}^{u,L}_{n}}{x}{y} 
\big] - n\, e^{- c(\log L)^{\gamma}} 
\ge \P \big[\lr{}{{\mathcal{V}}^{u_0, L_{0}}}{x}{y} 
\big] -  n\, e^{- c(\log L)^{\gamma}},
\end{multline}
where $(u_0, L_0) \stackrel{{\rm def}.}{=} (u(1 + \delta_1(L)), 2L)$ if $\mathcal{V}_{\cdot}^{u,L} = \widetilde {\mathcal{V}}_{\cdot}^{u,L}$ and $(u(1 + \delta_1(L)), L)$ otherwise. Now consider 
the sequence of scales $(L_i)_{i \ge 0}$ along with an increasing sequence of levels $(u_i)_{i \ge 0}$ defined as $L_{i + 1} = 2L_i$  and $u_{i + 1} =u_i(1 + 
\delta_1(L_i))$ for all $i \ge 0$. Similarly as above, let $\{ x_0^i,  
x_1^i, \ldots\}$ denote an enumeration of $(4L_{i} + 1)\Z^d$ (resp.~$(2L_{i} + 1)\Z^d$) with the property that $B_{N}(z)\subset \bigcup_{k \le n_i}B_{\couprad L_{i}}( x_{k}^i)$ where $n_i$ is the number of points $w$ satisfying $B_{\couprad L_{i}}(w) \cap B_{N}(z) \neq \emptyset$. Repeating the steps leading to \eqref{eq:compare_pure_pure} successively for models defined with respect 
to the sequences $\{ x_0^i,  x_1^i, \ldots\}$ over all $1 \le i \le m $ yields, for arbitrary $m 
\ge 1$, 
\begin{equation}
\label{eq:compare_mixed_largepure}
\P\big[\lr{}{{\mathcal{V}}^{u,L}_{\ell}}{x}{y}\big] \ge  \P \big[\lr{}{{\mathcal{V}}^{u_{m}, L_{m}}}{x}{y} 
\big]  
- n e^{- c(\log L)^{\gamma}} - \sum_{i > 0} n_i e^{- c(\log L_i)^{\gamma}}.
\end{equation}
Since $\gamma > 2\gamma_M$  
and 
$n \vee \sup_{i} \{ n_i \} \le |B_N| \leq C M_0(L)^d = C (M(10^3 L))^d$ by condition on $N$, this immediately gives in view of \eqref{eq:def_M} that 
\begin{equation}
\label{eq:compare_mixed_largepure}
\P\big[\lr{}{\mathcal{V}^{u,L}_{\ell}}{x}{y}\big] \ge \P \big[\lr{}{{\mathcal{V}}^{u_{m}, L_{m}}}{x}{y} \big] - e^{- c(\log L)^{\gamma}}, \text{ for all } m \geq1.
\end{equation}
We now arrive at the second step. By choosing $L_m$ large enough in the previous display, \eqref{e:loc-limit-I-u-L} applies and allows to deduce from \eqref{eq:compare_mixed_largepure} that
\begin{equation}
\label{eq:compare_mixed_full}
\P\big[\lr{}{\mathcal{V}^{u,L}_{\ell}}{x}{y}\big] \ge \P \big[\lr{}{{\mathcal{V}}^{u_{\infty}(1 + \frac{ \delta}{100})}}{x}{y} \big] - e^{- c'(\log L)^{\gamma}}.
\end{equation}
where $u_{\infty}=\lim_n \uparrow u_n (< \infty)$. We are free to prove \eqref{eq:connectionLB} for $L \geq C(\delta)$ only since the remaining cases amount to adapting the constants by restriction on $N$. For $L \geq C(\delta)$, one can ensure that $u_\infty(1 + \frac5{100} \delta) \le u_{**}(1 - 0.9\delta)$, so \eqref{eq:twopointsbound} applies and yields that the first term on the right-hand side of \eqref{eq:compare_mixed_full} is at least $e^{-\Cr{ctwopoints}(\delta) (\log N)^2}$.
The second term is then negligible since $N \leq M_0(L)$ and using $\gamma > 2 \gamma_M > 2$. Overall, this yields \eqref{eq:connectionLB} for $\P$.

It remains to explain how to deduce the corresponding estimate under $\P_{\tilde{\sigma}}$. The (stronger) quenched analogue of \eqref{eq:compare_mixed_pure0} (with $\P_{\tilde{\sigma}}$ everywhere) holds so we may assume that $\ell \in \mathbb{N}$. Then, one simply notes that the analogue of \eqref{eq:compare_pure_pure} with $\P_{\tilde{\sigma}}$ in place of $\P$ everywhere holds whenever $\tilde{\sigma} \in G_N(z)$. To see this, one proceeds exactly as in \eqref{eq:compare_pure_pure} but replaces the coupling \eqref{eq:coupling_smallbig_annealed} by \eqref{eq:coupling_smallbig}, which is in force in view of \eqref{def:GN}. Averaging the lower bound thus obtained over $\tilde \sigma$ and subsequently using \eqref{eq: Gxk_bnd} and a union bound (note to this effect that the intersection in \eqref{def:GN} is over at most $n$ elements), it follows that for all $\tilde \sigma \in G_N(z)$,
$$
\P_{\tilde \sigma}\big[\lr{}{\mathcal{V}_{\ell}^{u,L}}{x}{y}\big] \ge \P \big[\lr{}{{\mathcal{V}}^{u_0, L_{0}}}{x}{y} 
\big] -  n\, e^{- c(\log L)^{\gamma}} -\P[G_N(z)^c] \geq   \P \big[\lr{}{{\mathcal{V}}^{u_0, L_{0}}}{x}{y} 
\big]  -2n\, e^{- c(\log L)^{\gamma}}.
$$
One now proceeds exactly as above to deal with $\P \big[\lr{}{{\mathcal{V}}^{u_0, L_{0}}}{x}{y} 
\big]$ and arrives at \eqref{eq:connectionLB}.
\end{proof}

We now turn to disconnection estimates. The following is an analogue of \eqref{eq:twopointsbound}. By \eqref{eq:tildeu}, one knows that for all $\delta >0$, there exists $\Cl[c]{disco}(\delta)> 0$ such that for $ u \geq \tilde u(1 + \delta)$, $N \geq 1$, $x \in \Z^d$,
\begin{equation}
\label{eq:disconnection_Iu}
\P \big[\nlr{}{\mathcal{V}^u }{B_N(x)}{\partial B_{M(N)}(x)}  \big] \geq {\Cr{disco}}M(N)^{-d}.
\end{equation}
Unlike Lemma~\ref{L:connection}, we only state the relevant disconnection estimate under the annealed 
measure~$\mathbb P$, which is all we need. A 
similar quenched result could however also be derived.
\begin{lem}[Disconnection lower bound]
\label{L:disconnection}
Let $\delta \in (0, 1/2)$, $\gamma > 2\gamma_M$ and $u\in (\tilde u 
(1+\delta), \tilde u \delta^{-1})$. Then for some $c = c(\delta, \gamma) > 0$, all dyadic $L > 2$, $x \in \Z^d$ and $1 \leq N \leq 10^3 L$, we have
\begin{align*}
\label{eq:disconnectionLB}
&\P \big[\nlr{}{\mathcal{V} }{B_N(x)}{\partial B_{M(N)}(x)}  \big] \geq 
cM(N)^{-d}, \quad \text{with $\mathcal{V} = \mathcal V_{\ell}^{u, L}$}, 
\quad \ell \in \mathbb{N}/2.
\end{align*}
\end{lem}
\begin{proof}
The proof is very similar to that of Lemma~\ref{L:connection} except that various steps are performed in the `reverse' direction. More precisely, instead of \eqref{eq:compare_pure_pure}, with $\mathbb{L}=\{x_0,x_1,\ldots\}$ and $n$ similarly in the previous proof, i.e.~the ordering chosen as to minimize $n$ such that $B_{M(N)}(z)\subset \bigcup_{k < n} U_k$, whence in particular, \eqref{e:conn-LB-1} holds with $M(N)$ in place of $N$, one has that
\begin{multline*}
\P\big[\nlr{}{\mathcal{V}_{\ell}^{u,L}}{B_N(x)}{\partial B_{M(N)}(x)}\big] \geq \P\big[\nlr{}{\mathcal{V}_{\lfloor \ell \rfloor}^{u,L}}{B_N(x)}{\partial B_{M(N)}(x)}\big]
=\mathbb Q_{\lfloor \ell \rfloor \wedge n} \big[\nlr{}{\widehat{\mathcal{V}}_{\lfloor \ell \rfloor \wedge n}^{u,L}}{B_N(x)}{\partial B_{M(N)}(x)}\big] \\
\ge \mathbb Q_{0}\big[\nlr{}{\widehat{\mathcal{V}}^{u,L}_{0}}{B_N(x)}{\partial B_{M(N)}(x)} \big] - n e^{- c(\log L)^{\gamma}} \ge \P \big[\nlr{}{{\mathcal{V}}^{u_0, L_{0}}}{B_N(x)}{\partial B_{M(N)}(x)} \big] -  n\, e^{- c(\log L)^{\gamma}},
\end{multline*}
where $(u_0, L_0) = (u, L)$ if $\mathcal{V}_{\cdot}^{u,L} = \widetilde 
{\mathcal{V}}_{\cdot}^{u,L}$ and $(u, 2L)$ if $\mathcal{V}_{\cdot}^{u,L} = \overline 
{\mathcal{V}}_{\cdot}^{u,L}$. Observe that the restriction on $N$ implies that $n \leq e^{C (\log L)^{\gamma_M}}$, which is readily absorbed since $\gamma > 2\gamma_M$. Then one considers the dyadic sequence of 
scales $(L_i)_{i \ge 0}$ as before but with a decreasing sequence of levels
$(u_i)_{i \ge 0}$ defined as $u_{i + 1} = \tfrac{u_i}{1 + \delta_1(L_i)}$ for all $i \ge 0$,
to arrive at the following analogue of \eqref{eq:compare_mixed_largepure},
\begin{equation*}
	\P\big[\nlr{}{{\mathcal{V}}^{u,L}_{\ell}}{B_N(x)}{\partial B_{M(N)}(x)}\big] \ge  \P \big[\nlr{}{{\mathcal{V}}^{u_{m}, L_{m}}}{B_N(x)}{\partial B_{M(N)}(x)} 
	\big]  
	-  e^{- c(\log L)^{\gamma}} ,
\end{equation*}
valid for all integers $m \geq 1$. The remainder of the proof now follows the same steps as in the previous lemma except one 
substitutes \eqref{eq:disconnection_Iu} for \eqref{eq:twopointsbound} while deriving the relevant estimates.
\end{proof}

\subsection{Pivotality and switching}\label{subsec:pivotal-mixed}
We now prove two results that will allow to switch between different
pivotal configurations in $\mathcal{V}_{\ell}^{u,L}$. These will be used in the course of proving Lemma~\ref{L:closed_piv2piv} in the next paragraph, which states in \eqref{eq:compa_a_state.1} an (inductive) bound for the function $f(x)= \P [\text{Piv}_{x, {\the\numexpr\couprad+\rangeofdep\relax}L}(\mathcal {V}_{k +  1 /2}) ]$ introduced in \eqref{def:fx}. Recall to this effect the coarse pivotal event $\piv_K(\mathcal V)$ and close coarse pivotal event $\overline{\textrm{Piv}}_{K}$ from \eqref{def:coarse_piv} and \eqref{def:closed_piv}, which implicitly depend on $R \geq 2r$.
When $K=B_N(x)$, which will frequently occur, we continue to write $\piv_{x, N}(\mathcal V) = \piv_{B_{N}(x)}(\mathcal V)$ and similarly $\overline{\piv}_{x, N}(\mathcal V) = \overline{\piv}_{B_{N}(x)}(\mathcal V)$. Observe that the events in \eqref{def:coarse_piv} and \eqref{def:closed_piv} are both increasing in $K$, and in particular that both $\piv_{x, N}(\mathcal V)$ and $\overline{\textrm{Piv}}_{x,N}(\mathcal{V})$ are increasing in $N$. This fact will be frequently used in the sequel.

Our first lemma concerns a simple yet very useful relationship between the 
probability of pivotal events in two different configurations. The utility of this result 
will become manifest later; see, e.g., 
Remark~\ref{R:passingtoVk+1} for a brief discussion.
\begin{lemma}\label{lem:sandwich_simple}
For any $K \subset \Z^d$ and $\mathcal V' \subset \mathcal V \subset \Z^d$ (under $\P$), one has
	\begin{equation}\label{eq:cor_piv_k+1_to_k+1/2}
	\P \big[\, \overline{\piv}_{K}(\mathcal V')  \big]\le \P[\nlr{}{ \mathcal V' }{U}{V}, \lr{}{ \mathcal V}{U}{V}] + 	\P \big[\, \overline{\textnormal{Piv}}_{K}(\mathcal V) \big].
\end{equation}
In particular, \eqref{eq:cor_piv_k+1_to_k+1/2} holds with $(\mathcal V', \mathcal V) = (\mathcal V_{k + 1}, 
\mathcal V_{k + \frac12})$.
\end{lemma}
\begin{proof}
One first decomposes	
\begin{equation}
	\label{eq:sandwich_parition_con0}
	\P\big[\, \overline{\piv}_{K}(\mathcal V') \big] = \P \big[\, 
	\overline{\piv}_{K}(\mathcal V') , \lr{}{ \mathcal V }{U}{V} \, \big] + \P 
	\big[\, \overline{\piv}_{K}(\mathcal V') , \nlr{}{ \mathcal V }{U}{V} \, \big].
	\end{equation}
	Now by \eqref{def:coarse_piv} and \eqref{def:closed_piv} one has that
	\begin{equation}\label{eq:closedpiv_t_b}
	\big\{\overline{\piv}_{K}(\mathcal V') , \lr{}{ \mathcal V }{U}{V}\big\} \subset \big\{\nlr{}{ 
	\mathcal V' }{U}{V}, \lr{}{ \mathcal V}{U}{V}\big\}
	\end{equation}
	whereas, since $\mathcal V' \subset \mathcal V$, 
	\begin{equation}\label{closedpiv_t_closedpiv}
	\big\{\overline{\piv}_{K}(\mathcal V') , \nlr{}{ \mathcal V
		}{U}{V}\big\} \subset 
	\big\{\lr{}{ \mathcal V \cup K}{U}{V} , \nlr{}{ \mathcal V }{U}{V}\big\} = \overline{\piv}_{K}(\mathcal V).
	\end{equation}
	Feeding \eqref{eq:closedpiv_t_b} and \eqref{closedpiv_t_closedpiv} into \eqref{eq:sandwich_parition_con0} yields \eqref{eq:cor_piv_k+1_to_k+1/2}. Finally, $\mathcal V_{k+1} \subset \mathcal V_{k + 
	\frac12}$ in view of \eqref{eq:Itildeinclusions}, resp.~\eqref{eq:Ibarinclusions}, and thus 
\eqref{eq:cor_piv_k+1_to_k+1/2} is satisfied by $(\mathcal V_{k+1}, \mathcal V_{k + \frac12})$.
\end{proof}

Our second result will be used in various places in order to turn pivotal into closed 
pivotal configurations. Recall the definition of $\widetilde{\omega}_K^L$ following \eqref{eq:barequalstilde}. Below we refer to $\omega_K^{u, L}$ as the point measure induced by $\widetilde{\omega}_K^L$ obtained by retaining only the (finite length) trajectories $w$ corresponding to points 
$(j, v, w) \in \text{supp}(\widetilde{\omega}_K^L)$ with label $v \leq u / L$. 
\begin{lemma}
\label{L:boxpiv2closedpiv} For $\delta \in (0, 1)$ and $\gamma > 2\gamma_M$, there exists $C=C (\delta,\gamma) \in (0, \infty)$ such that the following holds. For all $u\in (\tilde u (1+\delta),u_{**}(1-\delta))$, $L\geq C$ integer power of 2 and all $r \geq 1$, $R \geq 
 2(r \vee 3 M(N))$, with $\mathcal{V}= \mathcal{V}_\ell^{u,L}$, one has
\begin{align}
\label{eq:boxpiv2closedpiv_gen_state}
\P\left[{\piv}_{x,N}(\mathcal{V})\right] \leq CM(N)^d \P\big[ \, 
\overline{\textnormal{Piv}}_{x,\widetilde{M}(N)}(\mathcal {V})\big], \quad \ell \in \mathbb N / 2, \, x \in 
\mathbb{L}, \, 1 \le N \le 10^3 L,
\end{align}
where $\widetilde{M}(N) \stackrel{{\rm def}.}{=} M(N ) + 10 L$. 
Moreover, for any $10 \leq \Delta < 10L$ and $x \in \mathbb{L}$, defining
\begin{equation}
\begin{split}
\label{eq:G_Delta}
&\qquad F_{\Delta} = F_{\Delta, N}(x) \stackrel{{\rm def}.}{=} \big\{\textnormal{diam}_{\ell^{\infty}}(\textnormal{range}(w)) \leq \Delta/4\textnormal{ for all $w \in \text{supp}(\omega_{B_{2N}(x)}^{u_{**}, L})$}\big\}, \mbox{ and}\\
&\qquad H = H_{N}(x) \stackrel{{\rm def}.}{=} G_{2N}(x) \ (\in \mathcal{F}_{\tilde \sigma}) \quad \text{$($see \eqref{def:GN}$)$},
\end{split}
\end{equation}
one has, for $\ell \in \mathbb N / 2$, $x \in 
\mathbb{L}$,  $10 \leq \Delta \leq 10L$ and $10 \le 
N \le \frac1{10}M_0(L)$,
\begin{align}
&\text{ if $\Delta < 10L$: }\label{eq:boxpiv2closedpiv_gen_state2}
  \P_{\tilde \sigma}\big[F_{\Delta},  \, \overline{\textnormal{Piv}}_{x,N}(\mathcal {V})\big] \leq e^{C 
  (\log N)^{2}} \textstyle \P_{\tilde \sigma}\big[\, \bigcup_{z \in 
  B_{2N}(x)}\overline{{\piv}}_{z, \Delta}(\mathcal{V})\,\big] \text{  on $H$, and}\\
&\text{ if $\Delta=10L$: }\label{eq:boxpiv2closedpiv_gen_state2'}
  \P\big[ \overline{\textnormal{Piv}}_{x,N}(\mathcal {V})\big] \leq e^{C 
  (\log N)^{2}} \textstyle \P\big[\, \bigcup_{z \in 
  B_{2N}(x)}\overline{{\piv}}_{z, \Delta}(\mathcal{V})\,\big]. 
\end{align}
\end{lemma}

\begin{proof}
As in the proof of previous lemmas in this section, generic constants $c,C$ may implicitly 
dependent on $\delta, \gamma$ (and $d$). We first show 
\eqref{eq:boxpiv2closedpiv_gen_state}. Let $D_x(N)$ refer to the event that $B_N(x)$ is 
disconnected from $\partial B_{M(N)}(x)$ in $\mathcal{V} = \mathcal V_{\ell}^{u, L}$.  Since 
$u \ge \tilde{u}(1 + 3\delta)$ and for $L \ge C$, which we will henceforth tacitly assume, 
we have by Lemma~\ref{L:disconnection} that
\begin{equation}
\label{eq:compa_a00}
\P[D_x(N)] \geq c M(N)^{-d},\ \text{ for all $N \le 10^3 L$.}
\end{equation}
As we now explain, since $R \geq 2(r \vee 3M(N))$ 
one has the inclusion
\begin{equation}
\label{eq:compa_a0}
D_x(N) \cap\{ \nlr{}{ {\mathcal V \setminus B_N(x)}}{U}{V}\} \subset  \{\nlr{}{ {\mathcal V}}{U}{V}\}.
\end{equation}
To see this, observe that at least one of $U$ and $V$ does not intersect $B_{M(N)}(x)$ due to the bound on $R$. Hence, on the event 
$D_x(N)$ one knows that $A_1\stackrel{{\rm def}.}{=}  \{\lr{}{ {\mathcal V}}{B_N(x)}{U}\}^c \cup  \{\lr{}{ {\mathcal V}}{B_N(x)}{V}\}^c$ 
occurs. On the other hand, the complement of the event $\displaystyle \{ \lr{}{ {\mathcal V \setminus 
		B_N(x)}}{U}{V}\}$ implies that any path in $\mathcal{V}$ joining $U$ and $V$ must 
	intersect $B_N(x)$. Since the occurrence of such a path in $\mathcal{V}$ is disjoint from 
	$A_1$, the event $ \{\lr{}{ {\mathcal V}}{U}{V}\}$ cannot occur, and \eqref{eq:compa_a0} follows. 

Now, recalling the relevant definition(s) relative to monotonicity from above Lemma~\ref{lem:FKG}, abbreviate 
$\mathcal{F}= \mathcal{F}_{B(x,M(N))}^L$ and observe that $D_x(N) \in \mathcal{F}$ is 
decreasing in $B(x,M(N))$ (under $\mathbb P_{\tilde \sigma}$). Moreover, defining $F_x(N)= 
\{  \lr{}{ S}{U}{V}\}$ where $S\stackrel{{\rm def}.}{=} \mathcal V \cup B(x,\widetilde{M}(N))$, the event 
$\displaystyle F_x(N) \cap \{ \lr{}{{\mathcal V \setminus B_N(x)}}{U}{V}\}^c$ is decreasing 
in $B( x, M(N))$ by our choice of $\widetilde{M}(N)$ as well as the radii of the boxes in 
$\widetilde{\mathcal{B}}$ and $\overline{\mathcal{B}}$ (cf.~\eqref{eq:barequalstildeRANGE}). 
In view of \eqref{def:coarse_piv}, we thus obtain by monotonicity and Lemma~\ref{lem:FKG} that
\begin{equation}
\label{eq:compa_a1}
\begin{split}
\P\big[ \, \overline{\textnormal{Piv}}_{x, \widetilde{M}(N)}(\mathcal {V})\big] &\stackrel{\eqref{eq:compa_a0}}{\geq} \P\big[ \, D_x(N), \nlr{}{ {\mathcal V \setminus B_N(x)}}{U}{V}, \, F_x(N)\big] 
\stackrel{\eqref{eq:FKG}}{\geq} E^{\tilde \sigma}[f(\tilde \sigma) g(\tilde \sigma)],
\end{split}
\end{equation}
where $f(\tilde \sigma)=\P_{\tilde \sigma}[ D_x(N)] $ and $g(\tilde \sigma) = \displaystyle \P_{\tilde \sigma}\big[ \nlr{}{ {\mathcal V \setminus B_N(x)}}{U}{V} , F_x(N) \big]$. Now  
recalling that $(\Sigma_{K}^L)^c = 
\{\sigma_{B'}^B: B, B' \in \mathcal{B}, B \cap K = 
\emptyset\}$ where $\mathcal B \stackrel{{\rm def}.}{=} \widetilde{\mathcal{B}} \cup \overline{\mathcal{B}}$,  
consider the $\sigma$-algebra $\widetilde{\Sigma}^c$ 
generated by $((\Sigma_{K}^L)^c, \mathsf U_{K^c})$ for $K = B(x, M(N))$. The function $f$ is clearly 
increasing in $(\sigma \setminus (\Sigma_{K}^L)^c, 1 - \mathsf U_{K})$ under $P^{\tilde \sigma}[ \,\cdot\, | 
\,\widetilde{\Sigma}^c]$. The same is true of $g$ as $\widetilde{M}(N) = M(N) + 10 L$. 
Since $f(\tilde \sigma)$ is independent of $\widetilde{\Sigma}^c$ and 
$E^{\tilde \sigma}[ f(\tilde \sigma)] = \P[D_x(N)]$, it 
follows by the FKG-inequality for independent variables that
$$E^{\tilde \sigma}[f(\tilde \sigma) g(\tilde \sigma) \, | \,\widetilde{\Sigma}^c] \ge \P[D_x(N)] \, \P\big[ \nlr{}{ {\mathcal V 
\setminus B_N(x)}}{U}{V} , F_x(N) \, | \, \widetilde{\Sigma}^c \big].$$
Hence, taking expectation on both sides, we obtain that the left-hand side of \eqref{eq:compa_a1} is bounded from below by
$$\P[ D_x(N)] \, \P\big[ \nlr{}{ {\mathcal V \setminus B_N(x)}}{U}{V} , F_x(N) \big] \stackrel{\eqref{eq:compa_a00}}{\geq} cM(N)^{-d} \P\left[{\piv}_{x,N}(\mathcal{V})\right],$$
where we used monotonicity and \eqref{def:closed_piv} in the last step. This completes the 
proof of \eqref{eq:boxpiv2closedpiv_gen_state}.

\bigskip

We now show \eqref{eq:boxpiv2closedpiv_gen_state2} and \eqref{eq:boxpiv2closedpiv_gen_state2'}. We will assume that $\Delta < N$ 
since otherwise both inequalities are trivial. Let $\mathscr{C}$ denote the cluster of 
$U$ in $\mathcal V \cap B_R$. That is, $\mathscr{C}$ is the union of $U$ and all connected components of $\mathcal V \cap B_R$ intersecting $U$. For convenience, we set $H = F_{\Delta} = \Omega$ (the full space) when $\Delta = 10L$. Recall that $\overline S = S \cup \partial_{{\rm out}} S$ for $S \subset \Z^d$.
We can then write on the event $H$ (cf.~\eqref{eq:G_Delta}), for all $10 \leq \Delta \leq 10L$, 
\begin{equation}\label{eq:pivtolocalizey}
\begin{split}
\P_{\tilde \sigma}\big[F_{\Delta},  \, \overline{\textnormal{Piv}}_{x,N}(\mathcal {V})\big]
	&\leq \sum_{\substack{\mathcal C: \, \mathcal C \cap \overline{B}_N(x) \neq \emptyset, \\ 
			\mathcal C \cap V =\emptyset}}  \P_{\tilde \sigma}[F_{\Delta}, \mathscr C = \mathcal C, 
		\lr{}{ {\mathcal V} \setminus \overline{\mathcal C}}{\overline {B}_N(x)}{V}]   \\
		&\le \sum_{y \in \overline{B}_N(x)} \, \sum_{\substack{\mathcal C: \, \mathcal C \cap \overline{B}_N(x) \neq \emptyset, \\ 
				\mathcal C \cap V =\emptyset}}  \P_{\tilde \sigma}[F_\Delta, \mathscr C = \mathcal C, 
		\lr{}{ {\mathcal V} \setminus \overline{\mathcal C}}{y}{V}]  \\
		&=  \sum_{y \in \overline{B}_N(x)} (\Sigma_{1, y} + \Sigma_{2, y}),
\end{split}
\end{equation}
where $\Sigma_{1, y}$ and $\Sigma_{2, 
y}$ consist of the terms corresponding to $\mathcal C$ satisfying $\mathcal C_{\Delta - 1} \ni y$ and $\not 
\ni y$, respectively; here $\mathcal C_{r}$ refers to the $r$-neighborhood of $\mathcal{C}$, see \S\ref{s:not}. In fact, one immediately gets
\begin{equation}\label{eq:Sigma1ypivbnd}
\Sigma_{1, y} \leq \P_{\tilde \sigma} \big[y \in \mathscr C_{\Delta - 1},\,  \overline \piv_{y, \Delta}(\mathcal V)  \big] \le \P_{\tilde \sigma} \Big[y \in \mathscr C_{\Delta - 1},\, \bigcup_{z \in B_{2N}(x)} \overline \piv_{z, \Delta}(\mathcal V)  \Big].
\end{equation}
To deal with the case $ y \notin \mathcal C_{\Delta - 1}$, we start by claiming that
\begin{equation}
\label{eq:monot200}
\text{ the event $\{F_\Delta, \mathscr C = \mathcal C, \lr{}{{\mathcal V} \setminus \overline{\mathcal C}}{y}{V}\}$ is increasing in $B_{{2}N}(x) 
\setminus \mathcal C_{\Delta - 1}$}
\end{equation}
(with `increasing' as defined above Lemma~\ref{lem:FKG}). Indeed, the monotonicity of $\{ \lr{}{ {\mathcal V} 
\setminus \overline{\mathcal C}}{y}{V}\}$ is clear. To see that $ \{\mathscr C = \mathcal C, F_\Delta \}$ is also increasing in $B_{2N}(x) \setminus \mathcal C_{\Delta - 1}$, first 
recall that $\mathcal V \cap K$ is measurable relative to $\omega_{K}^{u, L}$ (see above 
Lemma~\ref{L:boxpiv2closedpiv}). Now observe that the diameter of any trajectory in $\text{supp}(\omega_{B_{2N}(x)}^{u, L})$ is at most 
$\Delta / 4$ on the event $F_\Delta$ when $L \ge C$ (the latter ensures that $u (1+ \delta_1(L)) \leq u_{**}$, cf.~\eqref{eq:G_Delta}). It follows that $ \{\mathscr C = \mathcal C, F_\Delta \}$
is increasing since $F_\Delta$ is and on the event $F_\Delta$, no trajectory can intersect both 
$B_{{2}N}(x) \setminus \mathcal C_{\Delta - 1}$ and $\mathcal{C}$. 
Thus \eqref{eq:monot200} is shown.

Now, write $\Sigma_{2,y}(\mathcal{C})$ for the quantity obtained from $\Sigma_{2,y}$ upon disintegrating over $\{ \mathscr C = \mathcal C\}$. Hence, $\mathcal{C}$ varies over all connected sets such that $\mathcal C \cap \overline{B}_N(x) \neq \emptyset$, $\mathcal C \cap V =\emptyset$ and $y \notin \mathcal C_{\Delta - 1}$. In view of Lemma~\ref{lem:FKG}, using  \eqref{eq:monot200}, for all such $\mathcal{C}$ we have the bound
\begin{align}\label{eq:Sigma2ybnd}
\Sigma_{2, y}(\mathcal{C}) \le 
		p_y(\tilde \sigma)^{-1} \P_{\tilde \sigma}[F_\Delta, \mathscr C = \mathcal C, 
\lr{}{ {\mathcal V} \setminus \overline{\mathcal C}}{y}{V}, \lr{B_{2N}(x)}{ {\mathcal V} \setminus \mathcal C_{\Delta - 1}}{y}{\mathcal C_{\Delta}}] ,
\end{align}
where $p_y(\tilde \sigma) \stackrel{{\rm def}.}{=} \P_{{\tilde \sigma}}[\lr{B_{2N}(x)}{ {\mathcal V} \setminus 
\mathcal C_{\Delta - 1}}{y}{\mathcal C_{\Delta}}]$ (here and elsewhere $\lr{K}{ \mathcal V}{U}{V}= \lr{}{  K \cap \mathcal V}{U}{V}$).

We now consider the cases $\Delta < 10L$ and $\Delta = 10L$ separately, and first complete the proof of \eqref{eq:boxpiv2closedpiv_gen_state2}. Thus assume that $\Delta < 10L$. Fixing an arbitrary point $y' \in \mathcal C \cap \overline{B}_N(x)$, one finds a point $z$ such that $y, y' \in B_K(z)$ for some $K \leq N+1$, $B_K(z) \subset B_{2N}(x)$ and one of $y,y'$ is contained in $\partial B_K(z)$. Hence, we get from Lemma~\ref{L:connection} that on the event
$ H$ (recall \eqref{eq:G_Delta}) and for any $\mathcal{C}$ as above, 
\begin{equation}\label{eq:pysigma_bnd}
p_y(\tilde \sigma) = \P_{{\tilde \sigma}}[\lr{B_{2N}(x)}{ {\mathcal V} \setminus 
\mathcal C_{\Delta - 1}}{y}{\mathcal C_{\Delta}}] \ge \P_{{\tilde \sigma}}[\lr{B_{K}(z)}{ {\mathcal V}}{y}{y'}] \ge e^{-C (\log N)^2}.
\end{equation}
Plugging this into \eqref{eq:Sigma2ybnd} and summing over $\mathcal{C}$, we obtain, on the event $H$,
\begin{equation}
\label{eq:sigma2final}
\Sigma_{2, y} \le e^{C (\log N)^2}\P_{{\tilde \sigma}} \Big[y \notin \mathscr C_{\Delta - 1},\, \bigcup_{z \in B_{2N}(x)} \overline \piv_{z, \Delta}(\mathcal V)  \Big].
\end{equation}
Together with \eqref{eq:Sigma1ypivbnd} this gives the bound
\begin{equation*}
	\Sigma_{1, y} + \Sigma_{2, y} \le e^{C (\log N)^2}\P_{{\tilde \sigma}} \Big[\bigcup_{z \in B_{2N}(x)} \overline \piv_{z, \Delta}(\mathcal V)  \Big] \text{ on $H$,}
	\end{equation*}
which then leads to \eqref{eq:boxpiv2closedpiv_gen_state2} for all $\Delta < 10L$ in view of 
\eqref{eq:pivtolocalizey} and the bound on $N$.

\medskip

Returning to \eqref{eq:Sigma2ybnd}, we now assume that $\Delta=10L$ and supply the proof of 
\eqref{eq:boxpiv2closedpiv_gen_state2'}. 
The functions $p_y(\tilde \sigma)$ and $\Sigma_{2, y}(\mathcal{C}) \,( =\P_{\tilde \sigma}[\mathscr C = \mathcal C, 
\lr{}{ {\mathcal V} \setminus \overline{\mathcal C}}{y}{V}]\mbox{ since }F_{\Delta} = \Omega)$ 
are both decreasing in the variables 
$(\Sigma \setminus \Sigma_{\overline{\mathcal C}}^L,  1 - \mathsf U_{\overline{\mathcal C}^c})$.  This is plain 
in the former case. In the latter simply observe that the event $\{ \mathscr C = \mathcal C \}$ does not depend on the variables $(\Sigma \setminus \Sigma_{\overline{\mathcal C}}^L, \mathsf U_{\overline{\mathcal C}^c})$ by definition.
Therefore, 
\begin{equation}
\label{eq:FKG300}
\begin{split}
&E^{{\tilde \sigma}}\big[ \P_{\tilde \sigma}[ \mathscr C = \mathcal C, 
\lr{}{ {\mathcal V} \setminus \overline{\mathcal C}}{y}{V}, \lr{B_{2N}(x)}{ {\mathcal V} \setminus \mathcal C_{\Delta - 1}}{y}{\mathcal C_{\Delta}}]  \, \big| \, \sigma(\Sigma_{\overline{\mathcal C}}^L, \mathsf U_{\overline{\mathcal C}})\big]\\
&\qquad \stackrel{\eqref{eq:Sigma2ybnd}}{\geq} E^{{\tilde \sigma}}\big[  \Sigma_{2, y}(\mathcal{C}) \cdot p_y(\sigma)   \, \big| \, \sigma(\Sigma_{\overline{\mathcal C}}^L, \mathsf U_{\overline{\mathcal C}})\big] \geq E^{{\tilde \sigma}}\big[  \Sigma_{2, y}(\mathcal{C})    \, \big| \, \sigma(\Sigma_{\overline{\mathcal C}}^L, \mathsf U_{\overline{\mathcal C}})\big] \cdot \P[\lr{B_{2N}(x)}{ {\mathcal V} \setminus 
\mathcal C_{\Delta - 1}}{y}{\mathcal C_{\Delta}}],
\end{split}
\end{equation}
where for the last inequality, we have used the FKG-inequality for independent random variables and the fact that $E^{{\tilde \sigma}}[ p_y(\sigma)    \,| \, \sigma(\Sigma, \mathsf U_{\overline{\mathcal C}})]= E^{{\tilde \sigma}}[ 
p_y(\sigma) ] $ by independence (recall that $\Delta = 10L$). In order to obtain a lower bound on this 
probability, we use the same line of argument as the one used for \eqref{eq:pysigma_bnd}. To thid end, foregoing the restriction to stay outside $\mathcal 
C_{\Delta - 1}$, fixing a point in $\partial  \mathcal{C}_{\Delta} $, shrinking the ball 
$B_{2N}(x)$ suitably (see the argument preceding \eqref{eq:pysigma_bnd}) and using the bound \eqref{eq:connectionLB} in its annealed version, the 
last term in \eqref{eq:FKG300} is readily seen to be bounded from below by $e^{-C (\log 
N)^2}$ (cf.~\eqref{eq:pysigma_bnd}). Rearranging in \eqref{eq:FKG300}, summing over 
$\mathcal{C}$ and taking expectations with respect to $E^{\tilde \sigma}$, one readily deduces that 
$E^{{\tilde \sigma}}[\Sigma_{2, y}]$ is bounded by the right-hand side of \eqref{eq:sigma2final} with 
$\P$ in place of $\P_{{\tilde \sigma}}$. Together with \eqref{eq:Sigma1ypivbnd}, this yields \eqref{eq:boxpiv2closedpiv_gen_state2'}. 
\end{proof}

\subsection{Proof of Lemma~\ref{L:closed_piv2piv}} \label{subsec:pf-supercrit} We have now gathered all the tools required to prove Lemma~\ref{L:closed_piv2piv}, which, among other things, makes frequent use of Lemma~\ref{L:boxpiv2closedpiv}. Recall the notions of pivotality ($\text{Piv}$) and closed pivotality ($\overline{\text{Piv}}$) introduced 
in \eqref{def:coarse_piv}, \eqref{def:closed_piv}. Roughly speaking, by a combination of \eqref{eq:boxpiv2closedpiv_gen_state} and \eqref{eq:boxpiv2closedpiv_gen_state2'}, one replaces the pivotal probability $f$ (see \eqref{def:fx}) to be bounded by a nearby closed pivotality $\bar f$, see \eqref{def:fbarx}, at the same scale. The reduced cluster separation inherent to $q$ on the right-hand side of \eqref{eq:compa_a_state.1} is then engineered by application of \eqref{eq:boxpiv2closedpiv_gen_state2} for the choice $\Delta= R_T$ (see \eqref{eq:defR_T}). An important intermediate step is to enforce the conditioned measure from \eqref{def:Px}, which along with the noise warrants a slightly subtle use of correlation inequalities, and to switch into the configuration $\mathcal{V}_{k+1}$ entering $q(\cdot)$ in \eqref{eq:def_q1}. We return to the necessity for this switching at the end of the proof; see Remark~\ref{R:passingtoVk+1}. 
\begin{proof}[Proof of Lemma~\ref{L:closed_piv2piv}]
The proof consists of three 
steps. 

\medskip
\textbf{Step 1: from $\text{Piv}$ to nearby $\overline{\text{Piv}}$.} By analogy with \eqref{def:fx}, we define
\begin{equation}
\label{def:fbarx}
\bar f(x) \stackrel{{\rm def}.}{=} \P\big [\, \overline{\text{Piv}}_{x, \the\numexpr\couprad+\rangeofdep\relax L}(\mathcal {V}_{k + 1/2}) \, \big],\quad x \in \mathbb{L}.
\end{equation}
As we now explain, combining \eqref{eq:boxpiv2closedpiv_gen_state} with the choice $N=\the\numexpr\couprad+\rangeofdep\relax L$ and \eqref{eq:boxpiv2closedpiv_gen_state2'} with $N=\widetilde{M}(\the\numexpr\couprad+\rangeofdep\relax L) (\leq \frac1{10}M_0(L))$ 
to bound the resulting 
$\P[\overline{\textnormal{Piv}}_{x,\widetilde{M}(\the\numexpr\couprad+\rangeofdep\relax L)}(\mathcal V_{k + \frac12})]$, 
one obtains for all dyadic $L \geq C$ that \color{red} \color{black}
\begin{equation}
\label{eq:compa_a_state.2}
f(x) \leq e^{C (\log M)^{2}} \,\sum_{y \in B^{\mathbb{L}}(x,3\widetilde{M}(\the\numexpr\couprad+\rangeofdep\relax L))}\bar{f}(y)
\end{equation}
(with $M=M(L)$ as defined in \eqref{eq:def_M}). Let us briefly elaborate on the last part of our 
reasoning. We will perform similar arguments tacitly in the remainder of this section. In 
deducing \eqref{eq:compa_a_state.2}, applying \eqref{eq:boxpiv2closedpiv_gen_state2'} first 
yields a sum over $y' \in 
B(x,2\widetilde{M}(\the\numexpr\couprad+\rangeofdep\relax L))$ of pivotal events similar to 
$\bar{f}$ in \eqref{def:fbarx} but at scale $\Delta=10L$. Using monotonicity of 
$\overline{\text{Piv}}_{K}$ in $K\subset \Z^d$ and the fact that for any $y' \in \Z^d $, one 
finds $y \in \mathbb{L}$ with $|y-y'| \leq 5L$ such that $B(y',10L)\subset B(y, 20L)$, which 
produces an inconsequential multiplicity factor $CL^d$ when passing from a sum over $y'$ to a 
sum over $y$, the bound \eqref{eq:compa_a_state.2} follows.

\medskip

\textbf{Step 2: modifying $\mathcal{V}_{k + 1/2}$.}
Next, we switch the configuration inherent to $\bar f$ in \eqref{def:fbarx} from 
$\mathcal{V}_{k + 1/2}$ to the (smaller) set $\mathcal{V}_{k+1} (\subset 
\mathcal{V}_{k + 1/2})$, cf.~\eqref{eq:Itildeinclusions} and \eqref{eq:Ibarinclusions}, thereby slightly increasing the pivotality radius and triggering the sprinkling, i.e.~enforcing the conditioned measure from \eqref{def:Px}. The final outcome of this will be to replace the bound \eqref{eq:compa_a_state.2} by \eqref{eq:compa_a_state.2NEW} below.
Accordingly, let
\begin{equation}
\label{def:fbarxprime}
\bar{\bar f}(y) \stackrel{{\rm def}.}{=}  \P_{y}^{\varepsilon}\big [\, \overline{\text{Piv}}_{y, \the\numexpr\couprad+2*\rangeofdep\relax L}(\mathcal {V}_{k + 1}) \, \big],
\end{equation}
for $y \in \mathbb{L}$. With $b$ as defined in \eqref{eq:def_b}, we claim that 
\begin{equation}
\label{eq:compa_a_state.3}
\begin{split}
\bar{f}(y)  \leq  \bar{\bar f}(y) +  (\pi_{y}^{\varepsilon})^{-1} e^{C(\log L)^2}b, \quad y \in \mathbb{L}.
\end{split}
\end{equation}	
To see this, first note that, on account of \eqref{def:closed_piv} and since $\mathcal{V}_{k+1} \subset \mathcal{V}_{k+1/2}$, the intersection of the events $ 
\overline{\textnormal{Piv}}_{y,\the\numexpr\couprad+\rangeofdep\relax L}(\mathcal {V}_{k+1/2})$ and $\{ \lr{}{}{U}{V} \text{ in } (\mathcal {V}_{k+1}\cup \widetilde{B})\}$ with 
$\widetilde{B}=C_y$ is contained in $ 
\overline{\textnormal{Piv}}_{y,\the\numexpr\couprad+2*\rangeofdep\relax L}(\mathcal {V}_{k+1}) $, hence
\begin{equation}
\label{eq:compa_a_state.3.1}
\begin{split}
  \P_{y}^{\varepsilon}\big[\, \overline{\textnormal{Piv}}_{y,\the\numexpr\couprad+\rangeofdep\relax L}(\mathcal {V}_{k+ \frac 1 2}), \,  \lr{}{ {\mathcal V}_{k+1}\cup \widetilde{B}}{U}{V} \big] \leq \bar{\bar f}(y), \quad y \in \mathbb{L}.
\end{split}
\end{equation}
As we are about to show, one further has that 
\begin{equation}
\label{eq:compa_a_state.3.2}
\begin{split}
\P_{y}^{\varepsilon}\big[\, \overline{\textnormal{Piv}}_{y,\the\numexpr\couprad+\rangeofdep\relax L}(\mathcal {V}_{k+ \frac 1 2}), \,  \nlr{}{{\mathcal V}_{k+1}\cup \widetilde{B}}{U}{V} \big] \leq (\pi_{y}^{\varepsilon})^{-1} e^{C(\log L)^2} b, \quad y \in \mathbb{L}.
\end{split}
\end{equation}
Once \eqref{eq:compa_a_state.3.2} is shown, one concludes as follows. The sprinkling $s(\cdot)$ (see below \eqref{def:epsx}) underlying the definition of  $ \P_{y}^{\varepsilon}$ is independent of $\mathcal{V}_{k+\frac12}$, which follows upon recalling \eqref{eq:Itildeinclusions}, \eqref{eq:Ibarinclusions} and the sentence immediately following the latter. Hence, one obtains that $\bar{f}(y)= \P_{y}^{\varepsilon} [\, \overline{\text{Piv}}_{y, \the\numexpr\couprad+\rangeofdep\relax L}(\mathcal {V}_{k + \frac 1 2}) \, ]$, cf.~\eqref{def:fbarx}, and \eqref{eq:compa_a_state.3.1} and \eqref{eq:compa_a_state.3.2} immediately yield \eqref{eq:compa_a_state.3}.

We now explain \eqref{eq:compa_a_state.3.2}. First, retaining only the event $\{ \lr{}{ {\mathcal V}_{k+1/2}\cup {B}}{U}{V} \} \supset  \overline{\textnormal{Piv}}_{y,\the\numexpr\couprad+\rangeofdep\relax L}(\mathcal {V}_{k+ 1 /2}) $ with $B=B(y,\the\numexpr\couprad+\rangeofdep\relax L)$, we bound
\begin{equation}
\label{eq:compa_a_state.3.3}
\begin{split}
 \P_{y}^{\varepsilon}\big[\, \overline{\textnormal{Piv}}_{y,\the\numexpr\couprad+\rangeofdep\relax L}(\mathcal {V}_{k+ \frac 1 2}), \,  \nlr{}{ {\mathcal V}_{k+1}\cup \widetilde{B}}{U}{V} \big] \leq \sum_{u,v \in \partial_{\text{out}}B} \P_{y}^{\varepsilon}\big[\, \lr{}{ {\mathcal V}_{k+\frac12} }{u}{U} , \, \lr{}{ {\mathcal V}_{k+\frac12}}{v}{V} , \,  \nlr{}{ {\mathcal V}_{k+1}\cup \widetilde{B}}{U}{V} \big].
 \end{split}
\end{equation}
Let $\Sigma_y$ refer to the union of the collections $
\{\sigma_{B'}^{B''}: B', B'' \in {\mathcal{B}}, B'  \cap \widetilde{B}^c \neq \emptyset\}$ and $\{ \sigma^{B_k}_{B'} : B'\in {\mathcal{B}}\}$; recall that $B_k$ refers to the box of radius $L/2L$ depending on the model and centered at $x_k$.
Note that the sprinkling $s$ defined below \eqref{def:epsx} is a function of the latter collection and abbreviate $s_{| C_y} = \varepsilon$ 
by $s=\varepsilon$ in the sequel; note that this event has non-zero probability because $\varepsilon$ in \eqref{def:epsx} is an integer multiple of $\varepsilon_{{L^{\ast}}}$. The event appearing on the right-hand side of \eqref{eq:compa_a_state.3.3} and $C_{u,v}= \{ \lr{}{ {\mathcal V}_{k+1/2} \cap \overline{B}}{u}{v} \}$ being both increasing in $\overline{B}(=B \cup \partial_{\text{out}}B)$ for any $u,v \in \partial_{\text{out}}B$ under $\P_{\tilde \sigma}$, we see that
\begin{equation*}
\begin{split}
b &\stackrel{\eqref{eq:def_b}}{\geq} E^{\tilde \sigma}\Big[ \P_{\tilde \sigma}\big[ \, \lr{}{ {\mathcal V}_{k+\frac12} }{u}{U} , \, \lr{}{ {\mathcal V}_{k+\frac12}}{v}{V} , \,  \nlr{}{ {\mathcal V}_{k+1}\cup \widetilde{B}}{U}{V} , \, C_{u,v} \big] 1_{\{s=\varepsilon\}}\Big]\\
&\stackrel{\eqref{eq:FKG}}{\geq} E^{\tilde \sigma}\Big[ \P_{\tilde \sigma}\big[ \, \lr{}{ {\mathcal V}_{k+\frac12} }{u}{U} , \, \lr{}{ {\mathcal V}_{k+\frac12}}{v}{V} , \,  \nlr{}{ {\mathcal V}_{k+1}\cup \widetilde{B}}{U}{V}\big]   \P_{\tilde \sigma}\big[ C_{u,v} \big] 1_{\{s=\varepsilon\}}\Big]\\
&\ \, \stackrel{}{\geq} E^{\tilde \sigma}\Big[ \P \big[ \, \lr{}{ {\mathcal V}_{k+\frac12} }{u}{U} , \, \lr{}{ {\mathcal V}_{k+\frac12}}{v}{V} , \,  \nlr{}{ {\mathcal V}_{k+1}\cup \widetilde{B}}{U}{V} \, \big| \, \sigma(\Sigma_y, \mathsf U_{\widetilde{B}^c})\big] \P\big[ C_{u,v} \, \big| \, \sigma(\Sigma_y,  \mathsf U_{\widetilde{B}^c}) \big]1_{\{s=\varepsilon\}}\Big],
\end{split}
\end{equation*}
where, in passing from the second to the third line, we have conditioned on $\sigma(\Sigma_y, \mathsf U_{\widetilde {B}^c})$ under $P^{\tilde \sigma}$, whence 
$1_{\{s=\varepsilon\}}$ can be pulled out, applied the FKG-inequality to the independent family of random 
variables $(\sigma \setminus \Sigma_y, (1 - \mathsf U_x)_{x \in \widetilde B})$ (the two functions in question are both decreasing in these random 
variables). Now, using the fact that $\mathcal{V}_{k+\frac12}\cap \overline{B}$ is independent from 
$\sigma(\Sigma_y, \mathsf U_{\widetilde {B}^c})$ in view of the discussion preceding 
\eqref{eq:barequalstildeRANGE} and applying Lemma~\ref{L:connection}, one obtains that
$$
\P[ C_{u,v}| \sigma(\Sigma_y,  \mathsf U_{\widetilde {B}^c})]=\P[ C_{u,v}]  \ge e^{-C(\log L)^2}, \text{ $P^{\tilde \sigma}$-a.s.}
$$
Substituting the last two estimates into \eqref{eq:compa_a_state.3.3} readily gives \eqref{eq:compa_a_state.3.2}. Overall this completes the verification of \eqref{eq:compa_a_state.3}.
Substituting the latter into \eqref{eq:compa_a_state.2} yields that
\begin{equation}
\label{eq:compa_a_state.2NEW}
f(x) \leq e^{C (\log M)^{2}} \,\sum_{y \in 
B^{\mathbb{L}}(x,3\widetilde{M}(\the\numexpr\couprad+\rangeofdep\relax L))} \big( \bar{\bar 
f}(y) +  (\pi_{y}^{\varepsilon})^{-1} e^{C(\log L)^2}b\big). 
\end{equation}

\textbf{Step 3: reducing the cluster separation.} On account of \eqref{eq:compa_a_state.2NEW} and in view of the desired estimate \eqref{eq:compa_a_state.1},  it remains to reduce the range of pivotality in $ \bar{\bar f}(y)$ in \eqref{def:fbarxprime} from $\the\numexpr\couprad+2*\rangeofdep\relax L$ down to $R_T$.
On the one hand, writing $F_{\Delta}=F_{\Delta,N}(y)$, $H=H_N(y)$ for a given $y \in \mathbb{L}$ for the events in \eqref{eq:G_Delta} and applying \eqref{eq:boxpiv2closedpiv_gen_state2} with $N=\the\numexpr\couprad+2*\rangeofdep\relax L$, $\Delta= R_T $, yields that
\begin{equation}
\label{eq:compa_a_state.4}
\begin{split}
&  \P_{y}^{\varepsilon}\big [\, F_{\Delta}, \, H, \,  \overline{\text{Piv}}_{y,\the\numexpr\couprad+2*\rangeofdep\relax L}(\mathcal {V}_{k + 1}) \, \big] = (\pi_{y}^{\varepsilon})^{-1} E^{\tilde \sigma}\big[  \P_{\tilde \sigma} \big [\, F_{\Delta}, \,  \overline{\text{Piv}}_{y,\the\numexpr\couprad+2*\rangeofdep\relax L}(\mathcal {V}_{k + 1}) \, \big] 1_{\{H, s=\varepsilon\}} \big]\\[0.4em]
&\quad  \leq e^{C(\log \the\numexpr\couprad+2*\rangeofdep\relax L)^2} \sum_{z \in B(y,   \the\numexpr2*(\couprad+2*\rangeofdep)\relax  L)}  (\pi_{y}^{\varepsilon})^{-1}  E^{\tilde \sigma}\big[  \P_{\tilde \sigma} 
\big [\,  \overline{\text{Piv}}_{z,\Delta}(\mathcal {V}_{k + 1}) \, \big] 1_{\{H, 
s=\varepsilon\}} \big] \leq  e^{C'(\log L)^2} \sum_{z \in B(y,   
\the\numexpr2*(\couprad+2*\rangeofdep)\relax  L)}  q(z)
 \end{split}
\end{equation}
(see \eqref{eq:def_q1} regarding $ q(\cdot)$), where the last inequality follows by omitting $H$, observing that $B(z,\Delta) \subset B(z,20L)$ and noting that by choice of $\Delta$, the event $\overline{\text{Piv}}_{z,\Delta}(\mathcal {V}_{k + 1})$ implies 
in particular that $d_z(\mathscr{C}_U(\mathcal V_{k+1}), \mathscr{C}_V(\mathcal V_{k+1})) \leq R_T$ must occur. On the other hand, as we now explain, for all $ y \in B(x, 3( \widetilde{M}(\the\numexpr\couprad+\rangeofdep\relax L) )$ and $L$ large enough, 
\begin{equation}
\label{eq:compa_a_state.5}
  \P_{y}^{\varepsilon}\big [\,  (F_{\Delta} \cap H)^c,  \overline{\text{Piv}}_{y, \the\numexpr\couprad+2*\rangeofdep\relax L}(\mathcal {V}_{k + 1}) \, \big] \leq e^{-c(\log L)^{2 \gamma_2 \wedge \gamma}} \big(b+ Af(x)\big).
\end{equation}
To see this, first note that $ \overline{\text{Piv}}_{y, \the\numexpr\couprad+2*\rangeofdep\relax L}(\mathcal {V}_{k + 1}) \subset  {\text{Piv}}_{y,\the\numexpr4*(\couprad+2*\rangeofdep)\relax L}(\mathcal {V}_{k + 1})$ and the latter is independent of 
$\{ F_{\Delta} ,\, H, s_{| C_y}= \varepsilon\}$ under $\P$ on account of \eqref{eq:G_Delta}, see also \eqref{def:GN} and \eqref{def:Gxk} regarding the events entering the definition of $H$ and \eqref{eq:def_fs}, \eqref{eq:def_fsbar} (recall that $s (\cdot)= (s_{k+1}-s_k)(\cdot)$). Thus,
\begin{equation}
\label{eq:compa_a_state.5.1}
  \P_{y}^{\varepsilon}\big [\,  (F_{\Delta} \cap H)^c,  \overline{\text{Piv}}_{y, \the\numexpr\couprad+2*\rangeofdep\relax L}(\mathcal {V}_{k + 1}) \, \big] \leq   \P_{y}^{\varepsilon}\big [\,  (F_{\Delta} \cap H)^c] \cdot \P\big [\, {\text{Piv}}_{y, \the\numexpr4*(\couprad+2*\rangeofdep)\relax L}(\mathcal {V}_{k + 1}) \, \big].
\end{equation}
We consider each factor on the right-hand side of \eqref{eq:compa_a_state.5.1} separately. Regarding $\P [ {\text{Piv}}_{y,\the\numexpr4*(\couprad+2*\rangeofdep)\relax L}(\mathcal {V}_{k + 1}) \, ]$, applying a similar argument as the one leading to 
\eqref{eq:compa_a_state.2}, which combines \eqref{eq:boxpiv2closedpiv_gen_state} and \eqref{eq:boxpiv2closedpiv_gen_state2'}, then applying Lemma~\ref{lem:sandwich_simple} to 
change configurations from $\mathcal{V}_{k+1}$ to $\mathcal{V}_{k+1/2}$ in the resulting 
closed pivotal configurations, and using the fact that $3( \widetilde{M}(\the\numexpr\couprad+\rangeofdep\relax L) +\widetilde{M}(\the\numexpr4*(\couprad+2*\rangeofdep)\relax L)) + 5L \leq M_1/2 $, one arrives at
\begin{equation}
\label{eq:compa_a_state.5.2}
 \P\big [\, {\text{Piv}}_{y,\the\numexpr4*(\couprad+2*\rangeofdep)\relax L}(\mathcal {V}_{k + 1}) \, \big] \leq e^{C (\log M)^2} \big( b + A f(x)\big),
\end{equation}
(for all $ y \in B(x, 3( \widetilde{M}(\the\numexpr\couprad+\rangeofdep\relax L) )$) with $f$ as in \eqref{def:fx}. We now bound $ \P_{y}^{\varepsilon}\big [\,  (F_{\Delta} \cap H)^c] $ which will produce a (desired) small 
prefactor in \eqref{eq:compa_a_state.5.1}. Observing that $F_{\Delta}$ is independent of 
$\mathcal{F}_{\tilde \sigma}$ and recalling that $\Delta=  R_T $ and 
$N=\the\numexpr\couprad+2*\rangeofdep\relax L$, we can bound the number of 
trajectories in $\text{supp}(\omega_{B_{2N}(x)}^{u_{**}, L})$ by $CuL^{d-1}$ up to an 
exponentially small error in $L$ using the Poissonian tail estimate \eqref{eq:Poisson_tailbnd}.
Combined with the fact that a single trajectory has diameter exceeding $\Delta$ with 
probability at most $\sum_{|y|> \Delta} p_{CL}(0,y) \leq e^{-c(\Delta)^2/L}$, one deduces with the help of \eqref{eq:defR_T} that
 \begin{equation}
 \label{eq:compa_a_state.5.3}
  \P_{y}^{\varepsilon} [ (F_{\Delta})^c ] =  \P [  (F_{\Delta})^c ]  \leq   e^{-c(\log L)^{2 \gamma_2}}.
 \end{equation} 
 
Bounding $ \P_{y}^{\varepsilon} [  H^c ]$ requires a small amount of care because $H$ is not independent of the event $\{s=\varepsilon\}$ entering the conditioning under $\P_{y}^{\varepsilon}$. Recalling the definition of $H= H_{N}(y)$ from \eqref{eq:G_Delta}, \eqref{def:GN} and \eqref{def:Gxk}, we see that $H$ involves the fields $\mathsf U_{V_{k'}}$, $r_{k'}$ and $s_{k'}$ for suitable $k'$ (such that $x_{k'}$ is within a certain distance from $y \in \mathbb{L}$, which is arbitrary). In particular, $k'$ bears no relation to $k$, which indexes the configuration $\mathcal{V}_{k+1}$ of interest and to which $s=s_{k+1}-s_k$ refers. Whereas $\mathsf U_{V_{k'}}$ and $r_{k'}$ are independent of $s$ for any $k, k'$ (cf.~\eqref{eq:def_fs}, \eqref{eq:def_fsbar} and recall from above \eqref{eq:sigma} that $\sigma_L$ and $\sigma_{2L}$ are independent), the fields $s_{k'}$ and $s$ won't be when $k' \geq k$. To deal with this, first observe that 
$\varepsilon$ in \eqref{def:epsx} is bounded by 
$\tfrac{\delta_2(2L)}{200} = \tfrac{\delta_1(2L)}{200\Cr{sprinkling1}}$ for all $L$ large enough owing to the definitions of 
$\varepsilon$, $\varepsilon_{L}$ and $\delta_1(L)$  as laid out in \eqref{eq:varepsilon_L} and \eqref{eq:delta-2-mixed}, respectively.
Thus, writing $s_{k'}= (s_{k'}-s) + s$ for $k' \geq k$, noting that the term in parentheses is both non-negative and independent of $s$, one deduces that on the event $\{s=\varepsilon\}$, the excess deviation $s_{k'} \geq \tfrac{\delta_2(2L)}{100} $ implied by $H^c$ (cf.~\eqref{def:Gxk})
implies that $ (s_{k'}-s) \geq  \tfrac{\delta_2(2L)}{200}$, which is independent of $\{s=\varepsilon\}$, thus inducing the desired decoupling.
Putting the pieces together and applying a union bound, it follows that $ \P_{y}^{\varepsilon} [  H^c ]$ is bounded by
\begin{equation*}
  \sum_{\substack{z, K}} \sum_{z',\,  z''}\Big(\P\big [u( \varepsilon_{2L} \sigma_{2L}(z'') \vee  \varepsilon_{L}\sigma_{L}(z'')) \geq 
\textstyle\frac{\delta_2(2L)}{200}\big] 
 + \P\big[ \mathsf U_{z''} \notin [\tfrac{e^{-L}}{2}, 1 - \tfrac{e^{-L}}{2}] \big] \Big),
\end{equation*} 
where the first sum ranges over $z \in C_y$, $ 1 
\le K \le \the\numexpr\couprad+2*\rangeofdep\relax L$ and the second one over $z' \in \mathbb{L}$ such that $ B_{{\the\numexpr\couprad\relax} L}(z') \cap B_K(z) \ne \emptyset$ and $z'' \in B_{{\the\numexpr\couprad+\rangeofdep\relax}L}(z')$. Now, with a similar calculation as the one leading to \eqref{eq:coupling_smallbig_annealed} (see \eqref{eq: 
Gxk_bnd} in particular), one sees that
\begin{equation}
\label{eq:compa_a_state.5.4}
 \P_{y}^{\varepsilon} [  H^c ] \le e^{-c(\log L)^{\gamma}}.
\end{equation}
Combining \eqref{eq:compa_a_state.5.2}, \eqref{eq:compa_a_state.5.3} and \eqref{eq:compa_a_state.5.4} gives \eqref{eq:compa_a_state.5} for large enough $L$, since $2 
\gamma_2 \wedge \gamma > 3 \gamma_M $.

\smallskip

Finally, we put everything together to conclude the proof. Using \eqref{eq:compa_a_state.4} and \eqref{eq:compa_a_state.5} to bound $\bar{\bar f}(y)$ on the 
right hand side of \eqref{eq:compa_a_state.2NEW} we obtain, for all $L$ large enough,
\begin{equation}\label{eq:closed_piv2piv_final}
f(x) 
\leq e^{C(\log M)^2} b \sum_{y \in 
	B^{\mathbb{L}}(x,\frac34 M_2)}  (\pi_{y}^{\varepsilon})^{-1}  + e^{-c(\log L)^{2\gamma_2 \wedge \gamma}}Af(x) + e^{C (\log 
	M)^{2}} \,\sum_{y \in B^{ \mathbb{L}}(x,M_{2})}    q(y)
\end{equation}
where 
we used the relations $\gamma_M > 1$ and $2\gamma_2 \wedge \gamma > 3\gamma_M$ to arrive at the prefactors and 
the definition of $M_{2} = 4\widetilde{M}(\the\numexpr\couprad+\rangeofdep\relax L)$ for the 
limit of the last summation. On the other hand, since $M_1^2 + 3\widetilde{M}(\the\numexpr\couprad+\rangeofdep\relax L) \le 2 M_1^2$ and $(\pi_{x}^{\varepsilon})^{-1}$ is 
increasing in the distance $|x - x_k|_{\infty}$, see \eqref{eq:defpi}, we have from the relation $\gamma - \gamma_{1} + 3\gamma_M < \alpha \gamma$ implied by \eqref{eq:gammacond100} that for $L$ large enough,
\begin{equation*}
e^{C(\log M)^2} \sum_{y \in B^{\mathbb{L}}(x,\frac34 M_2)}  (\pi_{y}^{\varepsilon})^{-1}  \leq 
\begin{cases}
	\frac{e^{(\log L)^{\alpha \gamma}}}{10},  & \text{ if } |x-x_k|_{\infty} \leq M_1^{2},\mbox{}\\ 
	e^{(\log L	)^{\Cr{C:sprinkling_exp} \gamma}}|x-x_k|_{\infty}^{\Cr{C:sprinkling_exp} (\log L)^{\alpha  \gamma}}, & \text{ if } 
	|x-x_k|_{\infty} > M_1^{2}. \mbox{}
\end{cases}	
\end{equation*}
However 
the right-hand side is readily seen to be at most $g_{\alpha, \beta}(x)/2$ in view of the assumption that 
$\beta > 2\Cr{C:sprinkling_exp}$, thus yielding \eqref{eq:compa_a_state.1} after substituting into \eqref{eq:closed_piv2piv_final}.
\end{proof}

\begin{remark}[$\mathcal{V}_{k+\frac12}$ vs.~$\mathcal{V}_{k+1}$]\label{R:passingtoVk+1} In 
the sequel we will switch several times between the (closed) pivotal events in two different 
configurations in the proof of Lemma~\ref{lem:reduce_distance} in Section~\ref{sec:penelope}. Such switchings will take place in the process of progressively reducing the 
distance between $\mathscr{C}_U(\mathcal V_{k+1})$ and $\mathscr{C}_V(\mathcal V_{k+1})$ from $R_T$, see~\eqref{eq:def_q1}, to 
some power of $\log L$, see \eqref{def:lowest_scale}. However, it is equally 
important that the underlying pivotal region be preserved at \textit{all} scales below $R_T$ while we switch (much in the spirit of \eqref{eq:compa_a_state.1}). This turns out 
to be possible at all scales only when one switches {\em from} an (`almost') smaller configuration, see Lemma~\ref{cor:sandwich}, cf.~also 
Lemma~\ref{lem:sandwich_simple}. This explains the choice of starting (the proof of) Lemma~\ref{lem:reduce_distance} in the smallest possible configuration in the continuum (cf.~\eqref{eq:intermediate_cofig}) of models lying between $\mathcal V_{k + 1}$ and $\mathcal V_{k + 
\frac12}$.
\end{remark}

\section{The model $\mathcal{V}_T$}
\label{A:superhard}

In this section we prepare the ground for the proof of Lemma~\ref{lem:reduce_distance}. As opposed to the previous section, which dealt with super-diffusive scales, the surgery employed to further reduce the cluster separation is effectively a critical problem at near-diffusive scales $R_T= R_{T,m_0}$ and below (see \eqref{eq:defR_T} and \eqref{def:lowest_scale}), and correspondingly (much) more involved; we will return to this in the next section, where the actual surgery is performed. Matters are even worse because a welcome result to achieve separation of scales, such as suitable a conditional decoupling property akin to \cite[Proposition 2.3]{RI-II} (see also \cite{CaioSerguei2018,PopTeix}), is not readily available around diffusive length scales $\approx R_T$ for the walks of length $\approx L$ involved in the configurations $\mathcal{V}_{k,m}$ of interest, which interpolate between the models $\mathcal V_{k,\, m_0} =  \mathcal V_{k + 1}$ and $\mathcal V_{k, 0} = \mathcal V_{k + \frac 12}$ introduced in Section~\ref{sec:toolbox}. In a sense, the truncation to a finite time horizon $\approx L$ is most severely felt around spatial scales $R_{T,m} \approx \sqrt{L}$, because the trajectories are neither random walks, nor are they pure noise, i.e.~really small.

 Refining the considerations of Section~\ref{sec:toolbox}, the task of the present section is thus to modify the models $\mathcal{V}_{k,m}$ suitably inside a tubular region $T$ (in which the surgery will eventually occur), in a manner as to obtain a configuration $\mathcal{V}_T$ that
\begin{itemize}  
\setlength\itemsep{0.1em}
\item[i)] has good conditional decoupling properties, yet
\item[ii)] remains close to $\mathcal{V}_{k,m}$. 
\end{itemize}
The model $\mathcal{V}_T$ is a local modification of $\mathcal{V}_{k,m}$ in which `time runs for free' inside $T$ (cf.~\eqref{eq:phi_T} and \eqref{def:tau}). This is designed to facilitate i), but ii) may a-priori be completely out of control. Crucially, the tubular region~$T$, see \eqref{def:cylinder} below, has in fact narrow (near- but sub-diffusive) width, so the walks tend to quickly exit $T$, which deep down is the reason why the local modification does not spoil proximity to $\mathcal{V}_{k,m}$, and ii) is preserved. 

We begin by introducing in \S\ref{surgery-1} the tubes $T$ and defining the model $\mathcal{V}_T$. The two key features, Properties~i) and~ii) above, are encapsulated in Lemmas~\ref{lem:caio_T} and~\ref{L:superhardcoupling} below, respectively. The reader could choose to advance directly to Section~\ref{sec:penelope} at the end of \S\ref{surgery-1} in order to proceed more quickly through the proof of Lemma~\ref{lem:reduce_distance} and return to the proofs of Lemmas~\ref{lem:caio_T} and~\ref{L:superhardcoupling} afterwards. The remainder of Section~\ref{A:superhard} is devoted to the proofs of these two lemmas. We start with some preparation in \S\ref{sec:decouple}, which contains crucial a-priori random walk estimates related to $T$.
The proof of each lemma has a designated subsection among \S\ref{sec:caio} and \S\ref{subsubsec:superhard} below. 

Whereas the decoupling property i) (see Lemma~\ref{lem:caio_T}) is relatively straightforward to deduce (the definition of $\mathcal{V}_T$ is tailored to it), the proximity requirement inherent to ii)  (Lemma~\ref{L:superhardcoupling}) is more involved. The statement is formulated as a coupling, which roughly corresponds to a more elaborate version of Proposition~\ref{prop:couple_global} that concerned the `base' model $\mathcal{V}_k$. Its proof draws -- yet again -- heavily on our coupling results of \cite{RI-III}, see also \S\ref{subsec-coup-gen} (the philosophy is superficially similar to that underlying the proof of Proposition~\ref{prop:couple_global}, but it involves different models, notably $\mathcal{V}_T$). Among these couplings, the harder Theorem~\ref{thm:long_short} requires working against a suitably large background configuration $\mathcal{I}^{\rho}$ (cf.~Def.~\ref{def:background} and \eqref{eq:long-j-i}); recall that this is non-negotiable as it supplies the necessary environment configuration for the obstacle set. This also specifies the role of the disconnection estimate, i.e.~the condition $u \gtrsim \tilde{u} $ (see \eqref{eq:cond-2lemmas}) in the surgery argument of the next section.

\subsection{Definition and main properties} \label{surgery-1} 
We begin by introducing the relevant tubes, which come in two types, $\ell^\infty$- and $\ell^2$-tubes. Let $z \in \Z^d$, $ 1\leq j \leq d$ a coordinate direction and ${N, L \ge 0}$ be integers. The cross-section of $B(z,L)$ through $z$ orthogonal to $e_j$ is defined as $B^j_{\perp} (z,L) = B(z,L) \cap \{ z'\in \Z^d: z_j'=z_j \}$ and $B^{2,j}_{\perp} (z,L)$ is defined accordingly, replacing $B(z,L)$ by the $\ell^2$-ball $B^2(z,L)$. The ($\ell^{\infty}$-){\em tube} of length $N + 2L$ and (cross-sectional) radius 
$L$ in the $j$-th coordinate direction is the set
\begin{equation}
\label{eq:def-tube}
T_{L, N}^j(z) \stackrel{\text{def.}}{=} \bigcup_{0 \le n \le N}\, B(z + ne_j, L) = \bigcup_{-L \le n \le N+L}\, 
B^j_{\perp}(z + ne_j, L);
\end{equation}
the corresponding $\ell^2$-{tube} $T_{L, N}^{2,j}(z)$ is defined similarly, with $B^2/B^{2,j}_{\perp}$ in place of $B/B^j_{\perp}$ in \eqref{eq:def-tube}.

We now introduce three axis-aligned tubes $T \subset T' \subset T^{\circ}$, the outermost of which will be of $\ell^2$-type, which will play a central role in what follows. Recalling from \eqref{eq:defR_T} that $R_T =  \lfloor L^{1/2}(\log L)^{\gamma_2} \rfloor$ where $\gamma_2 > 1$, let 
\begin{equation}\label{eq:kappa}
r_T =  4\lceil L^{1/2}(\log L)^{- \bar \gamma_2}\rceil, \text{ with $\bar \gamma_2 \ge 3 \gamma_2$}, \quad  r_T' = 2 r_T, \quad r_T^{\circ}=  4\lceil L^{1/2}(\log L)^{-2{\gamma}_2}\rceil
\end{equation}
and with the notation of \eqref{eq:def-tube}, define
\begin{equation}\label{def:cylinder}
T=T_{r_T,R_T}^j(z), \quad   T' =T_{r_T',R_T}^j(z), \quad T^{\circ} = T_{r_T^{\circ},R_T}^{2,j}(z).
\end{equation}
The specification of $T, T'$ and $T^{\circ}$ thus depends on the choice of a vertex $z \in \Z^d$, a direction $j \in \{1,\dots, d\}$ and the parameters $\gamma_2$, $\bar{\gamma}_2$. The `thin width' alluded to in the above discussion is thus quantified by \eqref{eq:kappa}. We will later perform certain geometric constructions inside $T$ for suitable choices of these parameters. 

We now introduce the model $\mathcal{V}_T$ associated to $T$ in \eqref{def:cylinder}, which will be obtained by suitable modification of $\mathcal {V}_{k, m}$  in \eqref{eq:pixy}. First recall that $\mathcal {V}_{k,t}$ is declared under the measure $\P_{\tilde{\sigma}}$ (and $\P$) introduced below \eqref{e:disorder-mixed}, and involves four i.i.d.~Poisson processes $\omega_i$, $1\leq i \leq 4$ (see \eqref{eq:tildeI_k}, \eqref{eq:barI_k} and \eqref{eq:intermediate_cofig}), each having intensity $\nu$ given by \eqref{eq:mu_intensity}. In order to meaningfully manipulate the (time-)length of individual trajectories, we first restore a label corresponding to this information. Thus, to a realization of the point measure $\omega_i =\sum_j \delta_{(u_j, w_j)}$, we associate the extended point measure $\widehat{\omega}_i =\sum_j \delta_{(u_j, \ell_j, w_j)}$ which carries a (length) label $\ell_j \in \mathbb{N}^\ast$. By taking the product of $\nu$ in 
\eqref{eq:mu_intensity} with a measure $\mu_i$ on $\mathbb{N}^\ast$, this induces a 
Poisson process on $  \R_+ \times \mathbb{N}^\ast \times W_+$. Upon choosing $\mu_i=  \delta_{3L-{L^{\ast}}} $ for $i=1,2$ and $\mu_i=  \delta_{{L^{\ast}}} $ for $i=3,4$ with $L^*$ as in \eqref{eq:defL_*}, the random sets $\mathcal{J}_{k,t}$, $\mathcal{I}_{k,t}$, in 
\eqref{eq:intermediate_cofig}--\eqref{def:Ikm} can naturally be viewed as a function of $ \widehat{\omega} =({\widehat{\omega}_1}, \dots {\widehat{\omega}_4})$. With a slight abuse of notation, we assume henceforth that $\P_{\tilde{\sigma}}$ (and $\P$) carry the processes $\widehat{\omega}_i$ and we view $\mathcal{J}_{k,t}$, $\mathcal{I}_{k,t}$ as functions of $\widehat{\omega}$ under $\P_{\tilde{\sigma}}$.
Now set $\Phi_T({\widehat{\omega}})= (\Phi_T(\widehat{\omega}_1), \dots, \Phi_T({\widehat{\omega}_4}))$, where, writing $\widehat{\omega}_i =\sum_j \delta_{(u_j, \ell_j, w_j)}$ for a generic realization, we let
\begin{equation}
\label{eq:phi_T}
\Phi_T(\widehat{\omega}_i) \stackrel{{\rm def}.}{=}  \sum_{j } \delta_{(u_j, \phi_{T}(\ell_j, w_j))} 1\{ w_j(0) \notin T^{\circ}\},
\end{equation}
 with $\phi_{T}(\ell, w) = (\tau, w)$ and
\begin{equation}\label{def:tau}
\tau= \tau(\ell, w) \stackrel{{\rm def}.}{=} \inf \{ \ell' \in \mathbb{N}: |\{ w(n): 0 \leq n < \ell' \}| \cap (\Z^d \setminus T') \geq \ell\}.
\end{equation}
With this we define
\begin{equation}
\label{eq:I_T}
{\mathcal{J}}_T= 
{\mathcal{J}}_T^u({\widehat{\omega}}) = 
\mathcal{J}_{k,  m-\frac14} (\Phi_T({\widehat{\omega}}))
 \end{equation}
(cf.~\eqref{eq:intermediate_cofig} regarding $u$) and define $ \mathcal I_T$ in terms of ${\mathcal{J}}_T$ exactly as in \eqref{def:Ikm}, that is, replacing $\mathcal{J}_{k,t}$ by ${\mathcal{J}}_T$ everywhere in \eqref{def:Ikm}. The set $\mathcal{V}_T$ of interest is then simply its complement,  $\mathcal{V}_T=\Z^d \setminus {\mathcal{I}}_T$. In words, 
$\mathcal{J}_T$ {\em removes} all trajectories underlying $\mathcal{J}_{k,  m-1/4}$ that start 
inside $T^{\circ}$, and `lets time run for free inside $T'$', i.e.~the 
remaining trajectories $w_i$ (starting outside $T^{\circ}$) run until time $\ell_i$ has 
been accumulated outside $T'$. 

All the triplets of tubes $(T,T' , T^{\circ})$ we will consider in \eqref{def:cylinder} will have the property that
\begin{equation}
\label{eq:tube-cond}
(T\subset T'\subset) \ T^{\circ} \subset  
B(y, \the\numexpr\couprad+\rangeofdep+10\relax L),
\end{equation}
where $y \in \mathbb{L}$ refers to the point implicit in the construction of $\mathcal{J}_{k,t}$ in \eqref{eq:intermediate_cofig}; this is the same $y$ appearing in the statement of Lemma~\ref{lem:reduce_distance}, which we eventually want to prove. Condition \eqref{eq:tube-cond} will always be assumed to hold from here onwards. As a consequence of \eqref{eq:tube-cond} and by definition of $\mathcal{V}_T$ we also have, using the property in the second line of \eqref{eq:intermediate_config_inclusion},
\begin{equation}\label{eq:locally_identical} 
\mathcal V_T \cap B(y, \the\numexpr\couprad+3*\rangeofdep\relax L)^c = \mathcal V_{k+1} \cap B(y, 
\the\numexpr\couprad+3*\rangeofdep\relax L)^c.
\end{equation}
For $K \subset \Z^d$, let $\widehat{\omega}_K$ denote the process obtained by keeping the points $(v, 
\ell, w) \in {\rm supp}(\widehat{\omega}_i)$ with $w[0, \ell - 1] \cap K \ne \emptyset$ for all $1\leq i \leq 4$. It then follows in view of 
\eqref{eq:phi_T}--\eqref{eq:I_T} that, under $\P_{\tilde \sigma}$, 
\begin{equation}
	\label{eq:VTRANGE}
	\text{$\mathcal V_T \cap K$ is measurable relative to $\widehat{\omega}_{K\cup T'}$ and independent of $(\widehat{\omega}_{K\cup T'})^c \stackrel{\text{def.}}{=} \widehat{\omega} - \widehat{\omega}_{K\cup T'}$}
\end{equation}
(here with hopefully obvious notation the subtraction is meant coordinatewise, i.e.~individually for each $i$).

\medskip

The key properties of $\mathcal{V}_T$ are encapsulated in the next two lemmas, the first of which states that $\mathcal{V}_T$ has good conditional decoupling features. Let $\mathcal{F}_K= \sigma(1\{ x' \in 
\mathcal{V}_T\}: x' \in K)
$. We further denote by $ \Phi_T^{\textnormal{loc}}({\widehat{\omega}})= (\Phi_T^{\textnormal{loc}}(\widehat{\omega}_1), \dots, \Phi_T^{\textnormal{loc}}({\widehat{\omega}_4}))$ the localized version of $ \Phi_T({\widehat{\omega}})$, defined exactly as in \eqref{eq:phi_T} but retaining only starting points satisfying $w_j(0) \in 
B(T^{\circ}, 2L) \setminus T^{\circ}$. Observe that $\Phi_T^{\textnormal{loc}}({\widehat{\omega}})$ is measurable with respect to $\widehat{\omega}_{B(y,\the\numexpr\couprad+\rangeofdep+20\relax L)}$ on account of \eqref{eq:tube-cond}. 
\begin{lemma}[Conditional~decoupling for $\mathcal{V}_T$]\label{lem:caio_T} 
For all $u, \delta >0$, some $\Cl{C:gamma2}(\delta) \ge 3$, all
$\bar \gamma_2 \ge \Cr{C:gamma2}\, \gamma_2$, $L \ge C(\gamma_2, \bar \gamma_2)$, $ 1\leq r \leq  r_T$ and $x' \in \Z^d$ satisfying $B = B(x',r) \subset T$, the following holds: there exists an event $G_B \in \sigma(\Phi_T^{\textnormal{loc}}({\widehat{\omega}}))$ depending only on $\Sigma$ 
(given $\tilde \sigma$) such that
\begin{equation}
	\label{eq:decineq}
	\E_{\tilde \sigma}[\, f(\mathcal V_T) \,  | \, \sigma(\mathcal{F}_{\Z^d \setminus \widetilde{B}}, 1_{	G_B})\,] \geq \big(\E[f(\mathcal{V}^{u(1+\delta)})]- 
	C e^{- c r^{{\Cl[c]{c:box_gap}}}}\big) 1_{G_B}
\end{equation}
(with $\Cr{c:box_gap} \in (0,\frac12)$), for any increasing $f: \{0, 1\}^{\Z^d} \to [0, 1]$ depending only on the coordinates in $B$, where $\mathcal{V}_T=\mathcal{V}_T^u$ (cf.~\eqref{eq:I_T}), $\widetilde{B} \stackrel{{\rm def}.}{=} B(x',r + \lceil 
 r^{1-{\Cr{c:box_gap}}}\rceil)$, and 
 \begin{equation}
\label{eq:caioGbound}
\P_{\tilde \sigma}[ G_B] \geq 1- e^{- c r^{{\Cr{c:box_gap}}}}  
 \end{equation}
 holds on the event (recall \eqref{def:epsx}  regarding $\varepsilon$ and that $C_y= B(y, \the\numexpr\couprad+2*\rangeofdep\relax 
L)$)
 \begin{equation}\label{eq:F_y-superhard}
 F_y\equiv F_y (\tilde{\sigma}) \stackrel{{\rm def}.}{=} \big\{(s_k \vee r_{k+1}) \vert _{C_y} \le \textstyle \frac\delta{10}, \mathsf U_{B(y, \the\numexpr\couprad+4*\rangeofdep\relax L)} \in [\tfrac{e^{-L}}{2}, 1 - 
\tfrac{e^{-L}}{2}], s_{|C_y} = \varepsilon\big\}.
\end{equation}
\end{lemma}

From here onwards, we fix the parameter $\bar{\gamma}_2$ governing the length of the short side $r_T$ of the tube $T$ in \eqref{eq:kappa}-\eqref{def:cylinder} to be
\begin{equation}\label{eq:gamma2bar}
\bar{\gamma}_2 = \Cr{C:gamma2}\, \gamma_2 \ (\geq 3 \gamma_2, \text{ as required in \eqref{eq:kappa}}),
\end{equation}
so in particular, the conclusions of the previous result (Lemma~\ref{lem:caio_T}) hold.
The second result relates $\mathcal{V}_T$ from below \eqref{eq:I_T} and the models $\mathcal{V}_{k,m}$ from 
\eqref{def:Ikm} (which we are ultimately interested in) by means of a suitable coupling. The equalities in law in 
\eqref{eq:marginals} below refer to the quenched laws under $\P_{\sigma}$.  Let $D_y = 
B(y,\the\numexpr\couprad+4*\rangeofdep\relax L)$ and $\widetilde{D}_y =
B(y,\the\numexpr\couprad+6*\rangeofdep\relax L)$. Recall~\eqref{eq:cond-2lemmas}, which is in force.

\begin{lemma}
\label{L:superhardcoupling}  
Assume that $\gamma_2  \ge \gamma_{1} + 5$.
Then for every 
$1 \le m \le m_0$, there exists a coupling $ \mathbb{Q}_{\tilde \sigma,y}$ of five $\{0, 1\}^{\Z^d}$-valued random 
variables $(\widehat{\mathcal{V}}_p : 1\leq p \leq 5)$ such that the following hold:
\begin{align}
&(\widehat{\mathcal{V}}_5, \widehat{\mathcal{V}}_1) \stackrel{{\rm law}}{=} (\mathcal V_{k+ \frac 1 2}, \mathcal 
	V_{k + 1}), \  (\widehat{\mathcal{V}}_4, \widehat{\mathcal{V}}_2 ) \stackrel{{\rm law}}{=} (\mathcal V_{k,m-\frac12}, \mathcal V_{k,m} ) \text{ and } \widehat{\mathcal{V}}_3 \stackrel{{\rm law}}{=} \mathcal{V}_T \text{ (under $\P_{\tilde \sigma}$) } \label{eq:marginals}\\
	&\widehat{\mathcal{V}}_p \cap  D_y^c = \widehat{\mathcal{V}}_1 \cap D_y^c, \text{ for } p = 2, 3, 4, \label{eq:local_diff}\\
&\sigma(\widehat{\mathcal{V}}_p(x'), x' \in  \widetilde{D}_y^c, 1\leq  p \leq 5) \text{ is independent from } \sigma( \widehat{\mathcal{V}}_p(x'), x' \in D_y, 1\leq  p \leq 5)  \label{eq:local_couple}
\end{align}
(not.: $\widehat{\mathcal{V}}_p(x')= 1\{ x' \in \widehat{\mathcal{V}}_p \}$),
and for some 
$\Cr{c:coupling_suphard} = \Cl[c]{c:coupling_suphard}( \delta, \gamma_2, \gamma_3, \gamma_M, \gamma) > 0$, with $F_y$ as in \eqref{eq:F_y-superhard},
\begin{equation}\label{eq:superhardcoupling}
\mathbb{Q}_{\tilde \sigma,y}\big[ \mathscr{C}^{\partial}_{D_y}\big(\widehat{\mathcal{V}}_p\big) \subset  \mathscr{C}^{\partial}_{D_y}(\widehat{\mathcal{V}}_{p+1}) \text{ for all }p=1,2,3,4 \big] \geq (1- e^{-\Cr{c:coupling_suphard} (\log L)^{\gamma_2}})  1_{F_y}.
	\end{equation}
\end{lemma}
The proofs of both Lemmas~\ref{lem:caio_T} and~\ref{L:superhardcoupling} absolutely essential to our argument, and by no means standard, as we now explain. As alluded to above, the definition of $\mathcal{V}_T$ in \eqref{eq:phi_T}--\eqref{eq:I_T} is tailored for the arguments lurking behind \cite[Proposition 2.3,ii)]{RI-II}, which concern $\mathcal{V}^u$, to carry over. But whereas \eqref{eq:phi_T}--\eqref{eq:I_T} is readily written down, it is crucial for the resulting model $\mathcal{V}_T$ to act as a valid surrogate for $\mathcal{V}_{k,m}$ (else it is completely useless) and 
Lemma~\ref{L:superhardcoupling} valid for the delicate choice of tubes $T$ in \eqref{def:cylinder} essentially asserts that we are not tampering too much with the trajectories. Its proof is not at all an off-the-shelf result; indeed its main ingredients are the key coupling results derived in \cite{RI-III}, from which Theorems~\ref{thm:short_long} and \ref{thm:long_short} above are excerpted.

\subsection{Escape from narrow cylinders}\label{sec:decouple}

We begin with preliminary estimates that will serve the proofs of Lemmas~\ref{lem:caio_T} and~\ref{L:superhardcoupling}. Recall the $\ell^2$-tube  $ T_{L, N}^{2,j}(z) \equiv T_{L, N}^{\circ} $ introduced below~\eqref{eq:def-tube}. Its `long' boundary 
$\partial_{{\rm long}}T_{L, N}^{\circ}  \subset \partial T_{L, N}^{\circ} $ refers to the  set of all points $y \in \Z^d$ such that $y \in  \partial^{d-1}B^j_{\perp}(z + ne_j, L)$ for some $n$ with $-L \leq n \leq N+L$, where $ \partial^{d-1}$ refers to the inner boundary in the affine ($(d-1)$-dimensional) space $z + ne_j + (e_j)_{\perp}$,  and $(e_j)_{\perp} \subset \Z^d$ is the subspace isomorphic to $\Z^{d-1}$ orthogonal to $e_j$. The following pointwise lower bound on the 
equilibrium measure of $T_{L, N}^{\circ} $ on its long boundary will be used to prove Lemma~\ref{lem:missing_traject} below.
\begin{lem}[$z\in \Z^d$, $1\leq j \leq d$, $T_{L, N}^{\circ}  \equiv T_{L, N}^{2,j}(z)$]\label{lem:equil_new}
	For all $N \geq 2L \ge 2$ and $x \in \partial_{{\rm long}} T_{L, N}^{\circ} $, 
	\begin{equation}\label{eq:equi1_new}
		e_{T_{L, N}^{\circ} }(x) \ge 
		\begin{cases}
			\frac{\Cl[c]{equi1}}{L \, \log N/L}\,, \quad & \text{if $d = 3$}\\
			\frac{\Cr{equi1}}{L}, \quad & \text{if $d \ge 4$.}	 
		\end{cases}
	\end{equation}
\end{lem}

\begin{proof}
	By invariance of $P_x$ under translations and lattice rotations, it is enough to consider the case $z = 0$ 
	and $j = 1$. In the sequel, for any point $y \in \Z^d$ we write $y=(y_1, y^\star)$, where $y^\star$ denotes its projection onto coodinates $2,\dots, d$ and $|y^{\star}|$ for its $(d-1)$-dimensional Euclidean norm. By elementary geometric considerations, involving `chopping up' $T_{L, N}^{\circ} $ into (possibly overlapping) translates of the cylinder $T_{L, 0}^{1, \circ}(0)$, one sees that
	\begin{align}
	&\label{e:TB_new}\begin{array}{l} \text{for any point $y \in \Z^d$ satisfying $| y^{\star}| = M \ge L$, one can write $T_{L, N}^{\circ}$ as the union}\\ \text{of `cylinders' $T_k^{\ast}= T_{L, 0}^{1, \circ}(y_k)\,;\, -K_1 \le 
	k \le K_2$ such that $|K_1|, |K_2| \le 3 + \frac NL$}\\
		\text{and $d(y, T_k^{\ast}) 
			\ge |k|L \vee( M-L) $ 
			for all $k$}
	\end{array}
\end{align}
(recall that $d(\cdot,\cdot)$ refers to the $\ell^{\infty}$-distance between sets).
{Since $T_{L, N}^{\circ} \subset \{z \in \Z^d : | z^{\star}| \le  L\}$, one has by the strong Markov property:}
\begin{equation}\label{eq:eTnN_lower_bnd_new}
	{P_x [ \widetilde H_{T_{L, N}^{\circ}} = \infty ] \geq P_x[\, H_{S_M}(X^\star)  <  
		\widetilde H_{S_{L}}(X^\star)\,] \, \inf_{ y \in S_M} P_y [H_{T_{L, N}^{\circ}} = \infty]}
\end{equation}
for any $M > L$ where {$S_L \coloneqq \{x \in \Z^d : |x^{\star}| = L\}$}. Applying \cite[Exercise~1.6.8]{Law91} 
to the projection $(X^{\star}_n)_{n \ge 0}$, which has the law of a (lazy)
($d-1$)-dimensional simple random walk, one finds, in case $d=3$ and $M \ge CL$ {(recall that $|x^\star| = L$)}, for any $\alpha \geq 1$, by forcing the walk along a deterministic path to a distance $C'(\alpha)$ away from $S_L$,  leading to the prefactor $c(\alpha)$ below,
\begin{equation}\label{eq:gamblers_ruin_new}
	P_x[\, H_{S_M}(X^\star)  <  \widetilde H_{S_{L}}(X^\star)\,] \ge c(\alpha) \,  \Big(1 \vee  \frac{\log (L + \alpha) - \log L - C 
		L^{-1}}{\log M - \log L + C} \Big) \ge \frac{c}{L \log (M/L)}, \, 
\end{equation}
where the second bound follow by choosing $\alpha$ large enough so that the numerator is $\geq L^{-1}$ say. 
Similarly, for $d \geq 4$, using \cite[Proposition~1.5.10]{Law91}, one obtains when $M \ge CL$,
\begin{equation}\label{eq:gamblers_ruin'_new}
	P_x[\, H_{S_M}(X^\star)  <  \widetilde H_{S_{L}}(X^\star)\,] \ge c \,  \Big(1 \vee  \frac{ L^{2-d} -  (L+1)^{2-d}  - C 
		L^{1-d}}{ L^{2-d}  -  M^{2-d}} \Big) \ge \frac{c}{L}.
\end{equation}
We now to bound the second factor in \eqref{eq:eTnN_lower_bnd_new}. Combining \eqref{e:TB_new} and \eqref{eq:Greenasympt}, one has the bound {$\max_{z \in \partial B_k} g(y, z) \le C (|k|L + M)^{2 - d}$} 
valid for all $M \geq 2 L$. Using this along with \eqref{eq:lastexit} and the 
bounds on the capacity of a box given by \eqref{e:cap-box} in the third step below together with the 
monotonicity of capacity, see \eqref{eq:cap-K-increasing}, one infers that
\begin{multline}\label{eq:sweeping_low_bnd_new}
	{P_y [H_{T_{L, N}^{\circ}} = \infty] = 1 - \sum_{z \in \partial T_{L, N}^{\circ}} g(y, z) e_{T_{L, N}^{\circ}}(z)} \, \\ \stackrel{\eqref{e:TB_new}}{\ge} \, 
	{1 - \sum_{-K_1 \le k \le K_2} \, \sum_{z \in \partial B_k} g(y, z) \, e_{B_{y_k}(L)}(z) \ge 1 - C\sum_{-K_1 \le k \le K_2} (|k| +  M L^{-1})^{2 - d} \,  \stackrel{\eqref{e:TB_new}}{\ge} c,}
\end{multline}
provided one chooses $M = C L$ when $d \geq4$ in the last step, and $M = C N + L$ when $d = 3$ 
(observing that $\sum_{-K_1 \le k \le K_2} (|k| +  M L^{-1})^{-1} \leq {(K_1 + K_2 + 1) \frac LM}$ {and subsequently using \eqref{e:TB_new}}). Plugging \eqref{eq:sweeping_low_bnd_new} and \eqref{eq:gamblers_ruin_new} (resp.~\eqref{eq:gamblers_ruin'_new}) into 
\eqref{eq:eTnN_lower_bnd_new}, one deduces \eqref{eq:equi1_new}.
\end{proof}

Next, recall the the three cylinders $T \subset T' \subset T^{\circ}$ introduced 
in \eqref{def:cylinder}, which satisfy \eqref{eq:tube-cond}. Their size depends on two parameters $\gamma_2$ and $\bar{\gamma}_2$ satisfying \eqref{eq:kappa}, which govern the aspect ratios of the cylinders. Let $f_i: \Z^d \to \R_+$, $1\leq i \leq 4$ be given, by $f_1= g_{k+1}$, $f_2= r_{k+1} $, $f_3= h_{k+1}$ and $f_4=  s_{k}' + \frac{m-1/4}{m_0}s'$, respectively; see \eqref{eq:defgh^k}-\eqref{eq:def_fs}, \eqref{eq:defghbar^k}-\eqref{eq:def_fsbar} (recall the convention \eqref{eq:barequalstilde} by which we treat both cases at once), see also \eqref{def:s'} regarding $s_k'$ and $s'$. The functions $f_i$ naturally relate to $\mathcal J_{k, m - \frac 14}$, cf.~\eqref{eq:intermediate_cofig} and the discussion following \eqref{def:chi_i_ell} below. Further, let $L_1=L_2= 3L- L^\ast$ (see \eqref{eq:defL_*}) and $L_3=L_4=  L^\ast$ and define, for $\ell \geq 1$ and $x\in \Z^d$, 
\begin{equation}\label{def:v'}
 \rho(\ell, x)= \frac{4d\,u}{\ell} \sum_{1\leq i \leq 4} \delta_{L_i}(\ell) \cdot f_i(x)  1\{x\,  
\notin {T^{\circ}}\}
\end{equation}
and the measure $\nu: \Z^d \to \R_+$ given by
\begin{equation} \label{eq:reroot-new}
\nu (x) \stackrel{\text{def.}}{=} \sum_{\ell \ge 0} E_x \big[ \rho(\ell+\mathbb N^\ast, X_{\ell}) 1_{\{\widetilde{H}_{{T'}} > \ell\}} \big]1\{x \in \partial {T'}\}, \quad x \in \Z^d.
\end{equation}
The measure $\nu$ will soon be seen to correspond to a re-routing of trajectories upon first entrance in $T'$ (cf.~\eqref{eq:prelim4}). Note that both $\rho$ and $\nu$ depend through $f_4$ on $\Sigma$, part of the (quenched) disorder $\tilde{\sigma}$ in \eqref{e:disorder-mixed}. The following pointwise comparison between 
$\nu(\cdot)$ and the equilibrium measure $e_{{T'}}(\cdot)$ will be important. This result will later allow 
us to compare $\mathcal I_T$ with $\mathcal I^{u'}$ for $u'$ close to $u$. Recall the event $F_y= F_y(\tilde{\sigma})$ from \eqref{eq:F_y-superhard}.

\begin{lemma}[under \eqref{eq:tube-cond}]\label{lem:vT_equilib_compare}
For $u> 0,$ $\delta \in (0,1)$, $\bar \gamma_2 \ge \Cr{C:gamma2}(\delta)\, \gamma_2$ and $L \ge C(\gamma_2, \bar \gamma_2)$, on $F_y$,
\begin{equation}\label{eq:nu-compa} 
\nu(x) \le u(1 + \delta/5 )e_{{T'}}(x), \quad  x \in \partial {T'}.
\end{equation}
\end{lemma}
\begin{proof}
Starting with \eqref{eq:reroot-new}, using first that $\rho(A, z) = 0$ for any $z \in {T^{\circ}}$ and $A \subset \N^\ast$, which follows from \eqref{def:v'}, then applying the strong Markov property at time $H_{\partial {T^{\circ}}}$,
one obtains
\begin{multline}\label{eq:bnd_v''}
\nu(x) = \sum_{\ell \ge 0} E_x \big[ \rho(\ell+\mathbb N^\ast, X_{\ell}) 1_{\{ \widetilde{H}_{{T'}} \, > \,  \ell \, > \, H_{\partial {T^{\circ}}} \}} \big] \\
\leq \displaystyle E_x\Big[1_{\{\widetilde{H}_{{T'}} > H_{\partial {T^{\circ}}}\}} \,  \sum_{\ell'  \geq 0}  
E_{X_{H_{\partial {T^{\circ}}}}}\big[\rho(\mathbb{N}^\ast + H_{\partial {T^{\circ}}} + \ell', X_{\ell'})\big]\Big] \\
 \stackrel{\eqref{eq:occtime}}{\leq} 4d \,\displaystyle E_x\big[1_{\{\widetilde{H}_{{T'}} > H_{\partial {T^{\circ}}}\}} \, \bar \ell_{X_{H_{\partial {T^{\circ}}}}}\big] \leq  
4d \, \displaystyle \big( \max_{z \in \partial {T^{\circ}}} \, \bar \ell_z \big)  \cdot  P_x[\widetilde{H}_{{T'}} > H_{\partial 
	{T^{\circ}}}],
\end{multline}
for any $x \in \partial {T'}$. We now bound each of the two factors in the last bound individually, starting with the average occupation 
time $\bar \ell_z \equiv \bar \ell_z(\rho) $. For clarity, we focus on the case $\mathcal V_{\cdot} =  \widetilde{\mathcal{V}}^{u, 
L}_{\cdot}$ from now on. The other case is treated similarly. 

By considering separately the contributions to $\rho$ in \eqref{def:v'} stemming from $f_1$ and $f_2$ on the one hand, and $f_3$ and $f_4$ on the other, recalling the relevant definitions \eqref{eq:defgh}--\eqref{eq:defgh^k} (see also the 
discussion below \eqref{eq:Itildeinclusions} and \eqref{eq:delta-2-mixed} regarding the choice of $\delta_1$), one sees that
\begin{equation}\label{eq:rho-superhard-bd1}
\rho \le (1 +  \delta_1) \rho_1 + \rho_2, 
\end{equation}
where 
\begin{align*}
&\ell \rho_1(\ell, \cdot) \stackrel{{\rm def}.}{=} \sum_{j \ge k+1} \Big(\frac{1+P_L}{2} \Big) u 1_{B_{2L}(\tilde x_j)}(\cdot) \, 
\delta_{L}(\ell)  + \sum_{j < k+1} u 1_{B_{2L}(\tilde x_j)}(\cdot) \, \delta_{2L}(\ell),\\
&\ell \rho_2(\ell, \cdot) \stackrel{{\rm def}.}{=} u \big ( \tilde{r}_{k+1}(\cdot) \, \delta_{L}(\ell) +  \tilde{s}_{k+1}(\cdot) 
\, \delta_{2L}(\ell)\big), 
\end{align*}
and, in obtaining $\rho_2$, we also used that $f_4 \leq s_{k+1}$, as follows on account of \eqref{def:s'} and since $m \leq m_0$. Repeating the calculation of \eqref{eq:occup_time_compute} yields that
\begin{equation}\label{eq:rho-1-superhard-bd}
\bar \ell_z(\rho_1) =   u, \quad  z \in \Z^d. \ 
\end{equation}
On the other hand, on the event $F_y$ in \eqref{eq:F_y-superhard}, one has that $\ell \rho_2(\ell, x) \le 4d\,\frac{u\delta}{9} (\delta_{L}+ \delta_{2L})(\ell)$ for all $x \in C_y= B(y, \the\numexpr\couprad+2*\rangeofdep\relax L)$ whenever $L \ge C$. 
Now recall that ${T^{\circ}} \subset B(y, \the\numexpr\couprad+\rangeofdep+20\relax L)$ by  \eqref{eq:tube-cond}, which is in force, and 
hence $\bar \ell_z$ depends only the restriction of $\rho_2$ to $B_{\the\numexpr\couprad+\rangeofdep+20\relax L}(y)$ in the second coordinate. Consequently, by a similar computation as for $\rho_1$, one gets that $\bar \ell_z(\rho_2)  \leq \frac{u\delta}{9}$ for all $z \in \partial {T^{\circ}}$, for all $L \ge C$. Combining with \eqref{eq:rho-1-superhard-bd}, one obtains that
\begin{equation}\label{eq:barlz_bnd}
\max_{z \in \partial {T^{\circ}}}\, \bar \ell_z (\rho) \le (1+ \delta_1) \max_{z \in \partial {T^{\circ}}} \big (\bar \ell_z(\rho_1) + \bar \ell_z(\rho_2) \big) \le u(1 + \delta/8), \quad L \geq C.
\end{equation}

As to the term $P_x[\widetilde{H}_{{T'}} > H_{\partial {T^{\circ}}}]$ in \eqref{eq:bnd_v''}, one first observes, applying the strong Markov property at time $ H_{\partial {T^{\circ}}}$, that for all $x \in \partial T'$,
$$
P_x[\widetilde{H}_{{T'}} > H_{\partial {T^{\circ}}}] - (4d)^{-1}\,e_{{T'}}(x) = P_x[\, H_{\partial {T^{\circ}}} \le \widetilde{H}_{{T'}} < \infty] \leq p \cdot P_x[\widetilde{H}_{{T'}} > H_{\partial {T^{\circ}}}], 
$$
where $p \stackrel{{\rm def}.}{=} \max_{z \in \partial {T^{\circ}}}\, P_z[ H_{{T'}} < \infty ]$. Solving for $P_x[\widetilde{H}_{{T'}} > H_{\partial {T^{\circ}}}] $ thus yields that
\begin{equation}\label{eq:decomp_euilibrium}
P_x[\widetilde{H}_{{T'}} > H_{\partial {T^{\circ}}}]  \le (4d(1-p))^{-1} e_{{T'}}(x), \quad x \in \partial T'.
\end{equation}
In order to estimate $p$, one uses a last-exit decomposition, see \eqref{eq:lastexit}, similarly as in the proof of Lemma~\ref{lem:equil_new}. Note however that this somewhat delicate because the ratios of the widths of $T'$ and $T^{\circ}$ differ by powers of $\log L$ only, cf.~\eqref{eq:kappa}. This is where the condition on the parameters $\gamma_2$, $\bar{\gamma}_2$ comes into play. To start with, similarly as in the argument leading to \eqref{e:TB_new}, for any point $z \in \partial {T^{\circ}}$, the set ${T'}$ can be covered by (possibly overlapping) $\ell^{\infty}$-tubes $T_k^{\ast} \stackrel{{\rm def}.}{=} T_{r_{T'} , {r_T^{\circ}}}^1(z_k), \,$  with $-K_1 \le k \le K_2$ such that 
$|K_1|, |K_2| \le 2 + (\log L)^{3 \gamma_2 \color{black}}$ and $d_2(z, T_k^{\ast}) \ge L_k \stackrel{{\rm def}.}{=} \frac{r_T^{\circ}}{2} (|k| + 1)$ for all $k$. Thus, by \eqref{eq:lastexit}, one writes with $T^{\ast} = T_{r_{T'}, {r_T^{\circ}}}^1$,
\begin{equation}\label{eq:p_bnd}
p \le \sum_{-K_1 \le k \le K_2} \, 	 \max_{|z - z'| \ge L_k}\,g(z, z') \, {\rm cap}(T_k^{\ast}) 
\stackrel{\eqref{eq:Greenasympt}}{\le} C \, {\rm cap}(T^{\ast}) \sum_{-K_1 \le k \le K_2} ({r_T^{\circ}})^{2 - d} (|k|  + 1)^{2-d},
\end{equation}
where  the first step used monotonicity of $K \mapsto e_{K}$, which follows from the 
definition \eqref{eq:equilib_K}. It remains to bound the capacity of the tube 
$T^{\ast}$. From the definition of $r_T$ and ${r_T^{\circ}}$ in \eqref{eq:kappa}, one gets that 
$T^{\ast}$ is contained in a union of at most $C {r_T^{\circ}} r_T^{-1}$ many boxes of radius $r_{T'}$. Thus, by subadditivity,  ${\rm cap}(T^{\ast}) \le C {r_T^{\circ}} r_T^{-1} {\rm cap}(B_{r_T}) \le C' {r_T^{\circ}} 
r_T^{d-3} $. Plugging this into the right hand side of \eqref{eq:p_bnd} gives, with $\bar \gamma_2 = 3\gamma_2$, i.e.~$\Cr{C:gamma2} = 3$ and when $d \geq 4$,
\begin{equation}\label{eq:p_bndd=4}
p \le  	C \Big(\frac{r_T}{r_T^{\circ}}\Big)^{d-3} \le C (\log L)^{2\gamma_2 - \bar \gamma_2} \le 
\frac{\delta}{100}, \quad (d \geq 4)
\end{equation}
whenever $L \ge C$. The case $d = 3$ requires a slightly more careful analysis. In this case, using \eqref{eq:cap_bound_Green} and the Green function asymptotics \eqref{eq:Greenasympt}, one gets (see e.g.~\eqref{eq:NTexpect} below for a similar computation) that
\begin{equation*}
{\rm cap}(T^\star) \le C \frac{r_T^{2} {r_T^{\circ}}}{r_T^2 \bar \gamma_2 \log \log L} 	\le C \frac{r_T^{\circ}}{\bar \gamma_2 \log \log L}.
\end{equation*}
Plugging this into \eqref{eq:p_bnd} and using the upper bound 
$|K_1|, |K_2| \le 2 + (\log L)^{3\gamma_2}$ one obtains 
\begin{equation}\label{eq:p_bndd=3}
p \le C \frac{\gamma_2 \log \log L}{\bar \gamma_2 \log \log L} \le C' \frac{\gamma_2}{\bar \gamma_2} \le \frac{\delta}{100}, \quad (d = 3)
\end{equation}
when $\bar \gamma_2 \geq C \delta^{-1}\gamma_2$, i.e.~with $\Cr{C:gamma2} = C \delta^{-1}$, whenever $L \ge C$. Plugging the bound on $p$ given by \eqref{eq:p_bndd=4} and 
\eqref{eq:p_bndd=3} into  \eqref{eq:decomp_euilibrium}, one deduces that
\begin{equation*}
P_x[\widetilde{H}_{{T'}} > H_{\partial {T^{\circ}}}] \le (4d)^{-1}e_{{T'}}(x) \Big(1 + 
\frac{\delta}{100}\Big),	
\end{equation*}
for $L \ge C$ and $\bar \gamma_2 \geq \Cr{C:gamma2} \gamma_2$ with the above choice of $\Cr{C:gamma2}$. Feeding this and \eqref{eq:barlz_bnd} into \eqref{eq:bnd_v''} yields \eqref{eq:nu-compa}, which completes the proof.
\end{proof}

\subsection{Conditional decoupling for $\mathcal{V}_T$}\label{sec:caio}
We now set the stage for the (short) proof of Lemma~\ref{lem:caio_T}. Recall the process $ \widehat{\omega} =({\widehat{\omega}_1}, \dots {\widehat{\omega}_4})$ from above \eqref{eq:phi_T}. We start by re-rooting relevant trajectories in the support of ${\widehat{\omega}}$ at their first entrance point in $T'$ in much the same way as in the proof \eqref{eq:prelim5}, cf.~also \cite[Lemma 3.1]{RI-III}, to obtain $\widehat{\eta}=({\widehat{\eta}_1}, \dots {\widehat{\eta}_4})$, which depends implicitly on $T'$; that is, for $x \in {T'}$ and $\ell \in 
\mathbb{N}^\ast$, let $A_{\ell, x} \subset  (\mathbb{R}_+ \times \mathbb{N}^\ast \times W^+)$ consist of all triplets $(u, t, w)$ such that $\ell \leq t < \infty$, $H_{T'}(w) = t - \ell$ and $w(t - \ell) = x$. Given $\widehat{\omega}_i$ define the new point measures 
\begin{equation}\label{def:etaT'}
 \widehat{\eta}_{i, \ell, x} = \sum_{ \substack{(u, t, w) \in ( \text{supp}(\widehat{\omega}_i) \cap 
A_{\ell, x}):\\ w(0)\in 
B(T^{\circ}, 2L) \setminus T^{\circ}}} \delta_{(u,  w(0), \ell , \tilde w)}, \quad \widehat{\eta}_i = \sum_{x \in {T'}} \, \sum_{1\leq \ell (\leq 2L)}  \widehat{\eta}_{i, \ell, x},
\end{equation}
where $\tilde{w}(\cdot) = \theta_{t-\ell} w( \cdot)$. The processes $ \widehat{\eta}_{i, \ell, x}$'s are independent Poisson point processes, hence so are the $ \widehat{\eta}_i $ under $\P_{\tilde \sigma}$. The point measure $\Phi_T(\widehat{\eta}_i) $ is then defined exactly as in \eqref{eq:phi_T} but foregoing the constraint $w_j(0) \notin T^{\circ}$ and leaving $w(0)$ unchanged. In fact all trajectories $\tilde{w}$ in the support of $\widehat{\eta}_i$ start on $\partial T'$. 

We now collect two properties of this re-rooted framework, which will soon be used in the proof of  Lemma~\ref{lem:caio_T}. It follows with $\Phi_T({\widehat{\eta}})= (\Phi_T(\widehat{\eta}_1), \dots, \Phi_T({\widehat{\eta}_4}))$ that
\begin{equation}\label{eq:measurability_phiT}
\text{$\Phi_T({\widehat{\eta}})$ is measurable relative to $\Phi_T^{\textnormal{loc}}({\widehat{\omega}})$;}
\end{equation}
indeed, in view of \eqref{def:etaT'}, a point $ (u',x', \tau' , w') \in  \text{supp}( \Phi_T(\widehat{\eta}_i)) $ is precisely obtained from some point $(u, \tau, w) \in \text{supp}( \Phi_T^\textnormal{loc}(\widehat{\omega}_i))$ -- recall that `$\textnormal{loc}$' is precisely restricting $\widehat{\omega}_i$ to triplets with $ w(0)\in 
B(T^{\circ}, 2L) \setminus T^{\circ}  $, see below \eqref{eq:VTRANGE} --
 with $u'=u$, $x'=w(0)$ and $w'= \theta_{H_{T'}(w) }(w)$. To argue measurability of $\tau'$, first note that the original (time-)length $t \geq 1$ such that $\tau=\tau(t, w)$ can be re-constructed from $\tau$ and $w$ (indeed,  since $w$ starts outside $T'$ it is simply the total amount of time spent outside $T'$). From this one obtains measurably the new label $\ell$ since $\ell=t-H_{T'}(\omega)$ and finally $\tau'=\tau(\ell, w')$.

Furthermore,  as we now explain, $\Phi_T({\widehat{\eta}})$ hasn't lost any relevant information, in that 
\begin{equation}\label{eq:identity_intersect}
 (\mathcal J_{T} \, \cap  \, 
T' )\stackrel{\eqref{eq:I_T}}{=} \big( \mathcal J_{k, m - \frac14}( \Phi_T({\widehat{\omega}})) \, \cap  \,T' \big)= \big(\mathcal J_{k, m - \frac14}(\Phi_T({\widehat{\eta}})) \, \cap  \,T' \big),
\end{equation}
where, with a slight abuse of notation,
$$
\mathcal J_{k, m - \frac14}(\Phi_T({\widehat{\eta}})) \stackrel{\text{def.}}{=}\bigcup_{ i } \bigcup_{\substack{(v,y,\tau, w) \in \Phi_T({\widehat{\eta}_i}:)\\ v \le \frac{4du}{L_i} f_i(y) }} w[0, \tau -1].
$$
To see \eqref{eq:identity_intersect}, first observe that one can safely replace $\Phi_T(\widehat{\omega})$ by $\Phi_T^\textnormal{loc}(\widehat{\omega})$ in the middle expression: the restriction to starting points $w(0) \in 
B({T^{\circ}}, 2L)$ inherent to `$\textnormal{loc}$' is inconsequential since any triplet $(u, \ell, w) \in \widehat{\omega}_i$ such that ${\rm range} (w[0, \tau(\ell, w) - 1]) \cap {T'} \neq \emptyset$ (recall the definition of $\tau$ from \eqref{def:tau}) must also satisfy $w(0) \in 
B({T^{\circ}}, 2L)$. The second equality in \eqref{eq:identity_intersect} then follows because the trace $w'[0, \tau'-1] \cap T'$ of any point $ (u', x', \tau' , w') \in  \text{supp}( \Phi_T(\widehat{\eta}_i)) $ in $T'$ coincides with the trace of the corresponding point in $\text{supp}( \Phi_T^\textnormal{loc}(\widehat{\omega}_i))$ it originates from: indeed the new label $\ell$ in \eqref{def:etaT'} is exactly subtracting the time-length needed by the non re-rooted trajectory to first enter $T'$. The definition of $\tau$ in~\eqref{def:tau} thus implies that the two traces in question are equal.

To complete our setup, we now link $\Phi_T({\widehat{\eta}})$ to the intensity measure $\nu$ introduced in \eqref{eq:reroot-new}, by counting the number of starting points $w(0)$ of trajectories $w$ belonging to points $(u,x,\tau, w)$  in the support  $\Phi_T^\textnormal{loc}(\widehat{\eta}_i)$ with appropriate label $u$, i.e.~relevant for the construction of $\mathcal J_{k, m - \frac14}( \Phi_T({\widehat{\omega}}))$. Recall $f_i$ and $L_i$ from above \eqref{def:v'}. Then define
\begin{equation}\label{def:chi_i_ell}
 \chi= \chi(\Phi_T(\widehat{\eta})) = \sum_i \sum_{\substack{(v,y, \tau, w) \in \Phi_T(\widehat{\eta}_i): \\v \le \frac{4du}{L_i} f_i( y)}} \delta_{w(0)}.
\end{equation}
It then follows by applying \eqref{eq:prelim5} (see also \cite[(3.9) and (3.10)]{RI-III} for a similar calculation) individually to each 
$\chi(\Phi_T(\widehat{\eta}_{i,\ell,x}))$ and summing over $i, \ell,x$ while keeping in mind the definitions~\eqref{eq:tildeI_k}, \eqref{eq:barI_k} and \eqref{eq:intermediate_cofig}--\eqref{def:Ikm}, that ${\chi}$ is a Poisson point process under $\P_{\tilde \sigma}$ whose intensity measure is given by $\nu$ in \eqref{eq:reroot-new}, with $\rho$ as in \eqref{def:v'}.

\begin{proof}[Proof of Lemma~\ref{lem:caio_T}] We will make frequent reference to \cite[\S2.2]{RI-II}, in particular to matters surrounding Proposition 2.3 therein. With $B = B(x',r) \subset T$, define $A$ and $U$ in terms of $B$ exactly as in the proof of \cite[Proposition 2.3]{RI-II}. Note that necessarily $r\leq r_T$, cf.~\eqref{def:cylinder}. In particular, with $\widetilde{B}$ as above \eqref{eq:caioGbound}, one has the chain of inclusions 
\begin{equation}\label{eq:V_T-incl}
B \subset A \subset U \subset \widetilde{B} \subset T'
\end{equation} 
on account of \eqref{eq:kappa}. Recall now the Poisson point process $ \omega$ on $W^\ast \times \R_{+}$ from \eqref{e:RI-intensity}-\eqref{e:def-I-u}, along with its induced set of
excursions between $A$ and $\partial_{{\rm out}} U$ from \cite[(2.13)]{RI-II}, and the associated multi-set $\mathcal C_u= \mathcal C_u^{A, U}(\omega)$ from the discussion preceding Proposition~2.3 in the same reference. Notice that $\mathcal I^u \cap B$ is conditionally independent of $\sigma(1\{z \in \mathcal V^u\} : z 
\in \Z^d \setminus \widetilde B)$ given $\mathcal C_u$. 

Now consider 
$ (u',x', \tau' , w')  \in \Phi_T(\widehat{\eta}_{i})$ with $\widehat{\eta}_{i}$ as in \eqref{def:etaT'}, where $1 \leq i \leq 4$. By definition of $\tau$ 
and the fact that $w'(0) \in \partial T$, we can define excursions of $w'[0, \tau' - 1]$ 
between $A$ and $\partial_{{\rm out}} U$ just as in \cite[(2.13)]{RI-II} since 
both $A$ and $\partial_{{\rm out}} U$ are contained in $ T'$ by \eqref{eq:V_T-incl}. We denote by $\widehat{\mathcal C}=\widehat{\mathcal C}^{ A, U} (\Phi_T(\widehat{\eta}))$ the multi-set of endpoints of excursions thereby obtained, for any of the trajectories underlying ${\rm supp}( \Phi_T(\widehat{\eta}_{i}))$ with label at most $\frac{4d}{L_i}f_{i}(w(0))$  and as $i$ varies (i.e.~as in the sum in \eqref{def:chi_i_ell}). 

By construction of the interlacement, it follows that the conditional law of $\mathcal I^u \cap B$ given
$\{\mathcal C_u = \zeta\}$ is independent of $u > 0$. We denote it as $\P_\zeta$ in the sequel. Recalling that `time runs for free' inside $T'$ by the definition of 
$\tau$ in \eqref{def:tau}, it follows from \eqref{eq:identity_intersect} and the Markov property for simple random walks that, for any 
$\zeta$ such that $\P_{\tilde{\sigma}}[ \widehat{\mathcal C} = \zeta]>0$, 
\begin{equation}\label{eq:conditional_law_identity}
\begin{array}{l}\text{under $\P_{\tilde{\sigma}}$, $\mathcal J_{T} \cap B$ is distributed independently of $\mathcal F_{\Z^d \setminus 
\widetilde B}$}\\ \text{conditionally on $\{\widehat{\mathcal C} = \zeta\}$ and the conditional 
law is equal to $\P_\zeta$.}\end{array}
\end{equation}
Moreover, as we now explain, on the event $F_y=F_y(\tilde{\sigma})$,
\begin{equation}\label{eq:Caio-superhard-sd}
\widehat{\mathcal C}  \leq_{\textnormal{st}} \mathcal C_{u(1 + \delta/5)}
\end{equation}
(with $\widehat{\mathcal C}$ sampled under $\P_{\tilde{\sigma}}$ on the left-hand side), for all  $\bar \gamma_2 \ge \Cr{C:gamma2}(\delta)\, \gamma_2$ and $L \ge C(\gamma_2, \bar \gamma_2)$, which will be tacitly assumed from here on. To see \eqref{eq:Caio-superhard-sd}, first observe that the multi-set $\widehat{\mathcal C}$ is obtained by evolving independent random walks starting from each point in the support of the Poisson process $\chi$ introduced in \eqref{def:chi_i_ell}, and retaining the pairs of entrance and exit points in $A \times \partial_{\text{out}} U$ performed by each of the walks until a certain time (namely $\tau-1$, with $\tau$ as in \eqref{def:chi_i_ell} associated to the starting point $w(0)$ of each walk). In view of the representation for $\mathcal I^u \cap 
{T'}$ given by \eqref{e:def-I-u-K}, the multi-set $\mathcal C_{u(1 + \frac\delta 5)}$ is obtained in a similar way, except that starting points now follow a Poisson process $\chi_{u(1 + \frac\delta 5)}$ with intensity $u(1 + \frac\delta 5)e_{T'}$ and the time horizon for each walk $w$ is unrestricted, i.e.~one retains the pairs in $A \times \partial_{\text{out}} U$ induced by all excursions in $w[0,\infty)$ rather than $w[0,\tau)$. One then applies Lemma~\ref{lem:vT_equilib_compare} to couple $\chi$ and $\chi_{u(1 + \frac\delta 5)}$ so that $\chi \leq \chi_{u(1 + \frac\delta 5)}$ and runs the same walks for both processes to deduce \eqref{eq:Caio-superhard-sd}.

Combining \eqref{eq:conditional_law_identity}, \eqref{eq:Caio-superhard-sd} and \cite[Proposition~2.3, item~ii)]{RI-II}, one thus obtains the conditional decoupling inequality \eqref{eq:decineq} upon letting (see \cite[(2.17)]{RI-II} for the notation $\Xi_{B}^{\cdot, \cdot}$) 
\begin{equation}\label{eq:Caio-superhard-G}
 G_B \stackrel{{\rm def}.}{=} \{\widehat{\mathcal C}  \in \Xi_{B}^{u(1 + \delta/5), u \delta/5}\}.
 \end{equation}
 Note that the decoupling inequality a priori concerns the vacant set of $\mathcal J_T \cap B$ but the noise (cf.~\eqref{def:Ikm} and below \eqref{eq:I_T})
 does not interfere on the event $F_y$, i.e.~$(\mathcal I_T \cap T) = (\mathcal J_T \cap T)$ $\P_{\tilde{\sigma}}$-a.s.~for $\tilde 
\sigma \in F_y$. The claim that $G_B \in \sigma(\Phi_T^{\textnormal{loc}}({\widehat{\omega}}))$ follows from the fact that $\widehat{\mathcal C}=\widehat{\mathcal C}^{ A, U} (\Phi_T(\widehat{\eta}))$ and by \eqref{eq:measurability_phiT}. Finally the asserted lower bound for $\P_{\tilde \sigma}[ G_B]$ on the event $F_y$ (needed for \eqref{eq:Caio-superhard-sd} to hold) is inherited from \cite[(2.18)]{RI-II} on account of \eqref{eq:Caio-superhard-sd} and by monotonicity of $\Xi_{B}^{\cdot, \cdot}$, see \cite[Remark~2.4,1)]{RI-II}.
\end{proof}

\subsection{Coupling for $\mathcal{V}_T$} \label{subsubsec:superhard} 

In this section we prove the coupling relating $\mathcal{V}_T$, which has desirable technical features (such as the conditional decoupling proved in the previous paragraph), to the actual vacant sets of interest, as stated in Lemma~\ref{L:superhardcoupling}. By convention, throughout the proof, which occupies the remainder of this section, unlabeled constants may depend 
implicitly on $ \delta, \gamma_2, \gamma_3, \gamma_M, \gamma$. Length parameters like $L (\log L)^{- C \gamma_2}$ etc.~appearing in the sequel are routinely rounded 
off to the nearest larger integer power of 2 (i.e.~$f(L)$ is identified with $2^{\lceil\log_2 f(L) \rceil}$)
, so that they always divide $L$, which is dyadic.

\begin{proof}[Proof of Lemma~\ref{L:superhardcoupling}]
We will only discuss the case of $\mathcal I_{\,\cdot}^{u, L}  = \overline {\mathcal I}_{\,\cdot}^{u, L}$ as the proof for $\widetilde {\mathcal I}_{\,\cdot}^{u, L}$ is similar (in fact, simpler, see below \eqref{eq:superhard-''frombelow}).
We will routinely suppress the condition $L \geq C$. We will also implicitly 
assume $\tilde \sigma \in F_y$, for otherwise we can just set $\mathbb Q_{\tilde \sigma} = \P_{\tilde \sigma}$, which is a coupling of $(\widehat{\mathcal{V}}_p : 1\leq p \leq 5)$. Since on the event $F_y$, see \eqref{eq:F_y-superhard}, one has $\widehat{\mathcal{V}}_1 \cap D_y  \stackrel{\text{law}}{=}\mathcal{V}(\mathcal{I}_{k+1}) \cap D_y = \mathcal{V}(\mathcal{J}_{k+1}) \cap D_y $ and similarly for $ 2\leq p \leq 5$ (see~\eqref{eq:intermediate_cofig}--\eqref{def:Ikm} and \eqref{eq:I_T}), we only need to 
prove the desired inclusion for the `non-noised' versions of all processes.
 We will produce the couplings for $\widehat{\mathcal{V}}_p$, $p=2,3$, and $\widehat{\mathcal{V}}_p$, $p=3,4$ in 
\eqref{eq:superhardcoupling} separately (the other two being direct consequences of the definitions) and 
concatenate them using \cite[Lemma 2.4]{RI-III} at the end. In fact, 
we obtain each of these couplings by chaining several `elemental' couplings that crucially rely on Theorems~\ref{thm:short_long} and~\ref{thm:long_short}. We start with the more involved of the two.

\medskip

{\bf  Coupling between $\mathcal V(\mathcal J_{k, m}) \stackrel{\textnormal{law}}{=} \widehat{\mathcal V}_2$ and $\mathcal V(\mathcal J_T) \stackrel{\textnormal{law}}{=} \widehat{\mathcal V}_3$.}  Since $t \mapsto \mathcal{J}_{k,t}$ is increasing, see \eqref{eq:intermediate_cofig}, and $\mathcal J_T$ is obtained from $\mathcal J_{k, m-\frac14}$ by a mechanism letting trajectories run for a longer time, see \eqref{eq:phi_T}-\eqref{def:tau}, the task is to account for this over-shoot in terms of the additional `sprinkled' trajectories generated when passing from $m-\frac14$ to $m$. We first provide some intuition. A crucial observation is that the mechanism in \eqref{eq:phi_T}-\eqref{def:tau} by which one `forgets the clock' inside $ T'$ actually results in very little additional total time for the trajectories. This is ultimately due to the geometric features of $T'$, see \eqref{def:cylinder}, whose short direction scales sub-diffusively relative to the time length ($\approx L$) of the walks involved. Intuitively, the thin width entails that trajectories in $T'$ will tend to quickly exit $T'$, whence forgetting the clock is really a perturbation.
 
 We proceed to make this rigorous, by showing that with very high probability, the extra time caused by applying \eqref{eq:phi_T}-\eqref{def:tau} is accounted for by letting every trajectory in $\mathcal J_{k, m - 1/4}$ that visits $T'$ run for an additional time
 \begin{equation}\label{eq:Delta-L}
 \Delta L \stackrel{\text{def.}}{=} \text{dy} (L(\log L)^{-4\gamma_2}),
 \end{equation}
 where $\text{dy}(t)=2^{\lceil\log_2 t\rceil}$ for $t>0$. This is the content of Lemma~\ref{lem:elongation} below. In order to formulate 
this precisely, we fix any point $z \in \mathbb{L}$ (cf.~below \eqref{eq:barequalstilde}) such that ${T^{\circ}} \subset B_{2L}(z)$ and set, with $B= B_{6L}(z)$,
\begin{align*}
L^+(\cdot) &= \text{dy}(L + (\Delta L)\,1_{B}(\cdot)), \\
u^+(\cdot) &= u\tfrac{L^+(\cdot)}{L},
\end{align*}
and set $\bar{L}^+= \text{dy}( L + (\Delta L))$, $\bar{u}^+= u\frac{\bar{L}^+}{L}$, which are scalar. We now extend the definition of $\mathcal J^{\mu, \ell}$ in 
\eqref{eq:J} to include functions $\ell : \Z^d \to 
\N^{\ast}$ corresponding to spatially inhomogeneous time lengths depending on the starting point of the walk, i.e. 
\begin{equation*}
\mathcal{J}^{f, \ell} (\omega) \stackrel{{\rm def.}}{=}  \bigcup_{i \, : u_i \leq \frac{4d}{L} f(w_i(0))} w_i[0, 
\ell(w_i(0))-1].
\end{equation*}
With this we define (cf.~\eqref{eq:intermediate_cofig}, \eqref{eq:barI_k} and above \eqref{def:v'} regarding $f_i$)
\begin{equation}\label{def:I'km-14}
\hat{\mathcal J}_{0} = \bigcup_{ 1 \leq i \leq 4} {\mathcal{J}}^{\, u^{+}f_i, L_i^+}  (\omega_i),  
\end{equation}
with $L_1^+=L_2^+= 2L^+$ and $L_3=L_4= L^+$ (for the case of $\widetilde {\mathcal I}_{\,\cdot}^{u, L}$ not discussed here, one simply exchanges the roles of $L_i^+$, $i=1,2$ and $i=3,4$). The increase from $u$ to $u^{+}$ in \eqref{def:I'km-14} is precisely ensuring that the actual intensity of walks remains unaffected as their lengths increases, cf.~\eqref{eq:J}.

\begin{lemma}\label{lem:elongation}
With $\Delta L$ as in \eqref{eq:Delta-L} and on the event $F_y$, one has
\begin{equation}\label{eq:elongation}
\P_{\tilde \sigma}[ \mathcal J_ {T} \subset \hat{\mathcal J}_{0}] \ge 1 - e^{-c (\log 
L)^{-\gamma_2}}.
\end{equation}
\end{lemma}

\begin{proof}[Proof of Lemma~\ref{lem:elongation}] Recalling $T'$ from \eqref{def:cylinder}, let $\ell_{T'}(w)= \sum_{n \ge 0} 
1_{\{w(n) \in T'\}} $ denote the time spent in $T'$ by $w \in W^+$ and $\ell_{T'}=\ell_{T'}(X)$. With $\Delta L$ as in \eqref{eq:Delta-L}, we first argue that
\begin{equation}\label{eq:N_T_tail_bnd}
P_x[ \ell_{T'} \ge \Delta L ] \le e^{-c (\log L)^{\gamma_2}}
\end{equation}
 holds for all $x \in \Z^d$. To verify \eqref{eq:N_T_tail_bnd}, first note that, using the fact that $g(x, y)$ is equal to $E_x[g(X_{H_{T'}}, y) 1_{\{H_{T'} < 
\infty\}} ]$, performing the ensuing sum $\sum_{y \in T'} g(z,y)$ for a given $z \in T'$ first by summing over points in $B(z, 4r_{T})$, then over points at distance at least $\geq r_T$ from $z$ while using the asymptotics \eqref{eq:Greenasympt} for $g$, one finds that
\begin{multline}\label{eq:NTexpect}
E_x [ \ell_{T'}  ] = \sum_{y \in T'} g(x, y) \le  \sup_{z \in T'} \, \sum_{y \in T'} g(z, y) \\
 \le  C r_T^{2} +   Cr_T^{d-1}\sum_{ r_T \leq n \leq R_T}n^{2-d}  \stackrel{\eqref{eq:defR_T}, 
\eqref{def:cylinder}}{\le} C'(\gamma_2) L (\log L)^{-6\gamma_2} \log \log L.
\end{multline}
But by the strong Markov property for $X$ and Markov's inequality, we have that $$P_x\big [\ell_{T'} \ge  2t \, \sup_{z} E_z[\ell_{T'}] \big] \le 2^{-\lfloor t \rfloor},$$ from which \eqref{eq:N_T_tail_bnd} follows upon combining with \eqref{eq:NTexpect} and taking $t = C'(
\gamma_2)^{-1}\frac{(\log L)^{2\gamma_2}}{\log \log L}$.

Consider now $\Lambda$, the collection of all walks underlying $\mathcal J_{k, m - 1/4}$ starting inside $B(T', 2L) $, i.e.
$
\Lambda$ is the collection of all $ w \in W^+$ such that $(v,w) \in \omega_i $ with $w(0)\in B(T',2L)$ and $v \leq \frac{4d u}{L_i}f_i(w(0)) $ for some $1 \leq i \leq 4$; cf.~above \eqref{def:v'} regarding $L_i$ and $f_i$. Then 
\begin{equation}\label{eq:J_T_incl}
\{ \mathcal J_{T} \not\subset \hat{\mathcal J}_{k, m - 1/4} \}	\subset \{ \ell_{T'}(w)  \geq \Delta L \text{ for some }  w \in \Lambda  \}.
\end{equation}
However, since $B(T', 2L)  \subset B(y, 70L)$ on account of \eqref{eq:tube-cond} it follows that on the event $F_y$ (which controls $f_i$), the variable $
|\Lambda|$ which is Poisson under $\P_{\tilde \sigma}$ has mean bounded by $\tfrac{4u}{L}  |B(T', 2L)| \leq C L^{d-1}$. Thus, returning to \eqref{eq:J_T_incl}, using the Poisson tail bound \eqref{eq:Poisson_tailbnd} we obtain that
\begin{equation*}
\P_{\tilde \sigma} [   \mathcal J_{T} \not\subset \hat{\mathcal J}_{k, m - 1/4} ]	\le C L^{d-1} \sup_{x } P_x[\ell_{T'} \ge \Delta L] + e^{-c L^{d-1}},
\end{equation*}
from which Lemma~\ref{lem:elongation} follows by virtue of \eqref{eq:N_T_tail_bnd}.
\end{proof}

Continuing with the proof of Lemma~\ref{L:superhardcoupling}, we focus in view of Lemma~\ref{lem:elongation} on producing a coupling $\mathbb Q_{2} = \mathbb Q_{2, \tilde \sigma}$ between 
the laws of ${\mathcal V} (\hat{\mathcal J}_{0})$ and 
$\mathcal V({\mathcal J}_{k, m})$ with `good' properties, that is, satisfying a restricted version of \eqref{eq:marginals}-\eqref{eq:superhardcoupling} to $p=2,3$ only, with ${\mathcal V} (\hat{\mathcal J}_{0})$ playing the role of $\mathcal{V}_T (\stackrel{\text{law}}{=} \widehat{\mathcal V}_3)$. Similarly as in the proof of Proposition~\ref{prop:couple_global}, $\mathbb Q_2$ will be constructed by concatenating a chain of couplings $\mathbb Q_{2, a}$ through a 
sequence of intermediate configurations $(\hat {\mathcal J}_{a}: 0 \le a \le A)$ with $A = \lceil \frac{2}{\delta}\rceil$ that we will momentarily introduce. This slicing indexed by $a$ is owed to the fact that Theorem~\ref{thm:long_short} requires working against a suitable environment configuration.

We proceed to give the precise definition of $\hat {\mathcal J}_{a}$ and start with an informal description.
Recall from \eqref{eq:intermediate_cofig} that the sequence of configurations $(\mathcal J_{k, m - 
\frac14(1 - \frac aA)} : 0 \le a \le A)$ interpolate naturally between $\mathcal J_{k, m - 1/4}$ and $\mathcal J_{k, 
m}$ through the addition of independent point processes with intensity profile $\tfrac{a}{4Am_0}s'$ at 
each step. The configuration $\hat{\mathcal J}_{a}$ is obtained 
by changing the lengths of the `last $(1 - \tfrac aA)$-fraction' of the (labelled) trajectories 
underlying $\mathcal J_{k, 
m}$ with starting point inside $B$ from $(L, 2L)$ to $(L^+, 2L^+)$ respectively. In particular, for $a=0$ the resulting configuration is precisely $\hat{\mathcal J}_{0} $ in \eqref{def:I'km-14}. Formally, with $B= B(z,6L)$ as before, $\bar B=B(y,\the\numexpr\couprad+2*\rangeofdep\relax L)$, whence $B \subset \bar{B}$ by choice of $z$ and due to \eqref{eq:tube-cond}, and letting $U= \Z^d \setminus \bar{B}$, the set $\hat{\mathcal J}_{a}$ splits into the union of three independent `background', `current' and `outside' configurations, 
\begin{equation}\label{def:hat_Ia}
	\hat{\mathcal J}_{a} \stackrel{\text{def.}}{=}  \hat{\mathcal J}_{a}^{\rm curr} \cup \hat{\mathcal J}_{a}^{\rm back} \cup \hat{\mathcal J}^{{\rm out}},
\end{equation}
defined as (see \eqref{eq:two_fnc} for notation)
\begin{equation}\label{def:hathat_Ia}
\begin{split}
		\hat{\mathcal J}^{{\rm out}}&=\bigcup_{1\le i \le 4}\mathcal{J}^{uf_i 1_{U}, {L_i}}(\omega_i),\\  \hat{\mathcal J}_{a}^{\rm curr}  &=  \bigcup_{1 \leq i \leq 4} {\mathcal{J}}^{I_{a}^+f_i1_B, L_i^+}  (\omega_i), \\
		\hat{\mathcal J}_{a}^{\rm back} &=  \bigcup_{1 \leq i \leq 3}\big({\mathcal{J}}^{u_a f_i1_{\bar B},L_i} (\omega_i) \cup {\mathcal{J}}^{I_{a+1, A}^+ 
			f_i1_{B}, L_i^+}  (\omega_i)\big) \cup \Big({\mathcal{J}}^{u f_{4,a}, L_4} (\omega_4) \cup {\mathcal{J}}^{I_{a+1, A}^+	f_41_B, L_4^+} (\omega_4)\Big)
\end{split}
\end{equation}
where $I_{a, b}^+ =  (\frac{a}{A}, (\frac{b}{A}  \wedge  1)]u^+$ and $I_{a}^+ = I_{a, a + 1}^+$ for $0 \le a, b \le A$, $I_{a, b}^+=\emptyset$ if $a>b$ and
\begin{equation}\label{def:ua}
u_a = \frac{ua}{A}1_B + u 1_{B^c},\qquad 	f_{4,a} =  f_{4} 1_{\bar B \cap B^c} + \frac{a}{A}f_41_B + \frac{a}{4Am_0}s'.
\end{equation}
We first make a few elementary observations about $\hat{\mathcal J}_{a} $. First, the union in \eqref{def:hat_Ia} and \eqref{def:hathat_Ia}
is over independent sets under $\P_{\tilde \sigma}$. To see this, just observe that the relevant intensity profiles correspond to disjoint subsets of $(0,\infty) \times W^+$ when applied to the same underlying 
configuration $\omega_i$. From this point onwards, all the configurations appearing in a union such as \eqref{def:hat_Ia} will 
be tacitly implied to be independent under $\P_{\tilde \sigma}$. Next, observe that any two successive configurations in the sequence $(\hat {\mathcal J}_{a}: 0 \le a \le A)$ differ only inside $\bar B$; this follows by direct inspection of \eqref{def:hathat_Ia} upon recalling that $\text{supp}(s')\subset \bar{B}$. Finally, one verifies using that $L_i^+=L_i$ outside $U$ and $u^+=u$ outside $B$ that \eqref{def:hat_Ia} is consistent with the definition of $\hat{\mathcal J}_{0}$ in \eqref{def:I'km-14}, and similarly that $\hat{\mathcal{J}}_{A} = {\mathcal{J}}_{k, m}$, cf.~\eqref{eq:intermediate_cofig}. 

The desired measure $\mathbb Q_{2, a}$ will couple $\hat{\mathcal J}_{a}$ and $\hat{\mathcal J}_{a+1}$ in such a way that an analogue
of \eqref{eq:superhardcoupling} holds for these two configurations, cf. \eqref{eq:couplebndQ2a} below. Once this is achieved $\mathbb Q_{2}$ will readily follow, by chaining the resulting couplings over $a$. The measure $\mathbb Q_{2, a}$ will itself be built by concatenation of two measures $\mathbb Q_{2, a,1}$ and $\mathbb Q_{2, a,2}$, which bring into play an intermediate configuration $\mathcal{J}_{a, a+1}$
comprising shorter trajectories, that will be introduced shortly. The coupling $\mathbb Q_{2, a,1}$, which we define next, relates $\hat{\mathcal J}_{a}$ and $\mathcal{J}_{a, a+1}$ in essence by means of Theorem~\ref{thm:long_short}. The coupling $\mathbb Q_{2, a,2}$ will then link $\mathcal{J}_{a, a+1}$ with $\hat{\mathcal J}_{a+1}$ in three more (sub-)steps, giving rise to $\mathbb Q^{(k)}_{2, a, 2}$, $k=1,2,3$. 

We now define $\mathbb Q_{2, a,1}$. In preparation of the application of Theorem~\ref{thm:long_short}, we note that $\hat{\mathcal J}_{a}^{\rm back}$ will act as environment configuration $\mathcal{I}^{\rho}$. In fact the full configuration $\hat{\mathcal J}_{a}^{\rm back}$ is unnecessarily complicated so we first exhibit a part of it that will suffice to meet the constraint of satisfying $(\textnormal{C}_{\textnormal{obst}})$, see Definition~\ref{def:background}. To this end observe that
$\hat{\mathcal J}_{a}^{\rm back} =  \hat{\mathcal J}_{a,1}^{\rm back} \cup \hat{\mathcal J}_{a,2}^{\rm back}$ with $\hat{\mathcal J}_{a,1}^{\rm back}\stackrel{{\rm law}}{=} \mathcal I^{\rho}$ and $\hat{\mathcal J}_{a,2}^{\rm back}\stackrel{{\rm law}}{=} \mathcal I^{\rho'}$, see \eqref{eq:prelim2} for notation, where $\rho: \mathbb{N}^\ast \times \Z^d  \to \R_+$ is given by
\begin{equation}\label{def:v2}
\rho(\ell, x) = \textstyle  4d u \big(1 - A^{-1}\big)\big(\frac{1}{2L} 1_{2L}(\ell) \bar g_{k+1}(x) + \frac{1}{L} 1_{L}(\ell) \, \bar 
h_{k+1}(x)\big)1_{\bar B}(x)\,\end{equation}
(recall to this effect that $f_1= \bar g_{k+1}$, $f_3 = \bar 
h_{k+1}$ and $L_1^+(\cdot) \geq 2L$, $L_3^+(\cdot) \geq L$). The specifics of $\rho'$ are of no importance, $ \hat{\mathcal J}_{a,2}^{\rm back}$ (which is independent of $\hat{\mathcal J}_{a,1}^{\rm back}$) will not be involved in the coupling. In the next paragraph, we will apply Theorem~\ref{thm:long_short} for $K = 6L$ (recall that $B=B(z,6L)$) with $\hat{\mathcal J}_{a,1}^{\rm back}$ as background configuration. Hence, we need 
conditions~\eqref{eq:disconnect_background0}-\eqref{eq:disconnect_background2} to hold on the 
box $B(z,6 L + 5 \times 2L) = B(z,16L)$, which, in view of \eqref{def:v2} and \eqref{def:v}, follows by the very same computations that we did in the 
proof of Proposition~\ref{prop:couple_global}; see, in particular, \eqref{eq:local_time_lower_bnd} and 
\eqref{eq:occup_time_compute}.

The coupling $\mathbb Q_{2, a,1}$ will have the effect of `covering' all trajectories present in $\hat{\mathcal J}_{a}^{\rm curr} $, which have length $L^+$ or $2L^+$, by (shorter) trajectories of length $\Delta L$ as in \eqref{eq:Delta-L}. To fit the setup of Theorem~\ref{thm:long_short}, which is designed for exactly such purpose, we first observe that by \eqref{def:hathat_Ia},
\begin{equation}\label{eq:stoch_domcurr}
\hat{\mathcal J}_{a}^{\rm curr} \leq_{\textnormal{st.}} \mathcal J^{f_{1,2}^+, 2L^+} \cup \mathcal J^{f_{3,4}^+, L^+},
\end{equation}
where
\begin{equation}\label{def:f2}
\begin{split}
&f_{i,j}^+(\cdot) \stackrel{{\rm def}.}{=} \textstyle\frac {\bar u^{+}}A \big( f_i(\cdot)  + f_j(\cdot) \big) 1_{B}(\cdot) + 
\frac{u}{A} \varepsilon^2 1_{B_{10L}(z)} (\cdot)
\end{split}
\end{equation}
(recall $\varepsilon$ from \eqref{def:epsx} and \eqref{eq:varepsilon_L}). Notice that, due to the presence of the second term in \eqref{def:f2}, which is a-priori unnecessary for \eqref{eq:stoch_domcurr} to hold, one has since $\gamma_2 \ge (\gamma_1 + 5)$ that on the event $F_y$, the bound $u \ge f_{1,2}^+, f_{3,4}^+ \ge  u \, (\log L)^{-2\gamma_2}$ holds pointwise on $B_{10L}(z)$, so the required ellipticity condition \eqref{e:ass-f-coup} is indeed satisfied for all $L \geq C$ with $\bar \gamma = 32 \gamma_2$ (our implicit choice in all subsequent applications of Theorems~\ref{thm:short_long} and 
\ref{thm:long_short}). Thus overall, Theorem~\ref{thm:long_short} is in force, and applies separately to the functions 
$f=f_{1,2}$ and $f=f_{3,4}$ with $K = 6L$, centering at $z$ rather than $0$, and the choices $L' = \Delta L$, which divides $\bar L^+$ since $\Delta L$ is dyadic and with $\hat{\mathcal J}_{a,1}^{\rm back}$ as the background configuration $\mathcal{I}^{\rho}$
(see the discussion leading to \eqref{def:v2}). By first applying Theorem~\ref{thm:long_short} quenched on $\mathcal{I}^{\rho}$ (see \eqref{e:hard-coup-final3-quenched}) separately to $f_{1,2}$ and $f_{3,4}$, then integrating over $\mathcal{I}^{\rho}$ and finally concatenating the resulting coupling with the one 
implicit in \eqref{eq:stoch_domcurr}), we thus obtain a measure $\mathbb Q_{2, a, 1}$ satisfying the following property. Letting $l = \tfrac{L}{\Delta L}$ and $l^+ = \tfrac{\bar L^+}{\Delta L}$ (cf.~\eqref{e:couplings-n-L}) and
\begin{equation}
\label{eq:J-intermediate}
\hat{\mathcal J}_{a,a+1} \stackrel{\text{def.}}{=}  \hat{\mathcal J}_{a,a+1}^{\rm curr} \cup \hat{\mathcal J}_{a}^{\rm back} \cup \hat{\mathcal J}^{{\rm out}} , \qquad \hat {\mathcal J}_{a, a+1}^{\rm curr}\stackrel{\text{law}}{=} \mathcal J^{(1 + \Cl{C:sprinkDelL}l^{-1/2})f' , \Delta L}, 
\end{equation}
all sampled independently,
for suitable $\Cr{C:sprinkDelL} = \Cr{C:sprinkDelL}(\delta, \gamma_2, d) $, with
\begin{equation}\label{def:f'}
f' =(l^+)^{-1} \Big(\frac{1}{2}\sum_{0 \leq k < 2l^{+}}P_{k \Delta L} (f_{1,2}^+) + \sum_{0 \leq k < l^+}P_{k \Delta L} (f_{3,4}^+) \Big),
\end{equation} 
the measure $\mathbb Q_{2, a, 1}$ is a coupling between $  \hat {\mathcal J}_{a}$ and $\hat{\mathcal J}_{a,a+1}$ such that
\begin{equation}\label{eq:couple_L_by_L'2}\mathbb Q_{2, a, 1} \big[ \mathscr{C}^{\partial}_{D_y}\big(\mathcal V( \hat {\mathcal J}_{a})  \big)  \supset  \mathscr{C}^{\partial}_{D_y}\big(\mathcal V(\hat{\mathcal J}_{a,a+1})  \big)   \big] \geq 1 - e^{- c (\log L)^{\gamma_2}},
\end{equation} for suitable $c = c(\delta, \gamma_2, \gamma_M, d) > 0$. We emphasize that 
Theorem~\ref{thm:long_short} only gives us the inclusion between the boundary clusters of $\hat{\mathcal J}_{a}^{\rm curr} \cup \hat{\mathcal J}_{a,1}^{\rm back}$ and $\hat{\mathcal J}_{a,a+1}^{\rm curr} \cup \hat{\mathcal J}_{a,1}^{\rm back}$, where $\hat{\mathcal J}_{a,1}^{\rm back}\stackrel{\text{law}}{=}\mathcal{I}^{\rho}$ plays the role of the background configuration. However, such inclusions are preserved if we take union on both sides of \eqref{eq:couple_L_by_L'2} with some common 
$\mathcal J \subset \Z^d$, in this case $\hat{\mathcal J}_{a,2}^{\rm back} \cup \hat{\mathcal J}^{\rm out}$. We will use this observation implicitly in the sequel.

Next we define $\mathbb Q_{2, a,2}$, which will couple $\hat{\mathcal J}_{a,a+1}$ and $\hat{\mathcal J}_{a+1}$ in such a way that a feature similar to \eqref{eq:couple_L_by_L'2} holds. The coupling $\mathbb Q_{2, a,2}$ is a bit more involved (we return to this at the end of the proof) and will be obtained by concatenating $\mathbb Q_{2, a,2}^{(k)}$, $k=1,2$. In a nutshell, $\mathbb Q_{2, a,2}^{(1)}$ is an application of Theorem~\ref{thm:short_long} and does the bulk of transforming $ \hat{\mathcal J}_{a,a+1}^{\rm curr}$ back into longer trajectories underlying $ \hat{\mathcal J}_{a+1}$, but some error terms arise along the way, which require a slightly fine-tuned application of  
Theorems~\ref{thm:long_short} and~\ref{thm:short_long}, giving rise to $\mathbb Q_{2, a,2}^{(2)}$.

We start by introducing the configuration $(\hat {\mathcal J}_{a+1}^{ {\rm curr}})'$, defined in such a way that the equality
\begin{equation}\label{eq:J_a+1'}
(\hat {\mathcal J}_{a+1}^{ {\rm curr}})' \cup \hat{\mathcal J}_{a}^{\rm back} = \hat {\mathcal J}_{a+1}^{{\rm curr}} \cup \hat{\mathcal J}_{a+1}^{\rm back}.
\end{equation}
holds. In particular, it follows from \eqref{eq:J_a+1'} and \eqref{def:hat_Ia} that $\hat{\mathcal 
J}_{a + 1} = (\hat {\mathcal J}_{a+1}^{ {\rm curr}})' \cup \hat{\mathcal J}_{a}^{\rm back} \cup \hat{\mathcal 
J}_{{\rm out}}$. In view of \eqref{eq:J_a+1'}, our task can be reformulated as comparing the laws of 
$\hat{\mathcal J}_{a,a+1}^{\rm curr} \cup \hat{\mathcal J}_{a}^{\rm back}$ and $(\hat {\mathcal J}_{a+1}^{ {\rm curr}})' \cup \hat{\mathcal J}_{a}^{\rm back}$. The benefit of \eqref{eq:J_a+1'} is thus to retain the background configuration $\hat{\mathcal J}_{a}^{\rm back} $, which will be useful since Theorem~\ref{thm:long_short} is still to be used. With a bit of effort one checks that
\begin{equation}\label{def:hatIacurr1}
(\hat {\mathcal J}_{a+1}^{ {\rm curr}})' = \Big( \bigcup_{1 \leq i \leq 3} {\mathcal{J}}^{I_{a}f_i1_B, L_i^+}  (\omega_i) \Big) \cup  {\mathcal{J}}^{u [f_{4,a}, f_{4,a+1}], 
L_4}(\omega_4)
\end{equation}
where $I_{a} = [\frac aA, (\frac {a+1} A \wedge 1)]u$ is defined exactly as below \eqref{def:hathat_Ia}, but with $u$ in place of $u^+$. 

Instead of comparing directly $\hat{\mathcal J}_{a,a+1}^{\rm curr}$ with $(\hat {\mathcal J}_{a+1}^{ {\rm curr}})'$, cf.~below \eqref{eq:J_a+1'}, the coupling $\mathbb Q_{2, a, 1}^{(1)}$ will compare the former with a bulk contribution $(\hat {\mathcal J}_{a+1}^{ {\rm curr}})''$ to $(\hat {\mathcal J}_{a+1}^{ {\rm curr}})'$, tailored to the application of Theorem~\ref{thm:short_long}, and given by
\begin{equation}\label{def:hatJacurr0}
(\hat {\mathcal J}_{a+1}^{ {\rm curr}})'' \stackrel{{\rm law}}{=} 	\mathcal J^{f_{1,2}, 2L} \cup \mathcal{J}_{\varepsilon}^{2L} \cup \mathcal J^{f_{3,4}, L} 	\cup \mathcal{J}_{\varepsilon}^{L}, \qquad \mathcal{J}_{\varepsilon}^{r}\stackrel{{\rm law}}{=}  \mathcal J^{\frac {u\varepsilon} 
{20A m_0}1_{B(z,15L)}, r},
\end{equation}
where the union on the right-hand side of \eqref{def:hatJacurr0} is over independent sets and
\begin{equation}
\label{def:tildef1tildef2}
\begin{split}
&f_{i,j}(\cdot) \stackrel{{\rm def}.}{=} \frac {u}A \big( f_i(\cdot)  + f_j(\cdot) \big) 1_{B}(\cdot).
\end{split}
\end{equation}
We now claim that there exists a coupling $\mathbb Q_{2, a, 2}^{(1)}$ of $(\hat {\mathcal J}_{a+1}^{ {\rm curr}})'' $ and $\hat {\mathcal J}_{a, a+1}^{\rm curr}$ such that
\begin{equation}\label{eq:Q'2a2_bnd2}
\mathbb Q_{2, a, 2}^{(1)} \big[(\hat {\mathcal J}_{a+1}^{ {\rm curr}})''  \supset \hat {\mathcal J}_{a, a+1}^{\rm curr} \big] \geq 1 - e^{- c (\log L)^{\gamma_2}}.
\end{equation}
Indeed, this is an application of Theorem~\ref{thm:short_long}, as we now explain. Letting $ f^{\varepsilon}_{i,,j}= f_{i,j} + \frac{u}{A} \varepsilon^2 1_{B_{10L}(z)}$, with $ f_{i,j}$ as in \eqref{def:tildef1tildef2}, first notice in view of 
\eqref{def:f2} and \eqref{def:f'} that we can write
\begin{align*}
f' &= (l^+)^{-1} \Big(\frac 12\sum_{k=0}^{2l^+-1}P_{k \Delta L} (f_{1,2}^+) + \sum_{k=0}^{l^+ -1}P_{k \Delta L} (f_{3,4}^+) \Big) \le  \frac{u^+}{u}l^{-1} \Big(\frac 12\sum_{k=0}^{2l^+-1}P_{k \Delta L} (f^{\varepsilon}_{1,2}) + \sum_{k=0}^{l^+-1}P_{k \Delta L} (f^{\varepsilon}_{3,4}) \Big) \\
&= \frac{u^+}{u}l^{-1} \Big(\frac 12\sum_{k=0}^{2l-1}P_{k \Delta L} (f^{\varepsilon}_{1,2}) + 
\sum_{k=0}^{l-1}P_{k \Delta L} (f^{\varepsilon}_{3,4}) \Big) + \frac{u^+}{u}l^{-1} \Big(\frac 12\sum_{k= 2l}^{2l+1}P_{k \Delta L} (f^{\varepsilon}_{1,2}) + P_{L} (f^{\varepsilon}_{3,4}) \Big).
\end{align*}
We will refer to the two terms in parentheses as $g_1'$ and $g_2'$ 
respectively, so that $f' \leq \frac{u^+}{u}l^{-1} (g_1'+g_2')$. Since $f'$ is the profile controlling $\hat {\mathcal J}_{a, a+1}^{\rm curr}$
\eqref{eq:J-intermediate}, this gives a domination in terms of configurations involving $g_1'$ and $g_2'$. Theorem~\ref{thm:short_long} only deals with $g_1'$. Since $f^{\varepsilon}_{1,2}, f^{\varepsilon}_{3,4} \le u$ pointwise on $\Z^d$ on the event $F_y$,  we can apply 
Theorem~\ref{thm:short_long} separately for the triplets $(f, L', L) = (\frac{u^+}{u}f_{1,2} +  \frac{u\varepsilon} {40A m_0}1_{B_{15 L}(z)} , \Delta L, 2L)$ and $(\frac{u^+}{u}f_{3,4} + \frac{u\varepsilon}{40A m_0}1_{B_{15 
L}(z)} , \Delta L,  L)$ to get a coupling $\mathbb Q^{(1)}_{2, a , 2}$ between $\mathcal J_1 \stackrel{{\rm law}}{=}\mathcal J^{(1 + \Cr{C:sprinkDelL}\,l^{-1/2}) g_1', \Delta L}$ and 
\begin{equation}\label{def:hatI'acurr0}
\mathcal{J}_2\stackbin{{\rm law}}{=} \mathcal J^{(1 + \Cr{C:sprinkDelL}l^{-1/2})f_{1,2}, 2L} \cup  \mathcal J_{\varepsilon/2}^{2L} \cup \mathcal J^{(1 + 
		\Cr{C:sprinkDelL}l^{-1/2})f_{3,4}, L}  \cup \mathcal J_{\varepsilon/2}^L\end{equation}
(recall that $\frac{u^+}{u} = 1 + l^{-1}$), in such a way that $\{\mathcal{J}_2 \supset  \mathcal J_1\}$ holds with high probability, as in \eqref{eq:Q'2a2_bnd2}. To deal with $g_2'$, observe that on the event $F_y$, $P_{k\Delta L}( f_{i,i+1}^{\varepsilon}) \le \frac u A$ pointwise on $\Z^d$ ($i = 1, 
3$). Moreover, owing to our assumptions on $\gamma_1$, $\gamma_2$ and $\varepsilon$ and the fact that $m_ 0 = 
\Cr{subdivide}\lfloor \log L \rfloor$, see \eqref{def:lowest_scale}, we have 
$l^{-1/2} = (\log L)^{-2 \gamma_2} \le 2 \varepsilon^{1.5}$ (see~\eqref{def:epsx} and \eqref{eq:varepsilon_L} and recall that $\gamma_2  \ge \gamma_{1} + 5$). This implies in particular that $\mathcal J^{(1 + \Cr{C:sprinkDelL}l^{-1/2}) g_2', \Delta L}$ is dominated by $\mathcal J_{\varepsilon/4}^L \cup \mathcal J_{\varepsilon/4}^{2L}$. Together with \eqref{def:hatI'acurr0}, \eqref{def:hatIacurr1}, and recalling that $ \hat {\mathcal J}_{a, a+1}^{\rm curr} \leq_{\text{st}} \mathcal J^{(1 + \Cr{C:sprinkDelL}l^{-1/2})(g_1' + g_2'), \Delta L}$, the desired inclusion \eqref{eq:Q'2a2_bnd2} readily follows upon suitably extending the probability space underlying $\mathbb Q_{2, a, 2}^{(1)}$.

With \eqref{eq:Q'2a2_bnd2} at hand, the remaining task is effectively to couple $(\hat {\mathcal J}_{a+1}^{ {\rm curr}})''$ with $(\hat {\mathcal J}_{a+1}^{ {\rm curr}})'$ in an increasing fashion with high probability. This will be achieved against the joint environment $ \hat{\mathcal {J}}_{a}^{{\rm back}}$ through two further couplings $\mathbb Q_{2, a, 2}^{(2.1)}$ and $\mathbb Q_{2, a, 2}^{(2.2)}$, involving Theorem~\ref{thm:long_short} (for which $ \hat{\mathcal {J}}_{a}^{{\rm back}}$ is needed) and Theorem~\ref{thm:short_long}, respectively.
To this effect, returning to the explicit formula \eqref{def:hatIacurr1} for $(\hat {\mathcal J}_{a+1}^{ {\rm curr}})'$ and recalling $f_{i,j}$ from \eqref{def:tildef1tildef2}, we see that 
\begin{equation}\label{eq:superhard-''frombelow}
\mathcal J^{f_{1,2}, 2L} \cup  \mathcal J^{f_{3,4}, L} \cup \tilde{\mathcal J}^L_{2.5 \varepsilon} \leq_{\rm st.} (\hat {\mathcal J}_{a+1}^{ {\rm curr}})', \qquad \tilde{\mathcal J}^L_{2.5 \varepsilon} \stackrel{\text{law}}{=} \mathcal J^{\frac {u\varepsilon} 
{8Am_0}1_{B(z,20L)}, L}.
\end{equation}
As opposed to $\mathcal{J}_{2.5 \varepsilon}^{L}$ as declared by \eqref{def:hatJacurr0}, $\tilde{\mathcal J}^L_{2.5 \varepsilon}$ stands for a slightly larger support (radius $20L$ instead of $15L$) for the intensity function. Comparing \eqref{eq:superhard-''frombelow} and \eqref{def:hatJacurr0}, we deduce that it is enough to couple ${\mathcal J}^{2L}_{\varepsilon} \cup {\mathcal J}^{L}_{\varepsilon} \cup \hat{\mathcal {J}}_{a}^{{\rm back}}$ with  $\tilde{\mathcal J}^L_{2.5 \varepsilon} \cup\, \hat{\mathcal {J}}_{a}^{{\rm back}}$ in an increasing manner with high probability. The configuration ${\mathcal J}^{L}_{\varepsilon}$ is readily dispensed with but the longer length-$2L$ trajectories comprising ${\mathcal J}^{2L}_{\varepsilon}$ remain and we only have length-$L$ trajectories within $\tilde{\mathcal J}^L_{2.5 \varepsilon}$ at our disposal. In fact this step would simplify had we worked instead with $\mathcal I_{\cdot}^{u, L} = \widetilde {\mathcal I}_{\cdot}^{u, L}$. For, in that case the length of trajectories within ${\mathcal{J}}^{u [f_{4,a}, f_{4,a+1}], 
L_4}(\omega_4)$ appearing in \eqref{def:hatIacurr1}, which lead to $\tilde{\mathcal J}^L_{2.5 \varepsilon}$ in the present case, would instead have length $L_4=2L$, which then trivially covers all trajectories (length $L$ \textit{and} $2L$) present. 

We now first apply Theorem~\ref{thm:long_short} to the function $f \stackrel{{\rm def}.}{=} \frac {u\varepsilon} {20Am_0}1_{B_{17L}(z)}$ with $K = 15L$ (centering at $z$, i.e.~replacing $B_K$ by $B_K(z)$), with $L' =\Delta L$ and the background $\hat{\mathcal J}_{a}^{\rm back}$ as environment configuration (recall that this is 
stochastically larger than $\mathcal I^\rho$ defined in \eqref{def:v2}, which satisfies the required condition $(\textnormal{C}_{\textnormal{obst}})$), to obtain a measure $\mathbb Q_{2, a, 2}^{(2.1)}$ coupling $\mathcal J_1 
\stackrel{\text{law}}{=}  {\mathcal J}^{2L}_{\varepsilon} \cup \hat{\mathcal J}_{a}^{\rm back}$ and
$\mathcal J_{2} \stackrel{\text{law}}{=}  \mathcal J^{(1 + C (l')^{-1/2})f' , \Delta L} \cup \hat{\mathcal J}_{a}^{\rm back}$
with $l' = \frac{L}{L'}$ and $$f' = (2l')^{-1} \sum_{0 \leq k < 2l'}P_{k \Delta L} (f) = (l')^{-1} \sum_{0\leq k < l'}P_{k \Delta L} \Big(\frac{1 + P_L}{2}\Big)f$$ (not to be confused with \eqref{def:f'}, which is unrelated), in such a way that $\mathbb Q_{2, a, 2}^{(2.1)}\big[ \mathscr{C}^{\partial}_{D_y}\big(\mathcal V(\mathcal J_{1})  \big)  \supset  \mathscr{C}^{\partial}_{D_y}\big(\mathcal V(\mathcal J_{2})  
\big)   \big] \geq 1 - e^{- c (\log L)^{\gamma_2}}$. 
In view of the second formula for $f'$ and given that 
$(1 + C(l')^{-1/2})\frac{1 + P_L}{2} f \le \frac{u\varepsilon}{19Am_0}$ holds pointwise on $\Z^d$ by our assumptions on $\gamma_1, \gamma_2$ and 
$\varepsilon$, Theorem~\ref{thm:short_long} immediately applies to the triplet $(f, L', L) = (\frac 
{u\varepsilon} {19Am_0}1_{B_{20L}(z)}, L', L)$ and yields a coupling $\mathbb Q^{(2.2)}_{2, a, 2}$ between $\mathcal J_2$
 and $\mathcal J_3 \stackrel{{\rm law}}{=} 
\tilde{\mathcal J}^{L}_{20\varepsilon/19}$ such that $\mathbb Q_{2, a, 2}^{(2.2)}[ \mathcal J_{3}  \supset  \mathcal 
J_{2}    ] \geq 1 - e^{- c (\log L)^{\gamma_2}}$. Thus, concatenating $\mathbb Q_{2, a, 2}^{(2.1)}$ and $\mathbb Q_{2, a, 2}^{(2.2)}$ by virtue of \cite[Lemma 2.4]{RI-III}, recalling \eqref{def:hatJacurr0}, restoring by suitable extension the independent configurations $\mathcal J^{f_{1,2}, 2L} \cup  \mathcal J^{f_{3,4}, L}$  as well as $ {\mathcal J}^{L}_{\varepsilon} (\leq_{\rm st.})  \tilde{\mathcal J}^{L}_{\varepsilon}$ while keeping in mind \eqref{eq:superhard-''frombelow}, one arrives at the coupling $\mathbb Q_{2, a, 2}^{(2)}$ with the property that
\begin{equation}\label{eq:superhard-final_I'mrunningoutoflabels-1}
\mathbb Q_{2, a, 2} \big[ \mathscr{C}^{\partial}_{D_y}\big( (\hat {\mathcal J}_{a+1}^{ {\rm curr}})' \cup \hat{\mathcal J}_{a}^{\rm back} \big) \supset \mathscr{C}^{\partial}_{D_y}\big( (\hat {\mathcal J}_{a+1}^{ {\rm curr}})'' \cup \hat{\mathcal J}_{a}^{\rm back} \big) \big] \geq 1 - e^{- c (\log L)^{\gamma_2}}.
\end{equation}	

We now put together all the elements. Restoring the (independent) configuration $\hat{\mathcal J}^{\rm out}$ in $\mathbb Q_{2, a, 2}$, recalling that $(\hat {\mathcal J}_{a+1}^{ {\rm curr}})' \cup \hat{\mathcal J}_{a}^{\rm back} \cup \hat{\mathcal J}^{\rm out}= \hat {\mathcal J}_{a+1}$, see \eqref{eq:J_a+1'}, and concatenating with the coupling $\mathbb Q_{2, a, 2}^{(1)}$ from \eqref{eq:Q'2a2_bnd2}, by which one effectively replaces the configuration $(\hat {\mathcal J}_{a+1}^{ {\rm curr}})'' \cup \hat{\mathcal J}_{a}^{\rm back} \cup \hat{\mathcal J}^{\rm out}$ in \eqref{eq:superhard-final_I'mrunningoutoflabels-1} by the (smaller) configuration $\hat {\mathcal J}_{a, a+1}^{\rm curr} \cup \hat{\mathcal J}_{a}^{\rm back} \cup \hat{\mathcal J}^{\rm out}$, which is precisely $\hat{\mathcal J}_{a,a+1}$ (in law), see \eqref{eq:J-intermediate}, 
and concatenating with $\mathbb Q_{2, a, 1}$, which links $  \hat {\mathcal J}_{a}$ and $\hat{\mathcal J}_{a,a+1}$, see \eqref{eq:couple_L_by_L'2}, one obtains the coupling $\mathbb Q_{2, a}$ with the property that
\begin{equation}\label{eq:couplebndQ2a}
\mathbb Q_{2, a} \big[ \mathscr{C}^{\partial}_{D_y}\big(\mathcal V( \hat {\mathcal J}_{a+1})  \big)  \supset  \mathscr{C}^{\partial}_{D_y}\big(\mathcal V( \hat {\mathcal J}_{a})  \big)   \big] \geq 1 - e^{- c (\log L)^{\gamma_2}}.
\end{equation}
Finally, chaining $\mathbb Q_{2, a}$ over integers $a$ with $0 \leq a< A$ and recalling \eqref{eq:elongation} yields the desired coupling $\mathbb Q_{2}$, for which an analogue of \eqref{eq:couplebndQ2a} holds, with  $\hat {\mathcal J}_{A}\stackrel{\text{law}}{=}\mathcal{J}_{k,m}$ and $\mathcal{J}_T \, ( \leq_{\text{st.}} \hat {\mathcal J}_{0}) $ in place of $\hat {\mathcal J}_{a+1}$ and $\hat {\mathcal J}_{a}$, respectively. Inspection of the proof reveals that the coupling is in fact conditional on $\hat{\mathcal J}^{\rm out}$ in \eqref{def:hathat_Ia}, which implies that $\mathbb Q_{2}$ satisfies versions of \eqref{eq:local_diff} and \eqref{eq:local_couple} for $p=1,2,3$ (whereby $\hat{\mathcal{V}}_1$ is defined in terms of $\hat{\mathcal{V}}_2$ by suitable extension).

\bigskip

{\bf  Coupling between $\mathcal V(\mathcal J_T) \stackrel{\textnormal{law}}{=} \widehat{\mathcal V}_3$ and $\mathcal V(\mathcal J_{k, m-1/2}) \stackrel{\textnormal{law}}{=} \widehat{\mathcal V}_4$.} We now separately construct a coupling $\mathbb Q_{3} = \mathbb Q_{3, \tilde \sigma}$ such that an analogue of \eqref{eq:superhardcoupling} holds for $p=3,4$, which entails on $F_y(\tilde{\sigma})$ that $\mathcal J_{k,m-\frac12} \subset \mathcal{J}_T$ holds under $\mathbb Q_{3}$ with high probability. On account of \eqref{eq:I_T} and \eqref{eq:intermediate_cofig} this coupling should have been 
immediate since `forgetting the clock' can only increase the length of trajectories. However, one also has to 
take into account the effect of removing all trajectories underlying $\mathcal J_{k, m - 1/4}$ that start inside 
${T^{\circ}}$, cf.~\eqref{eq:phi_T}. The following result is key to this. Recall $T^{\circ}$ from \eqref{def:cylinder} and $B=B(z,6L) \, (\supset T^{\circ})$ with $z \in \mathbb{L}$ minimizing $d(\mathbb{L}, T^{\circ})$.
\begin{lemma}\label{lem:missing_traject}
For any $v > 0$ and $L \geq C(\gamma_2)$,
$$\mathcal J^{v (\log L)^{2\gamma_2} 1_{{T^{\circ}}},\, 2L} \leq_{\rm st.}
\mathcal J^{v1_{B \setminus {T^{\circ}}} ,\, 4L} \qquad \text{(see \eqref{eq:J} for notation).}$$
\end{lemma}

\begin{proof}[Proof of Lemma~\ref{lem:missing_traject}] 
Consider the function $\rho:\mathbb{N}^\ast \times \Z^d 
\to \mathbb{R}_{+}$  given by $\rho(\ell, x) = 	\frac{dv}{L} 1\{\ell = 4L, x \in  B \setminus {T^{\circ}} \}$ so that, in the notation of Section~\ref{sec:truncation}, see \eqref{eq:prelim3.1}, the set $\mathcal I^\rho$ has the same law as $\mathcal J^{v1_{B \setminus {T^{\circ}}},\, 4L}$. It then follows by the second item in \eqref{eq:prelim5} that
\begin{equation}\label{eq:insideT-1}
\mathcal I^{\rho_{{T^{\circ}}}}\, \leq_{\rm st.}\, \mathcal I^{\rho} \ \big( \stackrel{{\rm law}}{=} 
\mathcal J^{v1_{B \setminus {T^{\circ}}},\, 4L} \big),
\end{equation}
where $\rho_{{T^{\circ}}}$ is as defined in \eqref{eq:prelim4} with $K=T^{\circ}$. Now letting $\eta = \sum_{i}\delta_{(\ell_i, w_i)}$ refer to a generic realization of the Poisson point process underlying $\mathcal I^{\rho_{T^{\circ}}}$, whose intensity is $\nu_{\rho_{T^{\circ}}}$ as in \eqref{eq:prelim1}, consider the derived point process 
$$\eta_{{\rm hom}} \stackrel{{\rm def}.}{=} \sum_{i}\delta_{(\ell_i - t_{\text{h}},  \,w_i[t_{\text{h}}, \infty))}1 \{\ell_i \ge t_{\text{h}} + 2L, w_i(t_{\text{h}}) \in {T^{\circ}}\},
$$ 
where $t_{\text{h}}=({r_T^{\circ}})^{2}$ (recall 
from \eqref{eq:kappa}-\eqref{def:cylinder} that ${r_T^{\circ}}$ denotes the radius of the cylinder ${T^{\circ}}$). In words, $\eta_{{\rm hom}}$
takes the walks (originally starting in $B \setminus {T^{\circ}}$) hitting $T^{\circ}$ that have at least $t_{\text{h}} + 2L$ units of time  left on their clock after re-rooting, and records the trajectory and time left after letting them evolve freely for an extra (homogenization) time $t_{\text{h}}$.
 Clearly, by construction,
\begin{equation}\label{eq:insideT-2}
\mathcal I^{\rho_{\rm hom}}  \stackrel{{\rm law}}{=}  \bigcup_{(\ell, w)\in \eta_{{\rm hom}}} w[0,2L - 1] \subset \mathcal I^{\rho_{{T^{\circ}}}},
\end{equation}
where 
\begin{equation}\label{def:vT'}
\rho_{\rm hom}(\ell, x)  \stackrel{{\rm def}.}{=} \sum_{y \in \partial {T^{\circ}}}\, \Big(\sum_{\ell' \ge t_{\text{h}} + 2L}\, 
\rho_{{T^{\circ}}}(\ell', y) \Big)P_y[X_{t_{\text{h}}} = x] 1\{\ell = 2L, x \in {T^{\circ}} \}.	
\end{equation} 
In particular, by \eqref{eq:insideT-1} and \eqref{eq:insideT-2} we get
\begin{equation}\label{eq:insideT-3}
\mathcal I^{\rho_{\rm hom}} \leq_{\rm st.}\mathcal J^{v1_{B \setminus {T^{\circ}}},\, 4L}.
\end{equation}
We will now separately show that for $L \geq C(\gamma_2)$,
\begin{align}
& \label{eq:insideT-4} \sum_{\ell' \ge t_{\text{h}} + 2L}\, \rho_{{T^{\circ}}}(\ell', y)  \ge  (4d)^{-1}v e_{{T^{\circ}}}(y), \text{ $y \in \partial {T^{\circ}}$,}\\
& \label{eq:uniform_low_bnd}
\inf_{x \in T^{\circ}} P_{e_{{T^{\circ}}}}[X_{t_{\text{h}}} = x] \ge  (4d)^2 L^{-1}  (\log L)^{2\gamma_2}.
\end{align}
Indeed, feeding \eqref{eq:insideT-4} and \eqref{eq:uniform_low_bnd} into \eqref{def:vT'} immediately yields that $\rho_{\rm hom}(2L, x) \ge 4d  v\frac{(\log L)^{2\gamma_2}}{L}$ for all $x \in {T^{\circ}}$, which concludes the proof on account of \eqref{eq:prelim3.1}. It thus remains to argue that \eqref{eq:insideT-4} and \eqref{eq:uniform_low_bnd} hold. With regards to the former, using the explicit formula \eqref{eq:prelim4} for $\rho_{{T^{\circ}}}(\cdot, 
\cdot)$ and recalling $\rho$ from above \eqref{eq:insideT-1}, we get that for any $y \in \partial {T^{\circ}}$ and $L \geq C(\gamma_2)$,
\begin{multline*}
\sum_{\ell \ge t_{\text{h}} + 2L}\rho_{{T^{\circ}}}(\ell, y) 
\ge \sum_{2L - t_{\text{h}} \ge \ell' \ge 0} E_y \big[ \rho(\N + t_{\text{h}} + 2L + \ell', X_{\ell'}) 1_{\{\widetilde{H}_{{T^{\circ}}} > \ell'\}} \big]\\ 
 = \sum_{2L - t_{\text{h}} \ge \ell' \ge 0}\frac{dv}{L} P_y[\widetilde {H}_{{T^{\circ}}}> \ell' ] \ge \frac{2L - t_{\text{h}}}{L} (4d)^{-1} v \, e_{{T^{\circ}}}(y) \ge (4d)^{-1} v e_{{T^{\circ}}}(y).
\end{multline*}
To verify \eqref{eq:uniform_low_bnd}, we use the lower bounds from Lemma~\ref{lem:equil_new}. Recalling $\partial_{{\rm long}} {T^{\circ}}$ from above Lemma~\ref{lem:equil_new} we can estimate for any $x \in {T^{\circ}}$, using a standard heat kernel lower bound,
\begin{multline*}
\sum_{y \in \partial {T^{\circ}} }   e_{{T^{\circ}}}(y) P_y[X_{t_{\text{h}}} = x] \ge c  ({r_T^{\circ}})^{d-1}  \inf_{y \in \partial_{{\rm long}} {T^{\circ}}}  e_{{T^{\circ}}}(y) \inf_{|y - x| \le 2{r_T^{\circ}}}   P_y[X_{t_{\text{h}}} = x]  \\ \geq \frac{c}{r_T^{\circ}} \inf_{y \in \partial_{{\rm long}} {T^{\circ}}} e_{{T^{\circ}}}(y) \stackrel{\eqref{eq:equi1_new}}{\ge}\,  \frac{c(\gamma_2)}{t_{\text{h}} \log \log L} \ge (4d)^2 \frac{(\log L)^{2\gamma_2}}{L},
\end{multline*}
for all $L \geq C(\gamma_2)$, where in the final step we used that $t_{\text{h}}=({r_T^{\circ}})^{2}$ with ${r_T^{\circ}} = 4 L^{1/2}(\log L)^{-2\gamma_2}$.
\end{proof}

With Lemma~\ref{lem:missing_traject} at our disposal, we proceed to define $\mathbb{Q}_3$.
Let ${\widehat{\omega}}^{{\rm in}}_{{T^{\circ}}}$ be obtained from $ \widehat{\omega}=( \widehat{\omega}_1,\dots,  \widehat{\omega}_4)$, see above \eqref{eq:phi_T}, by retaining for each point measure $\widehat{\omega}_i$ only those 
points $(u, \ell, w) \in \widehat{\omega}_i$ such that $w(0) \in {T^{\circ}}$. Write $\widehat{\omega}= \widehat{\omega}^{{\rm in}}_{{T^{\circ}}}+ \widehat{\omega}^{{\rm out}}_{{T^{\circ}}}$ (understood coordinatewise). The processes $ \widehat{\omega}^{{\rm in}}_{{T^{\circ}}}$ and $\widehat{\omega}^{{\rm out}}_{{T^{\circ}}}$ are independent under $\P_{\tilde{\sigma}}$. Let $\mathcal J_{k, t}^{{\rm in}}=\mathcal J_{k, t}({\widehat{\omega}}^{{\rm in}}_{{T^{\circ}}})$ and similarly $\mathcal J_{k, t}^{{\rm out}}$ so that $\mathcal J_{k, t}= \mathcal J_{k, t}({\widehat{\omega}})= \mathcal J_{k, t}^{{\rm in}} \cup \mathcal J_{k, t}^{{\rm out}}$, the union being over independent sets. By
definition of $\mathcal{J}_T$, see \eqref{eq:phi_T} and \eqref{eq:I_T}, one has that 
\begin{equation}\label{eq:q3-1}
\mathcal{J}_T= \mathcal{J}_T(\widehat{\omega})= \mathcal{J}_T(\widehat{\omega}^{{\rm out}}_{{T^{\circ}}}) = \mathcal{J}_T^{(1)} \cup \mathcal{J}_T^{(2)},
\end{equation}
where $\mathcal{J}_T^{(1)} = \mathcal{J}_{k,  m-\frac12} (\Phi_T({\widehat{\omega}})) $ and $\mathcal{J}_T^{(2)}= \mathcal{J}_T \setminus \mathcal{J}_T^{(1)}$. In view of \eqref{eq:intermediate_cofig}, the union on the right-hand side of \eqref{eq:q3-1} is over independent sets. Now,
\begin{equation}\label{eq:q3-2}
\mathcal{J}_T^{(1)} = \mathcal{J}_{k,  m-\frac12} (\Phi_T(\widehat{\omega}^{{\rm out}}_{{T^{\circ}}}))
\supset\mathcal J_{k, m-\frac12}({\widehat{\omega}}^{{\rm out}}_{{T^{\circ}}}) = \mathcal J_{k, m-\frac12}^{{\rm out}}.
\end{equation}
The inclusion in \eqref{eq:q3-2} is plain by construction of $\Phi_T$. In view of \eqref{eq:q3-1}, \eqref{eq:q3-2}, \eqref{eq:intermediate_cofig} and by independence, it is thus enough to argue that there exists a coupling $\mathbb Q_3= \mathbb Q_{3,\tilde{\sigma}}$ of $\mathcal J_{k, m-\frac12}$ and $\mathcal{J} \cup  \mathcal J_{k, m-\frac12}^{{\rm out}} $, where $\mathcal J  \stackrel{{\rm 
law}}{=} \mathcal J^{u I \cdot 1_{B \setminus {T^{\circ}}}, L}(\omega_4) \subset \mathcal{J}_T^{(2)}$ with $I=\big[us_{k}' + \frac{m-1/2}{m_0}us', us_{k}' + \frac{m-1/4}{m_0}us'\big]$, independent of $\mathcal J_{k, m-\frac12}^{{\rm out}}$, with the property that on the event $F_y(\tilde{\sigma})$,
\begin{equation}\label{eq:Q3}
\mathbb Q_{3}\big[  \mathscr{C}^{\partial}_{D_y}\big(\mathcal V(\mathcal J_{k, m-\frac12}^{{\rm in}} \cup \mathcal J_{k, m-\frac12}^{{\rm out}})  \big)  \supset  \mathscr{C}^{\partial}_{D_y}\big(\mathcal V(\mathcal J \cup \mathcal J_{k, m-\frac12}^{{\rm out}})  \big)    \big] \geq 1 - e^{- c (\log L)^{\gamma_2}}.
\end{equation}
Indeed, in view of \eqref{eq:q3-1} and \eqref{eq:q3-2}, $\mathbb Q_{3}$ can be extended to a coupling between $\mathcal J_T$ and $\mathcal J_{k, m - 1/2}$ such that $\mathscr{C}^{\partial}_{D_y}(\mathcal V(\mathcal J_T)) \subset \mathscr{C}^{\partial}_{D_y}(\mathcal V(\mathcal J_{k, m-\frac12})) $ with high probability.

To obtain \eqref{eq:Q3}, recalling that $\varepsilon \le 2(\log L)^{-(\gamma_1 + 5)}$ by \eqref{def:epsx}, $m_0 = \Cr{subdivide}\lfloor \log L \rfloor$ by \eqref{def:lowest_scale} and $\gamma_2 \ge (\gamma_1 + 5)$ by the hypothesis of Lemma~\ref{L:superhardcoupling}, one first applies Lemma~\ref{lem:missing_traject} on the event $F_y$ with $v= \frac{u\varepsilon}{(\log L)^{\gamma_2} m_0}$ to deduce, as we now explain, that
\begin{equation}
\label{eq:q3-4}
\mathcal J_{k, m-\frac12}^{{\rm in}} \leq_{\rm st.} \mathcal J^{(\log L)^{-\gamma_2\frac{u\varepsilon}{ m_0} }1_B
,\, 4L};
\end{equation}
to see this, one simply notes that the intensity of relevant starting points of trajectories for $\mathcal J_{k, m-\frac12}^{{\rm in}}$ is bounded on the event $F_y$ by (see above \eqref{def:v'} for notation) by $u 1_{{T^{\circ}}} \sum_{i}f_i \leq Cu 1_{{T^{\circ}}}$, which is bounded by $v (\log L)^{2\gamma_2} 1_{{T^{\circ}}}$ for the above choice of $v$.

In view of \eqref{eq:q3-4} and \eqref{eq:Q3}, it is enough to suitably couple $(\mathcal J^{ (\log L)^{-\gamma_2} \frac{u\varepsilon}{m_0} 1_{\tilde{B}}, \, 4L}$, $\mathcal J^{u I \cdot 1_{\tilde{B}
}, L}(\omega_4))$ with $\tilde{B}= B(z, 17L)$, in the presence of $\mathcal J_{k, m-\frac12}^{{\rm out}}$ as in \eqref{eq:Q3}. This is achieved by successive application of Theorems~\ref{thm:long_short} and~\ref{thm:short_long} with $K=6L$, following a similar line of argument as in the construction of $\mathbb Q_{2, a, 2}^{(2.1)}$ and $\mathbb Q_{2, a, 2}^{(2.2)}$ above (leading up to \eqref{eq:superhard-final_I'mrunningoutoflabels-1}). We now specify the relevant intensity profile, called $\tilde{\rho}$ in the sequel, which plays the role of $\rho$ in Theorem~\ref{thm:long_short} and needs to satisfy $(\textnormal{C}_{\textnormal{obst}})$. Whereas in principle, the presence of $\mathcal J_{k, m-1/2}^{{\rm out}}$ allows for an application of Theorem~\ref{thm:long_short} with the same background configuration $\mathcal{I}^{\rho}$ as defined in \eqref{def:v2}, the superscript `out' forces one to  replace $1_{\bar{B}}$ in  \eqref{def:v2}  by $1_{\tilde{B} \setminus  {T^{\circ}}}$. The absence of trajectories starting in ${T^{\circ}}$ may spoil the lower bound on the average local time in \eqref{eq:disconnect_background0}. It also violates the ellipticity assumption in \eqref{eq:disconnect_background2}. 

To cope with this, we sacrifice a small proportion $A^{-1}$ of the trajectories present in $\mathcal J_{k, m-1/2}^{{\rm out}}$ when defining $\tilde\rho$ and only start looking at them after they evolve for a time $\text{dy} (t_{\rm h} + t_{\rm h}')= o(L)$ introduced shortly. As will become clear, this solves both afore mentioned problems at once. Let $\hat{L}/ 2 = L- \text{dy} (t_{\rm h} + t_{\rm h}')$ with $t_{\rm h}= L (\log L)^{-2\gamma_2}$ and $ t_{\rm h}'=  t_{\rm h} (\log L)^{\gamma_2}$. We will have $\tilde{\rho}(\ell,\cdot)= \tilde{\rho}(\ell,\cdot) 1_{\ell \in \{\hat{L}, \hat{L}/2\}}$. With $\rho$ given by \eqref{def:v2}, define for $ x \in \Z^d$
\begin{equation}\label{eq:-rhotilde}
\tilde{\rho}(\beta \hat{L},x)=  (1-A^{-1}) \rho(2\beta L,x) 1_{\tilde{B} \setminus  {T^{\circ}}} (x) +  \rho' (\beta \hat{L},x), \quad \beta \in \{ 2^{-1},1\} . 
\end{equation}
The intensity $\rho' $ is obtained as follows. First $\rho' ( \hat{L},x)=0$. To obtain $\rho' ( \hat{L}/2,x)$, one considers walks of either length $2 \beta L$, $\beta \in \{ 2^{-1},1\}$, with intensity of starting points equal to $A^{-1} \rho( 2\beta L,\cdot) 1_{\tilde{B} \setminus  {T^{\circ}}} (\cdot)$, keeps only those trajectories which hit $T^{\circ}$ before time $t_{\rm h}'$ and lets these evolve for another $t_{\rm h}$ units of time. The resulting points lying in 
$T^{\circ}$ form $\rho' ( \hat{L}/2, \cdot)$, and one attaches to each such point a random walk trajectory of length $\beta \hat{L}$. Thus, $\rho'$ is spatially supported in $T^{\circ}$. 
In view of \eqref{eq:-rhotilde} and by construction of $\rho'$, the set $\mathcal{I}^{\tilde{\rho}}$ is dominated by $\mathcal{I}^{{\rho 1_{\tilde{B} \setminus  {T^{\circ}}} }}$ in an obvious manner, and can thus be exhibited as part of $\mathcal J_{k, m-\frac12}^{{\rm out}}$. We now claim that
\begin{equation}
\label{eq-rho'-LB}
  \rho' ( \hat{L}/2,x) \geq c(\delta) L^{-1} (\log L)^{\gamma_2} , \quad x \in T^{\circ},
\end{equation}
for all $L \geq C(\gamma_2)$. To see this, one retraces the argument in the proof of Lemma~\ref{lem:missing_traject}, thereby witnessing $\rho'$ as being equal to the homogenization $\rho_{\text{hom}}$ of a suitable measure $\rho$ corresponding to \eqref{def:v2} up to a factor $A^{-1}$ and spatial support restricted to $\tilde{B} \setminus  {T^{\circ}}$. One thereby deduces \eqref{eq-rho'-LB} similarly as below \eqref{eq:uniform_low_bnd}. In doing so \eqref{eq:uniform_low_bnd} remains pertinent as is (due to our choice of $t_{\rm h}$), but the left-hand side of \eqref{eq:insideT-4} is replaced by a sum over $(L \geq) \ell' \geq L-t_h'$, which effectively leads to a factor $\frac{t_h'}{L}= (\log L)^{-\gamma_2}$ on the right-hand side of \eqref{eq:insideT-4}. Overall \eqref{eq-rho'-LB} follows.

Together with \eqref{eq:-rhotilde}, \eqref{eq-rho'-LB} is more than enough to conclude that $\tilde{\rho}$ satisfies $(\textnormal{C}_{\textnormal{obst}})$, since in particular \eqref{eq-rho'-LB} guarantees an intensity (much) larger than $\frac{8d u}{\hat{L}}$ of length-$\hat{L}/2$ trajectories starting inside $T^{\circ}$, which makes up for the afore mentioned loss. Theorems~\ref{thm:long_short} and~\ref{thm:short_long} now apply as desired with $\tilde{\rho}$ as background configuration (part of $\mathcal J_{k, m-1/2}^{{\rm out}}$). We omit further details.

\bigskip

We now conclude the proof of Lemma~\ref{L:superhardcoupling}.
Concatenating $\mathbb Q_2, \mathbb Q_3$ and $\mathbb Q_1 = \mathbb Q_4 = \mathbb P_{\tilde \sigma}$, we obtain a coupling $\mathbb Q = \mathbb Q_{\tilde \sigma}$ between the laws of 
$\mathcal J_{k+1/2}({\widehat{\omega}}_{\bar B}^{{\rm in}}), \mathcal J_{k, m - 1/2}({\widehat{\omega}}_{\bar B}^{{\rm in}}), \mathcal J_{T}({\widehat{\omega}}_{\bar B}^{{\rm in}}), \mathcal J_{k, m}({\widehat{\omega}}_{\bar B}^{{\rm in}})$ and $\mathcal J_{k+1}({\widehat{\omega}}_{\bar B}^{{\rm in}})$ satisfying the analogue of \eqref{eq:superhardcoupling} for these configurations
where 
${\widehat{\omega}}_{\bar B}^{{\rm in}}$ comprises all triplets $(u, \ell, w) \in \widehat{\omega}$ such that $w(0) \in \bar B$; to be precise, one checks by inspection of the proof that $\mathbb Q_2, \mathbb Q_3$ really define couplings of the restricted configuration ${\widehat{\omega}}_{\bar B}^{{\rm in}}$.
Now 
incorporating 
${\widehat{\omega}}_{\bar B}^{{\rm out}} = \widehat{\omega}- {\widehat{\omega}}_{\bar B}^{{\rm in}}$ into the underlying probability space, we obtain \eqref{eq:superhardcoupling} as a consequence of the corresponding properties for $\mathbb Q_2, \mathbb Q_3$, along with
properties \eqref{eq:local_diff} and \eqref{eq:local_couple}, which are immediate by construction of $\mathbb Q$ since $\bar B = B(y, \the\numexpr\couprad+2*\rangeofdep\relax L)$.
\end{proof}

\section{Near- and sub-diffusive scales}\label{sec:subdiff}
\label{sec:penelope}
We now carry out the proof of Lemma~\ref{lem:reduce_distance}, which is the most 
technically involved segment of this article. The proof roughly splits into three stages, each of which has a designated subsection among \S\ref{subsec_piv_V_T}-\S\ref{surgery-3} below. All resort on the preparatory results of Section~\ref{A:superhard} concerning the model~$\mathcal{V}_T$, which features prominently in the proof (albeit absent from the statement of Lemma~\ref{lem:reduce_distance}). In doing so, the proof thus implicitly relies on the findings of \cite{RI-III}. The actual (short) proof of Lemma~\ref{lem:reduce_distance} appears in \S\ref{sbusec:denouement} and assembles the various elements. 

We first provide an overview of these three individual stages, in order to convey the main lines of the argument, which we hope will ease the reading.
The reduced cluster separation $R_{T,m}$ inherent to $q_m$ in \eqref{eq:q_m-def} obtained at the outcome of of the first $m \leq m_0$ steps should be thought of as a partial reconstruction, in which an `almost-path' has been created, joining the two clusters in questions up to a  distance of at most $R_{T,m}$, while preserving the desired pivotal configuration in $\mathcal V_{k,\, m}$, which involves a dual disconnecting interface. The iteration over $m$, of which Lemma~\ref{lem:reduce_distance} represents a single step, is a matter of convenience. For  simplicity, the reader may choose to ignore this feature and imagine instead that Lemma~\ref{lem:reduce_distance} takes her directly from $\mathcal V_{k,\, m_0} =  \mathcal V_{k + 1}$ to $\mathcal V_{k, 0} = \mathcal V_{k + \frac 12}$ in a single step.

The separation between the two clusters is removed by constructing a connecting path in the pivotal configuration. We will return to this shortly when discussing stages $2$ and $3$ of the construction. An apparent difficulty with this surgery is the necessity to leave gaps when building (pieces of) paths. This is owed to degeneracies in the conditional laws, an enemy already faced in \cite{RI-II}, which warrant a buffer zone when conditioning on any part of the configuration such as the pivotal one to be preserved, whose geometry can be very wild -- these are de facto pieces of critical clusters. The design of $\mathcal{V}_T$, cf.~Lemma~\ref{lem:caio_T}, is precisely tailored to address this issue. To be able to actually work with $\mathcal{V}_T$ (rather than $\mathcal{V}_{k,m}$ for instance), Lemma~\ref{L:superhardcoupling} is applied to switch in and out of pivotal configurations involving $\mathcal{V}_T$. This important feature is derived separately in \S\ref{subsec_piv_V_T}, see Lemma~\ref{cor:sandwich}. It corresponds to a more elaborate version of Lemma~\ref{lem:sandwich_simple}. 

The second part, which is the content of \S\ref{surgery-2}, consists of forcing a connection, thereby reducing the overall separation while preserving pivotality. The buffer constraints stemming from decoupling naturally lend themselves to a hierarchical construction, and we build a deterministic (hierarchical) bridge inside the tubular region $T$ to this effect, which resorts on a construction introduced in \cite[Section 4]{RI-II}; see also Fig.~1 therein. The bridge roughly delimits safe zones within regions with possibly very rough boundaries (obtained from already explored (critical) clusters), in which pieces of paths  can be safely added conditionally on a partially reconstructed configuration. It is here that the polynomial connectivity lower bounds inherited from Lemma~\ref{L:connection} (via the implicit condition $u \lesssim {u}_{**}$) come into play. We note in passing that a similar - albeit streamlined (because applying directly for the full model $\mathcal{V}^u$ rather than the more intricate model $\mathcal{V}_T$) construction already played a central role in \cite{RI-II}. 

At the outcome of \S\ref{surgery-2} is a fully reconstructed path, save for gaps at the smallest scale, corresponding to the bottom level in our hierarchical construction. These gaps have poly-logarithmic size $(\log L)^C$ and are called \textit{holes}. The holes are dealt with separately with ad-hoc arguments in \S\ref{surgery-3}. From this, the proof of Lemma~\ref{lem:reduce_distance} quickly follows.

\medskip

We now draw up the above sketch. Throughout the remainder of this article, as in the statement of Lemma~\ref{lem:reduce_distance}, we always (implicitly) work under the assumption \eqref{eq:cond-2lemmas}.

\subsection{Pivotality switching reloaded} \label{subsec_piv_V_T}
To set the stage, we start by explaining how the results of Section~\ref{A:superhard} concerning $\mathcal{V}_T$ come into play in the context of Lemma~\ref{lem:reduce_distance}. Lemma~\ref{lem:caio_T} will be used extensively in the bridge construction presented in the next paragraph. Lemma~\ref{L:superhardcoupling} is used as a vehicle in the proof of the following key `pivotality switching' result, which roughly states that the pivotal 
region implicit in $q_m$, see \eqref{eq:q_m-def}, can in fact be preserved -- up to a multiple of $b$ and a 
small additive error term as in \eqref{eq:compa_a_state.1} -- when switching in and out of 
$\mathcal{V}_T$, cf.~Remark~\ref{R:passingtoVk+1}. This should be viewed as a refinement of Lemma~\ref{lem:sandwich_simple} to the case where one does not have a perfect inclusion between $\mathcal{V}'$ and $\mathcal{V}$, but rather a property like \eqref{eq:superhardcoupling}. The precise statement is as follows. Recall from \S\ref{sssec_subdiff} that $\Gamma= (\delta, \gamma, \gamma_1, \gamma_2, \gamma_3)$.
\begin{lemma}[Pivotality switching]
	\label{cor:sandwich}
Let $\gamma  \ge \gamma_2 \ge (\gamma_{1} + 5)  \vee 3 \gamma_M$, 
$(\mathcal{V}', \mathcal{V}) =(\mathcal{V}_{k, m}, \mathcal{V}_T)$ or $(\mathcal{V}_T, \mathcal{V}_{k, m - \frac12})$ for some $1 \le m \le m_0$ and  
$K \subset \widetilde{D}_y$. For all dyadic $L \ge L_0(\Gamma)$,
\begin{align}
\label{eq:cor_sandwich}
&\P_y^{\varepsilon} \left[\, \overline{\piv}_{K}(\mathcal V') \, \right] \le 
C(\pi_{y}^{\varepsilon})^{-1}\,b + \P_y^{\varepsilon} \left[\, \overline{\textnormal{Piv}}_{K}(\mathcal V) \, \right] +  e^{-c (\log L)^{\gamma_2}}
A f(y),
\end{align}
for some $c,C$ depending on $\Gamma$ 
(see \eqref{eq:superhardcoupling}, \eqref{eq:def_b}, \eqref{def:fx} and below \eqref{eq:compa_a_state.1} regarding $\pi(\cdot),b,f,$ and $A$, respectively).
\end{lemma}

While much simpler, the proof of Lemma~\ref{lem:sandwich_simple} (referred to multiple times below) captures the key idea behind the proof of Lemma~\ref{cor:sandwich}, which could be omitted upon first reading.

\begin{proof}[Proof of Lemma~\ref{cor:sandwich}] We start with a reduction step.
We claim that, under the hypotheses of Lemma~\ref{cor:sandwich}, recalling that $\widetilde{D}_y = B(y, \the\numexpr\couprad+6*\rangeofdep\relax L)$, it is enough to argue that
\begin{equation}
\label{eq:lem_sandwich}
\P_{y}^{\varepsilon}\left[\, \overline{\piv}_{K}(\mathcal V') \, \right] \le (\pi_y^{\varepsilon})^{-1}\,b + \P_{y}^{\varepsilon} \left[\, 
\overline{\textnormal{Piv}}_{K}(\mathcal V) \, \right] + e^{-c (\log L)^{\gamma_2}} \, \P\big[\,\piv_{\widetilde{D}_y}(\mathcal {V}_{k+1})\,\big];
\end{equation}
Indeed, if \eqref{eq:lem_sandwich} holds one completes the proof using  Lemma~\ref{L:boxpiv2closedpiv} as follows. First note that  \eqref{eq:boxpiv2closedpiv_gen_state} is in force with $N = \the\numexpr\couprad+6*\rangeofdep\relax L$. Hence, applying \eqref{eq:boxpiv2closedpiv_gen_state} and subsequently \eqref{eq:boxpiv2closedpiv_gen_state2'}, we obtain, in much the same way as  \eqref{eq:compa_a_state.2} was deduced, that
\begin{equation*}
\P\big[{\piv}_{\widetilde{D}_y}(\mathcal{V}_{k+1})\big] \leq e^{C (\log 
		M)^{2}}\sum_{z \in B^{\mathbb L}(y, 2M(\the\numexpr\couprad+6*\rangeofdep+10\relax L))}\P\left[\,\overline{{\piv}}_{z, {\the\numexpr\couprad+\rangeofdep\relax}L}(\mathcal{V}_{k+1})\,\right].
\end{equation*}
Using Lemma~\ref{lem:sandwich_simple} and the 
monotonicity of $K \mapsto \piv_{K}$, each summand in the above display is bounded by $b + f(z)$. Plugging the resulting estimate into the right-hand side of 
\eqref{eq:lem_sandwich} yields \eqref{eq:cor_sandwich} upon noticing that $\gamma_2 > 3\gamma_M$ and $8 M(\the\numexpr\couprad+6*\rangeofdep+10\relax L) \le  M_1 (= M (\the\numexpr5*(\couprad+3*\rangeofdep)\relax L))$ for large enough $L$.

We now return to the proof of \eqref{eq:lem_sandwich}. Invoking Lemma~\ref{L:superhardcoupling}, on account of \eqref{eq:marginals}, $(\mathcal V', \mathcal V) = (\mathcal V_{k, m}, 
\mathcal V_T)$ has the same law under $\mathbb{P}_{\tilde{\sigma}}$ as $(\widehat{\mathcal{V}}_2, \widehat{\mathcal{V}}_3)$ under $\mathbb{Q}_{\tilde \sigma,y}$.
We focus on $(\mathcal V', \mathcal V) = (\mathcal V_{k, m}, 
\mathcal V_T)$ in what follows; the other case is treated very similarly, the minor changes required when adapting the arguments below essentially boil down to exploiting different inclusions provided by~\eqref{eq:superhardcoupling}. Let $\mathbb Q_y^{\varepsilon}[\, \cdot \,] = P^{\tilde \sigma}[ \mathbb{Q}_{\tilde \sigma,y}[\, \cdot \,] \, | \, s_{\vert C_y} = \varepsilon]$, with $\mathbb{Q}_{\tilde \sigma,y}$ as given by Lemma~\ref{L:superhardcoupling}, and let $ \tilde F_y \supset F_y$ refer to the event comprising the requirements inherent to $F_y$ in \eqref{eq:F_y-superhard} except for $\{s_{\vert C_y} = \varepsilon\}$. That is, $\tilde F_y = \tilde F_y (\tilde{\sigma})= \big\{(s_k \vee r_{k+1}) \vert _{C_y} \le \frac\delta{10}, \mathsf U_{D_y} \in [\tfrac{e^{-L}}{2}, 1 - \tfrac{e^{-L}}{2}]\big\}$. By independence between $s_{| C_y}$ and $\big((s_k \vee r_{k+1}) \vert _{C_y}, \mathsf U_{D_y}\big)$, we have that
	\begin{equation*}
		\mathbb{Q}_y^{\varepsilon}[ \, \cdot \, ] \ge P^{\tilde \sigma}[  \mathbb{Q}_{\tilde \sigma,y}[\, \cdot \,] \, | \, F_y] \, P^{\tilde \sigma}[ \tilde F_y] 
	\stackrel{\eqref{eq:Poisson_tailbnd}, \eqref{eq: Gxk_bnd}, \eqref{eq:varepsilon_L}}{\ge} (1 - e^{-c (\log L)^{\gamma}}) 
	P^{\tilde \sigma}[  \mathbb{Q}_{\tilde \sigma,y}[\, \cdot \,] \, | \, F_y].
	\end{equation*}
Together with \eqref{eq:superhardcoupling}, abbreviating by $\text{Coup}$ the event on the left-hand side of the latter and using that
$\gamma_2 \le \gamma$, this is readily seen to imply the bound
\begin{equation}\label{eq:Qy_to_Qsigma-new}
		\mathbb Q_y^{\varepsilon}[\text{Coup}^c] \leq e^{-c(\log L)^{\gamma_2}},
\end{equation}
to which we will soon refer.
We start -- much like in \eqref{eq:sandwich_parition_con0} but now using $\mathbb{Q}_y^{\varepsilon}$ in absence of an immediate relation between $\mathcal{V}$ and $\mathcal{V}'$, where $(\mathcal V', \mathcal V) = (\mathcal V_{k, m}, 
\mathcal V_T)$ -- by partitioning the event on the left-hand side of \eqref{eq:lem_sandwich} according to 
	\begin{multline}
	\label{eq:sandwich_parition_con}
	\P_{y}^{\varepsilon} \left[\, \overline{\piv}_{K}(\mathcal V') \, \right]  \stackrel{\eqref{eq:marginals}}{=} 
	 \mathbb Q_y^{\varepsilon} \big[\, \text{Coup}^c, \, \overline{\piv}_{K}(\widehat{\mathcal 
	V}_2) \,  \big]  + \mathbb Q_y^{\varepsilon} \big[\, \text{Coup}, \,
	\overline{\piv}_{K}(\widehat{\mathcal{V}}_2) , \lr{}{ \widehat{\mathcal{V}}_3 }{U}{V}  \,  \big] \\ + \mathbb Q_y^{\varepsilon} \big[\, \text{Coup} , \, \overline{\piv}_{K}(\widehat{\mathcal 
	V}_2) , \nlr{}{ \widehat{\mathcal{V}}_3 }{U}{V} \,  \big]  .
	\end{multline}
		We consider the three terms on the right-hand side of \eqref{eq:sandwich_parition_con} separately, each of which will give rise to precisely one of the terms appearing on the right-hand side of \eqref{eq:lem_sandwich}.
		
		By monotonicity, $\overline{\piv}_{K}(\widehat{\mathcal{V}}_2) 
\subset \piv_{\widetilde{D}_y}(\widehat{\mathcal{V}}_2)$ (cf.~the definition in \eqref{def:closed_piv} and 
recall that $K \subset \widetilde{D}_y$) and $\piv_{\widetilde{D}_y}(\widehat{\mathcal{V}}_2)$ is measurable relative to 
$\sigma(\widehat{\mathcal{V}}_2(x'), x' \in \widetilde{D}_y^c)$. It follows by means of \eqref{eq:local_couple} and independence of the $\sigma_{B'}^B$'s, see~\eqref{eq:sigma_bb'}, that 
${\rm Coup}^c$ and ${\piv}_{\widetilde{D}_y}(\widehat{\mathcal{V}}_2)$ are independent under $\mathbb Q_y^{\varepsilon}$ (to see this one also uses the fact that $\mathbb{Q}_{\tilde{\sigma},y}[\piv_{\widetilde{D}_y}(\widehat{\mathcal{V}}_2) ]= \mathbb{P}_{\tilde{\sigma}}[\piv_{\widetilde{D}_y}({\mathcal{V}}_{k+1}) ]$ and $1\{s_{\vert C_y}=\varepsilon\}$ are independent under $P^{\tilde{\sigma}}$), whence
\begin{multline}\label{eq:sandwich_coupc2coarsepiv_fin_bnd}
		\mathbb Q_y^{\varepsilon} \big[\, {\rm Coup}^c , \, \overline{\piv}_{K}(\widehat{\mathcal{V}}_2)\, \big]	 \leq  \mathbb Q_y^{\varepsilon} \big[ \,  {\rm Coup}^c, \, \piv_{\widetilde{D}_y}(\widehat{\mathcal{V}}_2) \,\big]  = \mathbb Q_y^{\varepsilon} [{\rm Coup}^c] \, \cdot \, \mathbb Q_y^{\varepsilon} \big[{\piv}_{\widetilde{D}_y}(\widehat{\mathcal{V}}_2)\big]\\
		\stackrel{\eqref{eq:Qy_to_Qsigma-new}}{\leq} 	e^{-c(\log L)^{\gamma_2}} \, \mathbb Q_y^{\varepsilon} \big[{\piv}_{\widetilde{D}_y}(\widehat{\mathcal{V}}_2)\big]
			 \stackrel{ \eqref{eq:local_diff}, \eqref{eq:marginals}}{\le}  e^{-c(\log L)^{\gamma_2}} \,\P \big[\piv_{\widetilde{D}_y}(\mathcal V_{k+1})\big].
	\end{multline}
We now consider the second term on the right-hand side of \eqref{eq:sandwich_parition_con}, which we will treat in a manner similar 
 as \eqref{eq:closedpiv_t_b} except that we need to be careful due to the restricted nature of our coupling. To this end, we begin with the following observation. Since $D_y$ can not intersect both $U$ and $V$ owing to the hypothesis $R \ge 2\max(r, M_0)$ (recall the statement of Proposition~\ref{prop:comparisonLk}), we get in a similar fashion as \eqref{eq:compa202.a} that for any $\omega \in \{0, 1\}^{\Z^d}$,
	\begin{align}
		\label{eq:domain_decomp}
		1_{\{\lr{}{ \omega }{U}{V}\}} \mbox{ is an increasing function of } \big(\omega \cap D_y^c, \mathscr{C}^{\partial}_{D_y}(\omega)\big),
	\end{align}
where the underlying partial order (call it $\preceq$) is by inclusion as 
subsets of $\Z^d$. Now $\widehat{\mathcal{V}}_2 \cap D_y^c = \widehat{\mathcal{V}}_1 \cap D_y^c$ by \eqref{eq:local_diff} and on the event ${\rm Coup}$ we have that $\mathscr{C}^{\partial}_{D_y}(\widehat{\mathcal{V}}_1) \subset \mathscr{C}^{\partial}_{D_y}(\widehat{\mathcal{V}}_2)$. Therefore, on the event ${\rm Coup}$, we see that $$\big(\widehat{\mathcal{V}}_1 \cap D_y^c, \mathscr{C}^{\partial}_{D_y}(\widehat{\mathcal{V}}_1)\big)
	\preceq \big(\widehat{\mathcal{V}}_2 \cap D_y^c, \mathscr{C}^{\partial}_{D_y}(\widehat{\mathcal V}_2)\big)$$
	and consequently, owing to \eqref{eq:domain_decomp}, it follows that
		\begin{equation*}
		\big\{{\rm Coup}, \, \overline{\piv}_{K}(\widehat{\mathcal{V}}_2) \big\} \stackrel{\eqref{def:closed_piv}}{ \subset} \{{\rm Coup}, \, \nlr{}{ \widehat{\mathcal{V}}_2 }{U}{V}\} 
		 \subset \{\nlr{}{\widehat{\mathcal{V}}_1}{U}{V}\}.
	\end{equation*}
	On the other hand, from \eqref{eq:Itildeinclusions} and \eqref{eq:Ibarinclusions} we have that
	$\widehat{\mathcal{V}}_1 \subset \widehat{\mathcal{V}}_5$ holds almost surely as $(\widehat{\mathcal{V}}_5, \widehat{\mathcal{V}}_1) 
	\stackrel{{\rm law}}{=} (\mathcal V_{k + 1/2}, \mathcal V_{k + 1})$ by 
	\eqref{eq:marginals}. Using a similar reasoning as above, we then obtain
		$$\big\{\lr{}{ \widehat{\mathcal{V}}_3 }{U}{V}, {\rm Coup}\big\} \subset \{\lr{}{\widehat{\mathcal{V}}_5}{U}{V}\}.$$
	Combining the previous two displays, we deduce, writing $\mathbb Q_y[\cdot]= E^{\tilde \sigma}[ \mathbb{Q}_{\tilde \sigma,y}[\cdot]]$, that
	\begin{multline}\label{eq:term-2-switch}
		\pi^{\varepsilon}_y \, \mathbb Q_y^{\varepsilon} \big[\,  {\rm Coup}, \,  \overline{\piv}_{K}(\widehat{\mathcal{V}}_2) , \lr{}{ \widehat{\mathcal{V}}_3 }{U}{V} \, \big] = \mathbb Q_y\big[\,  {\rm Coup}, \,  \overline{\piv}_{K}(\widehat{\mathcal{V}}_2) , \lr{}{ \widehat{\mathcal{V}}_3 }{U}{V}, F_y \, \big]  \\	
		\le  \mathbb Q_y \big[\lr{}{\widehat{\mathcal{V}}_5}{U}{V}, \nlr{}{\widehat{\mathcal{V}}_1}{U}{V}\big] \stackrel{\eqref{eq:marginals}}{=} \P[\lr{}{\mathcal V_{k + \frac 12}}{U}{V}] -  \P[\lr{}{\mathcal V_{k + 1}}{U}{V}] \stackrel{\eqref{eq:def_b}}{=} b.
	\end{multline}

	We proceed similarly with the last term in \eqref{eq:sandwich_parition_con}. 
Taking hint from \eqref{closedpiv_t_closedpiv}, we start by observing that , since $\widehat{\mathcal V}_{p} \cap D_y^c = \widehat{\mathcal{V}}_1 \cap D_y^c$ for $p = 2, 3$ in view of \eqref{eq:local_diff} and $\mathscr{C}^{\partial}_{D_y}(\widehat{\mathcal{V}}_2) \subset \mathscr{C}^{\partial}_{D_y}(\widehat{\mathcal{V}}_3)$ on the event ${\rm Coup}$, we have that $$\big(\widehat{\mathcal{V}}_2 \cap D_y^c, \mathscr{C}^{\partial}_{D_y}(\widehat{\mathcal V}_2)\big) \preceq \big(\widehat{\mathcal{V}}_{3} \cap D_y^c, \mathscr{C}^{\partial}_{D_y}(\widehat{\mathcal{V}}_3)\big).$$
	Using \eqref{eq:domain_decomp} to compare $\omega = \widehat{\mathcal{V}}_2 \cup K$ and $\widehat{\mathcal{V}}_3 \cup K$, we then get
	\begin{equation*}
		\big\{ {\rm Coup}, \, \overline{\piv}_{K}(\widehat{\mathcal{V}}_2)\big\} \subset \big\{  {\rm Coup}, \, \lr{}{\widehat{\mathcal{V}}_2 \cup K}{U}{V} \big\} \subset \{\lr{}{\widehat{\mathcal{V}}_3 \cup K}{U}{V}\}.
	\end{equation*}
Taking intersections on either side with the complement of the event $\lr{}{ \widehat{\mathcal{V}}_3 }{U}{V}$, the right-hand side equals $\overline \piv_{K}(\widehat{\mathcal{V}}_3)$ and averaging over $\mathbb Q_y^{\varepsilon}$ yields that
	\begin{equation} \label{eq:term-3-switch}
		\mathbb Q_y^{\varepsilon} \big[\, {\rm Coup}, \, \overline{\piv}_{K}(\widehat{\mathcal{V}}_2), \, \nlr{}{ \widehat{\mathcal{V}}_3 }{U}{V} \, \big] \le \mathbb Q_y^{\varepsilon}\big[\,\overline \piv_{K}(\widehat{\mathcal{V}}_3)\, \big] \stackrel{\eqref{eq:marginals}}{=} \mathbb P_{y}^{\varepsilon}\left[\, \overline \piv_{K}(\mathcal V)\, \right].
	\end{equation}
Substituting \eqref{eq:sandwich_coupc2coarsepiv_fin_bnd}, \eqref{eq:term-2-switch} and \eqref{eq:term-3-switch} into \eqref{eq:sandwich_parition_con} leads to \eqref{eq:lem_sandwich}, which concludes the proof.
\end{proof}

We conclude this paragraph by explaining how Lemma~\ref{cor:sandwich} is used. We apply it as follows. Recalling $q_m$ from \eqref{eq:q_m-def}, for any integer $ m$ 
with $1\leq m \leq m_0$, by a union bound, we have that
\begin{equation}
 \label{eq:pixy}
\begin{split}
&q_m \leq  \sum_{\substack{z, w \in B(y,20L), \\ |z - w|_1 \le R_{T,m}}}  q_m(z,w), \text{ where } \\
&q_m(z,w) \stackrel{\text{def.}}{=} \P_{y}^{\varepsilon}\big[\, \nlr{}{ {\mathcal V}_{k, m}}{U}{V} , \, z \in \mathscr{C}_U(\mathcal V_{k, m}),  w \in \mathscr{C}_V(\mathcal V_{k, m}), |z - w|_1 =   d_y(\mathscr{C}_U(\mathcal V_{k, m}), \mathscr{C}_V(\mathcal V_{k, m}))\, \big ].
\end{split}
\end{equation}
with $d_y$ as in \eqref{eq:q_m-def}.  
We attach to this setup a triplet of tubes $(T,T' ,T^{\circ})$ as follows. For any $z,w$ as in \eqref{eq:pixy}, fix an $\ell^1$-geodesic $\pi'$ between $z$ and $w$ consisting of $d$ axis-aligned segments, and choose the direction $j \in \{ 1,\dots,d\}$ which contains the longest segment in $\pi'$ (if there 
are several such $j$'s we pick the smallest one). By `reordering' the segments comprising 
$\pi'$ we may assume that the longest segment has one endpoint at $z$. Let $w'$ denote its 
other endpoint. 
By construction, we have  $|w-w'|\leq R_{T,m} (1- d^{-1}) =  R_{m-1}$. We then let $T \subset T' \subset T^{\circ}$ be the tubes declared by \eqref{def:cylinder} attached to this choice of $z$ and coordinate direction $j$. Henceforth, the relevant tubes will always refer to this specific triplet. Note that \eqref{eq:tube-cond} holds whenever
$L \geq C(\gamma_2)$, which will tacitly be assumed from here on. Since by 
construction, any path connecting $z$ and $w$ yields a connection between $\mathscr 
C^m(U)$ and $\mathscr{C}_V(\mathcal V_{k, m})$ on the event in \eqref{eq:pixy}, we have the bound 
$$q_m(z,w) \leq \P_{y}^{\varepsilon}\big[\,\overline{\textnormal{Piv}}_{T \cup \pi}(\mathcal 
{V}_{k, m}) \, \big], \quad  1\leq m \leq m_0, 
$$
where $ \pi = \pi_{w',w}$ is an $\ell^1$-geodesic between $w'$ and $w$, which is tacitly identified with its range in writing $T \cup \pi$. Applying \eqref{eq:cor_sandwich} with 
$K = T \cup \pi$, which has the required property $K \subset \widetilde{D}_y$ since $z, w \in B(y, \the\numexpr\couprad+\rangeofdep\relax L)$ and $T$ satisfies \eqref{eq:tube-cond}, allows one to pass from $\mathcal V_{k,m}$ to $\mathcal{V}_T$,  thus obtaining
\begin{equation}
 \label{eq:finalbd2}
q_m(z,w)\leq  
C(\pi_{y}^{\varepsilon})^{-1}\, b + \mathbb{P}_{y}^{\varepsilon} \left[\, \overline{\textnormal{Piv}}
			_{T \cup \pi}(\mathcal V_T) \, \right] +  e^{-c (\log L)^{\gamma_2}} 
			A f(y) 
 \end{equation}
(with $c$ as in \eqref{eq:superhardcoupling}) for all $1 \leq m \leq m_0$ and $L \ge C'(\Gamma)$, whenever $\gamma  \ge \gamma_2 \ge (\gamma_{1} + 5)  \vee 3 \gamma_M$.

%
\subsection{Bridge and construction of an almost-path} \label{surgery-2} We now reduce the pivotal region $T \cup\pi$, $\pi= \pi_{w',w}$, to a smaller region $H \cup \pi$ where $H$ refers to a collection of (small) holes in 
the tube $T$; the main result is Lemma~\ref{L:renormalization}, which bounds the quantity $ \mathbb{P}_{y}^{\varepsilon} [ \overline{\textnormal{Piv}}_{T \cup \pi}(\mathcal V_T) ]$ appearing above in terms of $ \mathbb{P}_{y}^{\varepsilon}[\overline \piv_{H \cup \pi}( \mathcal{V}_{k,m-\frac12})]$ and similar additive error terms as those appearing in~\eqref{eq:finalbd2}.
 This will bring to bear the notion of a {\em hierarchical bridge} $\mathbb B$ introduced in \cite[Section~4]{RI-II}, to which we will frequently refer in the sequel. The bridge $\mathbb B$ represents the support on which the relevant (almost-)path is to be constructed.
 
 Broadly speaking, a (hierarchical) 
bridge $\mathbb B$ is a collection of boxes organized as an ordered family of 
sub-collections $(\mathbb B_j)_{1 \le j \le J}$ that altogether form a chain of boxes between two sets of interest; in the present case these will be $\mathscr{C}_U( 
\mathcal{V}_T)$ and $\mathscr{C}_V(\mathcal{V}_T)$.  Here $\mathscr{C}_U( 
\mathcal{V})$ refers here to the cluster of $U$ in $\mathcal{V}$, i.e.~the union of $U$ and all connected components of $\mathcal{V}$ adjacent to it. Intuitively, the higher $j$ the smaller the scale, so $(\mathbb B_i)_{1 \le i \le j}$ indicates a finer resolution as $j$ grows. Importantly, for all $j < J$, a given box $B\in \mathbb B_{j}$ of radius $r$, say, enjoys the additional property that it maintains a distance of {\em at least} $\lceil 
r^{1-{\Cr{c:box_gap}}} \rceil$ from all the other boxes comprised in $\bigcup_{j' \le j} \mathbb B_{j}$ as well as the sets 
$\mathscr{C}_U(\mathcal{V}_T) \cup \mathscr{C}_V(\mathcal{V}_T)$. This property will be useful for 
decoupling  certain crossing events below; cf.~the proof of Lemma~\ref{L:renormalization}. The decoupling concerns $\mathcal{V}_T$ and is facilitated by Lemma~\ref{lem:caio_T}, which is also the source of the exponent $1-{\Cr{c:box_gap}}$ determining the size of the gap around $B$; cf.~also $\widetilde{B}$ above \eqref{eq:caioGbound}.

\medskip

We now give a precise meaning to this outline. We will associate a bridge $\mathbb{B}= \mathbb{B} (\mathcal{C}_U,\mathcal{C}_V)$ to any realization $(\mathcal{C}_U,\mathcal{C}_V)$ of $(  
\mathscr{C}_U(\mathcal{V}_T), \mathscr{C}_V(\mathcal{V}_T))$. In fact we are only interested in realizations $(\mathcal{C}_U,\mathcal{C}_V)$ giving rise to a configuration in 
$\overline{\textnormal{Piv}}_{T \cup \pi}(\mathcal V_T)$ (recall that pivotality relates to the connection event $\{U \leftrightarrow V\}$, cf.~above \eqref{def:fx}). For definiteness, if $\{  (  
\mathscr{C}_U(\mathcal{V}_T), \mathscr{C}_V(\mathcal{V}_T))= (\mathcal{C}_U,\mathcal{C}_V) \} \not\subset \overline{\textnormal{Piv}}_{T \cup \pi}(\mathcal V_T)$, we set $J = 1$ and $\mathbb B = \mathbb B_1=  \emptyset$. 

Now suppose $\{  (  
\mathscr{C}_U(\mathcal{V}_T), \mathscr{C}_V(\mathcal{V}_T))= (\mathcal{C}_U,\mathcal{C}_V) \} \subset \overline{\textnormal{Piv}}_{T \cup \pi}(\mathcal V_T)$. Note that $\mathcal C_U$ and $\mathcal 
C_V$ are necessarily disjoint if the event $\overline{\textnormal{Piv}}_{.}(\mathcal V_T)$ occurs. We distinguish three mutually exclusive cases to define $\mathbb{B}= \mathbb{B} (\mathcal{C}_U,\mathcal{C}_V)$. {Case i):} both $\mathcal C_U \cup T$ and $\mathcal 
C_V \cup T$ are connected sets (i.e.~$\mathcal C_U, \mathcal C_U$ each intersect $\overline{T}=T \cup \partial_{\text{out}}T$). In this case, referring to the beginning of \cite[Section 4]{RI-II} for notation (see above (B.1) therein), we declare $\mathbb{B}$ to be the bridge associated to the septuple
\begin{equation}
\label{eq:bridge1-def}
\text{$(\mathcal C, \mathcal D, 200(\log L)^{4\gamma_3}, (\log L)^{\gamma_3},  \Cl{C:Mxi}, \xi,T)$}, \quad \xi =  1-{\Cr{c:box_gap}},
\end{equation}
with $\mathcal C= \mathcal C_U$, $\mathcal D= \mathcal C_V$, and $\Cr{C:Mxi}   = m(\xi)$ as supplied by \cite[Proposition 4.1]{RI-II}, which is in force whenever $L \geq C$, as tacitly assumed from here on; also recall that  $\gamma_3 \ge 10$ 
is a parameter, see \eqref{def:lowest_scale}, and ${\Cr{c:box_gap}} \in (0,\frac12)$ was defined in Lemma~\ref{lem:caio_T}. The statement of \cite[Proposition 4.1]{RI-II} being a pure existence result, if more than one bridge $\mathbb{B}$ associated to \eqref{eq:bridge1-def} exist we simply choose one according to some deterministic ordering.

Case ii): $\mathcal C_U \cup \mathcal C_V \cup \pi$ is a connected set and case i) does not occur. In this 
case we take $J = 1$, $\mathbb B_1 = \emptyset$ and $\mathbb B = \mathbb B_1$. Case iii): neither case i) or ii) occurs. Then since $\mathcal C_U \cup \mathcal C_V \cup \pi \cup T$ is connected on the event $\overline{\textnormal{Piv}}_{T \cup \pi}(\mathcal V_T)$, it follows that $\overline T$ and $\overline{\pi}=\pi \cup \partial_{\text{out}}\pi $ each intersect exactly one of $\mathcal C_U$ or $\mathcal C_V$. Let $\mathcal C_U$ 
intersect $\overline T$. The sets $\mathcal C_U$ and $\mathcal C_V \cup \pi$ are disjoint and both must intersect 
$T$. Thus \cite[Proposition 4.1]{RI-II} applies and yields a bridge $\mathbb B$ associated to the tuple in \eqref{eq:bridge1-def} with
$\mathcal C = \mathcal C_U$, $\mathcal D = \mathcal C_V \cup \pi$. If instead $\mathcal C_V$ intersects $\overline T$ one chooses $\mathcal C = \mathcal C_U\cup \pi$, $\mathcal D = \mathcal C_V $ in \eqref{eq:bridge1-def}.
\medskip

We now collect the properties of bridges that will be used in the sequel, referring to \cite[Section 4]{RI-II} for the items listed below. It follows from the above construction that in 
all cases,
\begin{equation}\label{eq:B1B3B4}
\text{\parbox{13cm}{$\mathbb{B}=\mathbb B(  
\mathscr{C}_U(\mathcal{V}_T), \mathscr{C}_V(\mathcal{V}_T))$ satisfies properties (B.1), (B.3) and (B.4) appearing in \cite[Section 4]{RI-II} for 
$\mathcal C, \mathcal D$ as above, $N = R_T$, $L = r_T$, $s= (200)^4(\log L)^{4\gamma_3}$, $s' = (\log L)^{\gamma_3}$, $m = \Cr{C:Mxi}$ and $\xi = 1-{\Cr{c:box_gap}}$}}
\end{equation} 
(in particular, in case the bridge is empty all these properties hold trivially).
For a given bridge $\mathbb{B}$, we refer to the union of boxes in $\mathbb{H} \stackrel{\text{def.}}{=}\mathbb B_J$ 
as {\em holes}, which we abbreviate by $H \,(= \bigcup_{B \in \mathbb{H}} B)$, and define $\mathcal{H} $ to be the collection of all sets $H$ obtained in this way from $\mathbb B= \mathbb B(  
\mathscr{C}_U(\mathcal{V}_T), \mathscr{C}_V(\mathcal{V}_T))$ as $  
\mathscr{C}_U(\mathcal{V}_T), \mathscr{C}_V(\mathcal{V}_T)$ range over all their possible realizations. Notice that ${H}$ may well be empty.
As we now explain, by the above construction, it then follows that on the event $\overline{\textnormal{Piv}}_{T \cup \pi}(\mathcal V_T)$, abbreviating $\mathbb B= \mathbb B(  
\mathscr{C}_U(\mathcal{V}_T), \mathscr{C}_V(\mathcal{V}_T))$,
\begin{equation} \label{B2'}  \tag{B2'}
\text{\parbox{14.0cm}{{\em (Connectivity).} For any path $\pi_B$ connecting the two marked vertices 
of $B$ for each $B \in \mathbb B \setminus \mathbb H$, the union of all $\pi_B$'s along with ${H}$ and $\pi$ 
connects $\mathscr{C}_U(\mathcal{V}_T)$ and $\mathscr{C}_U(\mathcal{V}_T)$.}}
\end{equation}
Indeed, in cases i) and iii) property \eqref{B2'} boils down to \cite[(B.2)]{RI-II}, which \cite[Proposition 4.1]{RI-II} guarantees. In case ii),~\eqref{B2'} remains valid since $\mathbb{H}=  \mathbb B \setminus \mathbb B_{J} = \emptyset$
and the union of $\mathscr{C}_U(\mathcal{V}_T),$  $\mathscr{C}_V(\mathcal{V}_T)$ and $\pi$ is a connected set. The characteristics \eqref{eq:B1B3B4} and \eqref{B2'} are the only features of bridges that will be used in the sequel.

\medskip
Returning to the quantity $q_m(v,w)$ in \eqref{eq:finalbd2} we aim to bound, we now 
reduce the pivotal region for the event on the right-hand side of \eqref{eq:finalbd2} using the above construction and employing the properties of $\mathcal{V}_T$ exhibited in \S\ref{surgery-1} to reconstruct pieces of open path between $\mathscr{C}_U(\mathcal{V}_T)$ and $\mathscr{C}_V(\mathcal{V}_T)$ while retaining the relevant pivotal configuration. The key relevant features of $\mathcal{V}_T$ include for one the conditional decoupling of Lemma~\ref{lem:caio_T}, as well as the pivotality switching involving $\mathcal{V}_T$ supplied by Lemma~\ref{cor:sandwich}. 


\begin{lemma}[$y \in \mathbb{L}$]
	\label{L:renormalization}
	If (cf.~Lemma~\ref{lem:caio_T} regarding ${\Cr{c:box_gap}}$)
\begin{equation}\label{eq:gamma-cond-lem-surgery-2}
 \gamma_2 \ge (\gamma_{1} + 5) \vee 3\gamma_M \vee C(d) \quad \text{and} \quad \gamma \wedge {\Cr{c:box_gap}}  \gamma_{3} \ge 5 \bar \gamma_2 \ \big(\stackrel{\eqref{eq:gamma2bar}}{=}5 \Cr{C:gamma2}\, \gamma_2\big),
 \end{equation} 
 then for all dyadic $L \geq L_0(\Gamma)$,
\begin{multline}
\label{eq:finalbd3}
\P_{y}^{\varepsilon} \left[\, \overline{\textnormal{Piv}}
_{T \cup \pi}(\mathcal {V}_{T})\,\right] \\\leq e^{C(\log L)^{4\bar \gamma_2}} \big( \sup_{H \in \mathcal H} \mathbb{P}_{y}^{\varepsilon}\big[\,\overline \piv
_{H \cup \pi}( \mathcal{V}_{k,m-\frac12}) \, \big] + (\pi_{y}^{\varepsilon})^{-1}\,  b \big) + e^{-c(\log L)^{\gamma_2}} A f(y).
\end{multline}
\end{lemma}
\begin{proof}
We will localize on the bridge $\mathbb B(  
\mathscr{C}_U(\mathcal{V}_T), \mathscr{C}_V(\mathcal{V}_T))$. To this end, for a given deterministic bridge $\mathbb{B}$ we abbreviate $\mathbb{P}_{\tilde{\sigma}, \mathbb{B}}[\,\cdot\,]=\P_{\tilde{\sigma}}[\, \cdot , \, \mathbb{B}(\mathscr{C}_U(\mathcal{V}_T), \mathscr{C}_V(\mathcal{V}_T))= \mathbb{B} ]$ and $\mathbb{P}_{ \mathbb{B}}[\,\cdot\,]= E^{\tilde{\sigma}}[\mathbb{P}_{\tilde{\sigma},\mathbb{B}}[\,\cdot\,]]$. All subsequent arguments operating under either $\mathbb{P}_{\tilde{\sigma}, \mathbb{B}}$ or $\mathbb{P}_{ \mathbb{B}}$ are tacitly implied to hold for any realization $\mathbb{B}$ such that $ \emptyset \neq \{ \mathbb{B}(\mathscr{C}_U(\mathcal{V}_T), \mathscr{C}_V(\mathcal{V}_T))=\mathbb{B}\} \subset \overline{\textnormal{Piv}}_{T \cup \pi}(\mathcal {V}_{T})$.

 We will prove that under \eqref{eq:gamma-cond-lem-surgery-2} and for any $\tilde \sigma \in F_y$ (see \eqref{eq:F_y-superhard}) and $L \geq C(\Gamma)$, one has
\begin{align}
\label{eq:almost_reconnect_localize}
\P_{\tilde \sigma, \mathbb{B}}\left[\, \overline{\textnormal{Piv}}_{T\cup\pi }(\mathcal {V}_{T} )\right] \leq 
e^{C(\log L)^{3\bar \gamma_2}} \P_{\tilde \sigma, \mathbb{B}}\left [\, \overline \piv_{ H\cup\pi }(\mathcal{V}_{T}) \right] +  e^{-c(\log L)^{5\bar\gamma_{2}}} \cdot 
\P_{\tilde \sigma}\big [\, {\text{Piv}}_{\widetilde{D}_y}(\mathcal {V}_{k + 1}) \, \big].
\end{align}
Let us first argue how \eqref{eq:almost_reconnect_localize} implies the claim \eqref{eq:finalbd3}. Averaging \eqref{eq:almost_reconnect_localize} over $E^{\tilde{\sigma}}$ on the event $F_y= F_y(\tilde{\sigma})$ and using that the latter is independent of ${\text{Piv}}_{\widetilde{D}_y}(\mathcal {V}_{k + 1})$, which follows upon recalling \eqref{eq:F_y-superhard}, \eqref{def:coarse_piv}, the discussion preceding \eqref{eq:barequalstildeRANGE} and that 
$\widetilde{D}_y =   B_{\the\numexpr\couprad+6*\rangeofdep\relax L}(y)$, one deduces that
\begin{multline}\label{eq:barpivTzw}
\P_{\mathbb B}\left[\, \overline{\textnormal{Piv}}_{T\cup\pi }(\mathcal {V}_{T}), F_y\right]  \\
\leq e^{C(\log L)^{3 \bar \gamma_2}} \P_{\mathbb B}\left[\, \overline \piv_{ H\cup\pi }(\mathcal{V}_{T}), F_y \right] 
+  e^{-c(\log L)^{5\bar \gamma_{2}}} \cdot 
\P\big [\, {\text{Piv}}_{\widetilde{D}_y}(\mathcal {V}_{k + 1})\, \big] \P[F_y].
\end{multline}
 In a similar vein, using that the three fields $1\{{\text{Piv}}_{\widetilde{D}_y}(\mathcal {V}_{k + 1})\}$, $ s_{| C_y}$ (where $s=s_{k+1}-s_k$, cf.~above \eqref{def:Px}) and 
$((s_k \vee r_{k+1}) \vert _{C_y}, \mathsf U_{D_y})$ are independent, writing $\tilde F_y \supset F_y$ for the event obtained from \eqref{eq:F_y-superhard} when removing the condition $s_{|C_y} = \varepsilon$
(so $F_y = \tilde{F}_y \cap \{ s_{|C_y} = \varepsilon\}$), one gets for all $L \geq C(\Gamma)$,
\begin{multline}\label{eq:barpivTzw-bis}
\P_{\mathbb{B}}\left[\, \overline{\textnormal{Piv}}_{T\cup\pi }(\mathcal {V}_{T}), s_{|C_y}  = 
	\varepsilon, F_y^c\right] \stackrel{ \eqref{eq:tube-cond}, \eqref{eq:locally_identical}}{\le}  
\P\big[\, \overline{\textnormal{Piv}}_{\widetilde{D}_y}(\mathcal {V}_{k+1})\big] \P[\tilde F_y^c,  s_{|C_y} = \varepsilon] \\
\stackrel{\eqref{eq:Poisson_tailbnd}, \eqref{eq:varepsilon_L}}{\le} 
	e^{-c (\log L)^{\gamma}} \pi_y^{\varepsilon} \, \P\big[\, \overline{\textnormal{Piv}}_{\widetilde{D}_y}(\mathcal {V}_{k+1})\big] \stackrel{\eqref{eq:compa_a_state.5.2}}{\le}  e^{-c (\log L)^{5\bar\gamma_2}} \pi_y^{\varepsilon} \big( b + A f(y)\big),
\end{multline}
where the last bound is obtained using the fact that $\gamma \geq 5 \bar\gamma_2 > 2\gamma_M$ implied by \eqref{eq:gamma-cond-lem-surgery-2} (omitting the localization inherent to the subscript $\mathbb{B}$ might seem overly wasteful but isn't; see below). Adding \eqref{eq:barpivTzw} and \eqref{eq:barpivTzw-bis}, observing that their left-hand sides combine to form $\P_{\mathbb{B}}[\, \overline{\textnormal{Piv}}_{T\cup\pi }(\mathcal {V}_{T}), s_{|C_y}  = \varepsilon ] $, and dividing the resulting inequality by $\pi_y^{\varepsilon}$, thereby replacing the occurrences of $F_y$ stemming from \eqref{eq:barpivTzw} by $\{s_{|C_y} = \varepsilon\} (\supset F_y)$, and applying \eqref{eq:compa_a_state.5.2} yet again to control the factor $\P [\, {\text{Piv}}_{\widetilde{D}_y}(\mathcal {V}_{k + 1})]$ appearing on the right-hand side of \eqref{eq:barpivTzw}, it follows that 
\begin{equation}\label{eq:almost_reconnect_localize_annealed}
\begin{split}	
\P_{y, \mathbb{B}}^{\varepsilon}&\left[\, \overline{\textnormal{Piv}}_{T\cup\pi }(\mathcal {V}_{T}) \right] 
\leq e^{C(\log L)^{3 \bar \gamma_2}} \P_{y, \mathbb{B}}^{\varepsilon}\left [\, \overline \piv_{ H\cup\pi }(\mathcal{V}_{T}) \right] + e^{-c (\log L)^{5\bar\gamma_2}}  \big( b + A f(y)\big),
\end{split}
\end{equation}
for all $L \geq C(\Gamma)$, where with hopefully obvious notation $ \pi_y^{\varepsilon} \P_{y, \mathbb{B}}^{\varepsilon}[\cdot]=  \mathbb{P}[\,\cdot \, , \mathbb{B}(\mathscr{C}_U(\mathcal{V}_T), \mathscr{C}_V(\mathcal{V}_T))= \mathbb{B}, s_{|C_y} = \varepsilon]$. To arrive at \eqref{eq:finalbd3} one now proceeds as follows. Upon summing over all possible realizations $\mathbb{B}$ of $ \mathbb{B}(\mathscr{C}_U(\mathcal{V}_T), \mathscr{C}_V(\mathcal{V}_T))$, the left-hand side of \eqref{eq:almost_reconnect_localize_annealed} yields $\P_{y}^{\varepsilon} [\, \overline{\textnormal{Piv}}_{T \cup \pi}(\mathcal {V}_{T})]$, as desired. On the right-hand side of \eqref{eq:almost_reconnect_localize_annealed} one first bounds
$\P_{y, \mathbb{B}}^{\varepsilon} [\, \overline \piv_{ H\cup\pi }(\mathcal{V}_{T}) ]$ by $\sup_{H \in \mathcal{H}} \P_{y}^{\varepsilon} [\, \overline \piv_{ H\cup\pi }(\mathcal{V}_{T})] $, thereby omitting the bridge localization. The resulting combinatorial complexity when summing over bridges is accounted for  as follows. Let $\mathcal B$ denote the set of all 
possible bridges $\mathbb B$  that can arise as realizations of $\mathbb{B}(\mathscr{C}_U(\mathcal{V}_T), \mathscr{C}_V(\mathcal{V}_T))$ on the event $\overline{\textnormal{Piv}}_{T\cup\pi }(\mathcal 
{V}_{T})$. Since any box $B$, part of a bridge $\mathbb{B} \in \mathcal{B}$, is determined by its radius, which is at most $s$ due to \cite[(B.3)]{RI-II} (see also \eqref{eq:B1B3B4} regarding $s$) and its center, a point in $B(T,s)$ with $T$ as in \eqref{eq:kappa}-\eqref{def:cylinder}, and the number of boxes comprising $\mathbb{B}$ is controlled by~\cite[(B.4)]{RI-II} (with $N$, $L$ as in~\eqref{eq:B1B3B4}), it follows that
\begin{equation*}
|\mathcal B| \le \big((8L)^{d} C(\log L)^{4\gamma_3}\big)^{(\log L)^{3\bar \gamma_2}} \le e^{(\log L)^{3\bar \gamma_2 + C(d)}},
\end{equation*}
for all $L \geq C$. The claim \eqref{eq:finalbd3} 
now readily follows from \eqref{eq:almost_reconnect_localize_annealed} followed by an application of Lemma
~\ref{cor:sandwich} to pass from $\mathcal{V}_T$ to $\mathcal V_{k,m-\frac12}$.

\medskip

The remainder of the proof is devoted to establishing $\eqref{eq:almost_reconnect_localize}$. To pass from $\piv_{ T\cup\pi }(\mathcal V_T)$ to $\piv_{ H\cup\pi }(\mathcal V_T)$, we will leverage \eqref{B2'} to build connections on the bridge $\mathbb{B}$ at a preferential cost. Exhibiting this cost will in turn rely on the decoupling for $\mathcal{V}_T$ supplied by Lemma~\ref{lem:caio_T}. Recall that $\mathbb{B}= \bigcup_{1\leq j \leq J} \mathbb{B}_j$ and that any box $B \in \mathbb{B} \setminus \mathbb{H}$ (where $\mathbb{H}= \mathbb{B}_J$) comes with two associated marked points $z_{i,B} \in B$, $i=1,2$, see \eqref{B2'}. For such $B$, let
\begin{equation*}
E_B = \big\{\lr{}{\mathcal V_T \cap B}{z_{1, B}}{z_{2, B}}\big\}
\end{equation*}
and define 
\begin{equation}\label{eq:bridge1-conn}
 E_j = \bigcap_{B \in \mathbb B_j} E_B, \quad E_{j_-}= \bigcap_{1\leq k < j} E_k, \quad 1\leq j \leq J.
 \end{equation} 
 The occurrence of $E_B$ for any $B \in \mathbb{B} \setminus \mathbb{H}$ implies the existence of a path $\pi_B \subset (\mathcal{V}_T \cap B)$ connecting the two marked vertices of $B$. Thus, owing to property~\eqref{B2'} and by definition of $\overline{\piv}_K$, see \eqref{def:coarse_piv} and \eqref{def:closed_piv}, it follows that
\begin{align}\label{eq:bridge1-final}
\P_{\tilde{\sigma}, \mathbb B}\left[\,\overline \piv_{ H\cup\pi }(\mathcal V_T)\right] \ge \P_{\tilde{\sigma}, \mathbb B}\big[\,\overline \piv_{ T\cup\pi }(\mathcal {V}_{T}), E_{J_-}\big].
\end{align}
We will recursively bound the right-hand side of \eqref{eq:bridge1-final} by removing the `excess' events $ E_j$ forming part of $E_{j_-}$ in steps. Specifically, we will show that for all $1 \leq  j < J$,
\begin{equation}
\label{eq:lower_bnd_caio}
\begin{split}
&\P_{\tilde{\sigma}, \mathbb B}\big[\, \overline \piv_{ T\cup\pi }(\mathcal {V}_{T}), E_{(j+1)_-}\big] \ge e^{-C(\log L)^{2 \bar \gamma_2}} \,  \P_{\tilde{\sigma}, \mathbb B}\big[ \, \overline \piv_{ T\cup\pi }(\mathcal {V}_{T}), 
 E_{j_-}\big] - q,
\end{split}
\end{equation}
where $q = e^{-c(\log L)^{5\bar{\gamma}_{2}}} \cdot \P_{\tilde \sigma} [\, {\text{Piv}}_{\widetilde{D}_y}(\mathcal {V}_{k + 1}) \, ]$.  
Iterating \eqref{eq:lower_bnd_caio} over $1 \leq j < J$, using the fact that $J \leq C \log \log L$ on account of \cite[(B.4)]{RI-II} and \eqref{eq:B1B3B4}, and combining the resulting estimate with \eqref{eq:bridge1-final} (see also \eqref{eq:kappa} regarding the value of $r_T$), the claim \eqref{eq:almost_reconnect_localize} follows.

To deduce \eqref{eq:lower_bnd_caio}, it will be convenient to fix an arbitrary ordering $B_1, \ldots, B_n$, where $n = |\mathbb B_{j}|$, of the boxes in
$\mathbb B_{j}$. Following the notation of Lemma~\ref{lem:caio_T} (see above \eqref{eq:caioGbound}, see also \eqref{eq:B1B3B4} regarding the value of $\xi$), if $B_m = B(x,r)$ for some $x \in \mathbb{Z}^d$ and $r> 0$, we write $\widetilde{B}_m= B(x,r + \lceil  r^{\xi}\rceil)$ in the sequel. Recall from above Lemma~\ref{lem:caio_T} that $\mathcal{F}_K= \sigma(1\{ x' \in \mathcal{V}_T\}: x' \in K)$ and let $\mathcal{F}_m= \mathcal{F}_{\Z^d \setminus 
\widetilde{B}_m}$. We first observe that for all $m$ (under $\mathbb{P}_{\tilde{\sigma}}$)
\begin{equation}\label{e:bridge1-meas} \big\{\, \overline{\piv}_{ T\cup\pi }(\mathcal {V}_{T}), \mathbb{B}(\mathscr{C}_U(\mathcal{V}_T), \mathscr{C}_V(\mathcal{V}_T))=\mathbb{B}\big\} \in \mathcal{F}_K (\subset \mathcal F_m),
\end{equation}
where $K=\Z^d \setminus 
\bigcup_{B \in (\mathbb{B} \setminus \mathbb{H})}\widetilde{B}$; in words, $\mathcal{F}_K$ amounts to revealing the configuration of $\mathcal{V}_T$ everywhere except in the enlargement of the boxes forming the bridge $\mathbb{B}$, save for the holes. This can be seen by writing the event in question in \eqref{e:bridge1-meas} as a disjoint union 
over realizations $\{ \mathscr{C}_U(\mathcal{V}_T)=\mathcal{C}_U, \mathscr{C}_V(\mathcal{V}_T)= 
\mathcal{C}_V \}$ of the clusters of $U$ and $V$ in $\mathcal{V}_T$, where $(\mathcal{C}_U,\mathcal{C}_V)$ range over all configurations such that the associated bridge equals $\mathbb B$ and such that $\overline{\piv}_{ T\cup\pi }(\mathcal {V}_{T})$. One then notices that the event $\{ \mathscr{C}_U(\mathcal{V}_T)=\mathcal{C}_U, \mathscr{C}_V(\mathcal{V}_T)= 
\mathcal{C}_V \}$ is $\mathcal{F}_K$-measurable owing to property~\cite[(B.1)]{RI-II} which is in force on account of~\eqref{eq:B1B3B4}. This also takes care of the event $\overline{\piv}_{ T\cup\pi }(\mathcal {V}_{T})$ appearing in \eqref{e:bridge1-meas}, which is a function of $(\mathscr{C}_U(\mathcal{V}_T), \mathscr{C}_V(\mathcal{V}_T))$, see \eqref{def:closed_piv}.

Now recall the event $G_{B}$ 
from Lemma~\ref{lem:caio_T} and abbreviate $G_{m}=G_{B_m}$ and $E_{j,m}= \bigcap_{1 \leq k \leq m} E_{B_k}$, so that $E_{j,n}= E_j$, see \eqref{eq:bridge1-conn}. We now argue that for all $ 1\leq m \leq n$, on the event in \eqref{e:bridge1-meas} and for any $\tilde \sigma \in F_y$ (cf.~\eqref{eq:F_y-superhard}),
\begin{align}
\P_{\tilde \sigma}\big[\,   E_{j_-}, E_{j,m}, G_{m} 
\, \big| \, \sigma(\mathcal F_m, 1_{G_{m}}) \big] \ge e^{-C(\log L)^{2}} \, 
1_{\{E_{j_-}\, \cap \, E_{j,m-1} \, \cap \, G_{m}\}}.
\label{eq:conditional_connect}
	\end{align}
Indeed \eqref{eq:conditional_connect} follows directly from 
Lemma~\ref{lem:caio_T}, using the 
fact that the event $\{ E_{j_-}, E_{j,m-1},$ 
$G_{m}\}$ is measurable relative to 
$\sigma(\mathcal F_m, 1_{G_{m}})$ in view of \eqref{eq:bridge1-conn} and property~\cite[(B.1)]{RI-II} and because 
$$\P_{\tilde \sigma}\big[ E_{B_m}| \sigma(\mathcal{F}_m, 1_{G_{m}})] \stackrel{\eqref{eq:twopointsbound}, \eqref{eq:decineq}}{\geq} e^{-C(\log L)^{2}}1_{G_{m}};$$
here, in applying \eqref{eq:twopointsbound}, \eqref{eq:decineq}, we used the (crude) bound that the radius of the box $B_m \subset B(T,s)$ (cf.~the beginning of \cite[Section 4]{RI-II}) is at most $Cr_Ts \leq {L}$ for $L \geq C'$ on account of \eqref{eq:kappa}-\eqref{def:cylinder} and \eqref{eq:B1B3B4}.
Using \eqref{eq:conditional_connect}, we now complete the proof of \eqref{eq:lower_bnd_caio}. To this 
effect, we first claim that
\begin{multline}
\label{eq:lower_bnd_caio2}
\P_{\tilde{\sigma}, \mathbb B}\big[\, \overline \piv_{ T\cup\pi }(\mathcal {V}_{T}), 
 E_{j_-}, E_{j,m-1}, G_{m}\big] \\ \ge \P_{\tilde{\sigma}, \mathbb B}\big[\, \overline \piv_{ T\cup\pi }(\mathcal {V}_{T}),  E_{j_-}, E_{j,m-1}\big] - 
\P_{\tilde \sigma}[G_{m}^c]\cdot \P_{\tilde \sigma}\big [\,{\text{Piv}}_{\widetilde{D}_y}(\mathcal {V}_{k + 1}) \, \big].
\end{multline}
To see this, one first notices that 
$\overline \piv_{ T\cup\pi }(\mathcal {V}_{T}) \subset {\text{Piv}}_{\widetilde{D}_y}(\mathcal {V}_{k + 1})$ because of \eqref{eq:locally_identical} 
and by monotonicity of $K \mapsto {\text{Piv}}_K $, using the fact that  
$T\cup\pi  \subset C_y \subset \widetilde{D}_y \, (= 
B(y, \the\numexpr\couprad+6*\rangeofdep\relax L))$. By definition of pivotality and 
\eqref{eq:VTRANGE}, we have that 
${\text{Piv}}_{\widetilde{D}_y}(\mathcal {V}_{k + 1})$ is independent of $\widehat{\omega}_{B(y,\the\numexpr\couprad+3*\rangeofdep\relax L)}$. On the other hand, the event 
$G_{m}$ 
is measurable relative to 
$\sigma(\Phi_T^{\textnormal{loc}}({\widehat{\omega}})) \subset 
\sigma(\widehat{\omega}_{B(y,\the\numexpr\couprad+3*\rangeofdep\relax L)})$ by Lemma~\ref{lem:caio_T} and the observation immediately preceding it. Therefore, considering the quantity $\P_{\tilde{\sigma}, \mathbb B}[\overline \piv_{ T\cup\pi }(\mathcal {V}_{T}), 
E_{j_-}, E_{j,m-1}, G_{m}^c]$, foregoing the events $E_{j_-}, E_{j,m-1}$ and the localization on $\mathbb{B}$, then using independence, one obtains that this probability is bounded by $\P_{\tilde 
\sigma}\big [G_{m}^c] \cdot 
\P_{\tilde \sigma}\big [\, {\text{Piv}}_{\widetilde{D}_y}(\mathcal {V}_{k + 1}) \, \big]$, and \eqref{eq:lower_bnd_caio2} follows.

Now, integrating \eqref{eq:conditional_connect}
on the event in \eqref{e:bridge1-meas} and combining with \eqref{eq:lower_bnd_caio2} leads to
\begin{multline*}
\P_{\tilde{\sigma}, \mathbb B}\big[\,\overline \piv_{ T\cup\pi }(\mathcal {V}_{T}), E_{j_-}, E_{j,m}\big]\\ \ge e^{-C(\log L)^{2}}  \P_{\tilde{\sigma}, \mathbb B}\big[\,\overline \piv_{ T\cup\pi }(\mathcal {V}_{T}), E_{j_-}, E_{j,m-1}\big] - \P_{\tilde \sigma}\big [G_{m}^c] \cdot 
	\P_{\tilde \sigma}\big [\, {\text{Piv}}_{\widetilde{D}_y}(\mathcal {V}_{k + 1}) \, \big].
	\end{multline*}
Iterating this inequality over all $m \le n = |\mathbb B_{j}|$, recalling that $E_{j,n}=E_j$ whence $E_{j_-} \cap E_{j,n}= E_{(j+1)_-}$ in view of  \eqref{eq:bridge1-conn}, we obtain that
\begin{multline}
	\label{eq:conditional_connect_onelevel}
\P_{\tilde{\sigma}, \mathbb B}\big[\, \overline \piv_{ T\cup\pi }(\mathcal {V}_{T}),  E_{(j+1)_-}\big] \\
\ge e^{-C |\mathbb B_{j}|(\log L)^{2}} \, \P_{\tilde{\sigma}, \mathbb B}\big[ \, \overline \piv_{ T\cup\pi }(\mathcal {V}_{T}), 
 E_{j_-}\big] -  \P_{\tilde \sigma}\big [\, {\text{Piv}}_{\widetilde{D}_y}(\mathcal {V}_{k + 1}) \, \big] \sum_{B \in \mathbb B_{j} }\P_{\tilde \sigma}[G_B^c].
\end{multline}
To deduce \eqref{eq:lower_bnd_caio} from \eqref{eq:conditional_connect_onelevel}, it thus remains to suitably bound the exponential pre-factor and the sum over $B \in \mathbb B_{j}$. We start with the former.
Using \cite[(B.4)]{RI-II} along with \eqref{eq:B1B3B4}, it follows that $|\mathbb B_{j}| 
\leq (\log L)^{C(d) + \gamma_2 + \bar \gamma_2}$, yielding that $e^{-C|\mathbb B_{j}|(\log L)^{2}} \ge e^{-C(\log L)^{2 \bar \gamma_{2}}}$ since $\bar \gamma_2 \ge 3\gamma_2 \ge C(d)$ by assumption, see \eqref{eq:gamma-cond-lem-surgery-2} and \eqref{eq:gamma2bar}. 

On the other hand, since $\tilde \sigma \in F_y$, the bound \eqref{eq:caioGbound} from 
Lemma~\ref{lem:caio_T} applies to the box $B \in \mathbb B_{j}$, the radius of which is bounded from below by $s'=  (\log L)^{\gamma_3}$ due to \cite[(B.3)]{RI-II} and \eqref{eq:B1B3B4}. This gives that 
$\P_{\tilde \sigma}[G_B^c] \le e^{-c(\log L)^{ {\Cr{c:box_gap}}\gamma_{3}}}$. Combining this with the previously obtained bound on $|\mathbb B_{j}|$ and 
using the second assumption in \eqref{eq:gamma-cond-lem-surgery-2}, which implies in particular that ${\Cr{c:box_gap}}\gamma_{3} > 2 \bar \gamma_{2}$ (and more), it follows that $$\sum_{B \in \mathbb B_{j}} \P_{\tilde \sigma}[G_{B}^c] \leq  e^{-c(\log L)^{{\Cr{c:box_gap}}\gamma_{3}}}$$
 for $L \geq C$. Plugging this along with the afore lower bound on $e^{-C|\mathbb B_{j}|(\log L)^{2}} $ back into \eqref{eq:conditional_connect_onelevel} then leads to \eqref{eq:lower_bnd_caio}, using now the full strength of \eqref{eq:gamma-cond-lem-surgery-2}, which implies that
 ${\Cr{c:box_gap}}\gamma_{3} \geq 5 \bar \gamma_{2}$.
\end{proof}

\subsection{Holes at small scales} \label{surgery-3}
We start with a brief summary of where things stand. With our final goal of showing \eqref{eq:reduce_distance1} in mind, the combination of \eqref{eq:finalbd2} (outcome of \S\ref{subsec_piv_V_T}) and Lemma~\ref{L:renormalization} (proved in \S\ref{surgery-2}) yields an estimate for $q_m$ in terms of the quantity $\mathbb{P}_{y}^{\varepsilon}[\,\overline \piv
_{H \cup \pi}( \mathcal{V}_{k,m-\frac12}) ]$ appearing on the right of \eqref{eq:finalbd3}.
The final step is to convert this quantity into $q_{m-1}$ defined in \eqref{eq:q_m-def}, which compared to $q_m$, has a slightly bigger vacant configuration, see the last line of \eqref{eq:intermediate_config_inclusion}, and importantly, a reduced cluster separation $R_{T,m-1}$, see \eqref{def:lowest_scale}. This will be achieved by removing the holes $H$ from the pivotal region $H \cup \pi$ as $\mathcal{V}_{k,m-\frac12}$ is transformed into $\mathcal{V}_{k,m-1}$. Indeed the resulting pivotal region, corresponding to the range of $\pi$, entails a cluster separation between $\mathscr{C}_U( \mathcal V_{k, m-1})$ and $\mathscr{C}_V( \mathcal V_{k, m-1})$ bounded by $\text{diam}(\pi) \leq R_{T,m-1}$ (recall to this effect from below \eqref{eq:pixy} that $\pi=\pi_{w,w'}$ is an $\ell^1$-geodesic between the two points $w,w'$, which are separated by a distance at most $R_{m-1}$). The `removal' of $H$ to obtain $q_{m-1}$ is the object of the next lemma.

Let $N_B(\tilde{\omega})$ (under $\mathbb{P}_{\tilde{\sigma}}$, cf.~below \eqref{eq:sigma} regarding $\tilde{\omega}$) denote the number of trajectories underlying $\mathcal J_{k, m-1/2}$ that intersect $B \subset \Z^d$; more precisely, $N_B$ is the sum of four independent Poisson variables $N_B^i$, one for each $\omega_i$, $1\leq i \leq 4$. With a view towards  \eqref{eq:intermediate_cofig} and \eqref{eq:tildeI_k} (and similarly for \eqref{eq:barI_k}), $N_B^1$ counts for instance the number of points $(v,w) \in \omega_1$ such that $v \leq \frac{u4d}{L}\tilde{g}_{k+1}(w(0))$ and $w[0,L-1] \cap B \neq \emptyset$. We will be interested in the events
\begin{equation}
\label{eq:H_m}
H_{y,m}= \big\{ N_B \leq (\log L)^{8d\gamma_{3}} : B= B(x,r) \subset B(y, 
 	\the\numexpr\couprad+1*\rangeofdep +20\relax L)  \text{ for some $x \in \Z^d$ and $r \leq s$}\big\}
\end{equation}
(recall from \eqref{eq:B1B3B4} that $s=(200)^4(\log L)^{4\gamma_{3}}$).
\begin{lemma}[Removing holes]\label{lem:coarsepiv2piv} 
Under \eqref{eq:gamma-cond-lem-surgery-2}, for all integer $1 \le m \le m_0$, $y \in \mathbb{L}$ and dyadic $L \geq C(\Gamma)$, one has (see \eqref{eq:F_y-superhard} regarding $F_y$)
\begin{equation}
\label{eq:coarsepiv2piv}
\sup_{H \in \mathcal H} \mathbb{P}_{y}^{\varepsilon}[ \, \overline \piv
_{H \cup \pi}( \mathcal{V}_{k,m-\frac12}), \, H_{y,m}, F_y]  \le e^{(\log L)^{25d \gamma_{3}}} \big(q_{m-1} + (\pi_{y}^{\varepsilon})^{-1}b\big).\end{equation}
\end{lemma}

The proof of Lemma~\ref{lem:coarsepiv2piv} makes use of the following simple result. 
\begin{lemma}
	\label{lem:mvmp}
	Let $(\mathcal S, 2^{\mathcal S}, P)$ be a probability space with $\mathcal S$ at most countable and $P(s) > 0$ for all $s \in \mathcal S$. Then for all $E, E' \subset \mathcal S$ and $\mathcal R \subset E \times E'$ one has that
	\begin{equation}\label{eq:mvmp}
P[E] \leq  P[E'] \cdot \sup_{s'\in E'}\sum_{s \in \mathcal R^{-1}(s')}\, \frac{P(s)}{P[\mathcal R(s)]},
	\end{equation}
where $\mathcal R^{-1}(s') = \{s \in E: (s,s') \in \mathcal R  \}$ and $\mathcal R(s) = \{s' \in E': (s,s') \in \mathcal R  \}$.
\end{lemma}
\begin{proof}
Referring to the supremum in question as $\beta$, one observes that
$$
P[E] = \sum_{s \in E} P(s) = \sum_{s \in E}\, \frac{P(s)}{P[\mathcal R(s)]} \sum_{s' \in \mathcal R(s)} P(s') = \sum_{s' \in E'} P(s') \sum_{s \in \mathcal R^{-1}(s')}\, \frac{P(s)}{P[\mathcal R(s)]} \le \beta P[E'].
$$
\end{proof}
Let us now move to:

\begin{proof}[Proof of Lemma~\ref{lem:coarsepiv2piv}]
In view of the event appearing on the left hand side of \eqref{eq:coarsepiv2piv}, we introduce, for $H \in \mathcal{H}$,
 \begin{equation}\label{events-mvmp1}
 E= \overline \piv
_{H \cup \pi}( \mathcal{V}_{k,m-\frac12}) \cap H_{y,m}, \qquad E' = \{\lr{}{ {\mathcal V}_{k, m-1} \cup \pi}{U}{V}, \nlr{}{ {\mathcal V}_{k, m-\frac 12}}{U}{V}\}.\end{equation}
We will prove by application of Lemma~\ref{lem:mvmp} that on the event $F_y$ given by \eqref{eq:F_y-superhard}
\begin{equation}\label{eq:coarsepiv2piv_modified}
\mathbb{P}_{\tilde{\sigma}}[ E ]  \le e^{(\log L)^{25d \gamma_{3}}} \mathbb{P}_{\tilde{\sigma}}[ E' ].
\end{equation}
To see this implies \eqref{eq:coarsepiv2piv}, observe that by partitioning the event $E'$ according to whether $U$ and $V$ are connected in ${\mathcal V}_{k, m-1}$ or not, one obtains that
\begin{equation*}
\P_{y}^{\varepsilon}[E'] \le \P_{y}^{\varepsilon}[\lr{}{ {\mathcal V}_{k, m-1}}{U}{V}, \nlr{}{ {\mathcal V}_{k, m-\frac 12}}{U}{V}] + \P_{y}^{\varepsilon}\big[\, \overline \piv_{\pi}(\mathcal V_{k, m-1}) \, \big] \le (\pi_{y}^{\varepsilon})^{-1} b + q_{m-1}
\end{equation*}
where in the last step we used the fact that $\mathcal V_{k, m-1} \subset \mathcal V_{k + \frac12}$ and  $\mathcal V_{k + 1} \subset \mathcal V_{k, m- \frac 12}$, 
cf.~\eqref{eq:intermediate_config_inclusion}), which leads to the factor $b$ upon recalling \eqref{eq:def_b}, and the event defining $q_{m-1}$ in \eqref{eq:q_m-def} is implied by $\overline \piv_{\pi}(\mathcal V_{k, m-1})$ because $\text{diam}(\pi)\leq R_{m-1}$ and $\text{range}(\pi) \subset B(y,20L)$ by construction, see below \eqref{eq:pixy}, which implies a similar bound for the separation between $\mathscr{C}_U(\mathcal V_{k, m-1})$ and $\mathscr{C}_V(\mathcal V_{k, m-1})$ as measured by $d_y$. Integrating  \eqref{eq:coarsepiv2piv_modified} suitably over $\tilde{\sigma}$ on the event $F_y$ and combining the resulting estimate with the last display readily gives \eqref{eq:coarsepiv2piv}.

We now show \eqref{eq:coarsepiv2piv_modified}, and start by preparing the ground in order to fit the discrete setup of the previous lemma. The inequality \eqref{eq:coarsepiv2piv}
we aim to prove solely deals with the restriction to $B_R$ (see e.g.~the statement of Proposition~\ref{prop:comparisonLk}; this is the box in which the relevant connection event is occuring) of the two configurations $\mathcal{I}_{k,m-\frac12}$ and $\mathcal{I}_{k,m-1} \subset \mathcal{I}_{k,m-\frac12}$, see \eqref{eq:intermediate_config_inclusion}, for some $1\leq m \leq m_0$. On account of \eqref{eq:intermediate_cofig} and \eqref{def:Ikm}, one has the decomposition
\begin{equation}
\label{eq:holes101}
 \mathcal{J}_{k,m-\frac12} =  \mathcal{J}^a \cup \mathcal{J}^b,
\end{equation}
where $ \mathcal{J}^a = \mathcal{J}_{k,m-1}$ and
$$
\mathcal{J}^b = 
\begin{cases}
 {\mathcal{J}}^{[ us_{k} , us_{k}' + \frac{ 1}{2m_0}us' ],\, {L^{\ast}}} ( \omega_4), & \text{when $m = 1$},\\
  {\mathcal{J}}^{[ us_{k}' + \frac{m- 1}{m_0}us' , us_{k}' + \frac{2m- 1}{2m_0}us' ],\, {L^{\ast}}} ( \omega_4), & \text{when $m > 1$}.
\end{cases}
$$
 Let $W_{R,L}^+$ denote the set of all finite-length (forward) trajectories in $\mathbb{Z}^d$ starting from $ B_{R+ 100L}$, and having time-length at most $2L$; here, as in \eqref{eq:Wdef}, a trajectory may at each step either jump to a neighbor or stay put.
Now define $\mathcal{S}= \Xi^2$, where
\begin{equation*}
\Xi = \big\{\eta  = \textstyle \sum_{1\leq i \leq n} \delta_{ w_i}:  n \in \N^*, \,  w_i \in W_{R, L}^+\, \forall \, 1 \le i \le n \big\}, 
\end{equation*}
which is the set of all finite point measures of finite-length trajectories originating inside $B_{R+ 100 L}$. We use $(\eta^a, \eta^b)$ to denote canonical coordinates on $\mathcal{S}$ in the sequel. We now consider the two $\Xi$-valued random variables under $\P_{\tilde{\sigma}}$, cf.~below \eqref{e:disorder-mixed}, obtained by retaining all points $(v,w) \in \mathbb{R}_+ \times W^+$ in the support of $\omega_i$, $1 \leq i \leq 4$, underlying $ \mathcal{J}^a \cap B_R$ and $ \mathcal{J}^b \cap B_R$, respectively, trimming them to their relevant time-length (either $L^*$ or $3L-L^*$; we explain this in the next sentence) while forgetting their label $v$. 
For instance, in the case of $ \mathcal{J}^b \cap B_R$, the only relevant process is $\omega_4$ (unlike $ \mathcal{J}^a \cap B_R$, cf.~\eqref{eq:tildeI_k} and~\eqref{eq:barI_k}) and the underlying points $(v,w) \in \omega_4$ are precisely those for which $w[0,L^*-1]\cap B_R \neq \emptyset$ and $v \cdot\frac{L^{\ast}}{4d} \in [ us_{k} , us_{k}' + \frac{ 1}{2m_0}us' ]$ (when $m=0$) or $v \cdot \frac{L^{\ast}}{4d}\in [ us_{k}' + \frac{m- 1}{m_0}us' , us_{k}' + \frac{2m- 1}{2m_0}us' ]$ (when $m\geq1$). The relevant trimming here is $L^*$, as can be seen from the definition of $\mathcal{J}^b$. Note also that the two random variables just defined indeed have values in $\Xi$, i.e.~the relevant trajectories $w$ always belong to $W_{R,L}^+$. We then set $P= P_{\tilde{\sigma}}$ to be the induced law on $\mathcal{S}$ of the above $\Xi$-valued random variables under $\P_{\tilde{\sigma}}$, by which $(\eta^a, \eta^b)$ are distributed as two independent Poisson processes on $W_{R,L}^+$ of a certain intensity. We will soon apply Lemma~\ref{lem:mvmp} with the choice of measure space $(\mathcal{S}, 2^{\mathcal{S}}, P)$.
Note to this effect that $\mathcal{S}$ is a countable set and that $P$ has full support.

Next, we observe from \eqref{eq:holes101} that the restriction of the two vacant sets $ \mathcal V_{k, m-1} $ and $ \mathcal V_{k, m- \frac 12} $ to $B_R$ and with them the events $E$ and $E'$ in \eqref{events-mvmp1} are all measurable with respect to the above two random variables; to reach this conclusion in the case of $E$ one notes in addition that the same measurability claim holds for the event $H_{y,m}$ introduced in \eqref{eq:H_m}: indeed the relevant variable $N_B$ counts precisely the number of trajectories in the support of $\eta^a + \eta^b$ intersecting $B$. In particular, setting $\mathcal{J}(\eta)= \bigcup_{w \in \eta} \text{range}(w)$ for $\eta \in \Xi$ and defining $\mathcal{V}(\eta)$ in terms of $\mathcal{J}(\eta)$ in exactly the same manner as $\mathcal{V}_{k,t}$ in terms of $\mathcal{J}_{k,t}$ in \eqref{def:Ikm}, one sees that
$$
\big(\mathcal V_{k, m-1} \cap B_R , \mathcal V_{k, m-\frac12} \cap B_R \big) \stackrel{\text{law}}{=} \big(\mathcal V(\eta^a) \cap B_R, \mathcal V(\eta^a +\eta^b) \cap B_R \big).
$$
 Thus, with a slight abuse of notation, the events $E=E(\eta^a,\eta^b)$, $E'=E'(\eta^a,\eta^b)$, are all naturally declared under $P$ and have the same joint law as their respective counterparts under $\P_{\tilde{\sigma}}$. For example, $E'(\eta^a,\eta^b)$ is simply obtained by replacing $\mathcal V_{k, m-1}$ by $\mathcal V(\eta^a)$ and $\mathcal V_{k, m-\frac12}$ by $\mathcal V(\eta^a + \eta^b)$ in~\eqref{events-mvmp1}. We emphasize that the sets $\mathcal V(\cdot)$ are  obtained from $\mathcal J(\cdot)$ by action of the noise operator as in \eqref{def:Ikm}, which is deterministic under $P(= P_{\tilde{\sigma}})$. Thus, on the event $F_y$, henceforth always tacitly assumed (as for \eqref{eq:coarsepiv2piv_modified}), due to the inactiveness of the noise inherent to \eqref{eq:F_y-superhard} and on account of \eqref{eq:noise1}-\eqref{eq:noise2}, we have that
\begin{equation}\label{eq:noise102}
\mathcal V(\cdot) \cap B(y,90L) = \big(\Z^d \setminus \mathcal{J}(\cdot) \big) \cap B(y,90L).
\end{equation}

We now introduce the relation $\mathcal{R}$ that will be relevant to our application of Lemma~\ref{lem:mvmp}.
To this end, we first define a sequence of trajectories $\pi_1, \pi_2, \ldots$ and a random time $\tau$ that will play a prominent role. Let $ (\eta^a, \eta^b) \in E$. We attach to every point in the support of $\eta^a$ and $\eta^b$ an independent label sampled uniformly in $(0,1)$. Let $\pi_1, \pi_2, \ldots$ 
denote the trajectories (elements of $W_{R,L}^+$) in the support of $ \eta^a $ intersecting $H$ -- recall that $H \subset B(T,s) \subset T' \subset 
B(y,60L)$, see \cite[Section 4]{RI-II} and \eqref{def:cylinder}, \eqref{eq:tube-cond} -- ordered according to increasing label. Since $ \overline \piv_{H \cup \pi}(\mathcal V(\eta^a + \eta^b))$ occurs on $E$, \eqref{eq:noise102} and the property~\cite[(B.4)]{RI-II} (see~\eqref{eq:B1B3B4}) of bridges imply that there exists an integer $ 0 \leq \tau < \infty$ such that
\begin{equation}\label{eq:holes-def-tau}
\begin{split}
&\text{$U$, 
$V$ are conn.~in $\mathcal V\big( (\eta^a \setminus \{\pi_{n}: n \leq \tau \}) \big) \cup \pi$, but not in $\mathcal V\big( (\eta^a \setminus \{\pi_{n}: n < \tau \}) \big) \cup \pi$};
\end{split}
\end{equation}
here, with hopefully obvious notation, $\eta^a \setminus \{\pi_{n}: n \leq N \}$ refers to the element of $\Xi$ obtained from $\eta^a$ by removing the points corresponding to $\pi_{n}$, $n \leq N$, from its support. The second requirement in \eqref{eq:holes-def-tau} is a minimality condition on $\tau$. If $\tau =0$ this simply means that $U$ and $V$ are already connected in $\mathcal V(\eta^a) \cup \pi$. Now, owing to the event $H_{y,m}$ inherent to $E$, see \eqref{events-mvmp1}, we obtain from \eqref{eq:H_m}, which yields an upper bound on the number of walks in the support of $\eta^a$ hitting any box in $\mathbb B_{J}$, and the fact that $H$ is composed of at most $|\mathbb B_{J}| 
\leq (\log L)^{C(d) + \gamma_2 + \bar \gamma_2}$ many such boxes, which follows from \cite[(B.4)]{RI-II} along with \eqref{eq:B1B3B4}, that
\begin{equation}\label{eq:tau-bound-holes}
\tau \le (\log L)^{\gamma_{\tau}}, \quad \gamma_{\tau} = (8d + 1)\gamma_{3},
\end{equation}
where we used that $\gamma_3 \ge 5 \bar \gamma_2$ and $\bar \gamma_2 \ge C\gamma_2 (\ge C')$.

We can now define $\mathcal R \subset E \times \mathcal S$, i.e.~for each $(\eta^a, \eta^b) \in E$ we specify the configurations $(\psi^a, \psi^b ) \in \mathcal S$ constituting $\mathcal{R}(\eta^a, \eta^b)$. We will verify a posteriori that $\mathcal{R}(\eta^a, \eta^b) \subset E'$. For a given $(\eta^a, \eta^b) \in E$ we set $\psi^a= \eta^a 
\setminus \{\pi_{n}; n \leq \tau\}$ with $\tau$ as defined by \eqref{eq:holes-def-tau}. A point measure $\psi^b$ is defined as follows. If $\tau=0$ then $\psi^b=\eta^b$. If $\tau \geq 1$ and $|\pi_\tau| = 
L^*$, we simply set $\psi^b = \eta^b + \delta_{\pi_\tau}$. Otherwise, we distinguish two cases. First, if $\mathcal 
V_{\cdot} = \widetilde{\mathcal V}_{\cdot}^{u, L}$, then $|\pi_\tau| = 3L-L^*= L$ on account of \eqref{eq:defL_*}, and we let $\psi^b$ be any point measure of the form $\psi^b = 
\eta^b + \delta_{\pi_\tau \cup \pi_\tau'}$ where $\pi_\tau \cup \pi_\tau'$ is the concatenation of $\pi_{\tau}$ with any nearest-neighbor path of 
length $L$ starting at $\pi_\tau(L-1)$. Secondly, if $\mathcal V_{\cdot} = \overline{\mathcal 
V}_{\cdot}^{u, L}$, whence $|\pi_\tau| = 2L$, we set $\psi^b = \eta^b +\delta_{\pi_\tau[0, 
L-1]} + \delta_{\pi_\tau[L, 2L-1]}$. The set of all $(\psi^a,\psi^b) \in \mathcal{S}$ thereby obtained defines  $\mathcal{R}(\eta^a, \eta^b)$. This fully specifies $\mathcal{R}$.

We now claim $\mathcal R \subset E \times E'$, i.e.~for a given $ (\eta^a, \eta^b) \in E$, any pair $(\psi^a, \psi^b)$ constructed by the above procedure satisfies $ (\psi^a, \psi^b) \in E'$. Indeed, by definition of $\psi^a$ and the first condition in \eqref{eq:holes-def-tau}, we know that $\{\lr{}{}{U}{V} \text{ in } {\mathcal V} (\psi^a) \cup \pi\}.$ On the other hand, the fact that $ (\eta^a, \eta^b) \in E$ implies that
 $ \{{{\nlr{}{}{U}{V}}} \text{ in }  {\mathcal V}(\eta^a + \eta^b)\}$, 
and owing to the second property in \eqref{eq:holes-def-tau} and the definition of $\psi^b$, which retains $\pi_{\tau}$, it follows that the disconnection persists in $ {\mathcal V}(\psi^a + \psi^b)$. Overall, in view of \eqref{events-mvmp1}, we obtain  that $(\psi^a, \psi^b) \in E'$, hence 
$\psi: E \to E'$, as desired.

\medskip
Now, Lemma~\ref{lem:mvmp} applies
and \eqref{eq:mvmp} yields (on the event $F_y$), 
\begin{equation}
\label{eq:mvmp_coarsepiv2piv}
 \P_{\tilde{\sigma}}[E] \leq \P_{\tilde{\sigma}} [E'] \,\cdot    
\max_{\psi= (\psi^a, \psi^b) \in E'} \, \sum 
_{\eta = (\eta^a, \eta^b) \in \mathcal R^{-1}(\psi)}\,\frac{P(\eta) }{P[\mathcal R(\eta)]}\,.
\end{equation}
We now bound the maximum on the right-hand side of \eqref{eq:mvmp_coarsepiv2piv} and focus on the case $\mathcal V_{\cdot} = \overline{\mathcal V}_{\cdot}^{u, L}$. The other 
case is dealt with similarly. In the sequel we write $p( \pi)$ 
for the probability attached to any path in $\pi \in W_{R,L}^+$ under $P_{\pi(0)}$. Let $\psi\in E'$. We first bound the ratio appearing in \eqref{eq:mvmp_coarsepiv2piv} and then perform the sum over $\eta$ separately. As we now explain,  we get with $\varepsilon$ as in 
\eqref{def:epsx}, for any $\eta$ as appearing in the sum in \eqref{eq:mvmp_coarsepiv2piv} with $\tau=\tau(\eta)\neq 0$ that
	\begin{align}\label{eq:ratio_bnd_mvmp}
		\frac{P(\eta) }{P[\mathcal R(\eta^a, \eta^b)]}  &\leq \frac {\displaystyle \prod_{1 \leq n \leq \tau}p(\pi_n) \,\,  \prod_{\tilde \pi \in (\eta^a + \eta^b) \setminus \{\pi_n: 1 \leq n \leq \tau \}} p(\tilde \pi)}{\displaystyle p(\pi_\tau)  \,\,\,\,\,\,\:\:\:\:\:\: \prod_{\tilde \pi \in (\eta^a + \eta^b) \setminus \{\pi_n: 1 \leq n \leq \tau \} } p(\tilde \pi)}\times\frac{CL^2(\log L)^{2\gamma_\tau}}{\varepsilon^2}\nonumber \\
		&\leq CL^2 (\log L)^{18d\gamma_{3}}\prod_{1 \leq n < \tau}p(\pi_n) , \quad \text{since $\gamma_3 \ge 5 \bar \gamma_2$, $\bar \gamma_2 \ge 3 \gamma_2$ and $\gamma_2 \ge \gamma_{1}+5$}.
	\end{align}
The second factor in the first line corresponds to a bound on the ratio of the relevant (Poisson) intensities. In obtaining this bound, we first used that the intensities of walks starting at 
	any $x' \in \Z^d$ for $\mathcal J(\eta^a)$ are bounded by $1$ for large enough $L$ 
	so that removal of trajectories leading to $\psi^a$ comes at no multiplicative cost. 
	We then used that the intensity $\lambda_{x'}$ of walks for $\mathcal{J}(\eta^b)$ starting at any $x'
	\in B = B(y, 60L)$ satisfies $\lambda_{x'}
	 \geq \varepsilon / 2\Cr{subdivide} \lfloor \log L \rfloor$ (recall to this effect that $\mathcal{J}(\eta^b) \cap B \stackrel{\text{law}}{=} \mathcal{J}^b \cap B$ and see below \eqref{eq:holes101}; then use \eqref{def:lowest_scale} and recall that $s'=\varepsilon$ on the event $F_y$). The bound on $\lambda_{x'}$ implies that adding at most two trajectories to produce $\psi^b$ costs no more 
	than $C L^2 (\log L)^2 N_B^2 / \varepsilon^2$, where $N_B$ is the maximum number of walks in the support of $\eta^b$ 
	intersecting $B$. However, since $\eta \in E$, we have $N_B \le (\log L)^{8d \gamma_{3}}$. Together, the previous observations lead to the first step in \eqref{eq:ratio_bnd_mvmp}. The second one is immediate by~\eqref{def:epsx}. We also note that \eqref{eq:ratio_bnd_mvmp} remains valid if $\tau(\eta)=0$ since by definition $\mathcal{R}$ acts trivially and the ratio is equal to $1$ in this case.
	
	Next, we perform the sum over $\eta$ in \eqref{eq:mvmp_coarsepiv2piv}. By \eqref{eq:tau-bound-holes}
we have the bound, valid for any $\psi \in E'$, 
	\begin{align}\label{eq:mvmp_reconstruction}
		&\sum _{\eta \in \mathcal R^{-1}(\psi)}    \,\prod_{1 \leq n < \tau}p(\pi_n) \leq \sum_{t =1}^{(\log L)^{\gamma_{\tau}}} \sum_{\substack{(w_1 \cup w_2)  \in \psi^b: \\ \, (w_1 \cup w_2) \cap H \neq \emptyset,\\|w_1| = |w_2| = L}} \sum_{(x_n',s_n, \ell_n)_n } \sum_{*} \prod_{1 \leq n < t}p(\pi_n),
	\end{align}
	where the triplets $(x_n',s_n, \ell_n)$, $1\leq n < t$, range over all $x_n' \in  H$, $\ell_n \in \{L, 2L\}$, and 
	$s_n \in \{0,1,\dots, 2L-1\}$. The summation $*$ is over all configurations $\eta \in \mathcal R^{-1}(\psi)$ satisfying $\tau=t$  with $\pi_n$ 
	having length $\ell_n$ and first entering $H$ at time $s_n$ through the point $x_n'$ for all $1\leq n\leq t -1$, and with $\pi_\tau$ arising by concatenating $w_1$ and $w_2$ ($\pi_\tau = w_1$ if $w_1 = w_2$). Crucially, \begin{equation*}
		\sum_{*} \prod_{1 \leq n < t}p(\pi_n) \leq 1.
	\end{equation*}
Then one bounds each of the two preceding sums in \eqref{eq:mvmp_reconstruction} individually, by noting that
\begin{equation*}
\sum_{(x_n',s_n, \ell_n)_n} 1 \leq  (C(\log L)^{d\gamma_{\tau}} L^2)^{(\log L)^{\gamma_{\tau} }}, \,\mbox{ and }\sum_{\substack{(w_1 \cup w_2)  \in \psi^b: \\ \, (w_1 \cup w_2) \cap H \neq \emptyset,\\|w_1| = |w_2| = L}}  1 \leq( (\log L)^{8d \gamma_3})^2,
\end{equation*}
where, we used again the fact that $\eta \in E $, which implies in particular a bound on the number of relevant trajectories. Plugging the two bounds into the right hand side of \eqref{eq:mvmp_reconstruction} yields that
\begin{equation}\label{eq:holes110}
		\sum _{\eta \in \mathcal R^{-1}(\psi)}  \:\:  \prod_{1 \leq n < \tau}p(\pi_n)  \leq e^{(\log L)^{2 \tilde{\gamma}}}.
\end{equation}
Summing  \eqref{eq:ratio_bnd_mvmp} over $\eta$, substituting  \eqref{eq:holes110}, and combining with  \eqref{eq:mvmp_coarsepiv2piv} yields \eqref{eq:coarsepiv2piv_modified}.
\end{proof}

\subsection{Proof of Lemma~\ref{lem:reduce_distance}} \label{sbusec:denouement}

Combining
\eqref{eq:pixy}, \eqref{eq:finalbd2} with Lemmas~\ref{L:renormalization} and~\ref{lem:coarsepiv2piv}, we supply the last missing piece.

\begin{proof}[Proof of Lemma~\ref{lem:reduce_distance}] Assume that \eqref{eq:gamma-cond-lem-surgery-2} is satisfied, as required for the conclusions of Lemmas~\ref{L:renormalization} and~\ref{lem:coarsepiv2piv} to hold. 
Let $\lambda_B$ denote the mean of the Poisson variable $N _B$ (under $\P_{\tilde{\sigma}}$) introduced above \eqref{eq:H_m}. By definition of $\mathcal J_{k, m}$ and on the event $F_y$, which allows to bound the relevant intensity profiles in~\eqref{eq:intermediate_cofig} by $C (\log L)^{\gamma_3}$, one has that $P^{\tilde{\sigma}}$-a.s., for any $B$ as in \eqref{eq:H_m},
\begin{equation}\label{eq:lambda_B-holesbd}
\lambda_B \le \frac{C u}{L} (\log L)^{\gamma_3}\sum_{z \in \Z^d} P_{z}[H_{B} < 2L] \leq C'  (\log L)^{\gamma_3} |\partial B| \leq C' (\log L)^{4d\gamma_3},
\end{equation}
where the penultimate step follows by disintegrating $P_{z}[H_{B} < 2L]$ over the discrete time $n < 2L$ at which $B$ is hit, the position $y=X_n (\in \partial B)$, using reversibility, which allows to sum over $z$, and bounding the resulting probabilities $P_y[\widetilde{H}_B > n]$ by $1$; the last inequality in \eqref{eq:lambda_B-holesbd} is because ${\rm rad}(B) \le  s=(200)^4(\log 
L)^{4\gamma_3}$. Using \eqref{eq:lambda_B-holesbd} together with the tail estimate \eqref{eq:Poisson_tailbnd} and a union bound over the balls $B$ in \eqref{eq:H_m}, it follows that
\begin{equation}\label{eq:rk_bnd}
\begin{split}
\P_{y}^{\varepsilon}[ (H_{y,m} \cap F_y)^c ] & \leq \P_{y}^{\varepsilon}[ F_y^c ] +  \frac1{\pi_y^{\varepsilon}} \sum_{B} P^{\tilde{\sigma}} \big[\P_{\tilde{\sigma}} [ N_B \geq (\log L)^{8d\gamma_3}] 1_{F_y} \big] \leq e^{-c(\log L)^{\gamma \wedge d\gamma_3}} 
\end{split}
\end{equation}
for $L \geq C(\Gamma)$, where the last step implicitly used the bound $\P_{y}^{\varepsilon}[ F_y^c ] \leq e^{-c(\log L)^{\gamma}} $ already at work in \eqref{eq:barpivTzw-bis}. Recalling that $D_y = 
B(y,\the\numexpr\couprad+4*\rangeofdep\relax L)$ and $\widetilde{D}_y =
B(y,\the\numexpr\couprad+6*\rangeofdep\relax L)$, one now observes that the events $H_{y,m}$, $F_y$ and $\{s_{| C_y}= \varepsilon\}$ resp.~given by \eqref{eq:H_m}, \eqref{eq:F_y-superhard} and \eqref{def:Px} are all measurable relative to 
$(\omega_{D_y}^L, \Sigma_{D_y}^L, \mathsf U_{D_y})$ whereas the event ${\textnormal{Piv}}_{\widetilde{D}_y}(\mathcal {V}_{k, m - \frac12})$ is independent (under $\P$) of this triplet. Noting that $(H\cup \pi) \subset \widetilde{D}_y$, which follows from \eqref{eq:tube-cond}, the fact that $H \subset B(T,s)$ and by construction of $\pi$ (cf.~\eqref{eq:pixy} and below), one further has the chain of inclusions $\overline{\piv}_{H \cup \pi}( \mathcal{V}_{k,m-\frac12}) \subset \piv_{H \cup \pi}( \mathcal{V}_{k,m-\frac12}) \subset \textnormal{Piv}_{\widetilde{D}_y}( \mathcal{V}_{k,m-\frac12})=  \textnormal{Piv}_{\widetilde{D}_y}( \mathcal{V}_{k+1})$, where the equality is due to the second line of \eqref{eq:intermediate_config_inclusion}. Combining with the previous observation thus yields that
\begin{multline}\label{eq:surgery-finalbit}
\sup_{H \in \mathcal H}\mathbb{P}_{y}[\overline \piv_{H \cup \pi}( \mathcal{V}_{k,m-\frac12}), \, 
(H_{y,m} \cap F_y)^c ] \\
\leq  \P_{y}^{\varepsilon}[(H_{y,m} \cap F_y)^c]\, \P \big[\, {\textnormal{Piv}}_{\widetilde{D}_y}(\mathcal {V}_{k+1})\,\big] \stackrel{\eqref{eq:compa_a_state.5.2}, \eqref{eq:rk_bnd}}{\le} e^{-c(\log L)^{\gamma \wedge d\gamma_3}} (b + Af(y)).
\end{multline}
for all $L \geq C(\Gamma)$, where we used that $(\gamma \wedge d\gamma_3) > 3\gamma_M$ which follows from \eqref{eq:gamma-cond-lem-surgery-2}. Adding \eqref{eq:surgery-finalbit} to 
\eqref{eq:coarsepiv2piv} and plugging the resulting bound into the right hand side of 
\eqref{eq:finalbd3} gives 
\begin{equation}\label{eq:surgery-final-bridge10}
	\P_{y}^{\varepsilon} \left[\, \overline{\textnormal{Piv}}_{T \cup \pi}(\mathcal {V}_{T})\,\right] \leq e^{(\log L)^{30d\gamma_3}} (q_{m-1} + (\pi_{y}^{\varepsilon})^{-1}\,b ) + e^{-c(\log L)^{\gamma_2}} A f(y)
\end{equation}
for $L \ge C(\Gamma)$, where we used that $( \gamma \wedge \gamma_3) \ge C \gamma_2$ implied by \eqref{eq:gamma-cond-lem-surgery-2}.

 Finally we select the constants in \eqref{eq:gammacond_intermed} as $\Cr{C:c11_inv}(\delta, d) = \frac{5\Cr{C:gamma2}}{{\Cr{c:box_gap}} \wedge 1}$,
 $\Cr{C:old6} = 5\Cr{C:gamma2}$ and $\Cr{C:Cd}(d) = C(d) \vee  40d$ with $C(d)$ as in the statement of Lemma~\ref{L:renormalization}, so that the conditions in  \eqref{eq:gammacond_intermed} imply \eqref{eq:gamma-cond-lem-surgery-2}, and \eqref{eq:surgery-final-bridge10} is in force (recall that the latter was derived under the hypothesis that \eqref{eq:gamma-cond-lem-surgery-2} holds). Plugging \eqref{eq:gamma-cond-lem-surgery-2} into \eqref{eq:finalbd2} and subsequently into \eqref{eq:pixy}, one deduces \eqref{eq:reduce_distance1}. Or, in~more vivid terms, cf.~the beginning of Section~\ref{sec:superdiff}: the arrow has traversed all axe heads.
 \end{proof}

\bigskip
\noindent \textbf{Acknowledgements.} This work was initiated at IH\'ES. It has received funding from the European Research Council 
(ERC) under the European Union's Horizon 2020 research and innovation programme, grant 
agreement No.~757296. The research of FS is currently supported by the same programme, grant agreement No.~851565. HDC acknowledges funding from the NCCR SwissMap, the Swiss FNS, and 
the Simons collaboration on localization of waves. SG’s research is supported by the SERB grant 
SRG/2021/000032, a grant from the Department of Atomic Energy, Government of India, under project 
12–R\&D–TFR–5.01–0500 and in part by a grant from the Infosys Foundation as a member of the 
Infosys-Chandrasekharan virtual center for Random Geometry.  During the duration of this project, AT has been supported by grants
``Projeto Universal'' (406250/2016-2) and ``Produtividade em Pesquisa'' (304437/2018-2) from
CNPq and ``Jovem Cientista do Nosso Estado'', (202.716/2018) from FAPERJ. 
PFR thanks the IMO in Orsay (including support from ERC grant agreement No.~740943) for its hospitality during the final stages of this project.
SG, PFR, FS and AT all thank the University of Geneva for its hospitality on several occasions.

{
\bibliography{biblicomplete}
\bibliographystyle{abbrv}}
\end{document}